\documentclass[12pt]{amsart}
\usepackage{amssymb,amsmath,amsthm}
\usepackage{bbm}
\usepackage{graphicx}
\usepackage{tikz}
\usetikzlibrary{cd,snakes}
\usepackage[pdfencoding=auto,psdextra]{hyperref}
\usepackage{textgreek}
\usepackage{stmaryrd}
\usepackage{blkarray}

\usepackage[hmargin=1in,vmargin=1.25in]{geometry}
\title[Flagged LLT polynomials and nonsymmetric plethysm] {Flagged LLT
polynomials, nonsymmetric plethysm, and nonsymmetric Macdonald polynomials}

\author[Blasiak]{Jonah Blasiak}
\address[Blasiak]{Dept.\ of Mathematics, Drexel University,
Philadelphia, PA}
\email{jblasiak@gmail.com}

\author[Haiman]{Mark Haiman}
\address[Haiman]{Dept.\ of Mathematics, University of California,
Berkeley, CA}
\email{mhaiman@math.berkeley.edu}

\author[Morse]{Jennifer Morse}
\address[Morse]{Dept.\ of Mathematics, University of Virginia,
Charlottesville, VA}
\email{morsej@virginia.edu}

\author[Pun]{Anna Pun}
\address[Pun]{Dept.\ of Mathematics, CUNY-Baruch College, New York,
NY}
\email{anna.pun@baruch.cuny.edu}

\author[Seelinger]{George H. Seelinger}
\address[Seelinger]{Dept.\ of Mathematics, University of Michigan, Ann
Arbor, MI}
\email{ghseeli@umich.edu}

\thanks{Authors were supported by NSF Grants DMS-2154282 and 2452208
(J.B.), DMS-2154281 and 2452209 (J.M.), DMS-2303175 (G.S.) and
DMS-1929284 (all authors, in residence at MathematicsCollaborate@ICERM)}

\newtheorem{thm}{Theorem}[subsection]
\newtheorem{lemma}[thm]{Lemma}
\newtheorem{prop}[thm]{Proposition}
\newtheorem{cor}[thm]{Corollary}
\newtheorem{conj}[thm]{Conjecture}

\theoremstyle{definition}
\newtheorem{defn}[thm]{Definition}

\theoremstyle{remark}
\newtheorem{example}[thm]{Example}

\newtheorem{remark}[thm]{Remark}

\newcommand{\multisymmetric}{(multi\mbox{-})symmetric }

\newcommand{\NN}{{\mathbb N}}
\newcommand{\QQ}{{\mathbb Q}}
\newcommand{\CC}{{\mathbb C}}

\newcommand{\ZZ}{{\mathbb Z}}

\newcommand{\kk}{{\mathbbm k}}

\newcommand{\aA}{{\mathbf a}}
\newcommand{\bb}{{\mathbf b}}

\newcommand{\pp}{{\mathbf p}}
\newcommand{\qq}{{\mathbf q}}
\newcommand{\rr}{{\mathbf r}}
\newcommand{\sS}{{\mathbf s}}
\newcommand{\xx}{{\mathbf x}}

\newcommand{\zz}{{\mathbf z}}
\newcommand{\fh}{{\mathfrak h}}
\newcommand{\tfh}{{\widetilde{\mathfrak h}}}
\newcommand{\sfrak}{{\mathfrak s}}

\newcommand{\lambold}{{\boldsymbol \lambda }}
\newcommand{\mubold}{{\boldsymbol \mu}}
\newcommand{\nubold}{{\boldsymbol \nu }}

\newcommand{\pibold}{{\boldsymbol \pi }}

\newcommand{\Weyl}{\pibold }
\newcommand{\Dem}{\pi }
\newcommand{\DemAt}{\widehat{\pi}}

\newcommand{\Kfrak}{{\mathfrak K}}
\newcommand{\Rfrak}{{\mathfrak R}}
\newcommand{\Sfrak}{{\mathfrak S}}

\newcommand{\Acal}{{\mathcal A}}
\newcommand{\Dcal}{{\mathcal D}}
\newcommand{\Ecal}{{\mathcal E}}

\newcommand{\Gcal}{{\mathcal G}}
\newcommand{\Hcal}{{\mathcal H}}

\newcommand{\Scal}{{\mathcal S}}

\newcommand{\That}{\widehat{T}}

\newcommand{\bhat}{\widehat{b}}
\newcommand{\mhat}{\widehat{m}}
\newcommand{\uhat}{\widehat{u}}
\newcommand{\vhat}{\widehat{v}}

\newcommand{\sigmahat}{\widehat{\sigma }}
\newcommand{\rrhat}{\widehat{\rr }}

\newcommand{\etahat}{\widehat{\eta }}
\newcommand{\zetahat}{\widehat{\zeta }}
\newcommand{\muhat}{\widehat{\mu }}
\newcommand{\muboldhat}{\widehat{\mubold }}

\newcommand{\nuboldhat}{\widehat{\nubold }}

\newcommand{\Acalhat}{\widehat{\Acal  }}

\newcommand{\FLambda}{\mathcal{F\!L}}

\newcommand{\sgngt}{\succ}
\newcommand{\sgnlt}{\prec}
\newcommand{\sgnle}{\preccurlyeq}
\newcommand{\sgnge}{\succcurlyeq}

\newcommand{\baug}{\bullet }
\newcommand{\aug}{\raisebox{.12ex}{\scalebox{.5}{$\baug $}}}

\newcommand{\Pisf}{{\mathsf \Pi }}

\newcommand{\domino}{\begin{tikzpicture}[scale=.25]
\draw (0,0) grid (1,2);
\node at (0.5,1.5) {\scriptsize $u$};
\end{tikzpicture}}

\newcommand{\smalldomino}{\resizebox{1.66ex}{!}{\domino}}

\newcommand{\flagwave}[1]{\vphantom{1}^{\#}\!\mathsf{#1}}

\DeclareMathOperator{\dg}{dg}
\DeclareMathOperator{\inv}{inv}
\newcommand{\oinv}{\inv _{{\rm o}}}
\DeclareMathOperator{\coinv}{coinv}
\DeclareMathOperator{\maj}{maj}
\DeclareMathOperator{\pol}{pol}

\DeclareMathOperator{\GL}{GL}

\DeclareMathOperator{\SSYT}{SSYT}
\DeclareMathOperator{\FSST}{FST}

\DeclareMathOperator{\Des}{Des}
\DeclareMathOperator{\op}{op}
\DeclareMathOperator{\ev}{ev}
\DeclareMathOperator{\Stab}{Stab}

\newcommand{\hsym}{\widehat{\Scal}}    %normalized Hecke symmetrizer

\newcommand{\defeq}{\mathrel{\overset{\text{{\rm def}}}{=}}}

\newcommand{\mybar}[1]{\ensuremath{\overline{\raisebox{0pt}[1.4ex][.3ex]{#1}}}}

\newcommand{\mE}{\mathsf H}   %modified nsmacs  (non q flipped)
\newcommand{\Jns}{J}           %stabilzed int form nsmacs

\newlength{\mylen}
\newcommand{\Xpb}[1]{\tikz[baseline=(s.base)]{\node[inner sep=0pt,
minimum width=\the\mylen] (s) {\(X_{#1}'\)}}^b}

\begin{document}

\begin{abstract}
The plethystic transformation $f[X] \mapsto f[X/(1-t)]$ and LLT
polynomials are central to the theory of symmetric Macdonald
polynomials.  In this work, we introduce and study nonsymmetric
\emph{flagged LLT polynomials}.  We show that these admit both an
algebraic and a combinatorial description, that they Weyl symmetrize
to the usual symmetric LLT polynomials, and we conjecture that they
expand positively in terms of Demazure atoms.  Additionally, we
construct a nonsymmetric plethysm operator $\Pi_{t,x}$ on
$\mathfrak{K}[x_1,\ldots,x_n]$, which serves as an analogue of $f[X]
\mapsto f[X/(1-t)]$.  We prove that $\Pi_{t,x}$ remarkably maps
flagged LLT polynomials defined over a signed alphabet to ones over an
unsigned alphabet.

Our main application of this theory is to formulate a nonsymmetric
version of Macdonald positivity, similar in spirit to conjectures of
Knop and Lapointe, but with several new features.  To do this, we
recast the Haglund--Haiman--Loehr formula for nonsymmetric Macdonald
polynomials $\mathcal{E}_{\mu }(x;q,t)$ as a positive sum of signed flagged
LLT polynomials.  
Then, after applying a suitable stable limit of
$\Pi_{t,x}$ to a stable version of $\mathcal{E}_{\mu }(x;q,t)$,
we obtain modified nonsymmetric
Macdonald polynomials which are positive sums of flagged LLT
polynomials and thus are conjecturally atom positive, strengthening the
Macdonald positivity conjecture.
\end{abstract}
\maketitle

\section{Introduction} \label{s:intro} In the 1980's Macdonald
\cite{Macdonald88} introduced a family of orthogonal polynomials
$P_\mu(X;q,t)$ in the algebra of symmetric functions
$\Lambda_{\QQ(q,t)}(X)$ and conjectured that carefully chosen scalar
multiples $J_\mu = c_\mu P_\mu$ satisfy a remarkable positivity
property.  In terms of plethystically modified versions $H_\mu(X;q,t)
= J_\mu[X/(1-t);q,t]$, Macdonald's conjecture becomes Schur positivity
of $H_\mu$.

Subsequent work revealed that this modified setting is rich in
combinatorics and representation theory.  Early on, Garsia--Haiman
\cite{GarsiaHaiman93} conjectured that $t^{n(\mu)} H_\mu(X;q,t^{-1})$
is the graded Frobenius character of a certain $\Sfrak_n$-submodule of
$\CC[x_{1}, y_{1}, \ldots, x_{n}, y_{n}]$.  This was eventually proven
in \cite{Haiman01} using geometry of Hilbert schemes, thus proving
Macdonald's positivity conjecture.  In the meantime, the program to
establish positivity spawned many new directions in the theory of
symmetric functions, including representation theoretic and
combinatorial studies of the $\nabla$ operator, the elliptic Hall
algebra, and $k$-Schur functions.

Each of these new directions is shaped by the central role played by
LLT polynomials \cite{LasLecTh97}, $t$-analogs of products of skew
Schur functions closely tied to Kazhdan-Lusztig theory.  LLT
polynomials appear in numerous results and conjectures in Macdonald
theory including
\cite{AggBorWhe23, CarlMell18, HaHaLoRU05, HRW18,
KimLeeOh22, LeeOhRho23, LoehWarr08, Mellit16}, where they are the
building blocks of combinatorial formulas for $H_\mu$, $\nabla e_n$,
$\nabla s_\lambda$, and other symmetric functions in a similar vein.
Haglund--Haiman--Loehr \cite{HagHaiLo05} gave a formula expressing
$H_\mu$ as a positive sum of LLT polynomials, which, combined with the
Schur positivity of LLT polynomials \cite{GrojHaim07}, gives a Hilbert
scheme-free proof of Macdonald positivity.

Concurrent with these symmetric function activities, a separate theory
of nonsymmetric Macdonald polynomials \(E_\alpha(x_1,\ldots,x_m;q,t)\)
$(\alpha\in \ZZ^m)$, initiated by Opdam \cite{Opdam}, Macdonald
\cite{Macdonald03}, and Cherednik \cite{Cheredniknsmac}, brought in
new tools from representation theory and topology and took on a life
of its own.  The nonsymmetric setting also offers new insights into
the symmetric one.  For instance, \( E_\alpha \) can be computed via a
recurrence relation \cite{Knop97, Sahiinterpolation} which has no
analog in the symmetric world, and then the \( P_\mu \) can be
recovered by using the Hecke algebra to symmetrize the \( E_\alpha \).

The web of research inspired by the study of \( H_\mu \) and its Schur
positivity, together with the advantages offered by the nonsymmetric
setting, leads to the question: \emph{\textsf{is there a nonsymmetric
analog of \( H_\mu \) that supports a framework equally rich in
combinatorics and representation theory?}} The question is guided by
the schematic diagram: \newcommand{\Toplabel}{f(X) \mapsto f[X/(1-t)]}
\begin{equation}
\label{e modnsmac square}
\begin{tikzcd}[column sep = 2.6cm, row sep = 1.1cm]
 J_{\mu}(X;q,t)
\arrow{r}[swap]{\Toplabel}
&
H_{\mu}(X;q,t)
\\
\arrow{u}{\, \raisebox{-2mm}{\text{\footnotesize{Hecke sym.}}}}
  \Ecal_{\alpha}(x;q,t)   \ar[dashed,r] &
\textbf{\large ?}
\ar[dashed, u, " "]
\end{tikzcd}
\end{equation}
in which the missing corner should be a nonsymmetric counterpart to the
modified Macdonald polynomials with positivity and the potential for
developing surrounding theories similar to those around $H_\mu$.

A significant first step in this direction was taken in 1997 by Knop
\cite{Knop97}, who defined the nonsymmetric integral forms
$\Ecal_\alpha = c_\alpha E_\alpha$ in the lower left corner of
\eqref{e modnsmac square}, analogous to the symmetric integral forms
$J_\mu$, with coefficients in $\ZZ[q,t]$.  As in the symmetric case,
identifying an appropriate integral form is necessary to even
formulate a positivity conjecture.  Building upon this a decade later,
Knop \cite{Knop07} formulated a positivity conjecture for (what we
call) the left stable versions of $\Ecal_\alpha$, involving
Kazhdan-Lusztig theory.  Progress stalled for many years until
Lapointe \cite{Lapointe22} arrived at a different positivity
conjecture inspired by his earlier joint work \cite{BFDLM12a,BFDLM12b}
on Macdonald superpolynomials.  His conjecture involves a left stable
version of $\Ecal_\alpha$ more general than Knop's and found
independently by Goodberry \cite{Goodberry22,Goodberry24}.  In other
related work, Goodberry and Orr \cite{GoodberryOrr} and Bechtloff
Weising and Orr \cite{BWOrr} studied these left stable versions from a
$K$-theoretic viewpoint, and Bechtloff Weising
\cite{BechtloffWeising23} investigated a right stable version of
$\Ecal_\alpha$.

Here, for the first time, we bring the powerful machinery of LLT
polynomials into the nonsymmetric setting. This approach allows us to
complete diagram \eqref{e modnsmac square} with (a stable version of)
the previously undiscovered corner, introducing a \emph{nonsymmetric
plethysm map} for the bottom arrow and \emph{modified $r$-nonsymmetric
Macdonald polynomials} $\mE_{\eta|\lambda}$ for the `\textbf{?}'.  We
introduce flagged LLT polynomials which Weyl symmetrize to ordinary
LLT polynomials, conjecture that they are positive in terms of
Demazure atoms, and show that $\mE_{\eta|\lambda}$ is a positive sum
of flagged LLT polynomials.  This yields a positivity conjecture for
$\mE_{\eta|\lambda}$ which strengthens Macdonald positivity.  These
contributions are highlighted in the summary below, which is divided
into three parts to match the layout of the paper.

\subsection{Flagged symmetric functions and nonsymmetric plethysm (\S
\S \ref{s:flagged-symmetric-functions}--\ref{s:ns-plethysm})}

We develop a theory of nonsymmetric plethysm by first formalizing
ideas related to flagged Schur functions and Demazure characters
implicit in works of Lascoux, Macdonald, Reiner--Shimozono and
others~\cite{Lascoux12, Macdonald91, RS}.  We identify a subspace of
\emph{flagged symmetric functions}, \(\FLambda\), of the algebra of
multi-symmetric functions $\Lambda(X_{1}) \otimes \cdots \otimes
\Lambda(X_{l})$.  It is defined as the linear span of flagged versions
of complete homogeneous symmetric functions, and contains the flagged
skew Schur functions first studied in the 80's and 90's by
Lascoux--Sch\"utzenberger \cite{LascouxSchutz82}, Gessel--Viennot
\cite{GessVien89}, and Wachs \cite{Wachs}.  Our viewpoint begins with
the observation that $\FLambda $ is linearly isomorphic to the
polynomial ring $\Kfrak[x_1, \dots, x_l]$; this is simple yet powerful
because it allows us to define operations on the latter using
plethystic manipulations on $\FLambda$.

These foundations laid, we define the \emph{nonsymmetric plethysm}
operator \(\Pi_{t,(x_1,\dots, x_n)} = \Pi_{t,x}\) on the polynomial
ring $\Kfrak[x_1, \dots, x_n]$ to be the linear map determined by
\begin{multline}\label{e nspleth intro}
  \Pi_{t,x} \Big(
 h_{a_{1}}[x_{1}]\, h_{a_{2}}[x_{2}+(1-t)\, x_{1}]\cdots
h_{a_{n}}[x_{n}+(1-t)(x_{1}+\cdots +x_{n-1})]\Big) \\ =
h_{a_{1}}[x_{1}]\, h_{a_{2}}[x_1+x_{2}]\cdots h_{a_{n}}[x_1+\cdots +x_n].
\end{multline}
We give several other descriptions of \(\Pi_{t,x}\) as well as
formulas for its inverse (\S \S
\ref{ss:PiA-via-inner}--\ref{ss:PiA-recursions}).  We prove that the
action of $\Pi_{t,x}$ on symmetric polynomials converges for large $n$
to the plethystic transformation \(f[X] \mapsto f[X/(1-t)]\) on
symmetric functions.  Better yet, we find formulas for the limit when
the input is symmetric in a growing block of variables.  Of particular
importance for our applications to nonsymmetric Macdonald polynomials,
we show that, restricting to polynomials symmetric in $x_{r+1}, \dots,
x_n$, $\Pi_{t,x}$ converges to an {\em $r$-nonsymmetric plethysm}
operator $\Pisf_r$ on the algebra \(\kk[x_1,\ldots,x_r] \otimes
\Lambda(Y) \) of \emph{\(r\)-nonsymmetric polynomials}, given by
\begin{equation}\label{e:Pir-preview}
\Pisf _{r} \bigl(g(x_{1},\ldots,x_{r})\, h[x_1+\cdots+x_r+Y] \bigr) =
(\Pi _{t,x_1,\dots, x_r} g(x_1, \dots, x_r))\,
h\Big[\frac{(x_1+\cdots+x_r+Y)}{1-t}\Big].
\end{equation}
See Theorem~\ref{thm:Pit-symm} and \eqref{e:lim-conclusions-0} in
Lemma \ref{lem:Pit-vs-Pir}.

\subsection{Flagged LLT polynomials (\S \ref{s:flagged-LLT})}
We define \emph{flagged
LLT polynomials} \(\Gcal_{\nubold,\sigma}\),
which are flagged
symmetric functions that generalize
symmetric LLT polynomials, just as
flagged Schur polynomials
generalize usual symmetric Schur polynomials.
While we give a purely algebraic definition of flagged LLT
polynomials, Theorem~\ref{thm:combinatorial-G-nu-sigma} shows that they
also have an explicit combinatorial formula as a weight generating function
over flagged tableaux, generalizing the usual attacking pair tableau
formula for symmetric LLT polynomials.
In full generality, these two descriptions work for any totally ordered
\emph{signed} alphabet.

Strikingly, flagged LLT polynomials interact well with nonsymmetric
plethysm: for a suitable choice of signed alphabet, \(\Pi_{t,x}\)
sends a signed flagged LLT polynomial, denoted
\(\Gcal^-_{\nubold,\sigma}\), to the unsigned flagged LLT polynomial
\(\Gcal_{\nubold,\sigma}\) with the same indexing data, i.e,
\begin{align}
\label{e Pi on G intro}
\Pi_{t,x}(\Gcal^-_{\nubold,\sigma}) = \Gcal_{\nubold,\sigma}.
\end{align}
See Proposition~\ref{prop:Pit-signed-to-unsigned}.  Additionally, we
propose two positivity conjectures (see \S \ref{ss:atom-positivity})
on the expansion of the \(\Gcal_{\nubold,\sigma}\) into Demazure
atoms, which generalize the known Schur positivity of symmetric LLT
polynomials in two directions.

\subsection{Nonsymmetric Macdonald polynomials (\S \ref{s:ns-Mac-pols})}
The main
application of our results above is
to the theory of nonsymmetric Macdonald polynomials.

We show that the Haglund--Haiman--Loehr formula \cite{HagHaiLo08} for
nonsymmetric Macdonald polynomials $\Ecal _{\alpha}$ can be recast as
a positive sum of signed flagged LLT polynomials
\(\Gcal^-_{\nubold,\sigma}\) for ribbons: specifically, using notation
defined in \S \S \ref{ss:integral-forms}--\ref{ss:ns-HHL}, it reads
\begin{equation}\label{e:ns-HHL intro}
\Ecal _{(\alpha _{1},\ldots,\alpha _{m})}(x;q,t) = t^{n(\alpha _{+})} \sum
_{(\nubold ,\sigma )\in R} \bigl( \prod \nolimits_{\domino }
q^{a(u)+1}\, t^{l(u)} \bigr)\, \Gcal _{\nubold ,\sigma
}^{-}(x_{1},\ldots,x_{m},0,\ldots,0;\, t^{-1})\, .
\end{equation}
We then show how limiting versions of this identity recover and unify
various stabilizations of nonsymmetric Macdonald polynomials
previously studied by Bechtloff Weising \cite{BechtloffWeising23},
Goodberry \cite{Goodberry22, Goodberry24}, Knop~\cite{Knop07}, and
Lapointe~\cite{Lapointe22}.

We introduce \emph{modified $r$-nonsymmetric Macdonald polynomials}
$\mE_{\eta |\lambda}(x,Y;\, q,t)$ ($\eta \in \NN^r$, partition
$\lambda$), defined as the \(\Pisf_r\) images of integral forms $\Jns
_{\eta |\lambda }(x,Y;\, q,t)$ given by rescaling the right stable
nonsymmetric Macdonald polynomials in \cite{BechtloffWeising23}.  The
modified $r$-nonsymmetric polynomials satisfy properties bearing a
striking resemblance to those of $H_\mu$ which made the modified
symmetric theory so fruitful.  We prove that the modified
$r$-nonsymmetric Macdonald polynomials are monomial positive and
conjecture that they are in fact Demazure atom positive.  We show that
$\mE_{\eta |\lambda }$ Weyl symmetrizes to $H_{\mu}$ for $\mu = (\eta
;\lambda )_{+}$ the partition rearrangement of $(\eta;\lambda)$; thus,
our atom positivity conjecture strengthens the Schur positivity of
$H_\mu$ (see \eqref{e:K-lam-mu-classical} and Remark
\ref{rem:K-lam-mu-classical}).  Further, applying a stable limit
version of \eqref{e Pi on G intro} to formula \eqref{e:ns-HHL intro}
yields an expression for $\mE_{\eta |\lambda }$ as a positive sum of
unsigned flagged LLT polynomials \(\Gcal_{\nubold,\sigma}\) (Theorem
\ref{thm:mod-mac}(b)).  This formula for $\mE_{\eta |\lambda}$ shows
that the atom positivity conjecture for modified \(r\)-nonsymmetric
Macdonald polynomials is implied by our atom positivity conjecture for
flagged LLT polynomials.

Our results imply the following commutative diagram (see \eqref{e Weyl
modnsmac2} and \eqref{e Hsym J2}).
\begin{equation}
\begin{tikzcd}[column sep = 2.47cm, row sep = .95cm]
 J_{(\eta;\lambda)_+}(X;q, t) \!
\arrow{r}[swap]{\Toplabel}
&
 \! H_{(\eta;\lambda)_+}(X;q,t)
\\
\arrow{u}{\, \raisebox{-2mm}{\text{\footnotesize{Hecke sym.}}}}
\Jns_{\eta|\lambda}(x,Y;q,t) \!    \ar[r, "\Pisf_r"] &
\! \mE_{\eta|\lambda}(x,Y;q,t)
 \arrow{u}[swap]{\, \raisebox{-2mm}{\text{\footnotesize{Weyl sym.}}}}
\end{tikzcd}
\end{equation}

\section{Preliminaries}
\label{s:prelim}

In this section we fix notation and terminology and summarize various
known results.

\subsection{Partitions and compositions}
\label{ss:partitions-etc}

A {\em weak} (resp.\ {\em strict}) {\em composition} of \(n \in \NN\)
is a finite sequence of nonnegative (resp.\ positive) integers \(\mu =
(\mu_1,\ldots,\mu_k)\) such that \(|\mu| \defeq \mu_1+\cdots+\mu_k =
n\).  We write $\mu ^{\op } \defeq (\mu _{k},\ldots,\mu _{1})$ for the
reversed sequence.

A {\em partition} $\lambda = (\lambda _{1}\geq \cdots\geq \lambda
_{k})$ of $n$ is a weakly decreasing strict composition of $n$.  Its
transpose is denoted by $\lambda ^{*}$.  Sometimes we also allow
partitions to have trailing zeroes.  Either way, the {\em length of
$\lambda $}, denoted $\ell (\lambda )$, means the number of nonzero
parts.  If \(\mu\) is a weak composition, $\mu _{+}$ denotes the
partition obtained by sorting the nonzero parts of $\mu $ into weakly
decreasing order (or all the parts of $\mu $, if allowing partitions
with trailing zeroes).  Of use later will be the number \(n(\lambda)\)
defined by
\begin{equation}\label{e:n(lambda)}
n(\lambda ) = \sum _{i}(i-1)\lambda _{i} = \sum _{i}\binom{\lambda
_{i}^{*}}{2} \,.
\end{equation}

Associated to any strict composition $\rr =(r_{1},\ldots,r_{k})$ of
$n$ is a partition of the set $[n] = \{1,\ldots,n \}$ into intervals
of sizes $r_{1},\ldots,r_{k}$.  We refer to these intervals as {\em
$\rr$-blocks}---e.g.,   for $\rr =(4,1,2)$, the $\rr$-blocks are
\begin{equation}\label{e:r-blocks}
[7]  = \{1,\ldots,4 \} \cup \{5 \} \cup \{6, 7 \}.
\end{equation}

\subsection{Symmetric groups}
\label{ss:Sn}

We write $\Sfrak _{n}$ for the symmetric group of permutations of
$\{1,\ldots,n \}$.  The group $\Sfrak _{n}$ is a Coxeter group with
generators given by the adjacent transpositions $s_{i} = (i
\leftrightarrow i+1)$ for $i=1,\ldots,n-1$.  The {\em length} $\ell
(w)$ of $w\in \Sfrak _{n}$, defined as the length of a reduced (i.e.,
shortest) factorization $w = s_{i_{1}}\cdots s_{i_{l}}$, is also given
by the number of inversions of the permutation $w$:
\begin{equation}\label{e:length=inv}
\ell (w) = \inv (w) \defeq |\{(i<j)\mid w(i)>w(j) \}|.
\end{equation}
We write $w_0$ for the longest permutation $w_0(i) = n+1-i$ in
$\Sfrak_n$.

Given a strict composition $\rr $ of $n$, $\Sfrak _{\rr } \cong
\Sfrak_{r_{1}}\times \cdots \times \Sfrak _{r_{k}}\subseteq \Sfrak
_{n}$ denotes the Young subgroup consisting of permutations that map
every $\rr $-block onto itself; $\Sfrak _{\rr }$ is a Coxeter
subgroup, generated by the $s_{i}$ for which $i$ and $i+1$ are in the
same $\rr $-block.

Given Young subgroups $\Sfrak _{\sS } \subseteq \Sfrak _{\rr}$, every
coset of $\Sfrak_{\sS}$ in $\Sfrak _{\rr}$ contains a unique element
of minimal (resp.\ maximal) length called a minimal (resp.\ maximal)
coset representative.  We let $\Sfrak _{\rr}/ \Sfrak _{\sS }$ denote
the set of minimal representatives of left cosets $w\, \Sfrak_{\sS}$.

The Bruhat order on $\Sfrak_n$ is the transitive closure of the relations
$w < w \, r$ for  $\ell(w) < \ell(w \, r)$ and  $r$ in the set of reflections
$\{s_{ij} = (i \leftrightarrow j) : 1 \le i < j \le n\}$.

\subsection{Symmetric functions}
\label{ss:symmetric-functions}

A {\em symmetric series} is a symmetric formal linear combination
$f(X)$ of monomials in infinitely many variables $X =
x_{1},x_{2},\ldots $, with coefficients in a commutative ring $\Kfrak
$.  A {\em symmetric function} is a symmetric series of bounded
degree.  A {\em multi-symmetric function} (resp.\ {\em series}) is a
function $f(X_{1},\ldots,X_{l})$ of several disjoint sets of variables
$X_{i}$ which is a symmetric function (resp.\ series) in each $X_{i}$.
The algebra of \multisymmetric functions is denoted by $\Lambda
_{\Kfrak }(X)$ or $\Lambda _{\Kfrak }(X_{1}, \ldots, X_{l})$, or just
$\Lambda (X)$ or $\Lambda (X_{1}, \ldots, X_{l})$ if $\Kfrak $ is
understood.  A multi-symmetric function or series can be viewed as a
function of any one set of variables $X_{i}$ with coefficients that
are functions of the other sets $X_{j}$.  Thus, properties of
symmetric functions and series in one set of variables also apply to
each set of variables in a multi-symmetric function or series.

We use Macdonald's notation \cite{Macdonald95} for the elementary,
complete, and power-sum symmetric functions $e_{k}$, $h_{k}$, $p_{k}$,
$e_{\lambda } = e_{\lambda _{1}} e_{\lambda _{2}}\cdots $, $h_{\lambda
} = h_{\lambda _{1}} h_{\lambda _{2}}\cdots $, $p_{\lambda } =
p_{\lambda _{1}} p_{\lambda _{2}}\cdots $, monomial symmetric
functions $m_{\lambda }$, Schur functions $s_{\lambda }$, and skew
Schur functions $s_{\lambda /\mu }$.  The functions $ e_{\lambda }$,
$h_{\lambda }$, $m_{\lambda } $, and $s_{\lambda }$ are graded free
$\Kfrak $-module bases of the algebra of symmetric functions $\Lambda
_{\Kfrak } = \Lambda _{\Kfrak }(X)$, in which the element indexed by a
partition $\lambda $ is homogeneous of degree $|\lambda |$.  If $\QQ
\subseteq \Kfrak $, the $p_{\lambda }$ are also a basis.  To specify
$X$, we write $e_{k}(X)$, $h_{k}(X)$, etc.

The algebra of symmetric functions is a polynomial ring $\Lambda
_{\Kfrak } \cong \Kfrak [h_{1}, h_{2},\ldots ] \cong \Kfrak [e_{1},
e_{2},\ldots ]$ in the $h_{k}$ or $e_{k}$, and, if $\QQ \subseteq
\Kfrak $, also in the $p_{k}$.  The algebra of symmetric series is a
formal power series ring in the same generators.  Using this, we may
identify symmetric functions or series $f(X)$ with polynomials or
power series in functions $e_{k}(X)$, $h_{k}(X)$, or (if $\QQ
\subseteq \Kfrak $) $p_{k}(X)$ of a symbolic {\em formal alphabet} $X$
consisting of unspecified variables.

We make frequent use of the symmetric series
\begin{equation}\label{e:Omega}
\Omega (X)\defeq \sum _{k=0}^{\infty } h_{k}(X) = \exp \bigl( \sum
_{k=1}^{\infty } \frac{p_{k}(X)}{k} \bigr) = \prod _{i}
\frac{1}{1-x_{i}},
\end{equation}
where $X = x_{1}, x_{2}, \ldots $ in the last formula, $h_{0}(X) = 1$
by convention, and the three formulas are equal by
\cite[(2.10)]{Macdonald95}.

\subsection{Plethystic evaluation}
\label{ss:plethysm}

Let $Z$ be an expression constructed from rational numbers and
indeterminates, and define $p_{k}[Z]$ to be the result of substituting
$z\mapsto z^{k}$ for every indeterminate appearing in $Z$.  Assume
that every $p_{k}[Z]$ makes sense as an element of some algebra
$\Rfrak $ over a coefficient ring $\Kfrak $ containing $\QQ $.  The
specific form of $Z$ is flexible, provided only that this last
assumption holds.  In particular, $Z$ could be a symmetric function,
symmetric series, or other formal series over $\QQ $ in a possibly
infinite set of indeterminates in $\Rfrak $, so long as $p_{k}[Z]\in
\Rfrak $ for all $k$.

Since $\Kfrak $ contains $\QQ $, $\Lambda _{\Kfrak } \cong \Kfrak
[p_{1}, p_{2}, \ldots ]$ is a polynomial ring in the power-sums.  For
any $f\in \Lambda _{\Kfrak }$, the {\em plethystic evaluation}
$f[Z]\in \Rfrak $ is defined to be the result of substituting
$p_{k}[Z]$ for $p_{k}$ in the expansion of $f$ as a polynomial in the
$p_{k}$.  In other words, $f\mapsto f[Z]$ is the unique $\Kfrak
$-algebra homomorphism $\Lambda _{\Kfrak }\rightarrow \Rfrak $ that
sends $p_{k}$ to $p_{k}[Z]$.  The plethystic evaluation $f[Z]$ of a
symmetric series $f$, if it exists, is likewise defined as the result
of substituting $p_{k}[Z]$ for $p_{k}$ in the expansion of $f$ as a
power series in the $p_{k}$.

Often the coefficient ring $\Kfrak $ and the target algebra $\Rfrak $
in which $f[Z]$ takes values will be implicit from the context.  For
$Z$ of a suitable form, $f[Z]$ can also be defined without assuming
that $\QQ \subseteq \Kfrak $, as explained in item (iv), below.

We recall some basic properties of plethystic evaluation.

{\it (i) } If $\Rfrak $ contains all symmetric functions of a finite
or infinite set of variables $z_{i}$, then the usual evaluation of
$f\in \Lambda _{\Kfrak }$ in these variables is given by the
plethystic evaluation
\begin{equation}\label{e:pleth=specialize}
f(z_{1},z_{2},\ldots ) = f[z_{1}+z_{2} + \cdots]\, ; \quad \text{in
particular,}\quad f(z_{1},\ldots,z_{n}) = f[z_{1}+\cdots +z_{n}]\, .
\end{equation}
Motivated by this, we use the term {\em plethystic alphabet} for a sum
of variables in the argument of a plethystic evaluation, and
sometimes also for other somewhat more general expressions.  By
convention, if $\Rfrak $ contains the algebra of symmetric functions
in an infinite set of variables $X = x_{1},x_{2},\ldots $, the symbol
$X$ stands for the
plethystic alphabet
\begin{equation}\label{e:pleth-convention}
X = x_{1}+x_{2}+\cdots
\end{equation}
whenever it appears in the argument of a plethystic evaluation.  Then
\eqref{e:pleth=specialize} implies the identity $f[X] = f(X)$.
Plethystic transformations such as $f(X)\mapsto f[(1-t)X]$ are defined
using this convention.  If $X$ is a formal alphabet, we apply
\eqref{e:pleth-convention} by introducing new variables $x_{1},
x_{2},\ldots $ and expressing the plethystic evaluation $f[Z]$ again
formally in terms of $X$.  Specifically, any formal power-sum
$p_{m}(X)$ occurring in $Z$ becomes $p_{km}(X)$ in $p_{k}[Z]$, since
$p_{m}(x_{1}^{k}, x_{2}^{k} ,\ldots ) = p_{km}(x_{1}, x_{2}, \ldots
)$.

{\it (ii) } In general, plethystic evaluation does not commute with
substituting values for indeterminates.  For example, if $X =
x_{1}+x_{2}+\cdots $ is a plethystic alphabet and $t$ is an
indeterminate, it is true that $f[t X] = f(t x_{1}, t x_{2}, \ldots )$
but it is {\em not} true that $f[-X] = f(-x_{1}, -x_{2}, \ldots )$,
i.e., setting $t = -1$ in $f[t X]$ does not yield $f[-X]$.  The reason
for this is that $p_{k}[t] = t^{k}$ for an indeterminate $t$, but
$p_{k}[-1] = -1 \not =(-1)^{k}$.

It is permissible, however, to substitute a monomial $z^{\aA }$ in
indeterminates for an indeterminate, because $p_{k}[z^{\aA }] =
(z^{\aA })^{k}$ for all $k$.  For example, the aforementioned identity
$f[tX] = f(t x_{1}, t x_{2}, \ldots )$ follows by substituting $t
x_{i}$ for $x_{i}$ in the identity $f[X] = f(X)$.

{\it (iii) } The symmetric series $\Omega (X)$ in \eqref{e:Omega} is
often useful in conjunction with plethystic evaluation.  In \S
\ref{s:ns-plethysm}, for example, we use formulas such as $\Omega [A\,
\sum _{i<j} z_{i}/z_{j}]$ to provide a convenient expression for a
symmetric series in $A$ with Laurent polynomial coefficients in
$z_{1},\ldots,z_{n}$ that would be difficult to work with
algebraically if expressed in some other way.

The formula $\Omega (X) = \exp \bigl( \sum _{k=1}^{\infty } p_{k}(X)/k
\bigr)$ implies the multiplicative property
\begin{equation}\label{e:Omega[A+B]}
\Omega [A+B] = \Omega [A]\, \Omega [B]
\end{equation}
whenever the plethystic evaluations $\Omega [A]$ and $\Omega [B]$ are
defined.  In particular, since $\Omega [0] = 1$, $\Omega [-A] = \Omega
[A]^{-1}$.  Since $\Omega (X)^{-1} = \sum _{k=0}^{\infty }(-1)^{k}
e_{k}(X)$ by \cite[(2.6)]{Macdonald95}, it follows that
\begin{equation}\label{e:Om[t(X-Y)]}
\Omega [t (A-B)] = \Omega [t\, A]\, \Omega [-t\, B] = \sum \nolimits
_{r,s\geq 0} (-1)^{s}\, h_{r}[A]\, e_{s}[B]\, t^{r+s}.
\end{equation}
Taking the coefficient of $t^{k}$ on both sides yields the plethystic
identity
\begin{equation}\label{e:hk[A-B]}
h_{k}[A-B] = \sum _{r+s=k} (-1)^{s}\, h_{r}[A]\, e_{s}[B].
\end{equation}

{\it (iv) } If $Z$ is a formal series with integer coefficients, we
can write it as $Z = a_{1}+a_{2}+\cdots -b_{1}-b_{2}-\cdots $, where
$a_{i}$, $b_{i}$ are monomials in the indeterminates, repeating terms
to allow for arbitrary coefficients.  Then \eqref{e:hk[A-B]} implies
\begin{equation}\label{e:hk[Z]}
h_{k}[Z] = h_{k}[a_{1}+a_{2}+\cdots -b_{1}-b_{2}-\cdots] = \sum _{r+s
= k} (-1)^{s} h_{r}(a_{1}, a_{2}, \ldots ) e_{s}(b_{1}, b_{2}, \ldots).
\end{equation}
whenever the symmetric functions $h_{r}(a_{1}, a_{2}, \ldots )$ and
$e_{s}(b_{1}, b_{2}, \ldots)$ make sense as elements of a $\Kfrak
$-algebra $\Rfrak $.  For $Z$ of this form, we can use \eqref{e:hk[Z]}
and the expansion of $f\in \Lambda _{\Kfrak }\cong \Kfrak [h_{1},
h_{2}, \ldots]$ as a polynomial in the $h_{k}$ to define $f[Z]$
without assuming that $\QQ \subseteq \Kfrak $.  The validity of
\eqref{e:hk[A-B]} implies that $f[Z]$ constructed in this way is
well-defined, independent of canceling any redundant terms $a_{i} =
b_{j}$, and that for $f = p_{k}$, we get back the expected formula
$p_{k}[Z] = a_{1}^{k}+a_{2}^{k}+\cdots -b_{1}^{k}-b_{2}^{k}-\cdots $.
In particular, if $\QQ \subseteq \Kfrak $, this construction of $f[Z]$
agrees with the one using power-sum expansions.

\subsection{Coefficients}
\label{ss:coefficients}

Various functions $f(x_{1},x_{2},\ldots, x_{n}, X_{1}, X_{2},\ldots ,
X_{l} ; q, t,\ldots )$ of the following form will be encountered in
this paper: $f$ is a \multisymmetric function or series in formal
alphabets $X_{i}$, and a (Laurent) polynomial or power or Laurent
series in variables $x_{i}$, with coefficients in the field $\kk =\QQ
(q,t,\ldots)$ of rational functions of parameters $q,t,\ldots $.  We
can freely regard $f$ as a function of some of the variables $x_{i}$
and formal alphabets $X_{i}$, with coefficients that are functions of
all the rest.  We denote the operation of taking a coefficient by
angle brackets, e.g., $\langle x^{\aA } \rangle \, f$ for the
coefficient of a monomial $x^{\aA } = x_{1}^{a_{1}}\cdots
x_{k}^{a_{k}}$, or $\langle m_{\lambda } \rangle \, f$ for the
coefficient of a monomial symmetric function.  Note that when $X =
x_{1}, x_{2}, \ldots$, we have $\langle m_{\lambda }(X) \rangle \,
f(X) = \langle x^{\lambda } \rangle \, f(x_{1},x_{2},\ldots )$.  When
taking coefficients with respect to a basis of $\Lambda (X)$, the
basis will either be clear from the context, or specified.

If $f$ happens to be a (Laurent) polynomial in some parameters---for
instance, if $\kk =\QQ (q,t)$, but $f$ has coefficients in the subring
$\ZZ [q,t]$---then we can consider the ground field to be smaller (in
this case, $\QQ $) and take coefficients with respect to the
parameters just as with other variables.

\subsection{Symmetric limits}
\label{ss:limits}

There are several possible senses in which a sequence of polynomials
or power series $f_{n}(x_{1},\ldots,x_{n})$ may converge to a
symmetric function or series
\begin{equation}\label{e:limit}
f(X) = \lim_{n\rightarrow \infty } f_{n}(x_{1},\ldots,x_{n})
\end{equation}
in $X = x_{1}, x_{2}, \ldots $, where $f$ and all $f_{n}$ have
coefficients in a ring $\Kfrak $.

{\it (i) } We say that the $f_{n}$ {\em converge strongly to $f$} if
\begin{equation}\label{e:strong-convergence}
f_{n}(x_{1},\ldots,x_{n}) = f[x_{1}+\cdots +x_{n}]
\end{equation}
for all $n$.  Strong convergence implies convergence in the other
senses below.  If the $f_{n}$ converge strongly, they are symmetric
and satisfy $f_{m}(x_{1},\ldots,x_{n},0,\ldots,0) =
f_{n}(x_{1},\ldots,x_{n})$ for $m\geq n$.  Conversely, these
conditions imply that the $f_{n}$ converge strongly to a unique
symmetric series $f(X)$.  This is essentially the definition of
symmetric series \cite[I.2, pp.~18-19]{Macdonald95}.

The sequence $f_{n}$ need not start at $n=0$.  For example, the
strongly convergent sequence in \eqref{e:left-stable-symm} is defined
for $n\geq \ell (\lambda )$ for a given partition $\lambda $.

{\it (ii) } Let $t$ be an indeterminate in $\Kfrak $.  We say that the
$f_{n}$ {\em converge $t$-adically to $f$} if for all $e>0$, we have
\begin{equation}\label{e:t-adic-convergence}
f_{n}(x_{1},\ldots,x_{n}) \equiv  f[x_{1}+\cdots +x_{n}] \pmod{(t^{e})}
\end{equation}
for $n$ sufficiently large.  Here, $f\equiv g \pmod{(t^{e})}$ means
that the coefficients of $f-g$ vanish to order $\geq e$ at $t = 0$.
Note that this interpretation makes sense even if $\Kfrak $ contains
$\QQ (t)$ or is a ring of Laurent polynomials or Laurent series in
$t$.  If $\Kfrak $ is a polynomial or power series ring in $t$, then
$f\equiv g \pmod{(t^{e})}$ means that the coefficients of $f-g$ are in
the ideal $(t^{e})\subseteq \Kfrak $.  Strictly speaking, $t$-adic
convergence does not require that the $f_{n}(x_{1},\ldots,x_{n})$ be
symmetric, but we will only use it in that case.

{\it (iii) } Suppose that elements of $\Kfrak $ are (Laurent)
polynomials or formal series in some designated variables $Y = y_{1},
y_{2}, \ldots $, so that taking coefficients with respect to these
variables is well-defined.  In particular, elements of $\Kfrak $ could
be any combination of ordinary or Laurent polynomials, power or
Laurent series, or symmetric functions or series in the variables $Y$.

We say that the $f_{n}$ {\em converge formally to $f$ in $X$ and $Y$}
if $\lim_{n\rightarrow \infty } f_{n}(x_{1},\ldots,x_{n}) = f(x_{1},
x_{2}, \ldots )$ as formal series in $X$ and $Y$---that is, if for
every monomial $x^{\aA } y^{\bb }$, we have $\langle x^{\aA } y^{\bb }
\rangle f_{n} = \langle x^{\aA } y^{\bb } \rangle f$ for $n$
sufficiently large.  If $\Kfrak $ contains symmetric functions or
series in a formal alphabet $A$, formal convergence in $A$ means
formal convergence in variables introduced for $A$.  Formal
convergence in $X$, $Y$, and $A$, for instance, is then equivalent to
the property that for all $\aA $, $\bb $, and $\lambda $, we have
$\langle x^{\aA } y^{\bb } m_{\lambda }(A) \rangle f_{n} = \langle
x^{\aA } y^{\bb } m_{\lambda }(A) \rangle f$ for $n$ sufficiently
large, or the same with any basis of $\Lambda (A)$ in place of
$m_{\lambda }$.

The $f_{n}$ need not be symmetric to have a symmetric formal limit
$f$.  In particular, if $f_{n}$ is symmetric in
$x_{1},\ldots,x_{p(n)}$, where $p(n)\leq n$ grows without bound as
$n\rightarrow \infty $, then the formal limit $f(x_{1}, x_{2}, \ldots
)$ is symmetric if it exists.  This situation arises, e.g., in
Corollary~\ref{cor:other-limits}.

\subsection{Diagrams and tableaux} \label{ss:diagrams}

Given $\alpha,\beta\in\mathbb Z^l$ with $\alpha_{j} \leq \beta_{j}$
for all $j$, the {\it diagram} $\beta/\alpha$ is the array of unit
lattice squares (or {\em boxes}) whose northeast corners have
coordinates $(i,j)$ for $j=1,\ldots,l$ and $\alpha _{j}< i\leq \beta
_{j}$ (if $\alpha _{j} = \beta _{j}$, then the $j$-th row of $\beta
/\alpha $ is empty).  When $\beta_1\geq \cdots\geq\beta_{l}$ and
$\alpha_1\geq \cdots \geq \alpha_{l}$, the diagram $\beta/\alpha$ is a
{\it skew diagram}.  In particular, the Young or Ferrers diagram of a
partition $\lambda $ is the skew diagram $\lambda /(0^{\ell (\lambda
)})$.  Note that we use the French style in which the diagram of a
partition has its southwest corner at the origin.

A {\em filling} of a diagram $\nu$ with letters from a set $\Acal $ is
a map $T\colon \nu \rightarrow \Acal $, pictured by placing the letter
$T(x)$ in each box $x$ of $\nu $.

A {\em signed alphabet} is a totally ordered set $\Acal $ equipped
with a partition $\Acal =\Acal ^{+}\coprod \Acal ^{-}$ into {\em
positive} and {\em negative} letters.  A {\em super filling} of a
general diagram $\nu $ is a map to a signed alphabet.  A super filling
$T$ is {\em row-increasing} if every row of $T$ is weakly increasing
with no repeated negative letters.  If $\nu $ is a skew diagram, a
{\em super tableau} on $\nu$ is a row-increasing super filling
$T\colon \nu \rightarrow \Acal $ such that, in addition, each column
of $T$ is weakly increasing with no repeated positive letters.  If
$\Acal =\Acal ^{+}$ has only positive letters, then a super tableau is
a semistandard tableau in the usual sense. The set of semistandard
tableaux is denoted by $\SSYT(\nu,\Acal)$.

For $a,b$ in a signed alphabet $\Acal $, we use the notation
\begin{equation}\label{e:inv-signed}
\begin{aligned}
a \sgngt b \quad \text{(or $b \sgnlt a$)}& \qquad \text{if $a>b$ or
 $a=b\in \Acal ^{-}$},\\
a \sgnle b \quad \text{(or $b \sgnge a$)}& \qquad \text{if $a<b$ or
 $a=b\in \Acal ^{+}$}.
\end{aligned}
\end{equation}
Note that $\sgnle $ is the negation of $\sgngt $, and that $\sgnlt $,
$\sgnle $ reduce to $<$, $\leq $ if $\Acal = \Acal ^{+}$ has only
positive letters.  In this notation, the increasing condition on a row
of a super filling is $a_1 \sgnle a_2 \sgnle \cdots \sgnle a_k$, and
on a column of a super tableau it is $a_1 \sgnlt a_2 \sgnlt \cdots
\sgnlt a_k$.

\subsection{Symmetric LLT polynomials}
\label{ss:classical-LLT}

Let $\nubold = (\nu ^{(1)},\ldots,\nu ^{(k)})$ be a tuple of diagrams.
We identify $\nubold $ with the disjoint union of its components,
i.e., a box in $\nubold $ means a box in a specific component $\nu
^{(j)}$.  A row-increasing super filling on $\nubold$ is a filling
$T\colon \nubold \rightarrow \Acal $ whose restriction to each
component $\nu ^{(j)}$ is a row-increasing super filling; likewise for
super tableaux if $\nubold$ is a tuple of skew diagrams.  We write
$\SSYT _{\pm }(\nubold ,\Acal )$ for the set of super tableaux on a
tuple of skew diagrams $\nubold $, or $\SSYT(\nubold,\Acal)$ if
$\mathcal A$ has only positive letters.

The {\em content} of a box $x$ in a diagram $\nu$ is $c(x)=i-j$, where
$(i,j)$ is the northeast corner of $x$.

\begin{defn} \label{def:inv_old}
Given a tuple $\nubold$ of diagrams, boxes $x\in \nubold ^{(i)}$ and
$y\in \nubold ^{(j)}$ form an {\em attacking pair} if either $i<j$ and
$c(y) = c(x)$, or $i>j$ and $c(y)=c(x)+1$.  For a row-increasing super
filling $T:\nubold\to \mathcal A$, an {\em attacking inversion} in $T$
is an attacking pair $(x,y)$ such that $T(x) \sgngt T(y)$.  The number
of attacking inversions in $T$ is denoted $\oinv (T)$.
\end{defn}

For a positive alphabet $\Acal = \Acal ^{+}$, the {\em LLT polynomial}
indexed by a tuple $\nubold$ of skew diagrams is defined by
\begin{equation}\label{e:classical-LLT}
\Gcal _{\nubold }(X;t) \defeq \sum _{T\in \SSYT (\nubold ,\Acal )}
t^{\oinv (T)}\, x^{T},
\end{equation}
where $x^{T} = \prod _{u\in \nubold } x_{T(u)}$.  It is known that
$\Gcal _\nubold (X;t)$ is symmetric in the variables $x_{a}$ ($a\in
\Acal $).  Up to factors of the form $t^{e}$, these are the same
polynomials as defined by Lascoux, Leclerc, and Thibon (LLT) in
\cite{LasLecTh97} using ribbon tableaux and spin---see also \cite[\S
5]{HaHaLoRU05} and \cite[\S 10]{HagHaiLo05}.

We can regard formula \eqref{e:classical-LLT} for an infinite positive
alphabet $\Acal $ as defining a symmetric function $\Gcal _{\nubold
}(X;t)\in \Lambda _{\ZZ [t]}(X)$ in a formal alphabet $X$.  Then we
have the following generalization of formula \eqref{e:classical-LLT}
to signed alphabets.

\begin{lemma}[{\cite[(81)--(82) and Proposition 4.2]{HagHaiLo05}}]
\label{lem:signed-classical-LLT}

Given a signed alphabet $\Acal $ and a tuple $\nubold $ of skew
diagrams, evaluating $\Gcal _{\nubold }(X;t)$ in the plethystic
alphabet $X^{\Acal } \defeq \sum _{a\in \Acal ^{+}} x_{a} - \sum
_{b\in \Acal ^{-}} x_b$ yields
\begin{equation}\label{e:signed-classical-LLT}
\Gcal _{\nubold }[X^{\Acal };t] = \sum _{T\in \SSYT _{\pm }(\nubold
,\Acal )} (-1)^{m(T)}\, t^{\oinv (T)}\, x^{T},
\end{equation}
where $m(T)=|T^{-1}(\Acal ^{-})|$ is the number of negative entries in
$T$.
\end{lemma}

\begin{remark}\label{rem:super-skew-schur}
If $\nubold = (\nu )$ has only one component, then there are no
attacking pairs, and $\oinv (T) = 0$ for every tableau $T$.  In this
case, $\Gcal _{\nubold }(X; t) = s_{\nu}(X)$ is a skew Schur function,
and \eqref{e:signed-classical-LLT} is the usual super tableau formula
for the signed alphabet specialization $s_{\nu }[X^{\Acal }]$.
\end{remark}

\subsection{\texorpdfstring{$\GL _{n}$}{GL\_n} characters}
\label{ss:GLn-characters}

We recall some standard notation associated to  $\GL_n$ weights and characters.
The weight lattice is $X = X(\GL_n) =  \ZZ ^{n}$.  Thus, its group
algebra $\Kfrak \, X$ (over any commutative ring $\Kfrak $) is the
Laurent polynomial ring $\Kfrak [\xx ^{\pm 1}] \defeq \Kfrak
[x_{1}^{\pm 1},\ldots,x_{n}^{\pm 1}]$, and the formal exponential of a
weight $\lambda \in X$ is the monomial $x ^{\lambda } = x_{1}^{\lambda
_{1}}\cdots x_{n}^{\lambda _{n}}$.

The dominant weights $X_{+} = X_+(\GL_n)$ are the weakly
decreasing vectors
\begin{equation}\label{e:X+}
X_{+} = \{\lambda \in \ZZ ^{n}\mid \lambda _{1}\geq
\cdots \geq \lambda _{n} \}.
\end{equation}
The antidominant weights $-X_+$ are the weakly increasing vectors in  $\ZZ^n$.
The Weyl group is $\Sfrak _{n}$, permuting coordinates in $X= \ZZ
^{n}$, and variables in $\Kfrak \, X = \Kfrak [x ^{\pm 1}]$.
Note that the stabilizer of a dominant or antidominant weight is always a Young subgroup
$\Sfrak_{\rr}$ for some strict composition $\rr $ of $n$.

A weight $\lambda $ is {\em regular} if its stabilizer in $\Sfrak _{n}$
is trivial; this means that $\lambda = (\lambda _{1},\ldots,\lambda
_{n})\in \ZZ ^{n}$ has distinct entries.  Hence, the set of regular
dominant weights is
\begin{equation}\label{e:X++}
X_{++}  = \{\lambda \in \ZZ ^{n}\mid \lambda _{1}>
\cdots > \lambda _{n} \}.
\end{equation}
Any weight $\rho $ such that  $\rho_{i+1} - \rho_{i} = 1$ for all $i=1,\ldots,n-1$, for example $\rho =
(n-1,n-2,\ldots,0)$, is minimally dominant and regular, in the sense
that
\begin{equation}\label{e:plus-rho}
X_{++} = \rho +X_{+}.
\end{equation}
Such a weight $\rho $ is unique up to adding a constant vector.

The irreducible characters of $\GL _{n}$ are given by the Weyl
character formula
\begin{equation}\label{e:chi-lambda}
\chi _{\lambda }(x) = \Weyl (x^{\lambda }) \quad \text{for
$\lambda \in X_{+}$},
\end{equation}
where $\Weyl$ is the {\em Weyl
symmetrization operator}
\begin{equation}\label{e:Weyl-symmetrizer}
\Weyl f(x) = \sum _{w\in \Sfrak _{n}} w
\left(\frac{f(x)}{\prod _{i<j} (1-x_{j}/x_{i})}
\right).
\end{equation}
Here we have written only $i<j$ for $1\leq i<j\leq n$, and will do so hereafter whenever the range is clear from context.

The {\em polynomial} irreducible characters of $\GL _{n}$ are the
$\chi _{\lambda }$ for $\lambda \in \NN ^{n}\cap X_{+}$, i.e., for
$\lambda _{1}\geq \cdots \geq \lambda _{n}\geq 0$, so $\lambda $ is a
partition of length $\ell (\lambda )\leq n$ with possible trailing
zeroes.  The polynomial characters are Schur functions, i.e.,
$\chi _{\lambda }(x) = s_{\lambda }(x_{1},\ldots,x_{n})$.

\subsection{Levi subgroups}
\label{ss:levi}

Given a strict composition $\rr =(r_{1},\ldots,r_{k})$ of $n$, we define
\begin{equation}\label{e:GLr}
\GL _{\rr } \cong  \GL _{r_{1}}\times \cdots \times \GL _{r_{k}}
\end{equation}
to be the subgroup of block diagonal matrices $g\in \GL _{n}$ with
blocks of size $r_{i}\times r_{i}$.

The Weyl group of $\GL _{\rr }$ is the Young subgroup $\Sfrak _{\rr }$
of $\Sfrak _{n}$.  A weight $\lambda $ is $\GL _{\rr }$
regular---i.e., its stabilizer in $\Sfrak _{\rr }$ is trivial---if
and only if its restriction $(\lambda _{i},\ldots,\lambda _{j})$ to each $\rr
$-block $\{i,\ldots,j \}$ has distinct entries.  Similarly, $\lambda \in \ZZ^n$
belongs to the set $X_{+}(\GL _{\rr })$ of $\GL _{\rr }$ dominant
weights if and only if $\lambda _{i}\geq \cdots \geq \lambda _{j}$ on each $\rr
$-block; the set of antidominant weights  $-X_+(\GL_\rr)$ of  $\GL_\rr$ consists of  $\lambda \in \ZZ^n$
which are weakly increasing on each  $\rr$-block.
As before, the set of regular dominant weights can be
expressed as
\begin{equation}\label{e:X++(GLr)}
X_{++}(\GL _{\rr }) = \rho _{\rr } + X_{+}(\GL _{\rr })
\end{equation}
if $(\rho _{\rr })_{i} - (\rho _{\rr })_{i+1}= 1$ whenever $i$ and $i+1$ are in the same $\rr $-block.
Such a
weight $\rho _{\rr }$ is unique up to adding an $\Sfrak _{\rr
}$-invariant weight, i.e., a vector constant on each $\rr $-block.
For combinatorial reasons, in \S \ref{s:flagged-LLT}, we will need the
specific choice
\begin{equation}\label{e:rho-r-preview}
\rho _{\rr } = (0,\, -1,\, \ldots,\, 1-r_{1},\, 0,\, -1,\, \ldots,\,
1-r_{2},\, \ldots, \, 0,\, -1,\, \ldots,\, 1-r_{k}) \,.
\end{equation}

\subsection{Demazure characters and atoms}
\label{ss:Demazure}

We review some basic facts and definitions for Demazure characters and
atoms for $\GL_n$, recall Lascoux's Cauchy formula
\cite{LascouxCrystal}, and prove some results for our discussion of
atom positivity in \S \ref{ss:atom-positivity}.  References for
Demazure characters and atoms for $\GL_n$ include
\cite{FuLascoux09,Lascoux12,RS}.

The \emph{Demazure operator} or \emph{isobaric divided difference
operator $\Dem_i$} ($1 \le i < n$) acting on a Laurent polynomial
$f(x) = f(x_1,\dots, x_n)$ is given by
\begin{align}
\Dem_i \, f(x) &= (1+s_{i})\frac{f(x)}{1- x_{i+1}/x_i},
\end{align}
which is the same as the Weyl symmetrization operator
in variables  $x_i, x_{i+1}$.
The Demazure operators satisfy the braid relations $\Dem_i \Dem_{i+1}
\Dem_i = \Dem_{i+1} \Dem_{i} \Dem_{i+1}$ and the 0-Hecke relations
$\Dem_i^2 = \Dem_i$.  Also define the operators $\DemAt_i = \Dem_i -
1$.  These satisfy the braid relations and $\DemAt_i^2 = -\DemAt_i$.
For $w\in \Sfrak_n$ with reduced expression $w = s_{j_1}s_{j_2}\cdots
s_{j_d}$, define the operators
\begin{align}
\label{e Demw def}
\Dem_w & = \Dem_{s_{j_1}} \cdots \Dem_{s_{j_d}}, \\
\label{e DemAtw def}
\DemAt_w & = \DemAt_{s_{j_1}} \cdots \DemAt_{s_{j_d}},
\end{align}
which do not depend on the choice of reduced expression.

The Demazure operator indexed by the longest permutation is equal to
the Weyl symmetrization operator, i.e.,
\begin{equation}
\Dem_{w_0} = \Weyl
\end{equation}
(see, e.g., \cite[eq.~(1.5.2)]{Lascoux12}).  It then follows from the
0-Hecke relations that for any $w \in \Sfrak_n$,
\begin{align}
\Weyl \, \Dem_w = \Weyl.
\end{align}

For $\lambda \in \ZZ^n$, the \emph{Demazure character} $\Dcal_\lambda
= \Dcal_\lambda(x) = \Dcal_\lambda(x_1,\dots, x_n)$ and \emph{Demazure
atom} $\Acal_\lambda = \Acal_\lambda(x) = \Acal_\lambda(x_1,\dots,
x_n)$ for $\GL_n$ are given by
\begin{align}\label{e:D-and-A}
\Dcal _{\lambda}(x) &= \Dem_{w}(x^{\lambda_{+}}),\\
\Acal_\lambda(x) &= \DemAt_{v} (x^{\lambda_+}),
\end{align}
where $\lambda_+$ is the rearrangement of $\lambda$ to a dominant
weight, $w \in \Sfrak_n$ is any permutation such that $\lambda
=w(\lambda_{+})$, and $v \in \Sfrak_n$ is the minimal permutation such
that $v(\lambda_+) = \lambda$.

If $\lambda \in \ZZ^n$ is dominant, then $\Dcal_\lambda(x) =
\Acal_\lambda(x) = x^\lambda$, while if $\lambda$ is antidominant,
then $\Dcal_\lambda(x)$ is the irreducible $\GL _{n}$ character $\chi
_{\lambda_+ }(x)$.  The sets $\{\Dcal _{\lambda }(x) \mid \lambda \in
\ZZ ^{n} \}$ and $\{\Acal_{\lambda }(x) \mid \lambda \in \ZZ ^{n} \}$
are bases of $\Kfrak [\xx ^{\pm 1}] = \Kfrak [x_{1}^{\pm 1}, \ldots,
x_{n}^{\pm 1}]$ for any coefficient ring $\Kfrak $, and the subsets
with $\lambda \in \NN^{n}$ are bases of $\Kfrak [\xx ]$ (see, e.g.,
\cite[Corollary 7]{RS}).

Demazure characters and atoms are related by
\begin{equation}\label{e:dem to atom}
\Dcal_{\lambda}(x) = \sum_{\mu \, \in \,
\Sfrak_n \lambda \colon  w_\mu \le w_\lambda} \Acal_\mu(x),
\end{equation}
where $w_\eta$ denotes the minimal permutation such that
$w_\eta(\eta_+)= \eta$ and $\le$ is Bruhat order on $\Sfrak_n$.

From $\Dem_i \, \DemAt_i = 0$, $\Weyl = \Dem_{w_0}$, and the fact that
\eqref{e DemAtw def} does not depend on the choice of reduced
expression, we have
\begin{gather}\label{e:for atom pos 1}
\Weyl \Acal _{\mu }(x) =
\begin{cases}
\chi _{\mu }(x) & \text{if $\mu =\mu _{+}$,} \\
0 & \text{ otherwise},
\end{cases}
\\[1ex]
\label{e:for atom pos 2}
\Dem _{i}\, \Acal _{\mu }(x) = \begin{cases}
\Acal _{\mu }(x)+\Acal _{s_{i}(\mu )}(x) &
	\text{if $\mu _{i}>\mu _{i+1}$},\\
\Acal _{\mu }(x) &	    \text{if $\mu _{i} = \mu _{i+1}$},\\
0& \text{if $\mu _{i}<\mu _{i+1}$}.
\end{cases}
\end{gather}

We record some consequences of these facts for our discussion of atom
positivity in \S \ref{ss:atom-positivity} and \S
\ref{ss:ns-Mac-atom-positivity}.

\begin{remark}\label{rem:for atom pos} 
(i) By \eqref{e:dem to atom}, a symmetric
polynomial is Schur positive if and only if it is atom positive.

(ii) By \eqref{e:for atom pos 1}, in the expansion of any Laurent
polynomial $f(x)$ into Demazure atoms, the coefficient of
$\Acal_{\lambda }(x)$ for dominant $\lambda $ is equal to the
coefficient of $\chi _{\lambda }(x)$ in the expansion of $\Weyl f(x)$
into irreducible $\GL _{n}$ characters.

(iii) It follows from \eqref{e:for atom pos 2} that $\Dem _{w}\, \Acal
_{\mu }(x)$ is a multiplicity-free sum of atoms $\Acal _{\eta }(x)$
with $\eta $ a permutation of $\mu $, for any $w$.

(iv) It also follows from \eqref{e:for atom pos 2} that a Laurent
polynomial $f(x_{1},\ldots,x_{m})$ is symmetric in
$x_{i},\ldots,x_{j}$ if and only if permuting the indices $\mu
_{i},\ldots,\mu _{j}$ does not change the coefficients $\langle \Acal
_{\mu }(x) \rangle\, f(x)$ in its atom expansion.
\end{remark}

Let $\langle -,- \rangle_{0}$ be the symmetric inner product
on the Laurent polynomial ring $\Kfrak [\xx ^{\pm 1}]$ defined by
\begin{align}\label{e:atom-demazure-inner}
\langle f(x),g(x) \rangle _{0} & = \langle x^{0} \rangle \, f(x)\,
g(x)\, \Omega [-\sum _{i<j} x_{i}/x_{j}].
\end{align}

\begin{prop}[{\cite{FuLascoux09, LascouxCrystal}}]\label{prop:Dem-At-orthog}

(a) Demazure characters $\Dcal _{\lambda }(x)$ and atoms $\Acal
_{-\lambda }(x)$ are dual bases for this inner product, i.e.,
\begin{equation}\label{e:atom-demazure-dual}
\langle \Dcal _{\lambda }(x), A_{-\mu }(x) \rangle_{0} = \delta
_{\lambda ,\mu }.
\end{equation}

(b) We have the following `Cauchy identities' for Demazure characters
and atoms:
\begin{equation}\label{e:atom-demazure-Cauchy}
\Omega [\sum _{i\leq j} x_{i}/y_{j}] = \sum _{\lambda \in \NN ^{n}}
\Dcal _{\lambda }(x)\mathcal A_{-\lambda }(y) = \sum _{\lambda \in \NN
^{n}} \Acal _{\lambda }(x)\mathcal \Dcal_{-\lambda }(y).
\end{equation}
\end{prop}

\begin{proof}
Our inner product differs from that of \cite{FuLascoux09} by the
involution $\theta \colon x_i \mapsto x_{w_0(i)}^{-1}$ in the second
factor.  Hence, \eqref{e:atom-demazure-dual} follows from
\cite[eq.~(35)]{FuLascoux09} and (the second of) the following two
identities
\begin{align}\label{e basic dem 2}
\Dcal _{\lambda}(x_{n}^{-1},\ldots,x_{1}^{-1}) = \Dcal
_{-w_{0}(\lambda )}(x) \ \ \, \text{ and } \ \ \, \Acal
_{\lambda}(x_{n}^{-1},\ldots,x_{1}^{-1}) = \Acal _{-w_{0}(\lambda
)}(x),
\end{align}
which in turn follow from $\theta \Dem_{i} \theta = \Dem_{n-i}$.
Likewise, both versions of the Cauchy identity in
\eqref{e:atom-demazure-Cauchy} can be obtained from the one in
\cite[Theorem 6]{LascouxCrystal} using \eqref{e basic dem 2}.
\end{proof}

By virtue of \eqref{e:atom-demazure-dual},
(either) Cauchy identity in \eqref{e:atom-demazure-Cauchy}
is equivalent to each of the two
`reproducing kernel' identities
\begin{align}\label{e:atom-demazure-reproducing-1}
f(x) & = \langle f(z), \Omega [\sum _{i\leq j} x_{i}/z_{j}] \rangle
_{0}\quad \text{for $f(x)\in \Kfrak [x_{1},\ldots,x_{l}]$;}
\\
\label{e:atom-demazure-reproducing-2} g(y) & = \langle g(z), \Omega
[\sum _{i\leq j} z_{i}/y_{j}] \rangle _{0} \quad \text{for $g(y)\in
\Kfrak [y_{1}^{-1},\ldots,y_{l}^{-1}]$,}
\end{align}
where the inner products are in the  $z$ variables.

\begin{lemma}\label{l dom regular atom}
For a dominant regular weight $\lambda \in X_{++}(\GL_n)$, $\Acal_{w_0
(\lambda)}(x) = x^{w_0 (\rho)} \chi_{\lambda-\rho}(x)$ for any
$\rho$ as in \S \ref{ss:GLn-characters}.
\end{lemma}

\begin{proof}
By \cite[eq.~(1.5.3)]{Lascoux12}, the operator $\DemAt_{w_0}$ can be
expressed in terms of the Weyl symmetrization operator by
\begin{equation}\label{e:antidominant-atom}
\DemAt_{w_0} \, = x^{w_0 (\rho)} \, \Weyl \, x^{-\rho}.
\end{equation}
Applying both sides to  $x^\lambda$ gives the result.
\end{proof}

\begin{lemma}\label{lem:Amu-with-zero}
Assume that $\mu \in \NN ^{n}$, so $\Acal _{\mu }(x)$ is a polynomial.
For any $k$, setting $x_{k}=0$ and re-indexing, we have
\begin{equation}\label{e:Amu-with-zero}
\Acal _{\mu }(x_{1},\ldots,x_{k-1},0,x_{k},\ldots,x_{n-1}) =
\begin{cases} \Acal _{\eta }(x_{1},\ldots,x_{n-1})& \text{if $\mu _{k}=0$,}\\
0& \text{if $\mu _{k}>0$,}
\end{cases}
\end{equation}
where $\eta = (\mu _{1}, \ldots, \mu _{k-1}, \mu _{k+1}, \ldots, \mu
_{n})$.
\end{lemma}

\begin{proof}
If $\mu =\mu _{+}$ is dominant, then $\Acal _{\mu } = x^{\mu _{+}}$,
and the result is clear.  Otherwise, we prove it (for all $k$) by
induction on the length of the minimal $w$ such that $\mu =w(\mu
_{+})$.

If $\mu _{i} < \mu _{i+1}$ for some $i\not \in \{k,k-1 \}$, then
$\Acal _{\mu } = \DemAt _{i}\Acal _{s_{i}\mu }$, and the result
holds for $s_{i}\mu $ by induction.  The result for $\mu $ follows
because $\DemAt _{i}$ commutes with setting $x_{k} = 0$.  Thus, we
can assume that we either have $\mu _{k-1} < \mu _{k}$ (with $k>1$),
or $\mu _{k} < \mu _{k+1}$ (with $k<n$).

If $\mu _{k-1} < \mu _{k}$, then we are in the case $\mu _{k} >0$, and
must show that $x_{k}$ divides $\Acal _{\mu } = \DemAt_{k-1} \Acal
_{s_{k-1}\mu }$.  We can assume by induction that $x_{k-1}$ divides $
\Acal _{s_{k-1}\mu }$.  The result now follows because $x_{k}$ divides
$\DemAt _{k-1} x_{k-1}^{a}x_{k}^{b}$ if $a>0$.

If $\mu _{k} <\mu _{k+1}$, then $\Acal _{\mu } = \DemAt_{k}\,
\Acal_{s_{k}\mu }$.  We can assume by induction that if $(s_k
\mu)_{k+1} = \mu_k > 0$, then $\Acal _{s_{k}\mu}(x_{1},\ldots,x_{k},0,
x_{k+1},\ldots,x_{n-1}) =0$, and if $(s_k \mu_k)_{k+1} = \mu_k = 0$,
then $\Acal_{s_{k}\mu}(x_{1},\ldots,x_{k},0, x_{k+1},\ldots,x_{n-1}) =
\Acal_\eta$, where $\eta = (\mu _{1}, \ldots, \mu _{k-1}, \mu _{k+1},
\ldots, \mu_{n})$ as above.  Since $(s_k \mu)_{k} >0$, we can also
assume by induction that $x_{k}$ divides $\Acal _{s_{k}\mu }$.  The
result now follows using that, for $a > 0$, $(\DemAt_{k}
x_{k}^{a}x_{k+1}^{b})|_{x_{k} = 0} = x_{k+1}^{a}0^{b}$ and hence,
$\big((\DemAt_{k} x_{k}^{a}x_{k+1}^{b})|_{x_{k} =
0}\big)_{x_{k+1}\mapsto x_k} = (x_{k}^{a}x_{k+1}^{b})|_{x_{k+1} = 0}$.
\end{proof}

\begin{cor}\label{cor:Amu-with-zeroes}
For $\mu \in \NN ^{n}$, setting any subset of the variables to zero
and re-indexing, we have that $\Acal _{\mu }(0,\ldots,0, x_{1},
0,\ldots,0, x_{2}, 0,\ldots,0, \ldots, 0,\ldots,0, x_{m}, 0,\ldots, 0)$
is equal to zero if we set $x_{i} = 0$ for any $i$ such that $\mu _{i}
>0$; otherwise it is equal to $\Acal _{\eta }(x_{1},\ldots,x_{m})$,
where $\eta $ is obtained from $\mu $ by deleting the entries $\mu
_{i} = 0$ in the positions $i$ corresponding to the variables $x_{i}$
that we set equal to zero.
\end{cor}

\section{Flagged symmetric functions}
\label{s:flagged-symmetric-functions}

The concept of a flagged symmetric function originated in the study of
flagged Schur functions and Schubert polynomials, with a history going
back at least to the notes by Macdonald \cite{Macdonald91}, who in
turn credits Lascoux for many of the main ideas.  For our work here,
we need to formalize this broadly understood concept in a more precise
way than was done historically.  To that end, we now introduce a
general definition of `flagged symmetric function.'  This concept will
provide a natural framework for the definition of nonsymmetric
plethysm in \S \ref{s:ns-plethysm}.  In addition, flagged LLT
polynomials---whose definition and study are at the heart of this
work---are examples of flagged symmetric functions.  The flagged
symmetric function formalism will thus turn out be useful for deriving
results about nonsymmetric plethystic transformations on flagged LLT
polynomials, such as Proposition~\ref{prop:Pit-signed-to-unsigned}.

\subsection{The space of flagged symmetric functions}
\label{ss:flagged-sym}

Let $\Lambda (X_{1},\ldots,X_{l}) = \Lambda _{\Kfrak
}(X_{1},\ldots,X_{l}) $ be the algebra of multi-symmetric functions
over a ring $\Kfrak $ which will usually be implicit from the context.
Although $\Kfrak $ is not assumed to be a field, we use the term
`(sub)space' as a shorthand for `free $\Kfrak $-(sub)module.'  All
plethystic evaluations in this section will be of the form in \S
\ref{ss:plethysm}, (iv), defined for \multisymmetric functions over
any coefficient ring $\Kfrak $.

We say that a multi-symmetric function $f(X_{1},\ldots,X_{l})$ is a
{\em flagged symmetric function} if it belongs to the distinguished
subspace $\FLambda (X_{1},\ldots,X_{l}) \subset \Lambda
(X_{1},\ldots,X_{l})$ defined as follows.

\begin{defn}\label{def:flagged-sym}
The space $\FLambda(X_{1},\ldots,X_{l})$ of {\em flagged symmetric
functions} is the subspace of $\Lambda (X_{1},\ldots,X_{l})$ spanned
by the functions
\begin{equation}\label{e:flagged-sym-basis}
{\fh}_{\aA}(X_{1},\ldots,X_{l}) \defeq
h_{a_{1}}[X_{1}]h_{a_{2}}[X_{1}+X_{2}]\cdots h_{a_{l}}[X_{1}+\cdots
+X_{l}]
\end{equation}
for $\aA \in \NN ^{l}$.
\end{defn}

We also write $\FLambda _{\Kfrak }(X_{1},\ldots,X_{l})$ if we need to
make $\Kfrak $ explicit.  Note that $\FLambda(X_{1},\ldots,X_{l})$ is
only a subspace, not a subalgebra, of $\Lambda (X_{1},\ldots,X_{l})$.

The expansion of ${\fh}_{\aA }(X_{1},\ldots,X_{l})$ in the basis
$\{h_{\lambda ^{(1)}}(X_{1})\cdots h_{\lambda ^{(l)}}(X_{l})\}$ of
$\Lambda (X_{1},\ldots, X_{l})$ has leading term
$h_{a_{1}}(X_{1})\cdots h_{a_{l}}(X_{l})$, with respect to any
ordering that extends lexicographic order on the reversed degree
sequence $(|\lambda ^{(l)}|,\ldots,|\lambda ^{(1)}|)$.  Since they
have distinct leading terms, the elements $\fh _{\aA
}(X_{1},\ldots,X_{l})$ are linearly independent, hence they form a
free $\Kfrak $-module basis of $\FLambda(X_{1},\ldots,X_{l})$.

Traditional notions such as flagged Schur functions fit into this
framework when we specialize the formal alphabets $X_{i}$ to blocks of
variables in a finite list $x_{1},\ldots,x_{n}$ in the following way.

\begin{defn}\label{def:flag-bounds}
Given integers $0 = b_{0}\leq b_{1}\leq \cdots \leq b_{l} \leq n$, the
{\em specialization with flag bounds $b= b_{1},\ldots,b_{l}$} of a
flagged symmetric function $ f(X_{1},\ldots,X_{l})$ is its evaluation
$f[X^b_{1},\ldots,X^b_{l}]$ in the finite plethystic alphabets
\begin{equation}\label{e:X-for-flag-bounds}
X_{i}^b = x_{b_{i-1}+1} +\cdots +x_{b_{i}} \quad \text{or,
equivalently,}\quad X_{1}^b+\cdots +X_{i}^b = x_{1} +\cdots +x_{b_{i}}.
\end{equation}
\end{defn}

Thus, the specialization $f[X_{1}^b,\ldots,X_{l}^b]$ is a polynomial
in $x_{1},\ldots,x_{b_{l}}$, symmetric in each of $l$ blocks of
consecutive variables of sizes $b_{1}, b_{2}-b_{1}$, \ldots,
$b_{l}-b_{l-1}$.  We allow the number $n$ of ambient variables to
exceed $b_{l}$ for flexibility, although $f[X_{1}^b,\ldots,X_{l}^b]$
is independent of any variables beyond $x_{b_{l}}$.

The following examples may help to motivate the definitions.

\begin{example}\label{ex:flagged-sym-1}
(i) The {\em flagged complete symmetric functions} ${\fh}_{\aA
}(X_{1},\ldots,X_{l})$ in \eqref{e:flagged-sym-basis} are flagged
symmetric functions by definition.  The specialization of ${\fh}_{\aA
}(X_{1},\ldots,X_{l})$ with flag bounds $b_{1},\ldots,b_{l}$ is the
product
\begin{equation}\label{e:flagged-h-special}
h_{a_{1}}(x_{1},\ldots,x_{b_{1}}) h_{a_{2}}(x_{1},\ldots,x_{b_{2}})  \cdots
h_{a_{l}}(x_{1},\ldots,x_{b_{l}}).
\end{equation}
In particular, the specialization of ${\fh}_{\aA
}(X_{1},\ldots,X_{l})$ with flag bounds $b_{1} = \ldots = b_{l} = n$
is the usual complete symmetric function $h_{\lambda
}(x_{1},\ldots,x_{n})$, where $\lambda $ is the partition whose parts
are the nonzero parts of $\aA $.  The polynomials ${\fh}_{\aA
}(X_{1},\ldots,X_{l})$ appear in the study of Schubert polynomials and
related polynomials---see, e.g.,~\cite[\S
11]{FomGelPos97},~\cite{LascouxNaruse14}, and~\cite[Chapter
III]{Macdonald91}.

(ii)
The {\em flagged skew Schur function} indexed by a skew diagram
$\lambda /\mu $ with $l$ rows is defined by the $l\times l$
determinant
\begin{equation}\label{e:flagged-skew-schur}
\sfrak _{\lambda /\mu }(X_{1},\ldots,X_{l}) = \det \biggl(
h_{\lambda _{i}-\mu _{j}+j-i}[X_{1}+\cdots +X_{i}]\\
\biggr),
\end{equation}
generalizing the Jacobi-Trudi formula for an ordinary skew Schur
function.  From the Leibniz formula for the determinant, we see that
${\mathfrak s}_{\lambda /\mu }(X_{1},\ldots,X_{l})$ is a flagged
symmetric function in the sense of Definition~\ref{def:flagged-sym}.
Up to a change of notation for the formal alphabets $X_{i}$, flagged
skew Schur functions are the same as the multi-Schur functions studied
in~\cite{LascouxNaruse14} and~\cite[eq.~(3.1)]{Macdonald91}.

Specializing with flag bounds $b = b_{1},\ldots,b_{l}$ recovers the
flagged skew Schur functions of \cite{LascouxSchutz82}, which by
\cite{GessVien89, Wachs} are also given by the
combinatorial formula
\begin{equation}
\label{eq:flagschuridentity} \sfrak_{\lambda /\mu
}[X_1^b,\ldots,X_l^b] = \sum \nolimits _{T} x^{T}\, ,
\end{equation}
where the sum is over tableaux $T\in \SSYT(\lambda /\mu )$ such that
all entries in the $i$-th row of $T$ are bounded above by $b_{i}$.
\end{example}

The following property is immediate from the definitions.

\begin{lemma}\label{lem:xk=0}
Let $f[X^b_{1},\ldots,X^b_{l}]$ be the specialization of a flagged
symmetric function with flag bounds $b = b_{1},\ldots, b_{l}$.
Setting $x_{k} = 0$ in $f[X^b_{1},\ldots,X^b_{l}]$ and re-indexing the
variables $x_{k+1},\ldots,x_{n}$ to $x_{k},\ldots,x_{n-1}$ yields the
specialization $f[X^{b'}_{1},\ldots,X^{b'}_{l}]$ of $f$ with flag
bounds
\[
b'_{i} = \begin{cases}
b_{i}-1,&	\text{if $b_{i}\geq k$}\\
b_{i},&		\text{if $b_{i}< k$}.
\end{cases}
\]
\end{lemma}

The specialization with flag bounds $1,\ldots,l$, that is, with
the formal alphabets specialized to individual variables $X_{i} =
x_{i}$, plays a special role.

\begin{lemma}\label{lem:X=x}
The evaluation map
\begin{equation}\label{e:X=x-map}
\xi \colon \FLambda(X_{1},\ldots,X_{l}) \rightarrow \Kfrak
[x_{1},\ldots,x_{l}],\quad f(X_{1},\ldots,X_{l}) \mapsto
f[x_{1},\ldots,x_{l}],
\end{equation}
defined by specializing each $X_{i}$ to a single variable $x_{i}$, is
a $\Kfrak $-linear isomorphism from $\FLambda (X_{1},\ldots,X_{l})$ to
the polynomial ring $\Kfrak [x_{1},\ldots,x_{l}]$ in $l$ variables.
\end{lemma}

\begin{proof}
The specialized elements ${\fh}_{\aA }[x_{1},\ldots,x_{l}] =
h_{a_{1}}(x_{1})h_{a_{2}}(x_{1},x_{2})\cdots
h_{a_{l}}(x_{1},\ldots,x_{l})$ are unitriangularly related to
$h_{a_{1}}(x_{1})h_{a_{2}}(x_{2})\cdots h_{a_{l}}(x_{l}) =
x_{1}^{a_{1}}\cdots x_{l}^{a_{l}}$.
\end{proof}

\begin{remark}\label{rem:X=x}
The isomorphism $\xi $ is independent of $l$ in the following sense.
For $l<m$, the space $\FLambda(X_{1},\ldots,X_{l})$ embeds in $\FLambda
(X_{1},\ldots,X_{m})$ as the subspace consisting of flagged symmetric
functions $f(X_{1},\ldots,X_{m})$ which are independent of $X_{i}$ for
$i>l$.  If $f(X_{1},\ldots,X_{m})$ is in this subspace, then
$f[x_{1},\ldots,x_{m}]$ is a polynomial in $x_{1},\ldots,x_{l}$, and
the isomorphism for $l$ is the restriction of the one for $m$.
\end{remark}

\begin{remark}\label{rem:big enough bounds}
An arbitrary multi-symmetric function $f(X_1,\dots, X_l)$ is
determined by its specialization $f[X^b_1,\dots, X^b_l]$ with $X^b_i$
as in \eqref{e:X-for-flag-bounds} when the $b_i-b_{i-1}$ are
sufficiently large.  More precisely, expanding $f$ in the basis of
elements $m_{\lambold } = m_{\lambda^{(1)}}(X_1)
m_{\lambda^{(2)}}(X_2) \cdots m_{\lambda^{(l)}}(X_l)$ of the algebra
of multi-symmetric functions, the coefficient of $m_{\lambold}$ in
$f(X_1, \dots, X_l)$ is equal to the coefficient of $x^{\lambold}$ in
$f[X^{b}_{1}, \ldots, X^{b}_{l}]$ when $b_i - b_{i-1}\ge
\ell(\lambda^{(i)})$.

Lemma~\ref{lem:X=x} implies, by contrast, that if $f$ is a flagged
symmetric function, then the specialization with flag bounds $b =
1,2,\dots, l$ already determines $f$.
\end{remark}

In a sense, Lemma~\ref{lem:X=x} means that flagged symmetric functions
are no different from ordinary (nonsymmetric) polynomials.  The point,
however, is that by passing between the two, we can make use of
plethystic operations on flagged symmetric functions to
give natural constructions of new operations on
ordinary polynomials that would otherwise appear quite
artificial, or be difficult to define at all.

\subsection{Action of Demazure operators and Weyl symmetrization}

Let $f[X^b_{1},\ldots,X^b_{l}]$ be a flagged symmetric function
specialized with flag bounds $b = b_{1},\ldots,b_{l}$.  If $k$ is
equal to a flag bound, then $x_{k}$ and $x_{k+1}$ are not in the same
block of variables, and $f[X^b_{1},\ldots,X^b_{l}]$ is generally not
symmetric in $x_{k}, x_{k+1}$.  The next proposition says that if the
flag bound $b_{j} = k$ is not repeated, then the Demazure operator
$\Dem _{k}$ symmetrizes in $x_{k}, x_{k+1}$ by raising $b_{j}$.  Note,
this observation was made in~\cite[eq.~(3.10)]{Macdonald91} for
flagged Schur functions using a similar argument.

\begin{prop}\label{prop:demazure-on-flagged}
Given flag bounds $b_{1}\leq \cdots \leq b_{l}$, suppose that $b_{j} =
k$, and that $b_{i}\not =k$ for $i\not =j$.  If $j=l$, assume further
that $n>k$, where $n$ is the number of variables.  Define new flag
bounds $b'_{i}$ by setting $b'_{j}=k+1$ and $b'_{i} = b_{i}$ for
$i\not =j$.  Then the specializations with flag bounds
$b = b_{1},\ldots,b_{l}$ and $b' = b'_{1},\ldots,b'_{l}$ of any flagged
symmetric function are related by
\begin{equation}\label{e:demazure-on-flagged}
\Dem _{k} f[X^b_{1},\ldots,X^b_{l}] = f[X^{b'}_{1},\ldots,X^{b'}_{l}].
\end{equation}
\end{prop}

\begin{proof}
By linearity, we can assume that $f(X_{1},\ldots,X_{l}) = \fh_{\aA
}(X_{1},\ldots,X_{l})$.  Our assumptions imply that each factor
$h_{a_{i}}[X^b_{1}+\cdots +X^b_{i}] = h_{a_{i}}(x_{1}, \ldots,x_{b_{i}})$
for $i\not =j$ is symmetric in $x_{k}$ and $x_{k+1}$, hence commutes
with $\Dem _{k}$.  The result now follows from the identity
\begin{equation}\label{e:demazure-on-h}
\Dem _{k}\, h_{a}(x_{1},\ldots,x_{k}) = h_{a}(x_{1},\ldots,x_{k+1}).
\qedhere
\end{equation}
\end{proof}

Starting with flag bounds $1,\ldots,l$, then applying
Proposition~\ref{prop:demazure-on-flagged} repeatedly to raise $l$ to
$b_{l}$, then $l-1$ to $b_{l-1}$, and so on, we obtain the following
corollary.

\begin{cor}\label{cor:demazure-on-flagged}
Assuming that $b_{i}\geq i$ for all $i$, the specialization
$f[X^b_{1},\ldots,X^b_{l}]$ with flag bounds $b = b_{1},\ldots,b_{l}$ of any
flagged symmetric function $f(X_{1},\ldots,X_{l})$ is given in terms
of its evaluation $f[x_{1},\ldots,x_{l}]$ in single variable alphabets
by
\begin{equation}\label{e:demazure-on-single}
f[X^b_{1},\ldots,X^b_{l}] = \Dem _{w} f[x_{1},\ldots,x_{l}],
\end{equation}
where $w =(s_{b_{1}-1} s_{b_{1}-2}\cdots s_{1})\, (s_{b_{2}-1}
s_{b_{2}-2}\cdots s_{2})\cdots (s_{b_{l}-1} s_{b_{l}-2}\cdots s_{l}).$
\end{cor}

\begin{remark}\label{rem:demazure-on-flagged}
By combining Corollary~\ref{cor:demazure-on-flagged} with
Lemma~\ref{lem:xk=0}, one can express the specialization
$f[X^b_{1},\ldots,X^b_{l}]$ with any flag bounds $b$ (dropping the
hypothesis $b_{i}\geq i$) as the result of applying a Demazure
operator to $f[x_{1},\ldots,x_{l}]$, then setting some variables to
zero and re-indexing.
\end{remark}

\begin{cor}\label{cor:Weyl-on-flagged}
Given flag bounds $b = b_{1},\ldots,b_{l}$ and $j\leq k$, assume that
$b_{i}\geq b_{j-1}+i-j+1$ for $i=j,\ldots,k$ (where $b_{0} =0$, by
convention, if $j=1$).  Let $\Weyl $ be the Weyl symmetrization
operator in the variables $x_{b_{j-1}+1},\ldots, x_{b_{k}}$ that occur
in $X^{b}_{j}, \ldots, X^{b}_{k}$.  Then
\begin{equation}\label{e:Weyl-on-flagged-partial}
\Weyl f[X^b_{1},\ldots,X^b_{l}] =
f[X^{b}_{1},\ldots,X^{b}_{j-1},X^{b}_{j}+\cdots +X^{b}_{k},
0,\ldots,0, X^{b}_{k+1},\ldots, X^{b}_{l}],
\end{equation}
i.e., applying $\Weyl $ to the flag bound specialization
$f[X^b_{1},\ldots,X^b_{l}]$ has the effect of raising all the flag
bounds $b_{j},\ldots,b_{k}$ to $b_{k}$.  If $b_{i}\geq i$ for all $i$,
then Weyl symmetrization in all variables $x_{i}$ yields
\begin{equation}\label{e:Weyl-on-flagged}
\Weyl f[X^b_{1},\ldots,X^b_{l}] = f[X,0,\ldots,0],
\end{equation}
where $X = x_{1}+\cdots +x_{n}$.
\end{cor}

\begin{proof}
Define new flag bounds $b'$ and $b''$ by $b'_{j} =\cdots =b'_{k} =
b_{k}$ and $b''_{i} = b_{j-1}+i-j+1$ for $i=j,\ldots,k$, and $b'_{i} =
b''_{i} = b_{i}$ for $i\not \in [j,k]$.  Note that $b''_{i}\leq
b_{i}\leq b'_{i}$ for all $i$, and that, by definition,
$f[X^{b'}_{1},\ldots,X^{b'}_{l}]$ is equal to the right hand side of
\eqref{e:Weyl-on-flagged-partial}.  In particular,
$f[X^{b'}_{1},\ldots,X^{b'}_{l}]$ is symmetric in
$x_{b_{j-1}+1},\ldots, x_{b_{k}}$.

By the same procedure that gave
Corollary~\ref{cor:demazure-on-flagged}, raising the bounds $b''_{i}$
for $i=j,\ldots,k$ one at a time to $b_{i}$, we can find a Demazure
operator $\Dem _{v}$ that involves only the variables
$x_{b_{j-1}+1},\ldots, x_{b_{k}}$, such that
$f[X^b_{1},\ldots,X^b_{l}] = \Dem _{v}
f[X^{b''}_{1},\ldots,X^{b''}_{l}]$.  In the same way, we can find
$\Dem _{w}$ such that $f[X^{b'}_{1},\ldots,X^{b'}_{l}] = \Dem _{w}
f[X^{b''}_{1},\ldots,X^{b''}_{l}]$.  Since $\Weyl\, \Dem _{u} = \Weyl
$ for any Demazure operator $\Dem _{u}$ in the variables
$x_{b_{j-1}+1},\ldots, x_{b_{k}}$, it follows that $\Weyl
f[X^b_{1},\ldots,X^b_{l}] = \Weyl f[X^{b'}_{1},\ldots,X^{b'}_{l}]$,
and we have $\Weyl f[X^{b'}_{1},\ldots,X^{b'}_{l}] =
f[X^{b'}_{1},\ldots,X^{b'}_{l}]$, since
$f[X^{b'}_{1},\ldots,X^{b'}_{l}]$ is symmetric in
$x_{b_{j-1}+1},\ldots, x_{b_{k}}$.

If $k=l$, the same argument with $n\geq b_{l}$ in place of $b_{k}$
shows that Weyl symmetrization in $x_{b_{j-1}+1},\ldots,x_{n}$ yields
\begin{equation}\label{e:Weyl-on-flagged-tail}
\Weyl f[X^b_{1},\ldots,X^b_{l}] = f[X^{b}_{1}, \ldots, X^{b}_{j-1},
x_{b_{j-1}+1} +\cdots + x_{n}, 0,\ldots,0].
\end{equation}
For $j=1$, this is \eqref{e:Weyl-on-flagged}.
\end{proof}

\begin{example}\label{ex:flagged-sym-2}
(i) Given $\mu \in \NN ^{l}$, so the Demazure character $\Dcal _{\mu
}(x_{1},\ldots,x_{l})$ is a polynomial, let $\Dcal _{\mu
}(X_{1},\ldots,X_{l}) = \xi ^{-1}\Dcal _{\mu }(x_{1},\ldots,x_{l})$,
where $\xi $ is the isomorphism in Lemma~\ref{lem:X=x}.  Thus, $\Dcal
_{\mu }(X_{1},\ldots,X_{l})$ is the unique flagged symmetric function
such that $\Dcal _{\mu }[x_{1},\ldots,x_{l}] = \Dcal _{\mu
}(x_1,\ldots ,x_l)$.  Corollary~\ref{cor:demazure-on-flagged} implies
that for flag bounds $b = b_{1},\ldots,b_{l}$ such that $b_{i}\geq i$
for all $i$, and $n\geq b_{l}$, the specialization $\Dcal _{\mu
}[X^b_{1},\ldots,X^b_{l}]$ is a Demazure character $\Dcal _{w (\mu ;\,
0^{n-l})}(x_{1},\ldots,x_{n})$, for $w$ as in the corollary.  For
distinct flag bounds $0<b_{1}< \cdots < b_{l}$, $w (\mu ;\, 0^{n-l})$
is the interleaving of $\mu $ in positions $b_{1},\ldots,b_{l}$ with
zeroes in all other positions, giving
\begin{equation}\label{e:flagged-demazure}
\Dcal _{\mu }[X^b_{1},\ldots,X^b_{l}] = \Dcal _{(0^{b_{1}-1},\, \mu
_{1},\, 0^{b_{2}-b_{1}-1},\, \mu _{2},\, \dots\, ,\,
0^{b_{l}-b_{l-1}-1},\, \mu _{l},\, 0^{n-b_{l}})}(x),
\end{equation}
In particular, this shows that for any given $\mu $, the Demazure
characters on the right hand side of \eqref{e:flagged-demazure} all
fit together as specializations with varying flag bounds of the same
flagged symmetric function $\Dcal _{\mu }(X_{1},\ldots,X_{l})$.
These flagged symmetric functions are special cases of the
`multi-symmetric Schur functions' in \cite{BechtloffWeising25} (in the
general case, however, the latter are not in
the subspace of flagged symmetric functions).

(ii) Let $\lambda $ be a partition, possibly padded with zeroes to
length $l$.  By~\eqref{eq:flagschuridentity}, the flagged Schur
function $\sfrak _{\lambda }(X_{1},\ldots,X_{l})$ with $\mu =
\emptyset$, evaluated with flag bounds $1,2,\dots, l$, is $\sfrak
_{\lambda }[x_{1},\ldots,x_{l}] = x^{\lambda }$, since there is a
unique tableau $T\in \SSYT (\lambda )$ with entries in row $i$ bounded
by $i$. But we also have $x^{\lambda } = \Dcal _{\lambda
}(x_1,\ldots,x_l)$.  Hence, using Lemma~\ref{lem:X=x} and example (i),
above, we recover the theorem of Reiner--Shimozono \cite[Theorem
23]{RS} that any flagged Schur function evaluated with flag bounds
$b_{i}\geq i$ is given by
\begin{equation}
\sfrak _{\lambda }[X^b_{1},\ldots,X^b_{l}] = \Dcal _{w\, (\lambda ;\,
0^{n-\ell (\lambda )})}(x_1,\ldots, x_n)\,,
\end{equation}
with $w$ as in Corollary~\ref{cor:demazure-on-flagged}.  Note that
when $b_{1}<\cdots <b_{l}$, \eqref{e:flagged-demazure} with $\lambda $
in place of $\mu $ gives a simpler formula for this.  Further, since
$\sfrak _{\lambda }[X^b_{1},\ldots,X^b_{l}]$ only depends on the first
$\ell (\lambda )$ plethystic alphabets $X^b_{i}$, we may assume that $l = \ell
(\lambda )$.  Then, if the flag bounds do not satisfy $b_{i}\geq i$
for all $i$, the set of tableaux $T\in \SSYT (\lambda )$ with entries
in row $i$ bounded by $b_{i}$ is empty, implying that $\sfrak
_{\lambda }[X^b_{1},\ldots,X^b_{l}] = 0$.
\end{example}

The next proposition is akin to Example~\ref{ex:flagged-sym-2}(i), but
for atoms instead of Demazure characters.  Here, $\xi $ is again the
isomorphism in Lemma~\ref{lem:X=x}.

\begin{prop}\label{prop:flagged-atom}
Given $\mu \in \NN ^{l}$, let $\Acal _{\mu }(X_{1},\ldots,X_{l}) = \xi
^{-1}\Acal _{\mu }$ be the unique flagged symmetric function such that
$\Acal _{\mu }[x_{1},\ldots,x_{l}]$ is the Demazure atom $\Acal _{\mu
}(x_{1},\ldots,x_{l})$.  Then, for any flag bounds $b =
b_{1},b_{2},\ldots,b_{l}$ and $n\geq b_{l}$, the specialization $\Acal
_{\mu } [X^b_{1},\ldots,X^b_{l}]$ is a multiplicity-free sum of atoms
\begin{equation}\label{e:flagged-atom}
\Acal _{\mu } [X^b_{1},\ldots,X^b_{l}] = \sum _{\eta \in S}\Acal _{\eta
}(x_{1},\ldots,x_{n}),
\end{equation}
where $S$ is a (possibly empty) set of elements $\eta \in \NN ^{n}$
such that the nonzero parts of $\eta $ are a permutation of the
nonzero parts of $\mu $.
\end{prop}

\begin{proof}
If $b_{i}\geq i$ for all $i$, this follows from
Corollary~\ref{cor:demazure-on-flagged} and Remark \ref{rem:for atom
pos}(iii), which says that for any $w\in \Sfrak _{n}$, $\Dem _{w}\,
\Acal _{\mu }$ is a sum of the form in \eqref{e:flagged-atom}, with
every $\eta $ a permutation of $(\mu ; 0^{n-l})$.  In the general
case, Lemma~\ref{lem:xk=0} implies that we can obtain $\Acal _{\mu }
[X^b_{1},\ldots,X^b_{l}]$ from another specialization $\Acal _{\mu }
[X^{b'}_{1},\ldots,X^{b'}_{l}]$ with flag bounds satisfying
$b'_{i}\geq i$, by setting some of the variables to zero and
re-indexing.  By Corollary~\ref{cor:Amu-with-zeroes}, this procedure
turns any sum of the form in \eqref{e:flagged-atom} into another sum
of the same form.
\end{proof}

\subsection{Flagged Cauchy identity} \label{ss:flagged-Cauchy}
The generating function $\Omega [X\, z^{-1}] = \sum _{k=0}^{\infty }
z^{-k}\, h_{k}(X)$ for complete symmetric functions immediately
implies the following `Cauchy identity' for flagged symmetric
functions:
\begin{equation}\label{e:flagged-Cauchy}
 \Omega [\sum _{i\leq j} \!\! X_{i}\, z_{j}^{-1} ] = \sum _{\aA
\in \NN ^{l}} z^{-\aA }\, \fh_{\aA }(X_{1}, X_{2}, \ldots, X_{l}).
\end{equation}
In particular, the coefficients of $\Omega [\sum _{i\leq j}X_{i}\,
z_{j}^{-1}]$ as a formal power series in
$z_{1}^{-1},\ldots,z_{l}^{-1}$ belong to $\FLambda (X_{1},\ldots,X_{l})$.
Hence,
\begin{equation}\label{e:xi-on-omega}
\Omega [\sum _{i\leq j}X_{i}\, z_{j}^{-1}] = \xi ^{-1} \Omega [\sum
_{i\leq j} x_{i}/z_{j}],
\end{equation}
where $\xi $ is the isomorphism in Lemma~\ref{lem:X=x}, and we apply
$\xi ^{-1}$ coefficient-wise to $\Omega [\sum _{i\leq j} x_{i}/z_{j}]$
considered as a power series in the $z_{i}^{-1}$ with coefficients in
$\Kfrak [x_{1},\ldots,x_{l}]$.

\subsection{Coproduct}
\label{ss:translation}

The standard coproduct $\Delta \colon \Lambda (X)\rightarrow \Lambda
(X)\otimes \Lambda (Y)$ on symmetric functions is given by the
plethystic evaluation $\Delta f(X) = f[X+Y]$.  Taking the coefficient
of $t^{k}$ in $\Omega [t(X+Y)] = \Omega [t X]\, \Omega [t Y]$ gives
the coproduct formula 
\begin{equation}\label{e:coproduct-of-h}
h_{k}[X+Y] = \sum _{p+q=k} h_{p}[X] h_{q}[Y],
\end{equation}
valid as a plethystic identity for any $X$ and $Y$.

The coproduct on symmetric functions extends to a coproduct on
multi-symmetric functions by applying \(X_i \mapsto X_i+Y_i\) in each
formal alphabet.  By the next proposition, this coproduct preserves
flagged symmetric functions.

\begin{prop}
\label{p coprod formula} For \(f(X_1,\ldots,X_l) \in
\FLambda(X_1,\ldots,X_l)\), the multi-symmetric function \(
f[X_1+Y_1,\ldots,X_l+Y_l]\) lies in \(\FLambda(X_1,\ldots,X_l) \otimes
\FLambda(Y_1,\ldots,Y_l)\). Explicitly,
\begin{equation}
\label{e coprod formula} \fh_{\aA}[X_1+Y_1,\ldots,X_l+Y_l] = \sum_{\pp
+ \qq = \aA} \fh_{\pp}(X_1,\ldots,X_l)\,  \fh_{\qq}(Y_1,\ldots,Y_l) \,,
\end{equation}
where all indices $\aA $, $\pp $, and $\qq $ are in $\NN ^{l}$.
\end{prop}

\begin{proof}
By definition, \(\fh_{\aA}[X_1+Y_1,\ldots,X_l+Y_l] =
h_{a_1}[X_1+Y_1]h_{a_2}[X_1+X_2+Y_1+Y_2] \cdots \). Now, apply the
coproduct formula $h_{a_k}[X_1+\dots +X_k+Y_1 + \dots + Y_k] =
\sum_{p_k+q_k = a_k} h_{p_k}[X_1+\cdots+X_k] h_{q_k}[Y_1+\cdots+Y_k]$
on each factor.
\end{proof}

\section{Nonsymmetric plethysm}
\label{s:ns-plethysm}

In this section we introduce the operation $\Pi _{A,x}$ of {\em
nonsymmetric plethysm} on polynomials in variables
$x_{1},\ldots,x_{n}$.  We view $\Pi _{A,x}$ and its inverse $\Pi_{A,x}^{-1}$ as natural nonsymmetric analogs of the plethystic
transformations on symmetric functions $f(X)\mapsto {f[X/(1-A)]}$ and
$f(X)\mapsto f[(1-A) X]$, respectively.

\subsection{Initial definition}
\label{ss:ns-pleth-def}

From here through \S \ref{ss:symmetric-limit}, we define $\Pi _{A,x}$,
characterize it in various ways, and show that the transformations
$f(X)\mapsto {f[X/(1-A)]}$ and $f(X)\mapsto f[(1-A) X]$ on symmetric
functions are limiting cases of $\Pi _{A,x}$ and $\Pi _{A,x}^{-1}$.
To develop the theory, we work in the context where $A$ is a formal
alphabet.  Afterwards we explain how to specialize to the
one-parameter case $A=t$, where $t$ is an indeterminate in the
coefficient ring, which is what we use for the applications to flagged
LLT polynomials and nonsymmetric Macdonald polynomials in \S \S
\ref{s:flagged-LLT}--\ref{s:ns-Mac-pols}.  For now, we don't restrict
ourselves to the one-parameter case, because the general case is no
harder, clarifies the reasoning, and may have other potential
applications.

To define $\Pi _{A,x}$ in the generic context, let $\Lambda (A) =
\Lambda _{\Kfrak }(A)$ be the algebra of symmetric functions in a
formal alphabet $A$ over a coefficient ring $\Kfrak $.  As in \S
\ref{s:flagged-symmetric-functions}, all plethystic evaluations of
\multisymmetric functions below will be of the form in \S
\ref{ss:plethysm}, (iv), defined for any $\Kfrak $.

Recall from Definition~\ref{def:flagged-sym} that the space $\FLambda
_{\Lambda (A)}(X_{1},\ldots,X_{n})$ of flagged symmetric functions
over $\Lambda (A)$ has a basis consisting of elements $\fh _{\aA
}(X_{1},\ldots,X_{n})$ for $\aA \in \NN^n$.  Evaluating them with
$X_{1} = x_{1}$ and $X_{i} = x_{i}-A\, x_{i-1}$ for $i>1$ yields
polynomials
\begin{multline}\label{e:PiA-preimage-of-ha}
\fh _{\aA }[x_{1},\, x_{2}-A\, x_{1},\, \ldots,\, x_{n}-A\, x_{n-1}] \\
= h_{a_{1}}[x_{1}]\, h_{a_{2}}[x_{2}+(1-A)\, x_{1}]\cdots
h_{a_{n}}[x_{n}+(1-A)(x_{1}+\cdots +x_{n-1})].
\end{multline}
By the same triangularity argument as in the proof of
Lemma~\ref{lem:X=x}, these polynomials form a basis of $\Lambda
(A)[x_{1},\ldots,x_{n}]$ as a free $\Lambda (A)$ module.

\begin{defn}\label{def:PiA}
The {\em nonsymmetric plethysm } $\Pi _{A,(x_{1},\ldots,x_{n})}$, or
$\Pi _{A,x}$ for short, is the invertible $\Lambda (A)$-linear
operation on $\Lambda (A)[x_{1},\ldots,x_{n}]$ given on the basis in
\eqref{e:PiA-preimage-of-ha} by
\begin{equation}\label{e:PiA-definition}
\Pi _{A,x} \, \fh _{\aA }[x_{1},\, x_{2}-A\, x_{1},\, \ldots,\,
x_{n}-A\, x_{n-1}] = \fh _{\aA }[x_{1},\, \ldots,\, x_{n}].
\end{equation}
\end{defn}
For convenience, we define
\begin{equation}\label{e:htild}
\tfh_{\aA}(X_1,\ldots,X_n) = \fh_{\aA}[X_1,X_2-AX_1, \ldots,
X_n-AX_{n-1}]
\end{equation}
(a multi-symmetric function, but not, in general, a flagged symmetric
function), and write
\begin{equation}\label{e:flag-h-pleth}
\fh_{\aA} = \fh_{\aA}[x_1, \ldots, x_n]\,, \qquad \tfh_{\aA} =
\tfh_{\aA}[x_1, \ldots, x_n] = \fh_{\aA}[x_1,x_2-Ax_1, \ldots,
x_n-Ax_{n-1}]\,,
\end{equation}
in terms of which \eqref{e:PiA-definition} takes the form $\Pi _{A,x}
\, \tfh_{\aA } = \fh_{\aA }.$

It can be helpful to picture the definition in the following way.
Define a deformation $\xi _{A} \colon \FLambda _{\Lambda (A)}
(X_{1},\ldots,X_{n}) \rightarrow \Lambda (A) [x_{1},\ldots,x_{n}]$ of
the isomorphism $\xi $ in Lemma~\ref{lem:X=x} by
\begin{equation}\label{e:xi-A}
\xi _{A}\, f(X_{1},\ldots,X_{n}) = f[x_{1},\, x_{2}-A\, x_{1},\,
\ldots,\, x_{n}-A\, x_{n-1}].
\end{equation}
Then $\Pi  _{A,x}$ and $\Pi  _{A,x}^{-1}$ are determined by the
following commutative diagram.
\begin{equation}\label{e:cd-for-PiA}
\begin{tikzcd}[row sep=small]
&&
\mbox{$\Lambda (A)[x_{1},\ldots,x_{n}]$}
\arrow[dd, shift left=3, "\Pi _{A,x}^{-1}"]
\\
\mbox{$\FLambda _{\Lambda (A)} (X_{1},\ldots,X_{n})$}
\arrow[urr, shift left=2, start anchor=east, end anchor=south west, "\xi "]
\arrow[drr, shift right=2, start anchor=east, end anchor=north west,
   "\xi _{A}"']
\\
&&
\mbox{$\Lambda (A)[x_{1},\ldots,x_{n}]$}
\arrow[uu, shift left=1, "\Pi _{A,x}"]
\end{tikzcd}
\end{equation}

\begin{remark}\label{rem:PiA-clarifications}
(i) In the definition of $\Pi _{A,(x_{1},\ldots,x_{n})}$, the ordering
of the variables $x_{1},\ldots,x_{n}$ matters.  However, the number of
variables does not really matter: if $f$ is a function of the first
$m$ variables $x_{1},\ldots,x_{m}$, then $\Pi
_{A,(x_{1},\ldots,x_{m})} f(x) = \Pi _{A,(x_{1},\ldots,x_{n})} f(x)$.
This follows directly from the definition, using the fact that
${\fh}_{\aA}(X_{1},\ldots,X_{m}) = {\fh}_{(\aA;\,
0^{n-m})}(X_{1},\ldots,X_{n})$ for $\aA \in \NN ^{m}$.

(ii) It also follows directly from the definition that the operations
$\Pi _{A,x}^{\pm 1}$ on $\Lambda _{\Kfrak }(A)[x_{1},\ldots,x_{n}] =
\Kfrak \otimes _{\ZZ } \Lambda _{\ZZ }(A)[x_{1},\ldots,x_{n}]$
coincide with those obtained by extension of scalars $\Kfrak \otimes
_{\ZZ } (-)$ from the operations $\Pi _{A,x}^{\pm 1}$ on $\Lambda
_{\ZZ }(A)[x_{1},\ldots,x_{n}]$.  In fact, all of our results on
flagged symmetric functions and flagged plethysm here and in \S
\ref{s:flagged-symmetric-functions} follow by extension of scalars
from the case $\Kfrak =\ZZ $.
\end{remark}

\begin{example}
\label{ex nspleth n2} We calculate the plethysm $\Pi_{A,x}$ on
monomials $x^{\aA }\in \Lambda(A)[x_1,\dots, x_n]$ in two cases.

(i) For monomials $x_{n}$ of degree $1$, for any $n$, we first note
that Definition \ref{def:PiA} gives
\begin{equation}\label{e:PiA-inv-h1 v0}
\Pi _{A,x}^{-1}\, \fh_{(0^{n-1};1)} = \tfh_{(0^{n-1};1)} =
x_{1}+\cdots +x_{n} - h_{1}(A)(x_{1}+\cdots +x_{n-1}).
\end{equation}
Writing $x_{n} = h_{1}[x_{1}+\cdots +x_{n}] - h_{1}[x_{1}+\cdots
+x_{n-1}] = \fh _{(0^{n-1};1)} - \fh _{(0^{n-2};1)}$,
we obtain
\begin{equation}\label{e:PiA-inv-y1}\Pi _{A,x}^{-1}\, x_n = x_n -
h_{1}(A)\, x_{n-1},\quad \text{or, equivalently,}\quad \Pi_{A,x}\, x_n
=x_n+h_1(A)\, \Pi_{A,x}\, x_{n-1}\,.
\end{equation}
Iterating this gives
\begin{gather}\label{e:PiA-xn}
\Pi _{A,x}\, x_{n} = x_{n} + h_{1}(A)\, x_{n-1} +\cdots
+h_{1}(A)^{n-1}\, x_{1}.
\end{gather}

(ii) For any monomial in $n=2$ variables, we note that the bases
$\fh_\aA$ and $\tfh_\aA$ of $\Lambda(A)[x_1,x_2]$ are given by
$\fh_{a_1, a_2} = h_{a_1}[x_1]h_{a_2}[x_1+x_2]$ and $\tfh_{a_1, a_2} =
h_{a_1}[x_1]h_{a_2}[x_2+(1-A)x_1]$.  Setting $Y = x_{2}+(1-A)x_{1}$,
we have
\begin{multline}\label{e:h-by-ht-n=2}
\fh _{a_{1},a_{2}} = h_{a_1}[x_1]\, h_{a_2}[Y+A\, x_{1}] =
h_{a_1}[x_1]\sum _{i+j=a_{2}} h_{j}[A \, x_{1}]\, h_{i}[Y] \\
= \sum _{i+j=a_{2}} h_{j}(A) \, h_{a_{1}+j}[x_{1}]\, h_{i}[Y] = \sum
_{i+j=a_{2}} h_{j}(A)\, \tfh _{a_{1}+j, i}.
\end{multline}
We can now compute $\Pi_{A,x}\, x_2^2$, for instance, by first
expressing $x_{2}^{2}$ in terms of the $\tfh _{a_{1},a_{2}}$, as
follows:
\begin{gather}\label{e:expand-x22}
\begin{aligned}
x_2^2 & = (x_1^2+x_1x_2 + x_2^2) - x_1 (x_1+x_2)= \fh_{0,2} -\fh_{1,1} \\
& = \big(h_2(A)\, \tfh_{2,0} + h_1(A)\, \tfh_{1,1} + \tfh_{0,2}\big)
-\big(h_1(A)\, \tfh_{2,0} + \tfh_{1,1}\big) \qquad
\text{by \eqref{e:h-by-ht-n=2}}\\
& = (h_2(A)-h_1(A)) \, \tfh_{2,0} + (h_1(A)-1)\, \tfh_{1,1}+
\tfh_{0,2}\, .
\end{aligned}\\
\intertext{Then}
\label{e:compute-PiA-xee}
\begin{aligned}
\Pi_{A,x}\, x_2^2 & = (h_2(A)-h_1(A)) \, \fh_{2,0} + (h_1(A)-1)\,
\fh_{1,1}+ \fh_{0,2}\\
& = x_2^2 + h_1(A)\, x_1 x_2 + h_2(A)\, x_1^2 \,.
\end{aligned}
\end{gather}
A similar computation shows that, more generally,
\begin{equation}
\label{e nspleth n2}
\Pi_{A, x}\, x_1^a x_2^b = x_1^a x_2^b + h_{1}(A) x_1^{a+1}x_2^{b-1} +
h_{2}(A) x_1^{a+2}x_2^{b-2} + \cdots + h_{b}(A) x_1^{a+b}.
\end{equation}
In particular, when $A = t$, this reduces to
\begin{equation}\label{e nspleth n2 t}
\Pi_{A, x}\, x_1^a x_2^b = x_1^a x_2^b + t \, x_1^{a+1}x_2^{b-1} + t^2
x_1^{a+2}x_2^{b-2} + \cdots + t^b x_1^{a+b}.
\end{equation}
\end{example}

We will present a number of alternative characterizations of $\Pi
_{A,x}$ and its inverse.  The first is essentially just a
reformulation of the definition.

\begin{prop}\label{prop:PiA-on-Omega}
We have the identity
\begin{equation}\label{e:PiA-on-Omega}
\Pi _{A,x}\, \Omega [\sum _{i\leq j}x_{i}/z_{j} - A\sum
_{i<j}x_{i}/z_{j}] = \Omega [\sum _{i\leq j}x_{i}/z_{j}],
\end{equation}
where $\Pi _{A,x}$ is applied to each coefficient in the expansion of
$\Omega [\sum _{i\leq j}x_{i}/z_{j} - A\sum _{i<j}x_{i}/z_{j}]$ as a
power series in $z_{1}^{-1},\ldots,z_{n}^{-1}$ over $\Lambda
(A)[x_{1},\ldots,x_{n}]$.
\end{prop}

\begin{proof}
Specializing the flagged Cauchy identity \eqref{e:flagged-Cauchy} in
two ways, we obtain
\begin{align}\label{e:flagged-Cauchy-with-A}
\Omega [\sum _{i\leq j}x_{i}/z_{j} - A\sum _{i<j}x_{i}/z_{j}] & = \sum
_{\aA \in \NN ^{n}} z^{-\aA }\, \fh _{\aA }[x_{1},\, x_{2}-A\, x_{1},\,
\ldots,\, x_{n}-A\, x_{n-1}];\\
\label{e:flagged-Cauchy-without-A} \Omega [\sum _{i\leq j}x_{i}/z_{j}
] & = \sum _{\aA \in \NN ^{n}} z^{-\aA }\, \fh _{\aA }[x_{1}, x_{2},
\ldots, x_{n}].
\end{align}
This given, \eqref{e:PiA-on-Omega} is immediate from the definition of
$\Pi _{A,x}$.
\end{proof}

Another straightforward consequence of the definition is a formula for
inverse nonsymmetric plethysm on a flagged symmetric function
specialized with strictly increasing flag bounds.

\begin{prop}\label{prop:PiA-inv-flag-bounds}
For any flagged symmetric function $f\in \FLambda _{\Lambda
(A)}(X_{1},\ldots,X_{l})$, the inverse nonsymmetric plethysm of its
specialization with flag bounds $b_{1}<\cdots <b_{l}$ is given by
\begin{equation}\label{e:PiA-inv-flag-bounds}
 \Pi _{A,x}^{-1}\, f[X_1^b,\, X_2^b,\, \ldots,\, X_l^b] = \, f
\bigl[X_1^b-A\,{X'_1}^{b} ,\, X_2^b-A\,{X'_2}^{b}
,\ldots,X_{l}^b-A\,{X'_{l\,}}^{b}\bigr]\,,
\end{equation}
where $X^b_i = x_{b_{i-1}+1} + \cdots + x_{b_i}$, as in
Definition~\ref{def:flag-bounds}, and we set $X' =
(0,x_{1},x_{2},\ldots)$, so ${X'_{i\,}}^b = x_{b_{i-1}}+ \cdots +
x_{b_i-1}$.
\end{prop}

\begin{proof}
If $b_{1} = 0$, the result reduces to the case with $l-1$ formal alphabets
$X_{i}$ and flag bounds $b_{2},\ldots,b_{l}$.  Assume now that
$b_{1}>0$.  Reducing by linearity to the case $f=\fh _{\aA
}(X_{1},\ldots,X_{l})$, the specialization of $f$ with flag bounds
$b_{i}$ becomes $\prod _{i} h_{a_{i}}[x_{1}+\cdots +x_{b_{i}}]$.
Since the $b_{i}$ are distinct and nonzero, this is equal to $\fh
_{\aA '}[x_{1},\ldots,x_{b_{l}}]$, where $\aA '_{b_{i}} = a_{i}$ and
$\aA '_{j} = 0$ for all other $j$.  The left hand side of
\eqref{e:PiA-inv-flag-bounds} is therefore $\Pi _{A,x}^{-1} \, \fh
_{\aA '}[x_{1},\ldots,x_{b_{l}}]$, which is equal to $\fh _{\aA
'}[x_{1},x_{2}-A\, x_{1},\ldots,x_{l}-A\, x_{l-1}]$ by definition.
This is the same as the right hand side of
\eqref{e:PiA-inv-flag-bounds} for $f=\fh _{\aA }$.
\end{proof}

\subsection{Inner product formulas for nonsymmetric plethysm}
\label{ss:PiA-via-inner}

Define an inner product
\begin{align}\label{e:A-inner-product}
\langle f(x),\, g(x) \rangle _{A}\, &	\defeq\, \langle x^{0} \rangle
f(x)\, g(x)\, \Omega [(A-1)\sum _{i<j} x_{i}/x_{j}]\\
\label{e:A-inner-product-2}
& = \langle f(x),\, g(x)\, \Omega [A\, \sum _{i<j} x_{i}/x_{j}]
\rangle _{0}
\end{align}
on $\Lambda (A)[x_{1}^{\pm 1},\ldots,x_{n}^{\pm 1}]$, generalizing the
inner product $\langle -,- \rangle _{0}$ in
\eqref{e:atom-demazure-inner}, as well as the inner product $\langle
-,- \rangle _{t}$ in \eqref{e:t-inner}, which will be discussed when
we get to nonsymmetric Hall-Littlewood polynomials.  We will express
the nonsymmetric plethysm $\Pi _{A,x}$ in terms of the inner product
$\langle -,- \rangle _{A}$ by means of two different formulas.  To
derive these formulas we first need the following `reproducing kernel'
property of $\langle -,- \rangle _{A}$.

\begin{lemma}\label{lem:reproducing-kernel}
For all $f(x)\in \Lambda (A)[x_{1}, ,\ldots,x_{n}]$, we have
\begin{equation}\label{e:reproducing-kernel}
f(x) = \langle f(z), \, \Omega [\sum _{i\leq j}x_{i}/z_{j} - A\sum
_{i<j}x_{i}/z_{j}] \rangle _{A},
\end{equation}
where the inner product is in the $z$ variables, and applies to each
coefficient in the expansion of $\Omega [\sum _{i\leq j}x_{i}/z_{j} -
A\sum _{i<j}x_{i}/z_{j}]$ as a power series in $x$ over the Laurent
polynomial ring $\Lambda (A)[z_{1} ^{\pm 1}, \ldots, z_{n} ^{\pm 1}]$.
\end{lemma}

\begin{proof}
The right hand side of \eqref{e:reproducing-kernel} is equal to
\begin{equation}\label{e:reproducing-kernel-rhs}
\begin{aligned}
& \langle z^{0} \rangle\, f(z)\, \Omega [\sum _{i\leq j}x_{i}/z_{j} -
\sum _{i<j} z_{i}/z_{j} + A \, (\sum _{i<j}z_{i}/z_{j} -\sum _{i<j}
x_{i}/z_{j})]\\
 = & \langle z^{0} \rangle\, f(z)\, \frac{\prod _{i<j}(1 -
z_{i}/z_{j})}{ \prod _{i\leq j} (1 - x_{i}/z_{j})}\, \Omega [A \,
(\sum _{i<j}z_{i}/z_{j} -\sum _{i<j} x_{i}/z_{j})].
\end{aligned}
\end{equation}
Since $f(z)$ is a polynomial, everything in
\eqref{e:reproducing-kernel-rhs} other than the factor $1 -
x_{1}/z_{1}$ in the denominator is a power series in $z_{1}$.
Expanding $(1-x_{1}/z_{1})^{-1}$ as a geometric series in
$x_{1}/z_{1}$, we have $\langle z_{1}^{0} \rangle \, g(z_{1})\,
(1-x_{1}/z_{1})^{-1} = g(x_{1})$ for any power series $g(z_{1})$.
Thus, after taking the constant term $\langle z_{1}^{0} \rangle$ in
\eqref{e:reproducing-kernel-rhs}, and observing that apart from the
factor $(1-x_{1}/z_{1})^{-1}$, everything with index $i=1$ cancels
when $z_{1} = x_{1}$, we are left with
\begin{equation}\label{e:reproducing-kernel-reduced}
\langle z_{2}^{0}\cdots z_{n}^{0} \rangle f(x_{1},
z_{2},\ldots,z_{n})\, \frac{\prod _{1<i<j}(1 - z_{i}/z_{j})}{ \prod
_{1<i\leq j} (1 - x_{i}/z_{j})}\, \Omega [A\, (\sum
_{1<i<j}z_{i}/z_{j} -\sum _{1<i<j} x_{i}/z_{j})].
\end{equation}
This is equal to $f(x)$ by induction on the number of variables, the
one-variable case being trivial.
\end{proof}

\begin{prop}\label{prop:PiA-via-inner-product}
For all $f(x)\in \Lambda (A)[x_{1},\ldots,x_{n}]$, we have the
identity
\begin{equation}\label{e:PiA-via-inner-product}
\Pi _{A,x}f(x) = \langle f(z),\, \Omega [\sum _{i\leq j}x_{i}/z_{j}]
\rangle _{A},
\end{equation}
where the inner product is in the $z$ variables.

\end{prop}

\begin{proof}
Apply $\Pi _{A,x}$ to each side of \eqref{e:reproducing-kernel},
noting that $\Pi _{A,x}$ moves inside the inner product, since it
operates on the $x$ variables.  Then use
Proposition~\ref{prop:PiA-on-Omega}.
\end{proof}

\begin{cor}\label{cor:PiA-inner-products}
For all $f(x)\in \Lambda (A)[x_{1},\ldots,x_{n}]$ and $g(x)\in \Lambda
(A)[x_{1}^{-1},\ldots,x_{n}^{-1}]$, we have
\begin{equation}
\label{e:PiA-inner-products}
\langle \Pi _{A,x}f(x),g(x) \rangle_{0} = \langle f(x), g(x)
\rangle_{A},
\end{equation}
where $\langle -,- \rangle_{0}$ is as in
\eqref{e:atom-demazure-inner}.
\end{cor}

\begin{proof}
Take $\langle -, g(x) \rangle_0$ on both sides of
\eqref{e:PiA-via-inner-product}.  On the right hand side, the inner product
in the $x$ variables moves inside inner product in the $z$ variables
in \eqref{e:PiA-via-inner-product}, replacing $\Omega [\sum _{i\leq
j}x_{i}/z_{j}]$ with $\langle g(x),\, \Omega [\sum _{i\leq
j}x_{i}/z_{j}] \rangle_{0}$, which is equal to $g(z)$, by
\eqref{e:atom-demazure-reproducing-2}.  This gives $\langle f(z), g(z)
\rangle_{A}$, which is the same as $\langle f(x), g(x) \rangle_{A}$.
\end{proof}

\subsection{Polynomial part formula for nonsymmetric plethysm}
\label{ss:PiA-via-pol}

For the next characterization of $\Pi _{A,x}$, we first introduce a
suitable notion of `polynomial part' of a Laurent polynomial.
For the moment, we denote the coefficient ring by $\Rfrak $
instead of $\Kfrak $, to avoid confusion when we take
$\Rfrak =\Lambda (A)\defeq \Lambda _{\Kfrak }(A)$.

Let $N$ be the orthogonal complement of $\Rfrak [\xx ^{-1}] = \Rfrak
[x_{1}^{-1}, \ldots, x_{n}^{-1}]$ in $\Rfrak [\xx ^{\pm 1}] = \Rfrak
[x_{1}^{\pm 1},\ldots,x_{n}^{\pm 1}]$, with respect to the inner
product $\langle -,- \rangle_{0}$ in \eqref{e:atom-demazure-inner}.
Since both $\{ \Dcal _{\lambda } \mid \lambda \in -\NN ^{n}\}$ and $\{
\Acal _{\lambda } \mid \lambda \in -\NN ^{n}\}$ span $\Rfrak [\xx
^{-1}]$, we have
\begin{equation}\label{e:N-explicit}
N = \Rfrak \{\Dcal _{\lambda }\mid \lambda \not \in \NN ^{n} \} =
\Rfrak \{\Acal _{\lambda }\mid \lambda \not \in \NN ^{n} \}.
\end{equation}
In particular, this implies that we have a direct sum decomposition
$\Rfrak [\xx ^{\pm 1}] = \Rfrak [\xx ]\oplus N$.

\begin{defn}\label{def:pol-part}
The {\em polynomial part} $f(x)_{\pol }$ of a Laurent polynomial
$f(x)\in \Rfrak [\xx ^{\pm 1}]$ is the image of $f(x)$ under the
projection of $\Rfrak [\xx ^{\pm 1}]$ on $\Rfrak [\xx ]$ with kernel
$N$.  In other words, $f(x)_{\pol }$ is the truncation of either
expansion $f(x) = \sum _{\lambda \in \ZZ ^{n}} c_{\lambda }\, \Dcal
_{\lambda }(x)$ or $f(x) = \sum _{\lambda \in \ZZ ^{n}} b_{\lambda }
\, \Acal _{\lambda }(x)$ to terms indexed by $\lambda \in \NN ^{n}$.
\end{defn}

\begin{example}
Using $x_1^{-1}x_2 = \Acal_{(-1, 1)} - \Acal_{(0,0)}$, we have
 $(x_1^{-1}x_2)_{\pol} = -\Acal_{(0,0)} = -1$.
This shows that the polynomial part operator is not the same as the na\"ive truncation of a Laurent
polynomial to monomials with nonnegative exponents.
\end{example}

\begin{remark}\label{rem:pol-part}
The restriction of the polynomial part operator to symmetric Laurent
polynomials recovers the following familiar notion of polynomial
truncation: it takes the sum $f(x) = \sum _{\lambda \in X_{+}}
c_{\lambda }\, \chi _{\lambda }$ to its truncation $f(x)_{\pol } =
\sum _{\lambda \in \NN ^{n}\cap X_{+}} c_{\lambda }\, \chi _{\lambda
}$ to polynomial irreducible characters of $\GL _{n}$.
\end{remark}

The Cauchy identity for Demazure characters and atoms
\eqref{e:atom-demazure-Cauchy} immediately implies the following
formula for $f(x)_{\pol }$.

\begin{lemma}\label{lem:pol-part}
For all $f\in \Rfrak [\xx ^{\pm 1}]$, we have
\begin{equation}\label{e:pol-part}
f(x)_{\pol } = \langle f(z), \Omega [\sum _{i\leq j} x_{i}/z_{j}]
\rangle _{0},
\end{equation}
where the inner product is in the $z$ variables.
\end{lemma}

Now we can give a formula for $\Pi _{A,x}$ in terms of the polynomial
part defined as above.

\begin{prop}\label{prop:PiA-via-pol}
For all $f\in \Lambda (A)[x_{1},\ldots,x_{n}]$, we have
\begin{equation}\label{e:PiA-via-pol}
\Pi _{A,x} f(x) = \bigl( f(x)\, \Omega [A \sum _{i<j} x_{i}/x_{j}]
\bigr)_{\pol },
\end{equation}
where $(-)_{\pol} $ applies term by term in the expansion of $f(x)\,
\Omega [A \sum _{i<j} x_{i}/x_{j}]$ as a symmetric formal series in
$A$ with coefficients in $\Kfrak [\xx ^{\pm 1}]$.
\end{prop}

\begin{proof}
Using Lemma~\ref{lem:pol-part}, formula \eqref{e:A-inner-product-2},
and Proposition~\ref{prop:PiA-via-inner-product}, we have
\begin{equation}\label{e:PiA-via-pol-proof}
\begin{aligned}
\bigl( f(x)\, \Omega [A \sum _{i<j} x_{i}/x_{j}] \bigr)_{\pol }
&  = \langle f(z) \, \Omega [A \sum
_{i<j} z_{i}/z_{j}],\, \Omega [\sum _{i\leq j} x_{i}/z_{j}]  \rangle _{0}\\
&  = \langle f(z),\Omega [\sum _{i\leq j} x_{i}/z_{j}] \rangle _{A}\\
& = \Pi _{A,x} f(x).
\end{aligned}
\end{equation}
All steps hold term by term as symmetric series in $A$.
\end{proof}

Note that, although the right hand side of \eqref{e:PiA-via-pol} is
{\it a priori} an infinite series, it is a consequence of the proof
above that the series terminates.  Alternatively, we can see directly
that only finitely many terms of the series contribute to the
polynomial part, using the following proposition.

\begin{prop} \label{p: pol part tail}
If ${\aA} \in \ZZ^n$ with $a_k + a_{k+1} + \dots + a_n < 0$ for some
$k$, then $(x^{\aA })_{\pol} = 0$.
\end{prop}
\begin{proof}
By Lemma \ref{lem:pol-part},
\begin{align}
(x^{\aA})_{\pol } = \langle z^0 \rangle \, z^{\aA} \, \Omega [\sum
_{i\leq j} x_{i}/z_{j} - \sum _{i < j} z_{i}/z_{j}] = \langle z^0
\rangle \frac{z^{\aA} \prod _{i < j} (1- z_{i}/z_{j})}{ \prod_{i\leq
j} (1-x_{i}/z_{j})}.
\end{align}
Every monomial $z^{\bb }$ that occurs in the expansion of the fraction
in the last expression satisfies $\sum_{i = k}^n b_i \le \sum_{i =
k}^n a_i < 0$. Hence the coefficient of $z^0$ is zero.
\end{proof}

\begin{example}\label{ex:ns-pleth-v2}
We revisit our computations in Example \ref{ex nspleth n2}, now using
Proposition~\ref{prop:PiA-via-pol}.

(i) When $n=2$ and $f(x) = x_1^ax_2^b$, as in Example \ref{ex nspleth
n2}(ii), formula \eqref{e:PiA-via-pol} becomes
\begin{equation}
\begin{aligned}
\Pi_{A,x} \, x_1^a x_2^b &= \big(x_1^a x_2^b \sum_{i \ge 0} h_i[A \, x_1/x_2] \big)_{\pol} \\
&= \big( x_1^a x_2^b + h_{1}(A) x_1^{a+1}x_2^{b-1} + \cdots + h_{b}(A) x_1^{a+b} +
h_{b+1}(A) x_1^{a+b+1} x_2^{-1} + \cdots \big)_{\pol} \\
 &= x_1^a x_2^b + h_{1}(A) x_1^{a+1}x_2^{b-1} + h_{2}(A)
x_1^{a+2}x_2^{b-2} + \cdots + h_{b}(A) x_1^{a+b},
\end{aligned}
\end{equation}
using Proposition~\ref{p: pol part tail} and
the fact that $(x^\aA)_{\pol } = x^\aA$ when all exponents are non-negative.

(ii) The case of $\Pi_{A,x} \, x_3$, as in Example \ref{ex nspleth
n2}(i), demonstrates the dependence of formula \eqref{e:PiA-via-pol}
on the fact that the polynomial part operator is not the same as
na\"ive truncation to monomials with nonnegative exponents.  In this
case, formula \eqref{e:PiA-via-pol} gives
\begin{equation}
\begin{aligned}
\Pi_{A,x} \, x_3 &= \big(x_3 \sum_{i,j,k \ge 0} h_i[A \, x_1/x_2]
h_j[A\, x_1/x_3]
h_k[A \, x_2/x_3] \big)_{\pol}  \\
&= \big(x_3+  h_1(A)(x_1x_2^{-1}x_3 + x_1 + x_2) + h_1(A)^2 x_1   \big)_{\pol}
\qquad  \text{by Proposition \ref{p: pol part tail}} \\
&= x_3+  h_1(A)x_2  + h_1(A)^2 x_1,
\end{aligned}
\end{equation}
where the last equality follows from
\begin{align}
(x_1x_2^{-1}x_3 + x_1 )_{\pol } = (\Acal_{(1,-1,1)})_{\pol } = 0.
\end{align}
\end{example}

\subsection{Formulas for the inverse}
\label{ss:PiA-inverse}

While the various formulas for $\Pi _{A,x}$ in \S\S
\ref{ss:ns-pleth-def}--\ref{ss:PiA-via-pol} of course characterize
$\Pi _{A,x}^{-1}$ implicitly, we can also give more explicit formulas
for $\Pi _{A,x}^{-1}$.

\begin{prop}\label{prop:PiA-inverse}
The inverse nonsymmetric plethysm $\Pi _{A,x}^{-1}$ is given
explicitly by either of the two formulas
\begin{align}\label{e:PiA-inverse-1}
\Pi _{A,x}^{-1} f(x) & = \langle f(z),\, \Omega [\sum _{i\leq j}
x_{i}/z_{j} - A\, \sum _{i<j} x_{i}/z_{j}] \rangle _{0}\, ,\\
\label{e:PiA-inverse-2}
\Pi _{A,x}^{-1} f(x) & = \bigl( f(x) \, \Omega [-A\, \sum _{i<j}
w_{i}/x_{j}] \bigr)_{\pol }\, |_{w\mapsto x}\ ,
\end{align}
where the inner product in \eqref{e:PiA-inverse-1} is in the $z$ variables.
\end{prop}

\begin{proof}
Applying $\Pi _{A,x}^{-1}$ to both sides of
\eqref{e:atom-demazure-reproducing-1} and using \eqref{e:PiA-on-Omega}
yields \eqref{e:PiA-inverse-1}. By Lemma~\ref{lem:pol-part}, {the right hand side of}
\eqref{e:PiA-inverse-2} is equal to $\langle f(z),\, \Omega [\sum
_{i\leq j} x_{i}/z_{j} -A\, \sum _{i<j} w_{i}/z_{j}] \rangle_{0}
|_{w\mapsto x}$.  Having expressed it in a form not involving
operations in the $x$ variables, we are now free to carry out the
substitution $w\mapsto x$, showing that \eqref{e:PiA-inverse-2} is
equivalent to \eqref{e:PiA-inverse-1}.
\end{proof}

\begin{remark}\label{rem:PiA-inverse}
Note that \eqref{e:PiA-inverse-2} is not the same as
\eqref{e:PiA-via-pol} with $-A$ substituted for $A$.  The difference
is that in \eqref{e:PiA-via-pol}, all variables $x_{i}$ and $x_{j}$ in
$\Omega [A \sum _{i<j} x_{i}/x_{j}]$ affect the polynomial part, while
in \eqref{e:PiA-inverse-2}, only the variables $x_{j}$ in the
denominators do so, as the substitution $w\mapsto x$ occurs after
taking $(-)_{\pol }$ in the $x$ variables.

Similarly, the definition of $\langle -,- \rangle _{A}$
implies that \eqref{e:PiA-via-inner-product} is equivalent to
\begin{equation}\label{e:PiA-via-inner-2}
\Pi _{A,x}f(x) = \langle f(z),\, \Omega [\sum _{i\leq j}x_{i}/z_{j} +
A \sum _{i<j} z_{i}/z_{j}] \rangle _{0},
\end{equation}
but \eqref{e:PiA-inverse-1} is not the same as
\eqref{e:PiA-via-inner-2} with $-A$ substituted for $A$.
\end{remark}

\subsection{Recursive formulas}
\label{ss:PiA-recursions}

We now derive recursive formulas for $\Pi _{A,x}$ and $\Pi
_{A,x}^{-1}$ in terms of the same operations in fewer variables.

\begin{prop}\label{prop:PiA-recursion}
For $0<m<n$, we have identities
\begin{multline}\label{e:PiA-recursion}
\Pi _{A,(x_1,\ldots,x_n)} \, f(x_{1},\ldots,x_{n})\\
 = \Bigl( \Pi _{A,(x_{1},\ldots,x_{m})}\bigl((\xi ^{-1} \Pi
_{A,(x_{m+1},\ldots,x_{n})}\, f)[x_{m+1}+W+(A-1)X,\, x_{m+2},\,
\ldots,\, x_{n}]\bigr) \Bigr) \Bigr|_{W\mapsto X},
\end{multline}
\vspace{-2ex}
\begin{multline}\label{e:PiA-inv-recursion}
\Pi _{A,(x_1,\ldots,x_n)}^{-1} \,f(x_{1},\ldots,x_{n}) \\
= \Bigl( \Pi _{A,(x_{1},\ldots,x_{m})}^{-1} \, \Pi
_{A,(x_{m+1},\ldots,x_{n})}^{-1}\, \bigl( (\xi ^{-1}
f)[x_{m+1}+(1-A)\, W - X,\, x_{m+2},\, \ldots,\, x_{n}] \bigr) \Bigr)
\Bigr|_{W\mapsto X},
\end{multline}
where $X = x_{1}+\cdots +x_{m}$, and $\xi $ is the isomorphism from
Lemma~\ref{lem:X=x} for polynomials in $x_{m+1},\ldots,x_{n}$.  More
explicitly, for $p(x)\in \Lambda (A)[x_{1},\ldots,x_{n}]$, $\xi
^{-1}\, p$ is the unique flagged symmetric function in $n-m$ formal alphabets
with coefficients in $\Lambda (A)[x_{1},\ldots,x_{m}]$ such that $(\xi
^{-1}\, p)[x_{m+1},\ldots,x_{n}] = p(x)$.
\end{prop}

Note that the plethystic substitution $W\mapsto x_{1}+\cdots +x_{m}$
in formulas \eqref{e:PiA-recursion}--\eqref{e:PiA-inv-recursion} takes
place after the recursively applied nonsymmetric plethysm.  Thus, the
formal alphabet $W$ acts as a placeholder for a copy of $x_{1}+\cdots
+x_{m}$ shielded from the action of $\Pi
_{A,(x_{1},\ldots,x_{m})}^{\pm 1}$.

\begin{proof}
It suffices to verify \eqref{e:PiA-recursion} formally for $f(x) =
\Omega [\sum _{i\leq j} x_{i}/z_{j} - A \sum _{i<j}x_{i}/z_{j}]$ and
\eqref{e:PiA-inv-recursion} for $f(x) = \Omega [\sum _{i\leq j}
x_{i}/z_{j}]$, because the coefficients with respect to $z$ of the
formal series in each case form a basis of the polynomial ring
$\Lambda (A)[x_{1},\ldots,x_{n}]$.

By Proposition~\ref{prop:PiA-on-Omega}, in \eqref{e:PiA-recursion} for
$f = \Omega [\sum _{i\leq j} x_{i}/z_{j} - A \sum _{i<j}x_{i}/z_{j}]$,
we have $\Pi _{A,x}\, f = \Omega [\sum _{i\leq j} x_{i}/z_{j}]$ on the
left hand side, and also
\begin{multline}\label{e:PiA-recurrence-1}
\Pi _{A, (x_{m+1},\ldots, x_{n})}\, f\\
 =  \Omega [\sum _{i\leq
j\leq m} x_{i}/z_{j} - A \sum _{i< j\leq m} x_{i}/z_{j}]\, \Omega
[(1-A) X (z_{m+1}^{-1}+\cdots +z_{n}^{-1})]\,\Omega [\sum _{m<i\leq j} x_{i}/z_{j}] \,,
\end{multline}
where $X = x_{1}+\cdots +x_{m}$.  Applying $\xi ^{-1}$ in the
variables $x_{m+1},\ldots,x_{n}$ to the right hand side of
\eqref{e:PiA-recurrence-1} only affects the factor $\Omega [\sum
_{m<i\leq j} x_{i}/z_{j}]$, as the other factors are independent of
these variables.  By \eqref{e:xi-on-omega}, $\xi ^{-1} \Omega [\sum _{m<i\leq j} x_{i}/z_{j}] = \Omega [\sum _{m<i\leq j} X_{i}/z_{j}]$.  Hence,
\begin{multline}\label{e:PiA-recurrence-2}
(\xi ^{-1} \Omega [\sum _{m<i\leq j} x_{i}/z_{j}])[x_{m+1}+W+(A-1)X,\,
x_{m+2},\, \ldots,\, x_{n}] \\
 =  \Omega [(W+(A-1)X)
(z_{m+1}^{-1}+\cdots +z_{n}^{-1})]\,\Omega [\sum _{m<i\leq j} x_{i}/z_{j}]\,,
\end{multline}
and therefore, combining~\eqref{e:PiA-recurrence-1}
and~\eqref{e:PiA-recurrence-2}, we get
\begin{multline}\label{e:PiA-recurrence-3}
(\xi ^{-1} \Pi _{A,(x_{m+1},\ldots,x_{n})}\, f)[x_{m+1}+W+(A-1)X,\,
x_{m+2},\,
\ldots,\, x_{n}]\\
 =  \Omega [\sum _{i\leq
j\leq m} x_{i}/z_{j} - A \sum _{i< j\leq m} x_{i}/z_{j}]\, \Omega [W
(z_{m+1}^{-1}+\cdots +z_{n}^{-1})]\,\Omega [\sum _{m<i\leq j} x_{i}/z_{j}] \,.
\end{multline}
Applying $\Pi _{A,(x_{1},\ldots,x_{m})}$, then setting $W =
x_{1}+\cdots +x_{m}$, this becomes the desired quantity
\begin{equation}\label{e:PiA-recurrence-4}
 \Omega [\sum _{i\leq j\leq
m} x_{i}/z_{j} ]\, \Omega [(x_{1}+\cdots +x_{m}) (z_{m+1}^{-1}+\cdots
+z_{n}^{-1})]\,\Omega [\sum _{m<i\leq j} x_{i}/z_{j}] \, = \Omega [\sum _{i\leq j} x_{i}/z_{j}].
\end{equation}

On the left hand side of \eqref{e:PiA-inv-recursion} for $f = \Omega
[\sum _{i\leq j} x_{i}/z_{j}]$, Proposition~\ref{prop:PiA-on-Omega}
gives $\Pi _{A,x}^{-1}\, f = \Omega [\sum _{i\leq j} x_{i}/z_{j} - A
\sum _{i<j}x_{i}/z_{j}]$.  On the right hand side, factoring out
$\Omega [\sum _{m< i\leq j} x_{i}/z_{j}]$ from $\Omega [\sum _{i\leq
j} x_{i}/z_{j}]$ and proceeding similarly to
\eqref{e:PiA-recurrence-2}--\eqref{e:PiA-recurrence-3}, we obtain
\begin{equation}\label{e:PiA-inv-recurrence-1}
\begin{aligned}
(\xi ^{-1} f)[x_{m+1} + (1& - A)\, W - X,\, x_{m+2},\, \ldots,\, x_{n}] \\
& = \Omega [\sum _{i\leq j} x_{i}/z_{j}]\, \Omega
[((1-A)W-X)(z_{m+1}^{-1}+\cdots +z_{n}^{-1})]\\
& = \Omega [\sum _{i\leq j\leq m} x_{i}/z_{j}]\,\Omega [\sum _{m<i\leq
j} x_{i}/z_{j}]\,\Omega [(1-A) W (z_{m+1}^{-1}+\cdots +z_{n}^{-1})].
\end{aligned}
\end{equation}
Applying $ \Pi _{A,(x_{1},\ldots,x_{m})}^{-1} \, \Pi
_{A,(x_{m+1},\ldots,x_{n})}^{-1}$ and then setting $W = x_{1}+\cdots
+x_{m}$ yields the desired quantity $\Omega [\sum _{i\leq j}
x_{i}/z_{j} - A \sum _{i<j}x_{i}/z_{j}]$.
\end{proof}

In one variable, $\Pi _{A,x_{n}}$ is the identity, and $\xi
^{-1}(x_{n}^{k})= h_{k}(X_{n})$.  Hence,
\eqref{e:PiA-recursion}--\eqref{e:PiA-inv-recursion} take the
following simpler form for $m = n-1$.

\begin{cor}\label{cor:PiA-recursions}
\begin{multline}\label{e:PiA-recursion-m=n-1}
\Pi_{A,(x_1,\ldots,x_n)} \, \big(f(x_{1},\ldots,\,  x_{n-1})\, x_{n}^{k}\big)\\
 = \bigl( \Pi _{A,(x_{1},\ldots,x_{n-1})}\,
f(x_{1},\ldots,x_{n-1})\,h_{k}[x_{n}+W+(A-1)X] \bigr) \bigr|_{W\mapsto
X},
\end{multline}
\begin{multline}
\label{e:PiA-inv-recursion-m=n-1}
\Pi _{A,(x_1,\ldots,x_n)}^{-1} \, \big(f(x_{1},\ldots,\, x_{n-1})\, x_{n}^{k}\big) \\
 = \bigl( \Pi _{A,(x_{1},\ldots,x_{n-1})}^{-1}\,
f(x_{1},\ldots,x_{n-1})\, h_{k}[x_{n}+(1-A)\, W - X] \bigr)
\bigr|_{W\mapsto X},
\end{multline}\
where $X = x_{1}+\cdots +x_{n-1}$.
\end{cor}

Using \eqref{e:PiA-recursion-m=n-1}--\eqref{e:PiA-inv-recursion-m=n-1}
recursively often seems to be a good way to compute $\Pi _{A,x}^{\pm
1}$ in practice.

\subsection{Symmetric limit}
\label{ss:symmetric-limit}

Although nonsymmetric plethysm does not preserve symmetric
polynomials, there is nevertheless a sense in which, if $f(X)$ is a
symmetric function, the nonsymmetric plethysms $\Pi
_{A,(x_{1},\ldots,x_{n})}\, f[x_{1}+\cdots + x_{n}]$ and $\Pi
_{A,(x_{1},\ldots,x_{n})}^{-1}\, f[x_{1}+\cdots + x_{n}]$ converge for
large $n$ to $f[X/(1-A)]$ and $f[(1-A)X]$, respectively.  More general
`factored' versions of these limits hold for polynomials symmetric in
all but the first $r$ and last $s$ variables.  We will state these
results in a form that gives precise bounds on how closely the
nonsymmetric plethysms approximate the symmetric limits.  The bounds
turn out to depend on the level where $f(X)$ appears in the following
filtration on symmetric functions.

\begin{defn}\label{def:Vl}
The {\em length filtration} of the algebra $\Lambda _{\Kfrak }$ of
symmetric functions is the filtration $\Kfrak = V_{0} \subset V_{1}
\subset V_{2} \subset \cdots $, where
\begin{equation}\label{e:Vl}
V_{l} = \Kfrak \cdot \{s_{\lambda } \mid \ell (\lambda )\leq l \} =
\Kfrak \cdot \{h_{\lambda } \mid \ell (\lambda )\leq l \}
\end{equation}
is the subspace spanned by the Schur functions $s_{\lambda }$, or
equivalently by the complete symmetric functions $h_{\lambda }$,
indexed by partitions $\lambda $ with at most $l$ parts.
\end{defn}

In the case of the forward plethysm $\Pi _{A,x}$, the convergence
bounds also depend on the order of approximation in $A$.  To express
this, we need an additional definition.

\begin{defn}\label{def:Lambda>e}
Let $f$ and $g$ be (Laurent) polynomials or \multisymmetric
functions with coefficients in $\Lambda (A)$.  We say that {\em $f$
and $g$ agree to order $e$ in $A$}, and write

\begin{equation}\label{e:Lambda>e}
f \equiv g \pmod{\Lambda (A)_{>e}},
\end{equation}
if all coefficients of $f-g$ belong to the ideal $\Lambda (A) _{>e} \defeq
\bigoplus _{d>e} \Lambda (A)_{d}$ in $\Lambda (A)$ spanned by
symmetric functions homogeneous of degree greater than $e$ in $A$.
Note that every $f$ has a canonical representative $g\equiv f$ (mod
$\Lambda (A) _{>e}$) such that $g$ has coefficients of degree at most
$e$ in $A$.  We say that $f$ has a given property (mod $\Lambda(A)
_{>e}$)---for instance, that $f(x_{1},\ldots,x_{n})$ is symmetric
(mod $\Lambda(A) _{>e}$)---if its canonical representative has the
property in question.
\end{defn}

Equipped with these definitions, we can now state the theorems.  The
proofs will be given at the end of this section. 

\begin{thm}\label{thm:PiA-symm}
Let $f(X)$ be a symmetric function, and let $g(x) =
g(x_{1},\ldots,x_{r})$ and $h(y) = h(y_{1},\ldots,y_{s})$ be
polynomials, all with coefficients in $\Lambda (A)$.  If $f(X)\in
V_{l}$, $\deg (h)\leq d$, and $n\geq r+l+d+e$, then
\begin{equation}\label{e:PiA-symm-in}
\Pi _{A,(x_{1}, \ldots, x_{n},\, y_{1},\ldots,y_{s})}
\bigl(g(x_{1},\ldots,x_{r})\, f[x_{1}+\cdots +x_{n}] \,
h(y_{1},\ldots,y_{s})\bigr)
\end{equation}
is symmetric {\rm (mod $\Lambda (A)_{>e}$)} in the variables
$x_{r+1},\ldots,x_{n-(l+d+e)+1}$, and upon setting $x_{i} = 0$ for
$n-(l+d+e)+1<i\leq n$, it reduces to
\begin{equation}\label{e:PiA-symm-out}
(\Pi _{A,x}\, g(x))\, f \! \left[\frac{x_{1}+\cdots
+x_{n-(l+d+e)+1}}{1-A} \right]\, (\Pi _{A,y}\, h(y)) \pmod{\Lambda
(A)_{>e}}.
\end{equation}
\end{thm}

Regarding \eqref{e:PiA-symm-out}, note that if $X = x_{1}+x_{2}+\cdots
$ is a finite or infinite plethystic alphabet, the plethystic
evaluation $f[X/(1-A)]$ is of the form in \S \ref{ss:plethysm}, (iv),
where $(1-A)^{-1}$ is understood as a geometric series
$1+A+A^{2}+\cdots $, with $A = a_{1}+a_{2}+\cdots $ in variables
$a_{i}$ introduced for the formal alphabet $A$.  Note also that
$f[X/(1-A)]$ is a symmetric function in $X$, but only a symmetric
series in $A$.

\begin{thm}\label{thm:PiA-inv-symm}
Let $f(X)$, $g(x_{1},\ldots,x_{r})$, $h(y_{1},\ldots,y_{s})$ be as in
Theorem~\ref{thm:PiA-symm}.  If $f(X)\in V_{l}$, $\deg (h)\leq d$, and
$n\geq r+l+d$, then
\begin{equation}\label{e:PiA-inv-symm-in}
\Pi _{A,(x_{1}, \ldots, x_{n},\, y_{1},\ldots,y_{s})}^{-1}
\bigl(g(x_{1},\ldots,x_{r})\, f[x_{1}+\cdots +x_{n}] \,
h(y_{1},\ldots,y_{s})\bigr)
\end{equation}
is symmetric in the variables $x_{r+1},\ldots,x_{n-(l+d)}$, and upon
setting $x_{i} = 0$ for $n-(l+d)<i\leq n$, it reduces to
\begin{equation}\label{e:PiA-inv-symm-out}
(\Pi _{A,x}^{-1}\, g(x))\, f \left[(1-A)(x_{1}+\cdots +x_{n-(l+d)})
\right]\, (\Pi _{A,y}^{-1}\, h(y)).
\end{equation}
\end{thm}

Theorems~\ref{thm:PiA-symm} and~\ref{thm:PiA-inv-symm} imply the
following weaker but simpler form of the limiting behavior that was
alluded to at the beginning of this subsection.

\begin{cor}\label{cor:other-limits}
Let $f(X)$ be a symmetric function, and let $g(x) =
g(x_{1},\ldots,x_{r})$ and $h(y) = h(y_{1},\ldots,y_{s})$ be
polynomials, all with coefficients in $\Lambda (A)$.  Then
\begin{multline}\label{e:other-limits-PiA}
\lim_{n\rightarrow \infty } \Pi _{A,(x_{1}, \ldots, x_{n},\,
y_{1},\ldots,y_{s})} \bigl(g(x_{1},\ldots,x_{r})\, f[x_{1}+\cdots
+x_{n}] \,
h(y_{1},\ldots,y_{s})\bigr) \\
= (\Pi _{A,x}\, g(x))\, f[X/(1-A)]\, (\Pi _{A,y}\, h(y)),
\end{multline}
\vspace{-3ex}
\begin{multline}\label{e:other-limits-PiA-inv}
\lim_{n\rightarrow \infty } \Pi _{A,(x_{1}, \ldots, x_{n},\,
y_{1},\ldots,y_{s})}^{-1} \bigl(g(x_{1},\ldots,x_{r})\, f[x_{1}+\cdots
+x_{n}] \,
h(y_{1},\ldots,y_{s})\bigr) \\
= (\Pi _{A,x}\, g(x))\, f[(1-A)X]\, (\Pi _{A,y}\, h(y)),
\end{multline}
where $X = x_{1}, x_{2}, \dots $, and the limits converge formally in
$X$, $y_{1},\ldots,y_{s}$, and $A$ (as defined in \S \ref{ss:limits},
(iii)).
\end{cor}

\begin{remark}\label{rem:other-limits}
Every polynomial $p(x_{1},\ldots,x_{n+s})$ symmetric in all but the
first $r$ and last $s$ variables is a linear combination of products
$g(x_{1},\ldots,x_{r})\, f[x_{1}+\cdots +x_{n}] \,
h(x_{n+1},\ldots,x_{n+s})$.  Thus, Theorems~\ref{thm:PiA-symm} and
\ref{thm:PiA-inv-symm} and Corollary~\ref{cor:other-limits} in effect
describe the limiting behavior of $\Pi _{A,x}^{\pm 1}$ on all such
polynomials $p(x)$.  The use of $y_{1},\ldots,y_{s}$ in place of
$x_{n+1},\ldots,x_{n+s}$ is merely a convenient way to assign indices
independent of $n$ to the last $s$ variables.
\end{remark}

\begin{example}\label{ex:PiA-limits}
(i) By definition, we have
\begin{equation}\label{e:PiA-inv-h1}
\Pi _{A,x}^{-1}\, h_{1}[x_{1}+\cdots +x_{n}] =
x_{n}+h_{1}[(1-A)(x_{1}+\cdots +x_{n-1})].
\end{equation}
As asserted by Theorem~\ref{thm:PiA-inv-symm} for $g=h=1$ and
$f=h_{1}$, with $l=1$ and $d=0$, the right hand side is symmetric in
all but $x_{n}$, and setting $x_{n} = 0$ gives $f[(1-A)(x_{1}+\cdots
+x_{n-1})]$.

(ii) After changing $n$ to $n+1$ and using variables $x_1, \dots, x_n,
y_1$ in place of $x_1,\dots, x_{n+1}$, we can rewrite the first
identity in \eqref{e:PiA-inv-y1} as
\begin{equation}
\Pi _{A,(x_{1},\ldots,x_{n},y_{1})}^{-1}\, y_{1} = y_{1} - h_{1}(A)\,
x_{n}.
\end{equation}
The right hand side is again symmetric in all but $x_{n}$, and setting
$x_{n} = 0$ gives $y_{1}$. Since $y_1=\Pi^{-1}_{A,y_1}(y_1)$, this
confirms the assertion of Theorem~\ref{thm:PiA-inv-symm} for $f = g =
1$ and $h(y_{1})=y_{1}$, with $l=0$ and $d=1$.

(iii) Summing the identity \eqref{e:PiA-xn} gives
\begin{gather}
\label{e:PiA-h1}
\Pi _{A,x}\, h_{1}[x_{1}+\cdots +x_{n}] = \sum _{i=1}^{n}
h_{1}(A)^{n-i}\, h_{1}[x_{1}+\cdots +x_{i}].
\end{gather}
Modulo $\Lambda (A)_{>e}$, only the terms in \eqref{e:PiA-h1} for
$i\geq n-e$ survive.  Hence, as asserted by Theorem~\ref{thm:PiA-symm}
for $g=h=1$ and $f=h_{1}$, with $l=1$ and $d=0$, the right hand side
is symmetric (mod $\Lambda (A)_{>e}$) in $x_{1},\ldots,x_{n-e}$, and
setting $x_{n-e+1} = \ldots =x_{n} = 0$ gives
\begin{equation}\label{e:PiA-h1-reduced}
(1+h_{1}(A)+\cdots +h_{1}(A)^{e}) h_{1}[x_{1}+\cdots +x_{n-e}] \equiv
h_{1}\! \left[\frac{x_{1}+\cdots +x_{n-e}}{1-A}\right] \pmod{\Lambda
(A)_{>e}}.
\end{equation}

(iv) After changing $n$ to $n+1$, we can rewrite \eqref{e:PiA-xn} as
\begin{equation}\label{e:PiA-y1}
\Pi _{A,(x_{1},\ldots,x_{n},y_{1})}\, y_{1} = y_{1} + h_{1}(A)\, x_{n}
+\cdots +h_{1}(A)^{n}\, x_{1}.
\end{equation}
The right hand side is
symmetric (mod $\Lambda (A)_{>e}$) in $x_{1},\ldots,x_{n-e}$, and on
setting $x_{n-e+1} = \ldots =x_{n} = 0$, it reduces (mod $\Lambda
(A)_{>e}$) to $y_{1}$. As $y_1=\Pi_{A,y_1}y_1$, this confirms the assertion of
Theorem~\ref{thm:PiA-symm} for $f = g = 1$ and
$h(y_{1})=y_{1}$, with $l=0$ and $d=1$.
\end{example}

For the proofs of Theorems~\ref{thm:PiA-symm} and \ref{thm:PiA-inv-symm}, we first need
to obtain formulas for the product of $\fh_{\aA }[x_{1},\ldots,x_{n}]$ or $\Pi
_{A,x}^{-1}\, \fh_{\aA }[x_{1},\ldots,x_{n}]$ with a symmetric
polynomial.

\begin{lemma}\label{lem:hra}
For $Y$ any plethystic alphabet and $x$ a single variable, we have the
identity
\begin{equation}\label{e:hra}
h_r[x+Y] h_a[x+Y] = \sum _{j=0}^{a+r} c_{j}\, h_{j}[Y]\, h_{a+r-j}[x+Y],
\end{equation}
where $c_{j} = 1$ for $0\leq j\leq \min (r,a)$, $c_{j} = -1$ for $\max
(r,a)< j \leq a+r$, and $c_{j} = 0$ otherwise.
\end{lemma}

\begin{proof}
Both sides of \eqref{e:hra} are symmetric in $r$ and $a$, so we may
assume that $r \leq a$.  The coproduct formula $h_{k}[x+Y]= \sum
_{i=0}^{k} x^{i} \, h_{k-i}[Y]$ implies the identity
\begin{equation}\label{e:hk}
h_{k}[x+Y] = h_{k}[Y] + x\, h_{k-1}[x+Y].
\end{equation}
For $k=j$ in $h_{j,b}[x+Y] = h_j[x+Y] h_b[x+Y]$ and for $k=b+1$ in $h_{j-1,
b+1}[x+Y]$, this gives
\begin{equation}\label{e:hjb}
h_{j,b}[x+Y] - h_{j-1, b+1}[x+Y] = h_{j}[Y]\, h_{b}[x+Y] -
h_{j-1}[x+Y]\, h_{b+1}[Y],
\end{equation}
as the terms $x\, h_{j-1,b}[x+Y]$ cancel.  Setting $b=a+r-j$ in \eqref{e:hjb}
and summing from $j=0$ to $j=r$ yields \eqref{e:hra}.
\end{proof}

\begin{lemma}\label{lem:fha-recursive}
Given $\aA \in \NN ^{n}$, $l>0$, and $f(X)\in V_{l}$, we have
\begin{equation}\label{e:fha-recursive}
f[x_{1}+\cdots +x_{i}] \, \fh_{\aA }[x_{1}, \ldots, x_{n}] = \sum _{k}
g_{k}[x_{1}+\cdots +x_{i-1}]\, \fh_{\aA +k\, \varepsilon _{i}}[x_{1},
\ldots, x_{n}],
\end{equation}
where $g_{k}(X)\in V_{l}$, or $g_{k}(X)\in V_{l-1}$ if $a_{i} = 0$,
and $g_{k}$ depends only on $a_{i}$, $k$, and $f$, independent of $n$,
$i$, and $a_{j}$ for $j\not =i$.

With different functions $g_{k}(X)$, whose coefficients are now in
$\Lambda (A)$, we also have
\begin{multline}\label{e:fha-A-recursive}
f[x_{1}+\cdots +x_{i}] \, \fh_{\aA }[x_{1}, x_{2} - A\, x_{1}, \ldots,
x_{n} - A\, x_{n-1}] \\
= \sum _{k} g_{k}[x_{1}+\cdots +x_{i-1}]\, \fh_{\aA +k\, \varepsilon
_{i}}[x_{1}, x_{2} - A\, x_{1}, \ldots,
x_{n} - A\, x_{n-1}],
\end{multline}
where $g_{k}(X)\in \sum _{d}\Lambda (A)_{\geq d}\, V_{l+d}$---in other
words, $g_{k}(X)\in V_{l+e}$ {\rm (mod $\Lambda (A)_{>e}$)} for all
$e$---or $g_{k}(X)\in \sum _{d}\Lambda (A)_{\geq d}\, V_{l-1+d}$ if
$a_{i} = 0$, and $g_{k}$ again depends only on $a_{i}$, $k$, and $f$,
independent of $n$, $i$, and $a_{j}$ for $j\not =i$.
\end{lemma}

\begin{proof}
The second formula implies the first, by specializing $A=0$ and noting
that this specialization collapses $\sum _{d}\Lambda (A)_{\geq d}\,
V_{l+d}$ to $V_{l}$.

To prove the second formula, we first consider the case when $l=1$. In
this case, we can assume that $f=h_m$.  Setting $X_{i-1} =
x_{1}+\cdots +x_{i-1}$ and $Y=(1-A)X_{i-1}$, we have
\begin{equation}\label{e:hk-ha-A-1}
h_{m}[x_{1} + \cdots + x_{i}] = h_{m}[x_{i} +Y+A\, X_{i-1}] = \sum
_{r=0}^{m} h_{r}[x_{i}+Y] \, h_{m-r}[A\, X_{i-1}].
\end{equation}
Using Lemma~\ref{lem:hra} with $a=a_{i}$, $x=x_{i}$ and this $Y$,
\eqref{e:hk-ha-A-1} implies
\begin{multline}\label{e:hk-ha-A-2}
h_{m}[x_{1} + \cdots + x_{i}]\, \fh_{\aA }[x_{1}, x_{2} - A\, x_{1},
\ldots, x_{n} - A \, x_{n-1}] =\\
\sum _{r=0}^{m} \sum _{j=0}^{a_{i}+r} c_{r,j}\, h_{j}[(1-A)X_{i-1}]\,
h_{m-r}[A\, X_{i-1}]\, \fh_{\aA + (r-j)\, \varepsilon _{i}}[x_{1}, x_{2}
- A\, x_{1}, \ldots, x_{n} - A \, x_{n-1}],
\end{multline}
where $c_{r,j}$ is $c_{j}$ in Lemma~\ref{lem:hra} for the given $r$
and $a=a_{i}$.
This has the form of \eqref{e:fha-A-recursive}, with
\begin{equation}\label{e:gk-l=1}
g_{k}(X) = \sum _{r-j = k} c_{r,j}\, h_{j}[(1-A)X]\, h_{m-r}[A\, X] =
\sum _{r-j = k} c_{r,j} \sum _{p+q = j} h_{p}(X)\, h_{q}[-A\, X]\, h_{m-r}[A\,
X].
\end{equation}
Since every symmetric function $h$ satisfies $h[A\, X]\in \sum _{d}
\Lambda(A) _{\geq d} V_{d}(X)$, we see that $g_{k}(X)\in \sum _{d}
\Lambda (A) _{\geq d} V_{1+d}$.
For $a=0$ in Lemma~\ref{lem:hra}, the
only nonzero term in \eqref{e:hra} is the term for $j=0$. Hence if
$a_{i}=0$, we only have the term with $j=0$ in \eqref{e:gk-l=1}, in which
case $g_{k}(X)\in \sum _{d} \Lambda(A) _{\geq d} V_{d}$.  It is clear
that the expression in \eqref{e:gk-l=1} only depends on $a_{i}$, $k$,
and $f$, so the result holds for $l=1$.

For $l>1$, we can assume that $f(X) = h_{m}(X) f_{1}(X)$ with
$f_{1}(X)\in V_{l-1}$, and that the result holds for $f_{1}(X)$ by
induction on $l$, with coefficients $(g_1)_{k}(X)$, say.  Then
\eqref{e:hk-ha-A-2} gives the expansion in \eqref{e:fha-A-recursive}
with
\begin{equation}\label{e:gk-l>1}
g_{k}(X) = \sum _{k'+r-j = k} c_{r,j}\, h_{j}[(1-A)X]\, h_{m-r}[A\,
X]\, (g_1)_{k'}(X).
\end{equation}
Given that $(g_1)_{k'}(X)\in \sum _{d} \Lambda(A) _{\geq d} V_{l-1+d}$
for all $k'$, it follows that $g_{k}(X)\in \sum _{d} \Lambda(A) _{\geq
d} V_{l+d}$, as desired.  If $a_{i}=0$, we have the same implication
with $l-1$ in place of $l$.  Since the coefficients $(g_1)_{k'}$ are
independent of $n$, $i$, and the $a_{j}$ for $j\not =i$, so is
$g_{k}$.
\end{proof}

Note that Lemma~\ref{lem:fha-recursive} does not hold for $l=0$ and
$a_{i} = 0$, since $f(X)$ may be a nonzero constant, but $V_{-1} = 0$.
However, the lemma is unnecessary for $l=0$, since the left hand sides
of \eqref{e:fha-recursive} and \eqref{e:fha-A-recursive} have trivial
expansions if $f(X)$ is constant.

\begin{prop}\label{prop:symm-by-ha}
(a) Given $\aA \in \NN ^{n}$, $1\leq i_{0}\leq i_{1}\leq n$, and
$f(X)\in V_{l}$, suppose that $a_{j} = 0$ for at least $l$ indices
$j\in [i_{0},i_{1}]$.  Then
\begin{equation}\label{e:symm-by-ha}
f[x_{1}+\cdots +x_{i_{1}}] \, \fh_{\aA }[x_{1}, \ldots, x_{n}] = \sum
_{\bb \in \NN ^{i_{1}-i_{0}+1} }c_{\bb }\, \fh_{(\aA';\bb
;\aA''')}[x_{1}, \ldots, x_{n}],
\end{equation}
where $\aA = (\aA '; \aA ''; \aA ''')$ with $\aA' = (a_1,\ldots ,
a_{i_0-1})$, $\aA'' = (a_{i_0},\ldots , a_{i_1})$, and $\aA''' =
(a_{i_1+1}\ldots , a_{n})$, and $c_{\bb }$ depends only on $\aA ''$,
$\bb $, and $f$.

(b) Given $\aA \in \NN ^{n}$, $1\leq i_{0}\leq i_{1}\leq n$, and
$f(X)\in V_{l}$, suppose that $a_{j} = 0$ for at least $l+e$ indices
$j\in [i_{0},i_{1}]$.  Then
\begin{multline}\label{e:symm-by-ha-A}
f[x_{1}+\cdots +x_{i_{1}}] \, \fh_{\aA }[x_{1}, x_{2} - A\, x_{1},
\ldots, x_{n} - A\, x_{n-1}]\\
\equiv \sum _{\bb \in \NN ^{i_{1}-i_{0}+1} } c_{\bb }\, \fh_{(\aA';\bb
;\aA''')}[x_{1}, x_{2} - A\, x_{1}, \ldots, x_{n} - A\, x_{n-1}]
\pmod{\Lambda (A)_{>e}}
\end{multline}
where $\aA = (\aA '; \aA ''; \aA ''')$ with $\aA' = (a_1,\ldots ,
a_{i_0-1})$, $\aA'' = (a_{i_0},\ldots , a_{i_1})$, and $\aA''' =
(a_{i_1+1}\ldots , a_{n})$, and $c_{\bb }$ depends {\rm (mod $\Lambda
(A)_{>e}$)} only on $\aA ''$, $\bb $, and $f$.
\end{prop}

\begin{proof}
If $l=0$, then $f(X)$ is constant, and the sole term on the right hand
side in either formula is for $\bb =\aA ''$, with $c_{\bb }$ equal to
the constant value of $f$.

For $l>0$, if $i_{0} = i_{1}$, then we must have $l=1$ and $a_{i_1} =
0$, and in part (b) we must have $e=0$.  Then
Lemma~\ref{lem:fha-recursive} gives an expansion with coefficients
$g_{k}(X)$ that are constant (mod $\Lambda (A)_{>0}$) in part (b), or
just constant, in part (a); the coefficients $c_{\bb } = c_{(k)}$ are
then given by $c_{(k)} = g_{k}$.

If $l>0$ and $i_{0}<i_{1}$, then the result follows by induction on
$i_{1}$, using Lemma~\ref{lem:fha-recursive} and the induction
hypothesis for the same $l$ if $a_{i}>0$, or for $l-1$ if $a_{i} = 0$.
\end{proof}

Proposition~\ref{prop:symm-by-ha} supplies the essential facts needed
for the proofs of Theorems~\ref{thm:PiA-symm} and
\ref{thm:PiA-inv-symm}, to which we now turn.

\begin{proof}[Proof of Theorem~\ref{thm:PiA-symm}] We set $\tfh _{\aA
}(X_{1},\ldots,X_{k}) = \fh_{\aA }[X_{1}, X_{2} - A\, X_{1}, \ldots,
X_{k} - A\, X_{k-1}]$, as in \eqref{e:htild}.  By linearity, we can
assume that $g(x) = \tfh _{\aA }[x_{1}, \ldots, x_{r}]$ and $h(y) =
\tfh _{\aA '}[y_{1}, \ldots, y_{s}]$, where $\aA \in \NN ^{r}$, $\aA
'\in \NN ^{s}$.  Using the coproduct formula \eqref{e coprod formula},
we can express $\tfh _{\aA '}[y_{1}, \ldots, y_{s}] $ in terms of
functions $\tfh _{\bb }$ of the full list of variables $x, y$, as
\begin{equation}\label{e:hy-A-extended}
\begin{aligned}
\tfh _{\aA '}[y_{1},& \ldots, y_{s}] \\
& = \sum _{\pp +\qq =\aA '} \fh_{\pp }[(A-1)X,0,\ldots,0]\, \fh_{\qq
}[(1-A)X + y_{1}, y_{2} - A\, y_{1}, \ldots, y_{s} - A\, y_{s-1}]\\
& = \sum _{\pp +\qq =\aA '} h_{\pp _{+}}[(A-1)X]\, \tfh
_{(0^{n};\, \qq) }[x_{1}, \ldots, x_{n}, y_{1}, \ldots, y_{s}],
\end{aligned}
\end{equation}
where $X = x_{1}+\cdots +x_{n}$, and $h_{\pp _{+}}$ is the complete
symmetric function indexed by the partition $\pp _{+}$ whose parts are
the nonzero entries of $\pp $.  Then
\begin{equation}\label{e:fgh-by-ha}
g(x)\, f[X]\, h(y) = \sum _{\pp +\qq =\aA '} f[X]\,
h_{\pp _{+}}[(A-1)X]\, \tfh_{(\aA ;\, 0^{n-r};\, \qq) }[x_{1},
\ldots, x_{n}, y_{1}, \ldots, y_{s}].
\end{equation}

Since $V_{d}$ contains all symmetric
functions of degree $\leq d$, we have $f[X]\, h_{\pp _{+}}[(A-1)X] \in
V_{l+d}$.  By Proposition~\ref{prop:symm-by-ha}(b) with
$[i_{0},i_{1}]= [n-(l+d+e)+1,n]$, it follows that
\begin{multline}\label{e:p-term}
f[X]\, h_{\pp _{+}}[(A-1)X]\, \tfh_{(\aA ;\, 0^{n-r};\, \qq)
}[x_{1}, \ldots, x_{n}, y_{1}, \ldots, y_{s}]\\
\equiv \sum _{\bb \in \NN ^{l+d+e}} c_{\pp ,\bb }\, \tfh_{(\aA ;\,
0^{n-r-(l+d+e)};\, \bb ;\, \qq )} [x_{1}, \ldots, x_{n}, y_{1},
\ldots, y_{s}] \pmod{\Lambda (A)_{>e}},
\end{multline}
where the coefficients $c_{\pp ,\bb }\in \Lambda (A)$ are independent
of $n$, $\aA $, and $\qq $.  In particular, by taking $\aA = 0^{r}$
and $\qq = 0^{s}$, we conclude that
\begin{equation}\label{e:cpb-1}
f[X]\, h_{\pp _{+}}[(A-1)X] \equiv \sum _{\bb \in \NN ^{l+d+e}} c_{\pp
,\bb }\, \tfh_{( 0^{n-(l+d+e)};\, \bb )} [x_{1}, \ldots, x_{n}]
\pmod{\Lambda (A)_{>e}},
\end{equation}
for all $n\geq l+d+e$.  Letting $m=n-(l+d+e)$ and setting $x_{m+1} =
\ldots = x_{n} = 0$, this becomes
\begin{equation}\label{e:cpb-2}
f[Z]\, h_{\pp _{+}}[(A-1)Z] \equiv \sum _{\bb \in \NN ^{l+d+e}} c_{\pp
,\bb }\, h _{\bb _{+}} [(1-A)Z] \pmod{\Lambda (A)_{>e}},
\end{equation}
where $Z = x_{1}+\cdots +x_{m}$.  Since \eqref{e:cpb-2} holds for
arbitrarily large $m$, it holds as an identity of symmetric functions
in a formal alphabet $Z$, and hence for plethystic evaluation in any
$Z$.

Now, \eqref{e:fgh-by-ha} and \eqref{e:p-term} imply
\begin{multline}\label{e:PiA-p-term}
\Pi _{A,(x_{1},\ldots,x_{n},y_{1},\ldots,y_{s})}
\bigl( g(x)\, f[x_{1},\ldots,x_{n}]\, h(y) \bigr)\\
\equiv \sum _{\pp +\qq =\aA '}\sum _{\bb \in \NN ^{l+d+e}} c_{\pp ,\bb
}\, \fh_{(\aA ;\, 0^{n-r-(l+d+e)};\, \bb ;\, \qq )} [x_{1}, \ldots,
x_{n}, y_{1}, \ldots, y_{s}] \pmod{\Lambda (A)_{>e}}.
\end{multline}
The block of zeroes in positions $r+1$ through $n-(l+d+e)$ of $(\aA
;\, 0^{n-r-(l+d+e)};\, \bb ;\, \qq )$ implies that each term is
symmetric in $x_{r+1},\ldots,x_{n-(l+d+e)+1}$, as claimed.
Furthermore, on setting $x_{j} = 0$ for $n-(l+d+e)+1 < j \leq n$, the
right hand side of \eqref{e:PiA-p-term} becomes
\begin{equation}\label{e:PiA-p-term-x=0}
\fh_{\aA }[x_{1},\ldots,x_{r}] \sum _{\pp +\qq =\aA '}\sum _{\bb \in \NN
^{l+d+e}} c_{\pp ,\bb }\, h_{\bb_+}[X]\, \fh_{\qq } [X + y_{1}, \ldots,
y_{s}]
\end{equation}
where we now set $X = x_{1}+\cdots +x_{n-(l+d+e)+1}$.  Taking $Z =
X/(1-A)$ in \eqref{e:cpb-2}, we see that \eqref{e:PiA-p-term-x=0} is
congruent (mod $\Lambda (A)_{>e}$) to
\begin{equation}\label{e:PiA-p-term-x=0-bis}
\fh_{\aA }[x_{1},\ldots,x_{r}] f[X/(1-A)] \sum _{\pp +\qq =\aA '} h_{\pp
_{+}}[-X]\, \fh_{\qq } [X + y_{1}, \ldots, y_{s}].
\end{equation}
The summation factor here reduces to
\begin{equation}\label{e:PiA-p-term-x=0-ter}
 \sum _{\pp +\qq =\aA '} \fh_{\pp }[-X,0,\ldots,0]\, \fh _{\qq } [X +
y_{1}, \ldots, y_{s}] = \fh_{\aA '}[y_{1},\ldots,y_{s}],
\end{equation}
so \eqref{e:PiA-p-term-x=0-bis} is equal to
\begin{equation}\label{e:PiA-p-term-x=0-iv}
\fh_{\aA }[x_{1},\ldots,x_{r}] \, f \! \left[\frac{x_{1}+\cdots
+x_{n-(l+d+e)+1}}{1-A} \right] \, \fh_{\aA '}[y_{1},\ldots,y_{s}],
\end{equation}
which is \eqref{e:PiA-symm-out} for $g(x) = \tfh_{\aA
}[x_{1},\ldots,x_{r}]$ and $h(y) = \tfh _{\aA
'}[y_{1},\ldots,y_{s}]$.
\end{proof}

\begin{proof}[Proof of Theorem~\ref{thm:PiA-inv-symm}]

In essence, this is the same as the preceding proof with the roles of
$\tfh _{\aA }$ and $\fh_{\aA }$ interchanged.  In more detail, we
first assume by linearity that $g(x) = \fh_{\aA }[x_{1}, \ldots,
x_{r}]$ and $h(y) = \fh_{\aA '}[y_{1}, \ldots, y_{s}]$.  Specializing
$A=0$ in \eqref{e:hy-A-extended}, we have
\begin{equation}\label{e:hy-extended}
\fh_{\aA '}[y_{1}, \ldots, y_{s}] = \sum _{\pp +\qq =\aA '} h_{\pp
_{+}}[-X]\, \fh_{(0^{n};\, \qq) }[x_{1}, \ldots, x_{n}, y_{1}, \ldots,
y_{s}],
\end{equation}
where $X = x_{1}+\cdots +x_{n}$.  The counterparts of
\eqref{e:fgh-by-ha}--\eqref{e:p-term} are now
\begin{gather}\label{e:fgh-by-h}
g(x)\, f[X]\, h(y) = \sum _{\pp +\qq =\aA '} f[X]\,
h_{\pp _{+}}[-X]\, \fh_{(\aA ;\, 0^{n-r};\, \qq) }[x_{1},
\ldots, x_{n}, y_{1}, \ldots, y_{s}],\\
\label{e:p-term-1}
\begin{aligned}
f[X]\, h_{\pp _{+}}[-X] \, \fh_{(\aA ;\, 0^{n-r};\, \qq)
}[x_{1}, \ldots & \, , x_{n}, y_{1}, \ldots, y_{s}]\\
& =  \sum _{\bb \in \NN ^{l+d+e}} c_{\pp ,\bb }\, \fh_{(\aA ;\,
0^{n-r-(l+d)};\, \bb ;\, \qq )} [x_{1}, \ldots, x_{n}, y_{1}, \ldots,
y_{s}]
\end{aligned}
\end{gather}
with integer coefficients $c_{\pp ,\bb }$ independent of $n$, $\aA $,
and $\qq $.  For \eqref{e:p-term-1}, instead of
Proposition~\ref{prop:symm-by-ha}(b), we have used
Proposition~\ref{prop:symm-by-ha}(a) with $[i_{0},i_{1}] =
[n-(l+d)+1,n]$.  The same argument that gave \eqref{e:cpb-2} now
yields
\begin{equation}\label{e:cpb-3}
f[Z]\, h_{\pp _{+}}[-Z] = \sum _{\bb \in \NN ^{l+d+e}} c_{\pp
,\bb }\, h _{\bb _{+}} [Z]
\end{equation}
as a plethystic identity in $Z$.

From \eqref{e:fgh-by-h} and \eqref{e:p-term-1}, we obtain the analog of
\eqref{e:PiA-p-term}:
\begin{multline}\label{e:PiA-inv-p-term}
\Pi _{A,(x_{1},\ldots,x_{n},y_{1},\ldots,y_{s})}^{-1} \bigl( g(x)\,
f[x_{1},\ldots,x_{n}]\, h(y) \bigr)\\
= \sum _{\pp +\qq =\aA '}\sum _{\bb \in \NN ^{l+d+e}} c_{\pp ,\bb }\,
\tfh _{(\aA ;\, 0^{n-r-(l+d)};\, \bb ;\, \qq )} [x_{1}, \ldots, x_{n},
y_{1}, \ldots, y_{s}].
\end{multline}
The block of zeroes in $(\aA ;\, 0^{n-r-(l+d)};\, \bb ;\, \qq )$
implies that each term is symmetric in $x_{r+1},\ldots,x_{n-(l+d)}$,
as claimed.  Note that this is slightly different from the proof of
Theorem~\ref{thm:PiA-symm}, because $\tfh _{(\aA ;\,
0^{n-r-(l+d)};\, \bb ;\, \qq
)}[x_{1},\ldots,x_{n},y_{1},\ldots,y_{s}]$ is only symmetric in the
variables $x_{i}$ corresponding to the positions of the zeroes, while
$\fh_{(\aA ;\, 0^{n-r-(l+d+e)};\, \bb ;\, \qq
)}[x_{1},\ldots,x_{n},y_{1},\ldots,y_{s}]$ is also symmetric in the
extra variable corresponding to the first position after the zeroes.

On setting $x_{n-(l+d)+1} = \ldots = x_{n} = 0$, the right hand side
of \eqref{e:PiA-inv-p-term} becomes
\begin{equation}\label{e:PiA-inv-p-term-x=0}
\tfh _{\aA }[x_{1},\ldots,x_{r}] \sum _{\pp +\qq =\aA '}\sum _{\bb
\in \NN ^{l+d+e}} c_{\pp ,\bb }\, h_{\bb_+}[(1-A)X]\, \tfh _{\qq }
[X + y_{1}, \ldots, y_{s}]
\end{equation}
where we now set $X = x_{1}+\cdots +x_{n-(l+d)}$.  Taking $Z = (1-A)X$
in \eqref{e:cpb-3} shows that this is equal to
\begin{equation}\label{e:PiA-inv-p-term-x=0-bis}
\tfh _{\aA }[x_{1},\ldots,x_{r}] f[(1-A)X] \sum _{\pp +\qq =\aA '}
h_{\pp _{+}}[(A-1)X]\, \tfh _{\qq } [X + y_{1}, \ldots, y_{s}].
\end{equation}
As in \eqref{e:PiA-p-term-x=0-bis}, the summation factor reduces to
$\tfh _{\aA '}[y_{1},\ldots,y_{s}]$, giving
\eqref{e:PiA-inv-symm-out} for $g(x) = \fh_{\aA }[x_{1},\ldots,x_{r}]$
and $h(x) = \fh_{\aA '}[y_{1},\ldots,y_{s}]$.
\end{proof}

\subsection{One-parameter specialization}
\label{ss:A=t}

For applications to flagged LLT polynomials and Macdonald polynomials,
we need nonsymmetric plethysm operations $\Pi _{t,x}^{\pm 1}$ on
$\Kfrak [x_{1},\ldots,x_{n}]$, where $t\in \Kfrak $, defined by
specializing $A$ to $t$ in the generic operations $\Pi _{A,x}^{\pm
1}$.

Given a homomorphism of $\Kfrak $-algebras $\ev \colon \Lambda
_{\Kfrak }(A)\rightarrow \Rfrak $, making $\Rfrak $ a $\Lambda
_{\Kfrak }(A)$-algebra, extension of scalars yields $\Rfrak $-linear
operations ${\Rfrak \otimes _{\Lambda _{\Kfrak }(A)} \Pi _{A,x}^{\pm
1}}$ on $\Rfrak [x_{1},\ldots,x_{n}]$ intertwined by $\ev (-)$ with
the operations $\Pi _{A,x}^{\pm 1}$ on $\Lambda _{\Kfrak
}(A)[x_{1},\ldots,x_{n}]$.  In general, this allows one to specialize
the nonsymmetric plethysm operations using an arbitrary homomorphism
$f(A)\mapsto \ev (f)$.  The following special case defines $\Pi
_{t,x}^{\pm 1}$.

\begin{defn}\label{def:Pit}
Given $t\in \Kfrak $, the nonsymmetric plethysm operations $\Pi
_{t,x}^{\pm 1}$ on $\Kfrak [x_{1},\ldots,x_{n}]$ are the $\Kfrak
$-linear operations given by extension of scalars
\begin{equation}\label{e:Pit-def}
\Pi _{t,x}^{\pm 1} = \Kfrak \otimes _{\Lambda _{\Kfrak }(A)} \Pi
_{A,x}^{\pm 1}
\end{equation}
from the operations $\Pi _{A,x}^{\pm 1}$ on $\Lambda _{\Kfrak
}(A)[x_{1},\ldots,x_{n}]$, using the $\Kfrak $-algebra homomorphism
\begin{equation}\label{e:ev-t}
\ev _{t} \colon \Lambda _{\Kfrak }(A) \rightarrow \Kfrak,\qquad \ev
_{t}f(A) = f[t] = f(t) \quad \text{(i.e., $\ev _{t}\, h_{k}(A) =
t^{k}$)}.
\end{equation}
\end{defn}

By definition, we have identities of maps from $\Lambda _{\Kfrak
}(A)[x_{1},\ldots,x_{n}]$ to $\Kfrak [x_{1},\ldots,x_{n}]$:
\begin{equation}\label{e:Pit-PiA}
\begin{gathered}
\Pi _{t,x} \circ \ev _{t} = \ev _{t} \circ \, \Pi _{A,x},\\
\Pi _{t,x}^{-1} \circ \ev _{t} = \ev _{t} \circ \, \Pi _{A,x}^{-1}.
\end{gathered}
\end{equation}
Thus, applying $\ev _{t}$ to
essentially any quantity expressed in terms of $A$ and the
nonsymmetric plethysm operations $\Pi _{A,x}^{\pm 1}$ has the effect
of changing $\Pi _{A,x}^{\pm 1}$ to $\Pi _{t,x}^{\pm 1}$, as well as
replacing all other occurrences of $A$ with $t$.  In particular,
from the defining formula \eqref{e:PiA-definition} we get that $\Pi
_{t,x}$ is given on the basis of polynomials $\fh_{\aA }[x_{1},\,
x_{2}-t\, x_{1},\, \ldots,\, x_{n}-t\, x_{n-1}]$ for $\Kfrak
[x_{1},\ldots,x_{n}]$ by
\begin{equation}\label{e:Pit-form-of-definition}
\Pi _{t,x} \, \fh_{\aA }[x_{1},\, x_{2}-t\, x_{1},\, \ldots,\, x_{n}-t\,
x_{n-1}] = \fh_{\aA }[x_{1},\, x_{2},\, \ldots,\, x_{n}].
\end{equation}

For most of the results on $\Pi _{A,x}^{\pm 1}$ in \S \S
\ref{ss:ns-pleth-def}--\ref{ss:symmetric-limit}, including
Theorem~\ref{thm:PiA-inv-symm}, it follows that the same results hold
for $\Pi _{t,x}^{\pm 1}$ after replacing $\Lambda (A)$ with $\Kfrak $
and setting $A=t$ in all formulas.  In particular, the results in \S
\ref{ss:PiA-via-inner} involving $\Pi _{A,x}^{\pm 1}$ and the inner
product $\langle -,- \rangle _{A}$ in \eqref{e:A-inner-product}
specialize to results for $\Pi _{t,x}^{\pm 1}$ and the inner product
$\langle -,- \rangle _{t}$ in \eqref{e:t-inner}, below.

If $\Kfrak $ consists of polynomials or power series in $t$, then
Theorem~\ref{thm:PiA-symm} also follows, with the additional
modification that we replace congruence (mod $\Lambda (A)_{>e}$) with
congruence (mod $(t^{e+1})$).  In more generality, we can allow
elements of $\Kfrak $ to be formal Laurent series in $t$, including
rational functions of $t$ by expanding them as formal Laurent series.
If we assume that the coefficients of $f(X)$, $g(x)$, and $h(y)$ have
no poles at $t=0$, then the Laurent series in question reduce to
ordinary power series, giving the following version of
Theorem~\ref{thm:PiA-symm}.

\begin{thm}\label{thm:Pit-symm}
Assume that the elements of $\Kfrak $ have formal Laurent series
expansions in $t$.  Let $f(X)$ be a symmetric function, and let $g(x)
= g(x_{1},\ldots,x_{r})$ and $h(y) = h(y_{1},\ldots,y_{s})$ be
polynomials, all with coefficients in $\Kfrak $.  Assume that their
coefficients have no poles at $t=0$.  If $f(X)\in V_{l}$, $\deg
(h)\leq d$, and $n\geq r+l+d+e$, then
\begin{equation}\label{e:Pit-symm-in}
\Pi _{t,(x_{1}, \ldots, x_{n},\, y_{1},\ldots,y_{s})}
\bigl(g(x_{1},\ldots,x_{r})\, f[x_{1}+\cdots +x_{n}] \,
h(y_{1},\ldots,y_{s})\bigr)
\end{equation}
is symmetric {\rm (mod $(t^{e+1})$)} in the variables
$x_{r+1},\ldots,x_{n-(l+d+e)+1}$, and upon setting $x_{i} = 0$ for
$n-(l+d+e)+1<i\leq n$, it reduces to
\begin{equation}\label{e:Pit-symm-out}
(\Pi _{t,x}\, g(x))\, f \! \left[\frac{x_{1}+\cdots
+x_{n-(l+d+e)+1}}{1-t} \right]\, (\Pi _{t,y}\, h(y)) \pmod{
(t^{e+1})},
\end{equation}
where $p\equiv q \pmod{(t^{e+1})}$ means that the coefficients of
$p-q$ vanish to order $\geq e+1$ at $t=0$.
\end{thm}

Corollary~\ref{cor:other-limits} specializes a little more
straightforwardly.

\begin{cor}\label{cor:other-limits2}
Assume that the elements of $\Kfrak $ have formal Laurent series
expansions in $t$.  Let $f(X)$ be a symmetric function, and let $g(x)
= g(x_{1},\ldots,x_{r})$ and $h(y) = h(y_{1},\ldots,y_{s})$ be
polynomials, all with coefficients in $\Kfrak $.  Then the identities
\eqref{e:other-limits-PiA} and \eqref{e:other-limits-PiA-inv} hold
with $A = t$, where now the limits converge formally in $X$,
$y_{1},\ldots,y_{s}$, and $t$.
\end{cor}

\begin{proof}
Since $f(X)$, $g(x)$, and $h(y)$ have finitely many coefficients, we
can multiply by a power of $t$ to remove any poles at $t=0$.  The
corollary now follows from Theorem~\ref{thm:Pit-symm} and the fact
that $\lim_{n\rightarrow \infty } p_{n} = p$ converges formally in $t$
if and only if the same holds with $p_{n}$ and $p$ multiplied by a
power of $t$.
\end{proof}

\section{Preliminaries, continued}
\label{s:more-prelim}

\subsection{Hecke algebras}
\label{ss:hecke}

In \S \S \ref{ss:hecke}--\ref{ss:ns-HL-pols}, we review the action of
the Hecke algebra $\Hcal (\Sfrak _{l})$ on Laurent polynomials, as in
Lusztig \cite{Lusztig83,Lusztig89} and Macdonald \cite{Macdonald03},
and define nonsymmetric Hall-Littlewood polynomials.  Our treatment
here closely follows that in \cite[\S \S 4.2--4.3]{BHMPS23}.

The \emph{Demazure-Lusztig} operators
\begin{equation}\label{e:Ti}
T_{i} = t\, s_{i} +\frac{1-t}{1-x_{i}/x_{i+1}}\, (s_{i}-1)
\end{equation}
for $i = 1,\dots, l-1$ generate an action of the Hecke algebra $\Hcal
(\Sfrak _{l})$ on the ring of Laurent polynomials $\Kfrak[x_{1}^{\pm
1},\ldots,x_{l}^{\pm 1}]$.  Here, $\Kfrak$ is any coefficient ring and
$t \in \Kfrak$ any invertible element (typically, $\Kfrak$ will be
$\QQ(t)$ or $\QQ(q,t)$).  The generators of $\Hcal (\Sfrak _{l})$
satisfy the braid relations $T_{i}T_{i+1}T_{i} = T_{i}T_{i+1}T_{i}$
and $T_{i}T_{j} = T_{j}T_{i}$ for $|j-i|>1$, and the quadratic
relations
\begin{equation}\label{e:quadratic-Hecke-rels}
(T_{i}- t)(T_{i}+1) = 0.
\end{equation}
The elements $T_w$, defined by $T_w = T_{i_1} T_{i_2}\cdots T_{i_k}$
for any reduced expression $w=s_{i_1}s_{i_2}\cdots s_{i_k}$, form a
$\Kfrak$-basis of $\Hcal (\Sfrak _{l})$, as $w$ ranges over
$\Sfrak_l$. 

Using an overbar $\overline{\, \vphantom{t} \cdot \, }$ to signify
inverting the variables $t$, $x_{i}$, one has $\overline{T_{i}} =
T_{i}^{-1}$ and hence $\overline{T_{w}} = T_{w^{-1}}^{-1}$.

\begin{remark}\label{rem: Ti conventions}
The version of the Demazure-Lusztig operators $T_{i}$ used here agrees
with \cite{Goodberry24,HagHaiLo08, Lapointe22}.  In terms of the
notation here, the Demazure-Lusztig operators in \cite{BHMPS25a,
BHMPS23} are $\overline{T_{i}}|_{t^{-1}\mapsto q}$, and those in
\cite{BechtloffWeising23} are $t\, \overline{T_{i}}$.
\end{remark}

A Laurent polynomial $f(x) = f(x_{1},\ldots,x_{l})$ is $s_{i}$
invariant---i.e., symmetric in $x_{i}$ and $x_{i+1}$---if and only if
$T_{i} \, f(x) = t\, f(x)$.  For any Young subgroup $\Sfrak _{\rr
}\subseteq \Sfrak _{l}$, it follows that if $f(x)$ is in the image of
the symmetrizer
\begin{equation}\label{e:Hecke-symmetrizer}
\Scal _{\rr} = \sum _{w\in \Sfrak _{\rr }} T_{w}\, ,
\end{equation}
then $f(x)$ is $\Sfrak _{\rr }$ invariant, and that every $\Sfrak
_{\rr }$ invariant $f(x)$ satisfies $\Scal _{\rr }\, f(x) = W_{\rr
}(t)\, f(x)$, where $W_{\rr }(t) = \sum _{w\in \Sfrak _{\rr }} t^{\ell
(w)}$.  Equivalently (if $\QQ (t)\subseteq \Kfrak $), the normalized
symmetrizer
\begin{equation}\label{e:normalized-Hecke-symmetrizer}
\hsym  _{\rr} = \frac{1}{W_{\rr }(t)}\, \Scal _{\rr }
\end{equation}
is idempotent, with $\hsym _{\rr}\, f(x) = f(x)$ if and only if $f(x)$
is $\Sfrak _{\rr }$ invariant.  If $\Sfrak _{\sS }\subseteq \Sfrak
_{\rr }$, it also follows that $\hsym _{\rr}\, \hsym _{\sS } = \hsym
_{\rr }$.  For the groups $\Sfrak _{\rr }$, note that
\begin{equation}\label{e:W(t)}
W_{\rr}(t) = \prod _{i} [r_{i}]_{t}!, \quad \text{where}\quad
[r]_{t}!=\prod _{k=1}^{r} [k]_{t}, \text{ with } [k]_{t} =
\frac{1-t^{k}}{1-t} = 1+t+\cdots +t^{k-1}.
\end{equation}

Lusztig's identity \cite[Theorem~6.6]{Lusztig83} implies that the full
symmetrizer $\Scal  _{(l)}$ acts as
\begin{equation}\label{e:full-Sl-formula}
\Scal _{(l)} f(x) = \sum _{w\in \Sfrak _{l}} w \left( f(x) \prod
_{i<j} \frac{(1 - t x_{j}/x_{i})}{(1 - x_{j}/x_{i})} \right).
\end{equation}
Hence, for $\lambda \in \ZZ ^{l}$ a dominant weight,
\begin{equation}\label{e:Sl-to-symmetric-HL}
\Scal _{(l)} x^{\lambda } = W_{\sS }(t)\, P_{\lambda }(x;t),
\end{equation}
where $\Sfrak _{\sS } = \Stab (\lambda )$ and $P_{\lambda }(x;t)$ is a 
symmetric Hall-Littlewood polynomial \cite[III,~(2.1)]{Macdonald95}.

We say that $f(x)$ is {\em Hecke antisymmetric} in $x_{i}$ and
$x_{i+1}$ if $T_{i}f(x) = -f(x)$.  Analogous to the case of ordinary
antisymmetric polynomials, if $f$ is Hecke antisymmetric in $x_{i}$
and $x_{i+1}$, then $t\, x_{i} - x_{i+1}$ divides $f(x)$, and
$f(x)/(t\, x_{i} - x_{i+1})$ is symmetric.  The converse also holds.

\subsection{Nonsymmetric Hall-Littlewood polynomials}
\label{ss:ns-HL-pols}

For $\lambda \in \ZZ ^{l}$, define the {\em nonsymmetric
Hall-Littlewood polynomials} $E_\lambda(x;t)$ and $F_\lambda(x;t)$ by
\begin{gather}\label{e:ns-HL-E}
E_{\lambda }(x;t^{-1}) = t^{\ell(w)}\, \overline{T_{w}}\,
x^{\lambda _{+}},\\
\label{e:ns-HL-F} F_{\lambda }(x;t^{-1}) = \overline{T_{v}}\,
x^{\lambda _{-}},
\end{gather}
where $\lambda =w(\lambda _{+}) = v(\lambda _{-})$ with
$\lambda_{+}\in X_{+}$, $\lambda _{-}\in -X_{+}$, and $v$ chosen to be
minimal if $\lambda _{-}$ has non-trivial stabilizer.  The polynomials
$E_\lambda(x;t)$ and $F_\lambda(x;t)$ here are the same as in \cite[\S
4.3]{BHMPS23} with $q$ changed to $t$.

The $F_{\lambda }(x;t)$ for $\lambda \in \ZZ^{l}$ form a basis of the
Laurent polynomial ring $\Kfrak [x_{1}^{\pm 1},\ldots,x_{l}^{\pm 1}]$,
and the subset with $\lambda\in \NN^l$ (resp.  $\lambda \in -\NN^l$)
is a basis of $\Kfrak [x_{1},\ldots,x_{l}]$ (resp.\ $\Kfrak
[x_{1}^{-1},\ldots,x_{l}^{-1}]$).  The same holds for $E_{\lambda
}(x;t)$.

If $\lambda _{i} = \lambda _{i+1}$, then
$E_{\lambda }(x;t)$ and $F_{\lambda }(x;t)$ are $s_{i}$-invariant.  To
see this in the case of $F_{\lambda }$, observe that, for the minimal
$v$ such that $\lambda =v(\lambda _{-})$, we have $\overline{T_{i}} \,
\overline{T_{v}} = \overline{T_{s_{i}v}} = \overline{T_{v s_{j}}} =
\overline{T_{v}}\, \overline{T_{j}}$ for $s_{j}$ fixing $\lambda
_{-}$, hence $\overline{T_{i}}\, \overline{T_{v}}\, x^{\lambda _{-}} =
t^{-1}\, \overline{T_{v}}\, x^{\lambda _{-}}$.  A similar argument
works for $E_{\lambda }$.

Define an inner product on $\Kfrak [x_{1}^{\pm 1},\ldots,x_{l}^{\pm 1}]$ by
\begin{equation}\label{e:t-inner}
\langle f(x),g(x) \rangle_{t} = \langle x^{0} \rangle\, f(x)\, g(x)\,
\Omega [(t-1)\sum _{i<j} x_{i}/x_{j}],
\end{equation}
so that $\langle - , -\rangle_t$ is the $A = t$ specialization of
$\langle - , - \rangle_A$ from \eqref{e:A-inner-product}.

\begin{lemma}\label{lem:t-dual-bases}
(a) The elements $E_{\lambda }(x;t^{-1})$ and $\overline{F_{\lambda
}(x;t^{-1})}$ are dual bases with respect to $\langle -,- \rangle
_{t}$, and (b) The elements $w_{0}\, \overline{E_{\lambda }(x;t)}$ and
$w_{0}\, F_{\lambda }(x;t)$ are also dual bases.  In symbols:
\begin{equation}\label{e:t-dual-bases}
\langle E_{\lambda }(x;t^{-1}),\, \overline{F_{\lambda }(x;t^{-1})}
\rangle_{t} = \delta _{\lambda ,\mu } = \langle w_{0}\,
\overline{E_{\lambda }(x;t)}, \, w_{0}\, F_{\lambda }(x;t)
\rangle_{t}.
\end{equation}
\end{lemma}

\begin{proof}
Part (a) is \cite[Prop.~4.3.2]{BHMPS23} with $q$ there changed to
$t^{-1}$.  Part (b) follows from the identity $\langle w_{0}\,
f(\overline{x}),\, w_{0}\, g(\overline{x}) \rangle _{t} = \langle
f(x),\, g(x) \rangle_{t}$, which is immediate from \eqref{e:t-inner}.
\end{proof}

\begin{lemma}\label{lem:Ti-self-adjoint}
The operators $T_{i}$ are self-adjoint for the inner product in
\eqref{e:t-inner}.
\end{lemma}

\begin{proof}
Equivalent to \cite[Lemma~4.3.3]{BHMPS23}, using Remark \ref{rem: Ti
conventions}.
\end{proof}

\begin{lemma}\label{lem:F-basics}
(a) For all $\mu \in \ZZ ^{l}$ and $d\in \ZZ $,
\begin{equation}\label{e:F-shift}
F_{\mu +(d^{l})}(x; t) = (x_{1}\cdots x_{l})^{d}\, F_{\mu
}(x;t).
\end{equation}

(b) If $\mu =(\mu_1,\ldots,\mu_l)$ satisfies $\mu _{i}\leq
\mu _{l}$ for all $i$, then
\begin{equation}\label{e:F-factors}
F_{\mu }(x;t) = x_{l}^{\mu _{l}}\, F_{(\mu
_{1},\ldots,\mu _{l-1})}(x_{1},\ldots,x_{l-1};t).
\end{equation}

(c) With the same hypothesis as in (b),
\begin{equation}\label{e:Tis-on-F}
\overline{T_{m}}\cdots \overline{T_{l-1}}\, F_{\mu }(x;t^{-1}) =
t^{-e}\, F_{s_{m}\cdots s_{l-1}(\mu )}(x; t^{-1}),
\end{equation}
where $e$ is the number of indices $m\leq i< l$ such that $\mu_{i}
= \mu _{l}$.
\end{lemma}

\begin{proof}
Part (a) follows from the fact that $T_i$ commutes with multiplication
by any Laurent polynomial symmetric in $x_i, x_{i+1}$.  Part (b)
follows directly from the definition of $F_\mu$.  Part (c) follows
from the definition of $F_\mu$ and the fact noted above that if $\mu_i
= \mu_{i+1}$, then $F_{\mu }(x;t^{-1})$ is $s_{i}$-invariant, hence
$\overline{T}_i F_\mu(x;t^{-1}) = t^{-1} F_\mu(x;t^{-1})$.
\end{proof}

\subsection{Leading terms in Bruhat order}
\label{ss:Bruhat}

We need the concept of leading term of a Laurent polynomial with
respect to Bruhat order on the weight lattice $\ZZ ^{m}$ of $\GL
_{m}$.  We also need to adapt this concept to functions $f(x,Y,z)$
that are polynomials in $x = x_{1}, \ldots,x_{r}$ and
$z=z_{1},\ldots,z_{s}$ and symmetric functions in $Y$.

There is a Bruhat order on the weight lattice of any reductive group
$G$.  For $G = \GL _{m}$, the Bruhat order on the weight lattice $\ZZ
^{m}$ has the following combinatorial description, due to Bj\"orner
and Brenti \cite[Theorem~6.2]{BjorBren96}.  For $\mu \in \ZZ ^{m}$,
let $S(\mu ) = \{(k,i)\in \ZZ \times [m] \mid k \leq \mu _{i} \}$,
where $[m] = {1,\ldots,m}$.  For any subset $S\subseteq \ZZ \times
[m]$ such that $k$ is bounded above for $(k,i)\in S$, define $\varphi
_{S}(k,i)$ to be the number of elements $(k',i')\in S$ such that
$(k',i')\geq (k,i)$ in lexicographic order.  Then $\lambda \leq \mu $
in Bruhat order on $\ZZ ^{m}$ if and only if $|\lambda |= |\mu |$ and
$\varphi _{S(\lambda )}(k,i)\leq \varphi _{S(\lambda )}(k,i)$ for all
$(k,i)\in \ZZ \times [m]$.

Similar to the case of minimal coset representatives in \S
\ref{ss:Sn}, every orbit of a Young subgroup $\Sfrak _{\rr }\subseteq
\Sfrak _{m}$ on $\ZZ ^{m}$ has a unique minimal element in Bruhat
order on $\ZZ ^{m}$.  Comparing minimal orbit representatives induces
a Bruhat order on $\Sfrak _{\rr }$ orbits (comparing maximal elements
induces the same ordering on orbits).

We say that a Laurent polynomial $f(x_{1},\ldots,x_{m})$ has {\em
leading term} $x^{\mu }$ in Bruhat order if $x^{\mu }$ is greatest in
Bruhat order among all monomials in the support of $f$.

Given a polynomial $f(x_{1},\ldots,x_{m})$ symmetric in
$x_{r+1},\ldots,x_{r+n}$, let $m = r+n+s$ and set $(x_{1}, \ldots,
x_{r}, y_{1}, \ldots, y_{n}, z_{1}, \ldots, z_{s}) =
(x_{1},\ldots,x_{m})$, so that $f(x,y,z)$ is symmetric in $y$.  Then
the coefficients $\langle x^{\eta } \, y^{\nu }\, z^{\kappa } \rangle
\, f(x,y,z) = \langle x^{(\eta ;\nu ;\kappa )} \rangle\, f(x)$ are
constant on $\Sfrak _{[r+1,r+n]}$ orbits of $\GL _{m}$ weights $(\eta
;\nu ;\kappa )$, and the coefficient $\langle x^{\eta }\, m_{\lambda
}(y)\, z^{\kappa } \rangle\, f$, where $\ell (\lambda )\leq n$, is
equal to the coefficient $\langle x^{\eta } \, y^{\nu }\, z^{\kappa }
\rangle \, f$ for any permutation $\nu $ of $(\lambda ;0^{n-\ell
(\lambda )})$.  We say that $x^{\eta }\, m_{\lambda }(y)\, z^{\kappa
}$ is the {\em leading term} of $f(x,y,z)$ if the corresponding
$\Sfrak _{[r+1,r+n]}$ orbit of $\GL _{m}$ weights is greatest in
Bruhat order among all orbits in the support of $f$.  Given the
symmetry in $y$, this is equivalent to $f$ having leading monomial
term $x^{\eta } \, y^{ \nu }\, z^{\kappa }$, where $\nu =(0^{n-\ell
(\lambda )}; \lambda _{-})$, with $\lambda _{-}$ the rearrangement of
the parts of $\lambda $ in increasing order.

It follows from the Bj\"orner--Brenti description that the Bruhat
order on weights $(\eta ;\, 0^{n-\ell (\lambda )};\, \lambda ;\,
\kappa )$, for $\eta \in \NN ^{r}$, $\kappa \in \NN ^{s}$, and
$\lambda $ a partition, is independent of $n$ as long as $n\geq \ell
(\lambda )$.  This induces a well-defined `stable' Bruhat order on
$\NN ^{r}\times \operatorname{Par}\times \NN ^{s}$, where
$\operatorname{Par}$ is the set of all partitions.  We say that
$x^{\eta } \, m_{\lambda }(Y) \, z^{\kappa }$ is the leading term of
$f(x,Y,z)\in \kk [x_{1},\ldots,x_{r}]\otimes \Lambda _{\kk }(Y)\otimes
\kk [z_{1},\ldots,z_{s}]$ if it is the leading term in the stable
Bruhat order, which is equivalent to $f[x,y_{1}+\cdots +y_{n},z]$
having leading term $x^{\eta } \, m_{\lambda }(y) \, z^{\kappa }$ for
all (or for any) $n\geq \deg _{Y}(f)$.

\subsection{Nonsymmetric Macdonald polynomials}
\label{ss:ns-Macs}

The \emph{nonsymmetric Macdonald polynomials} $E_{\mu }(x_1,\dots,
x_m;q,t)$ for $\GL_m$ form a basis for the ring of Laurent polynomials
$\kk[x_1^{\pm 1}, \dots, x_m^{\pm 1}]$ as $\mu$ ranges over $\ZZ^{m}$,
where $\kk = \QQ(q,t)$.  Our conventions for $E_\mu(x;q,t)$ match
those in \cite{HagHaiLo08}, and we mainly follow this reference for
the following brief review of essentials.

The $E_\mu(x;q,t)$ are characterized uniquely by triangularity and
orthogonality properties (see, e.g., \cite[\S 2.1]{HagHaiLo08}).  The
triangularity property is
\begin{align}
E_{\mu }(x;q,t) = x^{\mu } + \sum _{\lambda
<\mu } c_{\lambda }(q,t) x^{\lambda },
\end{align}
where $<$ is the Bruhat order on $\ZZ ^{m}$, i.e., $E_{\mu }(x;q,t)$
is monic with leading term $x^{\mu }$ in the sense of \S
\ref{ss:Bruhat}.  In particular, this implies that $E_{\mu }(x;q,t)\in
\kk [x_{1},\ldots,x_{m}]$ for $\mu \in \NN ^{m}$ and hence that $\{
E_{\mu }(x;q,t)\mid \mu \in \NN ^{m} \}$ is a basis of $\kk
[x_{1},\ldots,x_{m}]$.

For $\GL _{m}$, the nonsymmetric Macdonald polynomials are also
determined by the Knop-Sahi recurrence \cite{Knop97,
Sahiinterpolation}, consisting of the following two identities and the
initial condition $E_{(0^m)}(x;q,t) = 1$.  The first is the
\emph{intertwiner relation}: for $\mu \in \ZZ ^{m}$ such that $\mu
_{i}>\mu _{i+1}$,
\begin{equation}\label{e:intertwiner}
E_{s_{i}\, \mu }(x;q,t) = \left(T_{i} + \frac{1-t}{1-q^{\mu _{i}-\mu
_{i+1}} t^{w_{\mu }(i) - w_{\mu }(i+1)}} \right) E_{\mu }(x;q,t),
\end{equation}
where $w_{\mu }\in \Sfrak _{m}$ is the minimal permutation such that
$w_{\mu }^{-1}(\mu )$ is antidominant, i.e., the permutation such that
for $i<j$, we have $w_{\mu }(i)>w_{\mu }(j)$ if and only if $\mu
_{i}>\mu _{j}$.

The second is the \emph{rotation identity}
\begin{equation}\label{e:rotation}
E_{(\mu _{m}+1,\mu _{1},\ldots,\mu _{m-1})}(x;q,t) = q^{\mu _{m}}\,
x_{1}\, E_{\mu }(x_{2},\ldots,x_{m},q^{-1} x_{1}; q,t)\, .
\end{equation}

We also note that if $\mu _{i} = \mu _{i+1}$, then $s_{i} E_{
\mu}(x;q,t) = E_{\mu }(x;q,t)$.

For any orbit $\Sfrak _{\rr }\, \mu $ of a Young subgroup $\Sfrak
_{\rr }\subseteq \Sfrak _{m}$ on $\ZZ ^{m}$, there exists an $\Sfrak
_{\rr }$ invariant $\kk $-linear combination of nonsymmetric
Macdonald polynomials
\begin{equation}\label{e:symmetric-combo}
\sum _{\gamma  \in \Sfrak _{\rr }\mu} c_{\gamma  }\, E_{\gamma }(x;q,t),
\end{equation}
unique up to a scalar factor.  More precisely, the intertwiner
relation \eqref{e:intertwiner} implies that such a linear combination
is given by $\Scal _{\rr } \, E_{\mu }(x;q,t)$, where $\Scal _{\rr }$
is the symmetrization operator in \eqref{e:Hecke-symmetrizer}.  Still
more precisely, if we take $\mu $ to be $\GL _{\rr }$ dominant, let
$\Sfrak _{\sS }= \Stab _{\Sfrak _{\rr }}(\mu )$, and let $\Sfrak _{\rr
}/\Sfrak _{\sS }$ denote the set of coset representatives $w\in \Sfrak
_{\rr }$ such that $w$ is minimal in $w\, \Sfrak _{\sS }$, then
\begin{equation}\label{e:monic-symmetric-combo}
\sum _{w\in \Sfrak _{\rr }/\Sfrak _{\sS }} T_{w}\, E_{\mu }(x;q,t)
\end{equation}
is the monically normalized version of \eqref{e:symmetric-combo}, in
which $x^{w(\mu )}$ has coefficient $1$ for all $w\in \Sfrak _{\rr }$,
and the leading term is $x^{w(\mu )}$ for $w = w_{0}^{\rr }$ the
longest element in $\Sfrak _{\rr }$.

The following lemma is equivalent to an identity of Knop.

\begin{lemma}[{\cite[Theorem~9.10]{Knop07}}]\label{lem:x1=0}
For $\mu \in \NN ^{m}$, on setting $x_{1} = 0$ in $E_{\mu }(x;q,t)$,
we get
\begin{equation}\label{e:x1=0}
E_{\mu }(0,x_{2},\ldots,x_{m};q,t) = \begin{cases}
E_{(\mu _{2},\ldots,\mu _{m})}(x_{2},\ldots,x_{m};q,t)&	\text{if $\mu
_{1} = 0$}\\
0&   \text{if $\mu _{1}>0$.}
\end{cases}
\end{equation}
\end{lemma}

\section{Flagged LLT polynomials} \label{s:flagged-LLT}

In this section we define and study {\em flagged LLT polynomials}
$\Gcal _{\nubold ,\sigma }(X_{1},\ldots,X_{l};\, t)$---flagged
symmetric functions (as defined in \S
\ref{s:flagged-symmetric-functions}) which generalize both symmetric
LLT polynomials and flagged skew Schur functions.  We present
algebraic and combinatorial formulas for flagged LLT polynomials,
taking the algebraic formula as the definition, and deriving the
combinatorial formula from it.  We then formulate an atom positivity
conjecture for a natural sub-class of flagged LLT polynomials, which
generalizes the Schur positivity of symmetric LLT polynomials.  In \S
\ref{s:ns-Mac-pols}, we will see that this leads to an atom positivity
conjecture for plethystically transformed nonsymmetric Macdonald
polynomials, which generalizes the positivity theorem \cite{Haiman01}
for symmetric Macdonald polynomials.

\subsection{Indexing data}
\label{ss:indexing-data}

Flagged LLT polynomials $\Gcal _{\nubold ,\sigma
}(X_{1},\ldots,X_{l};\, t)$ are indexed by pairs $\nubold ,\sigma $ in
which $\nubold $ is given either by a pair of dominant weights for a
Levi subgroup $\GL _{\rr }\subseteq \GL _{l}$, or by a tuple of {\em
ragged-right skew diagrams}
(Definition~\ref{def:ragged-right-diagram}), and $\sigma $ is a
permutation satisfying a certain compatibility condition with $\nubold
$.  We start with the definition in terms of weights.

\begin{defn}\label{def:indexing-data}
{\em Flagged LLT indexing data} $\rr $, $\gamma $, $\eta $, $\sigma $
of length $l$ consist of a strict composition $\rr
=(r_{1},\ldots,r_{k})$ of $l$, a pair of $\GL _{l}$ weights $\gamma
,\eta \in \ZZ ^{l}$, and a permutation $\sigma \in \Sfrak _{l}$,
satisfying the following conditions:
\begin{itemize}
\item [(i)] $\gamma _{i}\leq \eta _{i}$ for all $i=1,\ldots,l$;
\item [(ii)] $\gamma \in X_{++}(\GL _{\rr })$ and $\eta \in X_{+}(\GL
_{\rr })$ are $\GL _{\rr }$ dominant, and $\gamma $ is $\GL _{\rr }$
regular;
\item [(iii)] $\eta =\sigma (\eta _{-})$, where $\eta _{-}\in
-X_{+}(\GL _{l})$ is the $\GL _{l}$ antidominant weight in
$\Sfrak _{l} \, \eta $;
\item [(iv)] $\sigma $ is maximal in its coset $\Sfrak _{\rr }\,
\sigma $.
\end{itemize}
We say that a permutation $\sigma \in \Sfrak _{l}$ is {\em compatible}
with the other data $(\rr ,\gamma, \eta )$ when (iii) and (iv) hold
(this does not depend on $\gamma $).
\end{defn}

Spelling things out more, condition (ii) says that for each $\rr
$-block $\{i,\ldots,j \}$, we have $\eta _{i}\geq \cdots \geq \eta
_{j}$ and $\gamma _{i} >\cdots > \gamma _{j}$, condition (iii) is that
$\eta _{-} = \sigma ^{-1}(\eta )$ is weakly increasing, and condition
(iv) is that $\sigma ^{-1}(i) >\cdots > \sigma ^{-1}(j)$ on each $\rr
$-block.

Since $\eta $ is $\GL _{\rr }$ dominant, its stabilizer in $\Sfrak
_{\rr }$ is a Young subgroup $\Sfrak _{\sS } = \Stab _{\Sfrak _{\rr
}}(\eta )\subseteq \Sfrak _{\rr }$.  The condition $\eta = \sigma
(\eta _{-})$ only depends on the coset $\Sfrak _{\sS }\, \sigma $.
Since this condition means that $\sigma ^{-1}$ sorts $\eta $ into
weakly increasing order, we automatically have $\sigma ^{-1}(i)>\sigma
^{-1}(i+1)$ if $\eta _{i}>\eta _{i+1}$.  Condition (iv) then says that
for $i$, $i+1$ in the same $\rr $-block, we also have $\sigma
^{-1}(i)>\sigma ^{-1}(i+1)$ if $\eta _{i} = \eta _{i+1}$.  Thus, given
condition (iii), condition (iv) is equivalent to $\sigma $ being
maximal in $\Sfrak _{\sS }\, \sigma $.  In other words, the compatible
permutations $\sigma $ are the maximal elements of the cosets $\Sfrak
_{\sS }\, \sigma $ whose elements satisfy $\eta =\sigma (\eta _{-})$.

Since $\eta _{-}$ is antidominant, there is a unique minimal $\sigma
$ such that $\eta =\sigma (\eta _{-})$.  From the preceding, we see
that this minimal $\sigma $ is compatible with $(\rr ,\gamma ,\eta )$
if and only if $\Sfrak _{\sS}$ is trivial, i.e., if and only if $\eta
$ (like $\gamma $) is $\GL _{\rr }$ regular.

\begin{defn}\label{def:standard-compat}
Given $(\rr , \gamma , \eta )$ satisfying conditions (i) and (ii) in
Definition~\ref{def:indexing-data}, and assuming further that both
$\gamma $ and $\eta $ are $\GL _{\rr }$ regular, the {\em standard
compatible permutation} $\sigma $ is the minimal $\sigma $ such that
$\eta =\sigma (\eta _{-})$.
\end{defn}

For combinatorial purposes, we need to represent $(\rr ,\gamma ,\eta
)$ by means of a tuple $\nubold $ of diagrams somewhat more general
than ordinary skew diagrams.  For every positive integer $r$, fix the
distinguished minimal dominant regular $\GL _{r}$ weight
\begin{equation}\label{e:rho-r}
\rho _{r} = -(0,1,\ldots,r-1).
\end{equation}

\begin{defn}\label{def:ragged-right-diagram}
For $\alpha ,\, \beta +\rho _{r}\in X_{+}(\GL _{r})$ satisfying
$\alpha _{i}\leq \beta _{i}$ for all $i$, the diagram $\beta /\alpha $
(as defined in \S \ref{ss:diagrams}, i.e., row $j$ consists of boxes
with northeast corner $(i,j)$ for $\alpha _{j} < i \leq \beta _{j}$)
is a {\em ragged-right skew diagram}.
\end{defn}

If $\beta \in X_{+}(\GL _{r})$, the definition reduces to that of an
ordinary skew diagram.  The weaker condition $\beta +\rho _{r}\in
X_{+}(\GL _{r})$ means that $\beta _{j+1}\leq \beta _{j}+1$ for $1\leq
j <r$.  Thus, in a ragged-right skew diagram, the right end of a row
is allowed to overhang the row below by at most one box, as in the
example shown here.
\begin{equation}\label{e:ragged-right-diagram}
\beta /\alpha = (7,8,5,5,6,7)/(3,3,1,0,0,0) \qquad
\begin{tikzpicture}[scale=.3, baseline=.84cm]
\foreach \beta / \alpha / \y in {7/3/0, 8/3/1, 5/1/2, 5/0/3, 6/0/4, 7/0/5}
  \draw (\alpha,\y) grid (\beta,\y+1);
\end{tikzpicture}
\end{equation}
Our convention is that empty rows matter, i.e., length zero rows
$\beta _{j}/\alpha _{j}$ with $\alpha _{j} = \beta _{j}$ count as part
of the diagram $\beta /\alpha $.  The reason for this is that the
boxes with northeast corner $(\alpha _{j},j)$ and $(\beta _{j}+1,\,
j)$ to the left and right of each row $\beta _{j}/\alpha _{j}$, empty
rows included, play a role in the combinatorial formula for flagged
LLT polynomials.  We define the {\em end content} of row $j$ in a
ragged-right skew diagram to be the content $\beta _{j} +1 -j$ of
the box to its right.  By our choice of $\rho _{r}$, the sequence of
end contents is $\beta +\rho _{r}$.  {The example in
\eqref{e:ragged-right-diagram} has end contents
\begin{equation}\label{e:end-contents}
\beta +\rho _{6} = (7,8,5,5,6,7) - (0,1,2,3,4,5) = (7,7,3,2,2,2).
\end{equation}
Note that runs of equal end contents correspond to runs of overhanging
rows in the diagram.}

Given $(\rr ,\gamma ,\eta )$ as in Definition~\ref{def:indexing-data},
set
\begin{gather}\label{e:rho-concat}
\rho _{\rr }  = (\rho _{r_{1}}; \ldots ;\rho _{r_{k}}) = -(0, 1, \ldots,
r_{1}-1,\, 0, 1, \ldots, r_{2}-1, \ldots, 0, 1, \ldots, r_{k}-1)\\
\label{e:alpha-beta}
\alpha = \gamma -\rho _{\rr },\quad \beta = \eta -\rho _{\rr }.
\end{gather}
Decomposing $\alpha ,\beta $ as concatenations
\begin{equation}\label{e:alpha-beta-concat}
\alpha =(\alpha ^{(1)}; \ldots ;\alpha ^{(k)}),\quad \beta =(\beta
^{(1)}; \ldots ;\beta ^{(k)})
\end{equation}
of subsequences indexed by $\rr $-blocks, conditions (i) and (ii) in
Definition~\ref{def:indexing-data} say that $\nu ^{(i)} = \beta
^{(i)}/\alpha ^{(i)}$ is a ragged-right skew diagram with $r_{i}$ rows
for each $i$.  Thus, to specify $(\rr ,\gamma ,\eta )$ of length $l$, it is
equivalent to specify the tuple $\nubold = (\nu ^{(1)},\ldots,\nu
^{(k)})$ of ragged-right skew diagrams.  The only constraints on
$\nubold $ are that each $\nu ^{(i)}$ has at least one row (including
empty rows), and $\nubold $ has $l$ rows altogether.

\begin{defn}\label{def:diagram-and-weights}
The tuple of ragged-right skew diagrams $\nubold $ constructed above
is the {\em diagram} of the corresponding weight data $(\rr, \gamma
,\eta )$ satisfying (i), (ii) in Definition~\ref{def:indexing-data}
\end{defn}

In connection with this definition, we use the following conventions
and terminology.
\begin{itemize}
\item We regard the weight data $(\rr , \gamma ,\eta )$ and the
diagram $\nubold $ as interchangeable.  In particular, we refer to
both $(\nubold ,\sigma )$ and $(\rr , \gamma ,\eta , \sigma )$ as {\em
flagged LLT indexing data}, and say that $\sigma $ is compatible
either with $\nubold $ or with $(\rr , \gamma ,\eta )$.

\item As in \S \ref{ss:classical-LLT}, we identify $\nubold $ with the
disjoint union of its components, i.e., a box in $\nubold $ means a
box in a specific component $\nu ^{(j)}$.  The total number of boxes
in $\nubold $ is denoted
\[
|\nubold | \, \defeq\, |\nu ^{(1)}| +\cdots + |\nu ^{(k)}| = |\beta |
- |\alpha | = |\eta |- |\gamma |.
\]

\item By {\em rows of $\nubold $} we mean all rows in all components
$\nu ^{(i)}$, numbered consecutively from $1$ to $l$.  Thus, row $j$
of $\nubold $ means row $m$ of $\nu ^{(i)}$, where $j$ is the $m$-th
element of the $i$-th $\rr $-block.  In terms of $\alpha $ and $\beta
$ as defined in \eqref{e:alpha-beta}, row $j$ of $\nubold $ is $\beta
_{j}/\alpha _{j}$.
\end{itemize}
We also need the following ordering on boxes in any tuple of diagrams.

\begin{defn}\label{def:reading-order}
Let $\nubold = (\nu ^{(1)},\ldots,\nu ^{(k)})$ be a tuple of arbitrary
diagrams $\nu ^{(i)} = \beta ^{(i)}/\alpha ^{(i)}$ with a total of $l$
rows.  The {\em content/row reading order} on boxes in $\nubold $, or
in any extension of $\nubold $ with additional boxes in each row, is
the lexicographic ordering by (content, row number), where all rows of
$\nubold $ are numbered consecutively from $1$ to $l$.
\end{defn}

Figure~\ref{fig:flagged-LLT} shows an example of the correspondence
between $(\rr , \gamma ,\eta )$ and its diagram $\nubold $.  We
arrange the components $\nu ^{(i)}$ with $i$ increasing from bottom to
top, in keeping with our `French' convention (\S \ref{ss:diagrams}) on
row indices within each diagram.

\begin{figure}
\begin{tikzpicture}[scale=.5]
\node[left] at (-1,1.5) {$\nu ^{(1)} = \beta ^{(1)}/\alpha ^{(1)} =
(4,4,2)/(2,0,0)$};
\foreach \beta / \alpha / \y in {4/2/0, 4/0/1, 2/0/2}
  \draw (\alpha,\y) grid (\beta,\y+1);
\node[left] at (-1,6) {$\nu ^{(2)} = \beta ^{(2)}/\alpha ^{(2)} =
(5,6,5,6)/(2,1,1,0)$};
\foreach \beta / \alpha / \y in {5/2/4, 6/1/5, 5/1/6, 6/0/7}
  \draw (\alpha,\y) grid (\beta,\y+1);
\node at (2.5,2.5) {$\flagwave{7}$};
\node at (6.5,7.5) {$\flagwave{6}$};
\node at (4.5,1.5) {$\flagwave{5}$};
\node at (5.5,6.5) {$\flagwave{4}$};
\node at (4.5,0.5) {$\flagwave{3}$};
\node at (6.5,5.5) {$\flagwave{2}$};
\node at (5.5,4.5) {$\flagwave{1}$};
\end{tikzpicture}
\begin{align*}
\rr &  = (3,4) \\
\gamma & = \alpha + \rho _{\rr }  = (2,0,0,2,1,1,0) - (0,1,2,0,1,2,3) =
(2,-1,-2,2,0,-1,-3) \\
\eta & = \beta + \rho _{\rr } = (4,4,2,5,6,5,6) - (0,1,2,0,1,2,3)
 =(4,3,0,5,5,3,3) \\
\sigma & = (3,7,2,6,1,5,4)
\end{align*}
\[
w_0\,\sigma^{-1} = (3, 5, 7, 1, 2, 4, 6)\ ,\quad \sigma ^{-1}(\eta )
 = (0,3,3,3,4,5,5)
\]
\caption{\label{fig:flagged-LLT}Flagged LLT indexing data $(\rr
,\gamma ,\eta ,\sigma )$ of length $l=7$ and the corresponding
$(\nubold , \sigma )$, where $\nubold = (\nu ^{(1)}, \nu ^{(2)}) $ is
the diagram of $(\rr ,\gamma ,\eta )$, and $\sigma $ is encoded by
writing the flag number $w_0\, \sigma^{-1}(i) = l+1-\sigma ^{-1}(i)$
in the flag box in row $i$.  Since $ \sigma ^{-1}(\eta ) = \eta _{-}$
and each $\rr $-block of $w_0\,\sigma^{-1}$ is increasing, $\sigma $
is compatible.}
\end{figure}

We incorporate the permutation $\sigma $ into the diagram as follows.

\begin{defn}\label{def:flag-numbers}
Given a tuple of ragged-right skew diagrams $\nubold $, the boxes
immediately to the right of each row, with content equal to the end
content of that row, are {\em flag boxes}.  To represent a permutation
$\sigma \in \Sfrak _{l}$ on a tuple of diagrams with a total of $l$
rows, we label the flag box in row $j$ with the {\em flag number}
$w_{0}\, \sigma ^{-1}(j) = l+1-\sigma ^{-1}(j)$.  In other words,
since row $\sigma (i)$ gets flag number $l+1-i$, the sequence $\sigma
(1),\ldots,\sigma (l)$ lists the row indices in decreasing order of
flag number.
\end{defn}

Figure~\ref{fig:flagged-LLT} illustrates this.  Compatibility of
$\sigma $ and $\nubold $ is expressed in terms of flag numbers by the
following lemma.

\begin{lemma}\label{lem:flag-numbers}
(a) A permutation $\sigma \in \Sfrak _{l}$ is compatible with $\nubold
$ if and only if the corresponding flag numbers decrease in flag boxes
of increasing content, and increase from bottom to top in flag boxes
of equal content in the same component $\nu ^{(i)}$.

(b) The weight $\eta $ in the indexing data is $\GL _{\rr }$ regular
if and only if the corresponding tuple consists of ordinary
(non-ragged-right) skew diagrams.  In this case, the standard
compatible permutation corresponds to flag numbers decreasing in the
content/row reading order.
\end{lemma}

\begin{proof}
Exercise, using the fact that $\eta = \beta +\rho _{\rr }$ is the
sequence of end contents of the rows in the corresponding tuple of
diagrams.
\end{proof}

\begin{remark}\label{rem:flag-numbers}
(i) Part (a) implies that the flag numbers increase as a function of
the row index within each component $\nu ^{(i)}$.

(ii) The lemma gives a pictorial way to understand the observation
preceding Definition~\ref{def:standard-compat} that the minimal
$\sigma $ such that $\eta =\sigma (\eta _{-})$ is only compatible
with $\nubold $ when $\eta $ is $\GL _{\rr }$ regular.  Namely, if
$\eta $ is not regular, then some $\nu ^{(i)}$ is not an ordinary skew
diagram, hence it has at least two rows with flag boxes of equal
content.  Flag numbers corresponding to a compatible permutation must
increase in these rows, hence they are not decreasing in the
content/row reading order.
\end{remark}

\subsection{Triples}
\label{ss:triples}

Before defining $\Gcal _{\nubold ,\sigma }$, we need the notion of
{\em triples} for tuples of diagrams.  Although this will mainly come
into play when we prove the combinatorial formula for $\Gcal _{\nubold
,\sigma }$ in \S \ref{ss:proof}, it turns out that the algebraic
definition of $\Gcal _{\nubold ,\sigma }$ in \S \ref{ss:G-nu-sigma}
must depend on the number of triples of $\nubold $ in order for the
algebraic and combinatorial formulas to match.

For the purposes of the next definition, we allow $\nubold $ to be any
tuple of diagrams as defined in \S\ref{ss:diagrams}.  Adjoining the
boxes to the left and right of a row $\beta ^{(j)}_{r}/\alpha
^{(j)}_{r}$ gives an {\em extended row} $(\beta ^{(j)}_{r}+1)/(\alpha
^{(j)}_{r} - 1)$.  This includes the case of the two boxes to the left
and right of an empty row with $\alpha ^{(j)}_{r} =\beta ^{(j)}_{r}$.

\begin{defn}\label{def:triple}
A {\em triple of $\nubold $} consists of three boxes $x,y,z$ such that
\begin{itemize}
\item [(i)] $x$ and $z$ are adjacent boxes in an extended row $(\beta
^{(j)}_{r}+1)/(\alpha ^{(j)}_{r} - 1)$ of some component $\nu ^{(j)}$,
with $x$ left of $z$;
\item [(ii)] $y$ is a box in $\nu ^{(i)}$ for some $i<j$;
\item [(iii)] the contents of these boxes satisfy $c(y) = c(z)$
($=c(x)+1)$.
\end{itemize}
The number of triples of $\nubold $ is denoted $h(\nubold )$.
\end{defn}

This definition is similar to \cite[Definition 5.5.1]{BHMPS25a}, but
here we do not need triples twisted by a permutation.

We also need the number of certain special triples for flagged LLT
indexing data $\nubold ,\sigma $.

\begin{defn}\label{def:end-triples}
Let $\nubold $ be a tuple of ragged-right skew diagrams, $\sigma $ a
compatible permutation, encoded by flag numbers in the flag boxes, as
in Definition~\ref{def:flag-numbers}, and $\nubold '$ the diagram
consisting of $\nubold $ together with the flag boxes.  A {\em rising
end triple} is a triple $(x,y,z)$ of $\nubold '$ in which $y$ and $z$
are flag boxes, and the flag numbers $p$ in $y$ and $q$ in $z$ satisfy
$p<q$.  The number of rising end triples is denoted $e(\nubold ,\sigma
)$.
\end{defn}

Alternatively, taking $i,j$ to index the rows containing $y,z$ in a
rising end triple, we see that $e(\nubold ,\sigma )$ is the number of
pairs $i<j$ such that rows $i$ and $j$ are in different components of
$\nubold $, their end contents are equal, and $\sigma ^{-1}(i)>\sigma
^{-1}(j)$.

\subsection{Algebraic definition}
\label{ss:G-nu-sigma}

From here through the end of \S \ref{s:flagged-LLT}, all
\multisymmetric functions and (Laurent) polynomials have
coefficients in the field $\kk =\QQ (t)$.

\begin{defn}\label{def:G-nu-sigma}
The {\em flagged LLT polynomial} corresponding to flagged LLT indexing
data $\nubold ,\sigma $ {of length $l$} is the multi-symmetric
function
\begin{equation}\label{e:G-nu-sigma}
\Gcal _{\nubold , \sigma}(X_{1},\ldots,X_{l};\, t) = t^{h(\nubold ) +
e(\nubold ,\sigma )} \sum _{w\in \Sfrak _{\rr }} (-1)^{\ell (w)}
\langle F_{w(\gamma )}(z;\, t^{-1}) \rangle\, \overline{T_{\sigma }}\,
\Omega [\sum _{i\leq j} X_{i} /z_{w_{0}(j)} ] \, z^{\eta _{-}},
\end{equation}
where $(\rr ,\gamma ,\eta )$ are the weight data with diagram $\nubold
$, $w_{0}$ is the maximal permutation $w_{0}(j) = l+1-j$ in
$\Sfrak_{l}$, $\overline{T_{\sigma }}$ and $F_{\lambda }$ are as in \S
\S \ref{ss:hecke}--\ref{ss:ns-HL-pols}, and $h(\nubold )$ and
$e(\nubold ,\sigma )$ are as above in \S \ref{ss:triples}.
\end{defn}

The expression $\Omega [\sum _{i\leq j} X_{i} /z_{w_{0}(j)} ] \,
z^{\eta _{-}}$ in \eqref{e:G-nu-sigma} is a multi-symmetric series in
$X_{1},\ldots,X_{l}$ with coefficients in $\kk [\zz ^{\pm 1}]$.  The
operator $\overline{T_{\sigma }}$ acts on the variables $\zz =
z_{1},\ldots,z_{l}$ in the coefficients.  Since $\overline{T_{\sigma
}}$ is homogeneous of degree zero, only the homogeneous component of
$\Omega [\sum _{i\leq j} X_{i} /z_{w_{0}(j)} ] $ of degree $-|\nubold
|$ in $z$ and degree $|\nubold |$ in $X$ contributes to the
coefficient of $F_{w(\gamma )}(z;\, t^{-1})$.  By
\eqref{e:flagged-Cauchy}, the coefficients of $\Omega [\sum _{i\leq j}
X_{i} /z_{w_{0}(j)} ] $ with respect to $z$ belong to the space of
flagged symmetric functions $\FLambda  _{\kk }(X_{1},\ldots,X_{l})$,
giving the following result.

\begin{prop}\label{prop:G-nu-flagged}
The flagged LLT polynomial $\Gcal _{\nubold ,\sigma
}(X_{1},\ldots,X_{l};\, t)$ is a flagged symmetric function,
homogeneous of total degree $|\nubold |$ in the formal alphabets $X_{i}$.
\end{prop}

The definition implies that $\Gcal _{\nubold ,\sigma
}(X_{1},\ldots,X_{l};\, t)$ has coefficients in the subring $\ZZ
[t^{\pm 1}]\subseteq \kk $.  The combinatorial formula in the next
subsection further implies that its coefficients with respect to the
basis of monomial multi-symmetric functions are in $\NN [t]$---see
Remark~\ref{rem:monomial-positive}.

Orthogonality of nonsymmetric Hall-Littlewood polynomials yields an
equivalent alternative version of \eqref{e:G-nu-sigma}.

\begin{prop}\label{prop:G-nu-sigma-bis}
In terms of the inner product $\langle -,- \rangle _{t}$ in
\eqref{e:t-inner}, we have
\begin{multline}\label{e:G-nu-sigma-bis}
t^{h(\nubold ) + e (\nubold ,\sigma )} \Gcal _{\nubold ,
\sigma}(X_{1},\ldots,X_{l};\, t^{-1}) = \\
\bigl\langle\, z^{\eta _{+}}\, T_{w_{0}\, \sigma ^{-1} \, w_{0}}\,
w_{0} \sum _{w\in \Sfrak _{\rr }} (-1)^{\ell (w)} \,
\overline{E_{w(\gamma )}(z;\, t)}, \, \Omega [\sum _{i\leq j} X_{i}
/z_{j}]\, \bigr\rangle _{t}.
\end{multline}
\end{prop}

\begin{proof}
After making a change of variables $t\mapsto t^{-1}$,
$z_{i}\rightarrow z_{w_{0}(i)}$ in \eqref{e:G-nu-sigma}, this follows
from $\eta _{+} = w_{0}(\eta _{-})$ and the properties that $w_{0}\,
F_{\lambda }(z; t)$ and $w_{0}\, \overline{E_{\lambda }(z; t)}$ are
dual bases (Lemma~\ref{lem:t-dual-bases}(b)), that $T_{i}$ is
self-adjoint (Lemma~\ref{lem:Ti-self-adjoint}), and that $w_{0}\,
\overline{T_{i}}\, w_{0}\, |_{t\rightarrow t^{-1}} = T_{l-i}$,
which is immediate from the definition \eqref{e:Ti} of $T_{i}$, and
implies that the adjoint of $w_{0}\, \overline{T_{\sigma }}\, w_{0}\,
|_{t\rightarrow t^{-1}}$ is $T_{w_{0}\, \sigma ^{-1} \, w_{0}}$.
\end{proof}

Since $\Gcal _{\nubold ,\sigma}(X_{1},\ldots,X_{l};\, t)$ is a flagged
symmetric function, it is determined by its specialization $\Gcal
_{\nubold ,\sigma}[x_{1},\ldots,x_{l};\, t]$ to single-variable
alphabets $X_{i} = x_{i}$, using Lemma~\ref{lem:X=x}.
Propositions~\ref{prop:PiA-via-inner-product} and
\ref{prop:G-nu-sigma-bis} yield the following formula for this
specialization in terms of nonsymmetric plethysm $\Pi _{t,x}$.

\begin{cor}\label{cor:G-nu-sigma-ter}
The specialization of $\Gcal _{\nubold ,\sigma}(X_{1},\ldots,X_{l};\,
t)$ to single-variable alphabets is given by
\begin{equation}\label{e:G-nu-sigma-ter}
t^{h(\nubold ) + e (\nubold ,\sigma )} \Gcal _{\nubold ,
\sigma}[x_{1},\ldots,x_{l};\, t^{-1}] = \Pi _{t,x}\, x^{\eta _{+}}\,
T_{w_{0}\, \sigma ^{-1} \, w_{0}}\, w_{0} \sum _{w\in \Sfrak _{\rr }}
(-1)^{\ell (w)} \, \overline{E_{w(\gamma )}(x;\, t)},
\end{equation}
where the Hecke algebra operator $T_{w_{0}\, \sigma ^{-1} \, w_{0}} $
acts on $\kk [x_{1}^{\pm 1},\ldots, x_{l}^{\pm 1}]$.  Equivalently, by
Proposition~\ref{prop:PiA-via-pol},
\begin{equation}\label{e:G-nu-sigma-ter-pol}
t^{h(\nubold ) + e (\nubold ,\sigma )} \Gcal _{\nubold ,
\sigma}[x_{1},\ldots,x_{l};\, t^{-1}] = \Bigl( \frac{x^{\eta _{+}}\,
T_{w_{0}\, \sigma ^{-1} \, w_{0}}\, w_{0} \sum _{w\in \Sfrak _{\rr }}
(-1)^{\ell (w)} \, \overline{E_{w(\gamma )}(x;\, t)}}{\prod _{i<j} (1-
t\, x_{i} / x_{j})} \Bigr)_{\pol }.
\end{equation}
\end{cor}
\begin{remark}\label{rem:flip-E}
The functions $w_{0}\, \overline{E_{\lambda }(x;t)} $ that appear in
\eqref{e:G-nu-sigma-bis}--\eqref{e:G-nu-sigma-ter-pol} are also given
by $w_{0}\, \overline{E_{\lambda }(x;t)} = E_{-w_{0}(\lambda )}(x;
t^{-1})$
\end{remark}

As an important special case, when \(\rr=(1^l)\) and all the end
contents of \(\nubold\) are equal, we have \(\rho_\rr = (0^l)\), every
permutation $\sigma$ is compatible, the standard compatible
permutation is the identity, and the flagged LLT polynomial $\Gcal
_{\nubold, \sigma}$ becomes essentially the nonsymmetric plethysm
$\Pi _{t,x}$ applied to a nonsymmetric Hall-Littlewood polynomial.

\begin{cor}\label{cor:single-rows-G}
Let \(\nubold =
((d/\alpha_1),\ldots,(d/\alpha_l))\) be a tuple of single-row diagrams
with equal end contents $d$ (for any integer $d$).  Set
\(\lambda = (\lambda_1,\ldots,\lambda_l)\) where \(\lambda_i = d -
\alpha_i\).  Then, for any \(\sigma \in \Sfrak_l\),
\begin{equation}\label{e:general-rows-G}
t^{h(\nubold)+\inv(\sigma)}\, \Gcal_{\nubold,
\sigma}[x_1,\ldots,x_l;t^{-1}] = \Pi_{t,x} \, T_{w_0 \sigma^{-1}
w_0}\, E_{w_0(\lambda)}(x;t^{-1}) \,.
\end{equation}
In particular, if $\sigma $ is the identity, then
\begin{equation}\label{e:rows-G}
t^{h(\nubold)}\, \Gcal_{\nubold, \sigma}[x_1,\ldots,x_l;t^{-1}] =
\Pi_{t,x}\, E_{w_0(\lambda)}(x;t^{-1}) \,.
\end{equation}
\end{cor}

\begin{proof}
Apply Corollary~\ref{cor:G-nu-sigma-ter} with \(\beta = \eta = (d^l)\)
and \(\alpha = \gamma = (d-\lambda_1,\ldots,d-\lambda_l)\) and use
Remark~\ref{rem:flip-E}.  By definition, \(e(\nubold,\sigma) =
\inv(\sigma)\) when each shape has a single flag box and all the flag
boxes have the same content.  Equation \eqref{e:general-rows-G} then
follows.
\end{proof}

\subsection{Combinatorial formula}
\label{ss:combinatorial-formula}

Theorem~\ref{thm:combinatorial-G-nu-sigma}, below, gives a
combinatorial formula for $\Gcal _{\nubold ,\sigma
}(X_{1},\ldots,X_{l};\, t)$ with the formal alphabets $X_{i}$
specialized to (signed) sums of variables indexed by subsets of a
(signed) alphabet $\Acal $.  Like \eqref{e:classical-LLT},
\eqref{e:signed-classical-LLT} for symmetric LLT polynomials, our
formula is a weighted sum over tableaux---in this case, over {\em
flagged tableaux}.  Before stating the theorem, we define these
tableaux.

\begin{defn}\label{def:i-admissible}
Let $\nubold ,\sigma $ be flagged LLT indexing data, interpreted as a
tuple of ragged-right skew diagrams $\nubold $ with $\sigma $ encoded
by flag numbers in the flag boxes (Definition~\ref{def:flag-numbers}).
A box $x$ in a component $\nu ^{(j)}$ of $\nubold $ is {\em
$i$-admissible} if
\begin{itemize}
\item [(i)] $i$ is less than or equal to the flag number in the row
containing $x$, and
\item [(ii)] $i$ is greater than the flag number in any flag box below
$x$ in the column of $\nu^{(j)}$ containing $x$
\end{itemize}
(condition (ii) holds vacuously if $\nu ^{(j)}$ is an ordinary skew
diagram).
\end{defn}

Figure~\ref{fig:flagged-LLT-2} illustrates $i$-admissible regions.
The following lemma may help clarify the picture.

\begin{lemma}\label{lem:i-admissible}
Given a component $\nu ^{(j)}$ of $\nubold $, let $\theta \subseteq
\nu ^{(j)}$ be the subdiagram consisting of rows with flag number
greater than or equal to $i$.  The set of $i$-admissible boxes in $\nu
^{(j)}$ is the largest ordinary skew diagram contained in $\theta $,
i.e., the skew diagram obtained by deleting all boxes of $\theta $
that overhang the end of a lower row in $\theta $.
\end{lemma}

\begin{proof}
If the flag number in every row of $\nu ^{(j)}$ is greater than $i$,
then $\theta $ is empty, and there are no $i$-admissible boxes in $\nu
^{(j)}$.  Otherwise, let row $r$ be the first row of $\nu ^{(j)}$
which contains a flag number greater than or equal to $i$.  Then, by
Remark~\ref{rem:flag-numbers}(i), $\theta $ consists of rows $r$ and
above.  In particular, $\theta $ is a ragged-right skew diagram.  It
is straightforward to see that every ragged-right skew diagram
contains a largest ordinary skew diagram, obtained by deleting all
overhanging boxes.

By construction, a box $x$ in $\nu ^{(j)}$ satisfies condition (i) if
and only if $x\in \theta $, and $x$ satisfies condition (ii) if and
only if no flag box in a row of $\theta $ is below $x$ in the same
column.  Thus, the $i$-admissible boxes in $\nu ^{(j)}$ are the boxes
in $\theta $ that do not overhang any flag box in $\theta $.
\end{proof}

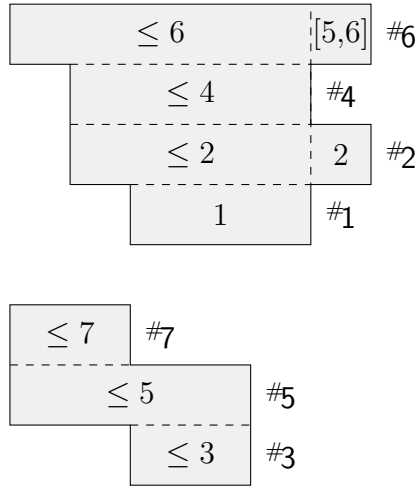
\begin{figure}
\begin{tikzpicture}[scale=.8]
\draw[fill=gray!12] (4,0)--(4,2)--(2,2)--(2,3)--(0,3)--(0,1)--(2,1)--(2,0)
 --cycle;
\draw[fill=gray!12] (5,4)--(5,5)--(6,5)--(6,6)--(5,6)--(5,7)--(6,7)--(6,8)
 --(0,8)--(0,7) --(1,7)--(1,5)--(2,5)--(2,4)--cycle;
\node at (2.5,2.5) {$\flagwave{7}$};
\node at (6.5,7.5) {$\flagwave{6}$};
\node at (4.5,1.5) {$\flagwave{5}$};
\node at (5.5,6.5) {$\flagwave{4}$};
\node at (4.5,0.5) {$\flagwave{3}$};
\node at (6.5,5.5) {$\flagwave{2}$};
\node at (5.5,4.5) {$\flagwave{1}$};

\draw[dashed] (1,7)--(5,7)--(5,8);
\node at (5.5,7.5) {$[5,\!6]$};
\node at (2.5,7.5) {$\leq 6$};
\draw[dashed] (1,6)--(5,6)--(5,7);
\node at (3,6.5) {$\leq 4$};
\draw[dashed] (2,5)--(5,5)--(5,6);
\node at (3,5.5) {$\leq 2$};
\node at (5.5,5.5) {$2$};
\node at (3.5,4.5) {$1$};

\draw[dashed] (0,2)--(2,2);
\node at (1,2.5) {$\leq 7$};
\draw[dashed] (2,1)--(4,1);
\node at (2,1.5) {$\leq 5$};
\node at (3,0.5) {$\leq 3$};
\end{tikzpicture}
\caption{\label{fig:flagged-LLT-2}Admissible regions for the indexing
data $\nubold , \sigma $ in Fig.~\ref{fig:flagged-LLT}.  The
$i$-admissible region is where $i$ belongs to the indicated range of
values.}
\end{figure}

\begin{defn}\label{def:FSST}
Let $\nubold ,\sigma $ be flagged LLT indexing data of length $l$. Let
$\Acal _{1}<\cdots <\Acal _{l}$ be subsets of a signed alphabet $\Acal
$, where $\Acal _{i}<\Acal _{j}$ means $a<b$ for all $a\in \Acal
_{i}$, $b\in \Acal _{j}$.  A {\em flagged tableau} $T$ on $\nubold $
is a filling $T\colon \nubold \rightarrow \bigcup _{i}\Acal _{i}$ such
that
\begin{itemize}
\item [(i)] every row of $T$ is weakly increasing with no repeated
negative letter;
\item [(ii)] within each component $\nu ^{(j)}$, every column of $T$
is weakly increasing with no repeated positive letter;
\item [(iii)] if $T(x)\in \Acal _{i}$, the box $x$ is $i$-admissible.
\end{itemize}
For brevity, we denote the set of flagged tableaux on $\nubold $ by
$\FSST (\nubold, \sigma ,\Acal )$, although it also depends on the
subsets $\Acal _{i}$.  We also use the term {\em flagged semistandard
tableau} for a flagged tableau with entries in an alphabet $\Acal
=\Acal ^{+}$ with only positive letters.
\end{defn}

\begin{remark}\label{rem:FSST}
Suppose that $x$ and $y$ are in the the same column of a component
$\nu ^{(j)}$, and are separated by a flag box $v$ in that column, with
$x$ below $v$ below $y$.  Then the flag numbers $p$ in the row of $x$
and $q$ in $v$ satisfy $p<q$.  By condition (iii), if $T(x)\in \Acal
_{i}$, then $i\leq p<q$, while if $T(y)\in \Acal _{i'}$, then $i'>q$,
hence $T(x)<T(y)$.  Thus, given that $T$ satisfies condition (iii),
condition (ii) reduces to the condition that $T$ is weakly increasing
with no repeated positive letters on each partial column of $\nu
^{(j)}$ consisting of contiguous boxes not separated by a flag box.
\end{remark}

The following alternative point of view on the definition of flagged
tableaux is often useful.  Assume that $\Acal $ contains letters
$b_{1},\ldots,b_{l}$, or adjoin them if needed, ordered so that
\begin{equation}\label{e:A-with-bi}
\Acal _{1} \sgnle b_{1} \sgnlt \Acal _{2} \sgnle b_{2} \sgnlt \cdots
\sgnle b_{l-1} \sgnlt \Acal _{l} \sgnle b_{l}
\end{equation}
in the notation of \eqref{e:inv-signed}.  More precisely---to
be unambiguous in case some of the $\Acal
_{i}$ are empty---we require that
\begin{equation}\label{e:A-with-bi-better}
\Acal _{1} \cup \cdots \cup  \Acal _{j} \sgnle b_{j} \sgnlt \Acal _{j+1}
\cup  \cdots \cup  \Acal _{l}.
\end{equation}
for all $j$.  In particular, these conditions hold if $b_{j}$ is
positive and is the largest letter in $\Acal _{1} \cup \cdots \cup
\Acal _{j}$, of if $b_{j}$ is negative and is the smallest letter in $
\Acal _{j+1}\cup \cdots \cup \Acal _{l}$.

Let $\nubold '$ be the extended tuple of ragged-right skew diagrams
consisting of $\nubold $ together with the flag box in each row, and
for any filling $T\colon \nubold \rightarrow \bigcup _{i}\Acal _{i}$,
define $T'\colon \nubold '\rightarrow \Acal $ by
\begin{equation}\label{e:T-extended}
T'(x)=\begin{cases}
T(x)&	\text{if $x \in \nubold$}\\
b_{j}&	\text{if $x$ is the flag box containing flag number $j$}.
\end{cases}
\end{equation}
By construction, for $a\in \Acal _{i}$, we have $a \sgnle b_{j}$ if
and only if $i\leq j$, and $a \sgngt b_{j}$ if and only if $i>j$.
Using this and Remark~\ref{rem:FSST}, we get the following
characterization of flagged tableaux in terms of the extended tableau
$T'$.

\begin{lemma}\label{lem:T-prime}
A filling $T\colon \nubold \rightarrow \bigcup _{i}\Acal _{i}$ is a
flagged tableaux if and only if $T'$ is a {\em ragged-right super
tableau} on $\nubold '$, meaning that for adjacent boxes $x, y$ in the
same row of $\nubold '$, with $x$ left of $y$, we have $T'(x) \sgnle
T'(y)$, and for adjacent boxes $x, y$ in the same column of any
component $\nu '\, ^{(j)}$, with $x$ below $y$ and $y$ not a flag box,
we have $T'(x) \sgnlt T'(y)$.
\end{lemma}

\begin{remark}\label{rem:T-prime}
(i) In certain contexts (e.g.,
Corollary~\ref{cor:G-nu-sigma-flag-bounds},
Definition~\ref{def:non-attacking}, and
Proposition~\ref{prop:non-attacking}) the letters $b_{j}$ in
\eqref{e:A-with-bi}--\eqref{e:T-extended} will be the flag bounds
$b_{j}$ in the flag bound specialization $\Gcal _{\nubold }[X^{b}_{1},
\ldots, X^{b}_{l}; t]$ of a flagged LLT polynomial.  We call the
letters $b_{j}$ in the general setting {\em flag bounds} as well.

(ii) If $\Acal _{i+1} = \emptyset $, the conditions in
\eqref{e:A-with-bi-better} do not force $b_{i} \sgnlt b_{i+1}$.  For
compatibility with flag bound specializations, we in fact want to
allow positive letters $b_{i} = b_{i+1}$ in this case.

(iii) If $b_{i} \sgnlt b_{i+1}$ for all $i$, then
Remark~\ref{rem:FSST} implies that the tableau $T'$ in
Lemma~\ref{lem:T-prime} has the stronger property that within each
$\nu '\, ^{(j)}$, its full, possibly non-contiguous, columns are
weakly increasing with no repeated positive letters.  However, if
$\Acal _{i+1} = \emptyset $ and $b_{i} \not \sgnlt b_{i+1}$, then $T'$
need not have this property, since the flag box $v$ with flag number
$i+1$ and $T'(v) = b_{i+1}$ might have a box $u$ below it with $T'(u)
= b_{i}$.
\end{remark}

\begin{example}
\label{ex:flagged-tableaux} Consider flagged LLT data with $\nubold =
(\nu ^{(1)})$ consisting of a single ordinary Young diagram $\nu
^{(1)} = (2,1)$, and $\sigma = (2,1)$ the unique compatible
permutation. The $i$-admissible regions are illustrated below, as in
Figure \ref{fig:flagged-LLT-2}, with flag numbers in the flag boxes to
the right of each row.
\begin{equation}\label{e:mini-admissible}
\begin{array}{c}
\begin{tikzpicture}[scale=0.66]
    \foreach \beta / \alpha / \y in {2/0/0,1/0/1}
        \draw (\alpha,\y)  grid (\beta, \y+1);
    \draw[fill=gray!12] (0,0)--(2,0)--(2,1)--(1,1)--(1,2)--(0,2)--(0,0);
    \draw[dashed] (0,1)--(1,1);
    \node at (2.5,0.5) {$\flagwave{1}$};
    \node at (1.5,1.5) {$\flagwave{2}$};
    \node at (1,0.5) {1};
    \node at (0.5,1.5) {$\leq\! 2$};
\end{tikzpicture}
\end{array}
\end{equation}
We list the tableaux in $\FSST
(\nubold ,\sigma ,\Acal )$ for positive alphabets $\Acal _{1} =
\{1,\ldots,b_{1} \}$, $\Acal _{2} = \{b_{1}+1,\ldots,b_{2} \}$, for
the two choices of flag bounds $(b_{1},b_{2})=(1,3)$ and
$(b_{1},b_{2})=(2,3)$.  We illustrate them by displaying the extended
tableaux $T'$ with flag bound $b_i$ written in the flag box with flag
number $i$.
\begin{gather}
\begin{array}{c}
\begin{tikzpicture}[scale=0.5,baseline=0]
    \node at (0,1.5) {\(\Acal_2 = \{2,3\}\)};
    \node at (0,0.5) {\(\Acal_1 = \{1\}\phantom{,2}\)};
\end{tikzpicture}
\quad
\begin{tikzpicture}[scale=0.5,baseline=0]
    \foreach \beta / \alpha / \y in {2/0/0,1/0/1}
        \draw (\alpha,\y)  grid (\beta, \y+1);
    \node at (2.5,0.5) {1};
    \node at (1.5,1.5) {3};
    \node at (0.5,0.5) {1};
    \node at (1.5,0.5) {1};
    \node at (0.5,1.5) {2};
\end{tikzpicture}\;
\begin{tikzpicture}[scale=0.5,baseline=0]
    \foreach \beta / \alpha / \y in {2/0/0,1/0/1}
        \draw (\alpha,\y)  grid (\beta, \y+1);
    \node at (2.5,0.5) {1};
    \node at (1.5,1.5) {3};
    \node at (0.5,0.5) {1};
    \node at (1.5,0.5) {1};
    \node at (0.5,1.5) {3};
\end{tikzpicture}
\end{array}
\\
\begin{array}{c}
\begin{tikzpicture}[scale=0.5,baseline=0]
    \node at (0,1.5) {\(\Acal_2 = \{3\}\phantom{,2}\)};
    \node at (0,0.5) {\(\Acal_1 = \{1,2\}\)};
\end{tikzpicture}
\quad
\begin{tikzpicture}[scale=0.5,baseline=0]
    \foreach \beta / \alpha / \y in {2/0/0,1/0/1}
        \draw (\alpha,\y)  grid (\beta, \y+1);
    \node at (2.5,0.5) {2};
    \node at (1.5,1.5) {3};
    \node at (0.5,0.5) {1};
    \node at (1.5,0.5) {1};
    \node at (0.5,1.5) {2};
\end{tikzpicture}\;
\begin{tikzpicture}[scale=0.5,baseline=0]
    \foreach \beta / \alpha / \y in {2/0/0,1/0/1}
        \draw (\alpha,\y)  grid (\beta, \y+1);
    \node at (2.5,0.5) {2};
    \node at (1.5,1.5) {3};
    \node at (0.5,0.5) {1};
    \node at (1.5,0.5) {1};
    \node at (0.5,1.5) {3};
\end{tikzpicture}\;
\begin{tikzpicture}[scale=0.5,baseline=0]
    \foreach \beta / \alpha / \y in {2/0/0,1/0/1}
        \draw (\alpha,\y)  grid (\beta, \y+1);
    \node at (2.5,0.5) {2};
    \node at (1.5,1.5) {3};
    \node at (0.5,0.5) {1};
    \node at (1.5,0.5) {2};
    \node at (0.5,1.5) {2};
\end{tikzpicture}\;
\begin{tikzpicture}[scale=0.5,baseline=0]
    \foreach \beta / \alpha / \y in {2/0/0,1/0/1}
        \draw (\alpha,\y)  grid (\beta, \y+1);
    \node at (2.5,0.5) {2};
    \node at (1.5,1.5) {3};
    \node at (0.5,0.5) {1};
    \node at (1.5,0.5) {2};
    \node at (0.5,1.5) {3};
\end{tikzpicture}\;
\begin{tikzpicture}[scale=0.5,baseline=0]
    \foreach \beta / \alpha / \y in {2/0/0,1/0/1}
        \draw (\alpha,\y)  grid (\beta, \y+1);
    \node at (2.5,0.5) {2};
    \node at (1.5,1.5) {3};
    \node at (0.5,0.5) {2};
    \node at (1.5,0.5) {2};
    \node at (0.5,1.5) {3};
\end{tikzpicture}
\end{array}
\end{gather}
\end{example}

We now adapt the notions of attacking pairs and attacking inversions
(\S \ref{ss:classical-LLT}) to the flagged setting.  A new feature is
that the second box in an attacking pair can be a flag box.

\begin{defn}\label{def:inversions}
Let $\nubold $ be a tuple of ragged-right skew diagrams, and let
$\nubold '$ be the extended tuple consisting of $\nubold $ together
with the flag boxes.  An {\em attacking pair} for $\nubold $ consists
of boxes $x\in \nu ^{(i)}$ and $y\in \nu '\, ^{(j)}$ such that either
$i<j$ and $c(y) = c(x)$, or $i>j$ and $c(y) = c(x)+1$.

An {\em attacking inversion} in a flagged tableau $T\in \FSST (\nubold
, \sigma ,\Acal )$ is an attacking pair $(x,y)$ such that $T'(x)
\sgngt T'(y)$ in the notation of \eqref{e:inv-signed}, where $T'$ is
the extended tableau in \eqref{e:T-extended}.  More explicitly,
$(x,y)$ is an attacking inversion if $x, y\in \nubold $ and $T(x)
\sgngt T(y)$, or if $y$ is the flag box with flag number $p$, and
$T(x)\in \Acal _{s}$ for $s>p$.  The number of attacking inversions in
$T$ is denoted $\inv (T)$.
\end{defn}

\begin{remark}\label{rem:no-funny-inv}
(i) If $\nubold $ is non-ragged-right and $\sigma $ is the standard
compatible permutation, then $T$ has no attacking inversions involving
flag boxes.  Hence, in this case, $\inv (T) = \oinv (T)$, where $\oinv
(T)$ is as in Definition~\ref{def:inv_old}.  To see this, let $(x,y)$
be an attacking pair with $y$ a flag box, and let $z$ be the flag box
in the row of $x$.  Then $T'(x)\sgnle T'(z)$, and $y$ precedes $z$ in
the content/row reading order, so $\sigma $ being standard implies
that $T'(z)\sgnle T'(y)$.

(ii) For similar reasons, an attacking pair $(x,y)$ with contents
$c(x)=c(y)$ and $y$ a flag box cannot be an attacking inversion in a
flagged tableau $T$.
\end{remark}

\begin{thm}\label{thm:combinatorial-G-nu-sigma}
Given a signed alphabet $\Acal $ and subsets $\Acal _{1}<\cdots <\Acal
_{l}$, define $X_{i}^\Acal = \sum _{a\in \Acal _{i}^{+}}\, x_{a} -
\sum _{b\in \Acal _{i}^{-}}\, x_{b}$.  Then, on specializing the
formal alphabets $X_i$ to $X^\Acal _i$, we have
\begin{equation}\label{e:combinatorial-G-nu-sigma}
\Gcal _{\nubold ,\sigma }[X^\Acal_{1}, \ldots, X^\Acal_{l};\, t] = \sum _{T\in
\FSST (\nubold , \sigma , \Acal )} (-1)^{m(T)} t^{\inv (T)} x^{T},
\end{equation}
where $x^{T} = \prod _{u\in \nubold } x_{T(u)}$, $\inv (T)$ is as in
Definition~\ref{def:inversions}, and $m(T) = |T^{-1}(\Acal ^{-})|$ is
the number of negative entries in $T$.
\end{thm}

We prove Theorem~\ref{thm:combinatorial-G-nu-sigma} in \S
\ref{ss:proof}.  For now, we note some of its consequences. To start,
we can deduce a combinatorial formula for every flag bound
specialization of $\Gcal _{\nubold ,\sigma }(X_{1},\ldots,X_{l};\,
t)$.  Given flag bounds $b_{1}\leq \cdots \leq b_{l}\leq n$, we write
$\Acal _{+}^{b}$ for the positive alphabet $\Acal = \{1<2<\cdots
<n\}$, with distinguished subsets
\begin{equation}\label{e:A-plus-b}
\Acal_1 = \{1, 2,\dots, b_1\}, \,\Acal_2 = \{b_1+1, \dots,
b_2\},\, \ldots,\, \Acal_l = \{b_{l-1}+1, \dots, b_l\}\, ,
\end{equation}
so that the $X^\Acal _i$ in Theorem~\ref{thm:combinatorial-G-nu-sigma}
for $\Acal = \Acal_+^b$ become the $X_i^b$ from Definition
\ref{def:flag-bounds}.  The set of flagged tableaux on these alphabets
is then denoted $\FSST (\nubold ,\sigma,\Acal_+^b)$.  Explicitly, a
tableau $T \in \FSST (\nubold ,\sigma,\Acal_+^b)$ is a filling
$T\colon \nubold \rightarrow \{1, \ldots, n\}$ such that the extended
filling $T'$ with fixed entry $b_{i}$ in the flag box with flag number
$i$ is a ragged-right semistandard tableau, as in
Example~\ref{ex:flagged-tableaux}.

\begin{cor}\label{cor:G-nu-sigma-flag-bounds}
The specialization of $\Gcal _{\nubold ,\sigma }(X_{1},\ldots,X_{l};\,
t)$ with flag bounds $b_{1}\leq \cdots \leq b_{l}$, as in
Definition~\ref{def:flag-bounds}, is given by
\begin{equation}\label{e:flag-bound-G-nu-sigma}
\Gcal _{\nubold ,\sigma }[X_{1}^b, \ldots, X_{l}^b;\, t] = \sum _{T\in
\FSST(\nubold , \sigma,\mathcal A_+^b) } t^{\inv (T)} x^{T}\,.
\end{equation}
\end{cor}

\begin{remark}\label{rem:monomial-positive}
Corollary \ref{cor:G-nu-sigma-flag-bounds} shows that the polynomial
$\Gcal _{\nubold ,\sigma }[X^{b}_{1}, \ldots, X^{b}_{l};\, t]$ has
coefficients in $\NN [t]$.  By Remark \ref{rem:big enough bounds}, it
follows that, in fact, the expansion of $\Gcal _{\nubold ,\sigma
}(X_{1}, \ldots, X_{l};\, t)$ in terms of the basis of multi-symmetric
functions $m_{\lambold } = m_{\lambda^{(1)}}(X_1)
m_{\lambda^{(2)}}(X_2) \cdots m_{\lambda^{(l)}}(X_l)$ has coefficients
in $\NN [t]$.
\end{remark}

\begin{example}\label{ex:flag-LLT-ex-2}
We illustrate Corollary~\ref{cor:G-nu-sigma-flag-bounds} and
Theorem~\ref{thm:combinatorial-G-nu-sigma} by computing flagged LLT
polynomials combinatorially and algebraically for two specializations:
one with a positive alphabet $\Acal $, and one with a signed alphabet.
Fix \(\nubold = ((2)/(1), (2))\).

(i) For each $\sigma\in\Sfrak_2$, we compute
$\Gcal_{\nubold,\sigma}[x_1, x_2; t]$ using
Corollary~\ref{cor:G-nu-sigma-flag-bounds} for the alphabet $\Acal
^{b}_{+}$ with flag bounds $(b_1,b_2)=(1,2)$. For each
$T\in\FSST(\nubold, \sigma, \Acal ^{b}_{+})$, we display the extended
tableaux $T'$ with $b_{i}$ in the flag box with flag number $i$. We
indicate inversions with arrows, and write each term $t^{\inv (T)}\,
x^{T}$ in \eqref{e:flag-bound-G-nu-sigma} below the corresponding
tableau $T'$.
\begin{equation}\label{e:Gnu-12-x1x2}
\arraycolsep=.33ex
\begin{array}{r@{\quad }cccc}
\FSST(\nubold,(1,2), \Acal^{b}_{+})
&&
\begin{tikzpicture}[scale=.5,baseline=.5cm]
  \foreach \beta / \alpha / \y in {2/1/0,2/0/2}
    \draw (\alpha,\y)  grid (\beta, \y+1);
  \node at (1.5,0.5) {1};
  \node at (2.5,0.5) {2};
  \node at (0.5,2.5) {1};
  \node at (1.5,2.5) {1};
  \node at (2.5,2.5) {1};
\end{tikzpicture}
&&
\begin{tikzpicture}[scale=.5,baseline=.5cm]
  \tikzstyle{vertex}=[outer sep=1.5pt]
  \foreach \beta / \alpha / \y in {2/1/0,2/0/2}
     \draw (\alpha,\y)  grid (\beta, \y+1);
   \node[vertex] (a) at (1.5,0.5) {2};
   \node at (2.5,0.5) {2};
   \node at (0.5,2.5) {1};
   \node[vertex] (b) at (1.5,2.5) {1};
   \node at (2.5,2.5) {1};
   \draw[-Latex] (a) -- (b);
\end{tikzpicture}
\\[3ex]
\Gcal _{\nubold ,(1,2)}[x_{1},x_{2};t] & = & x_{1}^{3} & + & t\,
 x_{1}^{2}x_{2}
\end{array}
\end{equation}
\begin{equation}\label{e:Gnu-21-x1x2}
\arraycolsep=.33ex
\begin{array}{r@{\quad }cccccc}
\FSST(\nubold,(2,1), \Acal^{b}_{+})
&&
\begin{tikzpicture}[scale=.5,baseline=.5cm]
  \foreach \beta / \alpha / \y in {2/1/0,2/0/2}
    \draw (\alpha,\y)  grid (\beta, \y+1);
  \node at (1.5,0.5) {1};
  \node at (2.5,0.5) {1};
  \node at (0.5,2.5) {1};
  \node at (1.5,2.5) {1};
  \node at (2.5,2.5) {2};
\end{tikzpicture}
&&
\begin{tikzpicture}[scale=.5,baseline=.5cm]
  \tikzstyle{vertex}=[outer sep=1.5pt]
  \foreach \beta / \alpha / \y in {2/1/0,2/0/2}
     \draw (\alpha,\y)  grid (\beta, \y+1);
  \node[vertex] (b) at (1.5,0.5) {1};
  \node[vertex] (d) at (2.5,0.5) {1};
  \node[vertex] (a) at (0.5,2.5) {1};
  \node[vertex] (c) at (1.5,2.5) {2};
  \node at (2.5,2.5) {2};
  \draw[-Latex] (c) -- (d.north);
\end{tikzpicture}
&&
\begin{tikzpicture}[scale=.5,baseline=.5cm]
  \tikzstyle{vertex}=[outer sep=1.5pt]
  \foreach \beta / \alpha / \y in {2/1/0,2/0/2}
     \draw (\alpha,\y)  grid (\beta, \y+1);
  \node[vertex] (b) at (1.5,0.5) {1};
  \node[vertex] (d) at (2.5,0.5) {1};
  \node[vertex] (a) at (0.5,2.5) {2};
  \node[vertex] (c) at (1.5,2.5) {2};
  \node at (2.5,2.5) {2};
  \draw[-Latex] (a) -- (b.north);
  \draw[-Latex] (c) -- (d.north);
\end{tikzpicture}
\\[3ex]
\Gcal _{\nubold ,(2,1)}[x_{1},x_{2};t] & = & x_{1}^{3} & + & t\,
x_{1}^{2}x_{2} & + & t^{2}\, x_{1}x_{2}^{2}
\end{array}
\end{equation}
For comparison, we now compute $\Gcal_{\nubold,\sigma}[x_1, x_2;
t^{-1}]$ algebraically using Corollary~\ref{cor:single-rows-G} with
\(d=2\) and \(\alpha = (1,0) \).  Note that \(h(\nubold) = 1\).  For
$\sigma=(1,2)$, we get
\begin{equation}\label{eq:flag-LLT-2-row-compat-perm}
t\, \Gcal_{\nubold,(1,2)}[x_1, x_2; t^{-1}] = \Pi_{t,x}(E_{21}(x;\,
t^{-1})) = \Pi_{t,x}(x_1^2 x_2) = t \, x_1^3 + x_1^2 x_2 \,,
\end{equation}
where we used \(E_{21}(x;\, t) = x_1^2 x_2\) and \eqref{e nspleth n2
t}.  For $\sigma=(2,1)$, noting that $\inv(\sigma)=1$, we get
\begin{equation}\label{eq:flag-LLT-ex-1}
t^2 \, \Gcal_{\nubold, (2,1)}[x_1,x_2;t^{-1}] = \Pi_{t,x} ( T_1 \,
E_{21}(x;t^{-1})) = \Pi_{t,x}(x_1 x_2^2) = t^2 \, x_1^3 + t \, x_1^2
x_2+ x_1 x_2^2 \,.
\end{equation}

(ii) For $\sigma = (2,1)$, we compute $\Gcal _{\nubold
,(2,1)}[x_{1},x_{2}- x_{\overline{1}}\, ; t]$ using
Theorem~\ref{thm:combinatorial-G-nu-sigma} with the signed alphabet
$\Acal = \{1 < \overline{1} < 2\}$, $\Acal ^{+} = \{1,2 \}$, $\Acal
^{-} = \{\overline{1} \}$ and subsets $\Acal_1 = \{1\}$, $\Acal_2 =
\{\overline{1},2\}$.
\begin{equation}\label{e:Gnu-21-x1x1barx2}
\arraycolsep=.25ex
\begin{array}{r@{\quad }cccccccccc}
\FSST(\nubold,(2,1), \Acal)
&&
\begin{tikzpicture}[scale=.5,baseline=.5cm]
  \foreach \beta / \alpha / \y in {2/1/0,2/0/2}
    \draw (\alpha,\y)  grid (\beta, \y+1);
  \node at (1.5,0.47) {1};
  \node at (2.5,0.47) {1};
  \node at (0.5,2.47) {1};
  \node at (1.5,2.47) {1};
  \node at (2.5,2.47) {2};
\end{tikzpicture}
&&
\begin{tikzpicture}[scale=.5,baseline=.5cm]
  \tikzstyle{vertex}=[outer sep=1.5pt]
  \foreach \beta / \alpha / \y in {2/1/0,2/0/2}
     \draw (\alpha,\y)  grid (\beta, \y+1);
   \node[vertex] (b) at (1.5,0.47) {1};
   \node[vertex] (d) at (2.5,0.47) {1};
   \node[vertex] (a) at (0.5,2.47) {1};
   \node[vertex] (c) at (1.5,2.47) {2};
   \node at (2.5,2.47) {2};
   \draw[-Latex] (c) -- (d.north);
\end{tikzpicture}
&&
\begin{tikzpicture}[scale=.5,baseline=.5cm]
  \tikzstyle{vertex}=[outer sep=1.5pt]
  \foreach \beta / \alpha / \y in {2/1/0,2/0/2}
     \draw (\alpha,\y)  grid (\beta, \y+1);
   \node[vertex] (b) at (1.5,0.47) {1};
   \node[vertex] (d) at (2.5,0.47) {1};
   \node[vertex] (a) at (0.5,2.47) {2};
   \node[vertex] (c) at (1.5,2.47) {2};
   \node at (2.5,2.47) {2};
   \draw[-Latex] (c) -- (d.north);
   \draw[-Latex] (a) -- (b.north);
\end{tikzpicture}
&&
\begin{tikzpicture}[scale=.5,baseline=.5cm]
  \tikzstyle{vertex}=[outer sep=1.5pt]
  \foreach \beta / \alpha / \y in {2/1/0,2/0/2}
     \draw (\alpha,\y)  grid (\beta, \y+1);
   \node[vertex] (b) at (1.5,0.47) {1};
   \node[vertex] (d) at (2.5,0.47) {1};
   \node[vertex] (a) at (0.5,2.47) {1};
   \node[vertex] (c) at (1.5,2.47) {\mybar{1}};
   \node at (2.5,2.47) {2};
   \draw[-Latex] (c) -- (d.north);
\end{tikzpicture}
&&
\begin{tikzpicture}[scale=.5,baseline=.5cm]
  \tikzstyle{vertex}=[outer sep=1.5pt]
  \foreach \beta / \alpha / \y in {2/1/0,2/0/2}
     \draw (\alpha,\y)  grid (\beta, \y+1);
   \node[vertex] (b) at (1.5,0.47) {1};
   \node[vertex] (d) at (2.5,0.47) {1};
   \node[vertex] (a) at (0.5,2.47) {\mybar{1}};
   \node[vertex] (c) at (1.5,2.47) {2};
   \node at (2.5,2.47) {2};
   \draw[-Latex] (c) -- (d.north);
   \draw[-Latex] (a) -- (b.north);
\end{tikzpicture}
\\[4ex]
\Gcal _{\nubold ,(2,1)}[x_{1},x_{2}-x_{\overline{1}}\,; t] & = &
x_{1}^{3} & + & t\, x_{1}^{2}x_{2} & + & t^{2}\, x_{1}x_{2}^{2} & - &
t\, x_{1}^{2}x_{\overline{1}} & - &\, t^{2} x_{1}x_{\overline{1}}x_{2}\,
.
\end{array}
\end{equation}
For comparison, we compute the flagged symmetric function $\Gcal
_{\nubold ,(2,1)}(X_{1},X_{2};t)$ by applying $\xi ^{-1}$ to $\Gcal
_{\nubold ,(2,1)}[x_{1},x_{2};t]$, and then compute $\Gcal _{\nubold
,(2,1)}[x_{1},x_{2}-x_{\overline{1}}\,; t]$ by specializing.  Using
the value $\Gcal_{\nubold,(2,1)}[x_1,x_2;t] = x_1 x_2^2 + t \, x_1^2
x_2 + t^2 \, x_1^3$ from part (i), this gives
\begin{gather}
\begin{aligned}\label{e:Gnu-21-X1X2}
\Gcal_{\nubold,(2,1)}(X_1, X_2;t)
& = \xi^{-1}(x_1 x_2^2 + t \, x_1^2  + t^2 \, x_1^3)
= \xi^{-1}(t^2\, \fh_{1,2} + (t-t^2)\,\fh_{2,1}
 +(1-t)\,\fh_{3,0} )\\
& = t^2\, \fh_{1,2}(X_1,X_2) +
(t-t^2)\,\fh_{2,1}(X_1,X_2)+(1-t)\,\fh_{3,0}(X_1,X_2).
\end{aligned}
\\
\begin{aligned}\label{e ex signG ii}
\Gcal_{\nubold,\sigma}[x_1, \, & x_2 -x_{\overline{1}};t]\\
 & = t^2\, h_1[x_1]h_2[x_1+ x_2-x_{\overline{1}}]+
(t-t^2)\,h_{2}[x_1]h_1[x_1+ x_2-x_{\overline{1}}]
 +(1-t)\,h_3[x_1] \\
 & = x_1^3 + t \, x_1^2 x_2 + t^2 \, x_1 x_2^2 - t \, x_1^2
x_{\overline{1}} - t^2 \, x_1 x_{\overline{1}} x_2 \,.
\end{aligned}
\end{gather}
\end{example}

\begin{example}\label{ex:demazure-on-flag-bounds}
Since $\Gcal_{\nubold,\sigma}[X_1, \dots, X_l]$ is a flagged symmetric
function, Proposition~\ref{prop:demazure-on-flagged} applies to its
flag bound specializations.  This implies that applying $\Dem _k$ to
the combinatorial formula in
Corollary~\ref{cor:G-nu-sigma-flag-bounds} has the effect of changing
the flag bound $k$ to $k+1$ in the extended tableaux $T'$, if the flag
bound $k$ is not repeated.  For instance, for $\nubold = ((2,1))$ and
$\sigma = (2,1)$ with flag bounds $(b_1,b_2) = (1,3)$ or $(2,3)$, as
in Example~\ref{ex:flagged-tableaux}, we have
\begin{align}
\Gcal _{\nubold, \sigma}[x_1, x_2+x_3; t] & \ = \sum_{T\in
\FSST(\nubold, \sigma,\Acal_{+}^{(1,3)}) \hspace{-2ex}} t^{\inv(T)}
x^T \  = \  x_1^2x_2+ x_1^2x_3 \\ 
\Gcal _{\nubold, \sigma}[x_1 +x_2,x_3; t] & \ = \sum_{T\in
\FSST(\nubold, \sigma,\Acal_{+}^{(2,3)}) \hspace{-2ex}} t^{\inv(T)}
x^T \  = \ x_1^2x_2+ x_1^2x_3 + x_1 x_2^2 + x_1 x_2 x_3 + x_2^2 x_3,
\end{align}
which agrees with   $\Dem_1 (x_1^2x_2+ x_1^2x_3) =
x_1^2x_2+ x_1^2x_3 + x_1 x_2^2 + x_1 x_2 x_3 + x_2^2 x_3 $.
\end{example}

\begin{example}\label{ex:nu=single-diagram}
We consider cases of Corollary~\ref{cor:G-nu-sigma-flag-bounds} for
which $\nubold$ consists of a single shape. Note that $\sigma=w_0$ is
then the only compatible permutation.

(i) When $\nubold = (\lambda/\mu)$ is an ordinary skew shape,
$\FSST(\nubold,w_0,\Acal ^{b}_{+})$ is the set of semistandard
tableaux of shape $\lambda/\mu$ such that all entries in the $i$-th
row are bounded above by $b_i$ and each such tableau has $\inv(T) =
0$.  Hence $\Gcal_{(\lambda/\mu),w_0}[X_1^b,\ldots,X_l^b] = \sum_{T
\in \FSST(\nubold,w_0,\Acal_+^b)} x^T =
\sfrak_{\lambda/\mu}[X_1^b,\ldots,X_l^b]$ by
Corollary~\ref{cor:G-nu-sigma-flag-bounds}
and~\eqref{eq:flagschuridentity}.  By Remark~\ref{rem:big enough
bounds}, choosing $b$ so that $X_i^b$ are large enough gives
\begin{equation}
\Gcal_{(\lambda/\mu),w_0}(X_1,\dots, X_l)=\sfrak_{\lambda/\mu}(X_1,
\dots, X_l)\,.
\end{equation}

(ii) When $\nubold=((m,m+1,\ldots,m+l-1)/\mu)$ is a single `totally
ragged-right' skew shape, where $\mu$ is a partition with $\mu_1 \le
m$, we can show that
\begin{equation}\label{e:total-ragged}
\Gcal_{\nubold, w_0}[x_1, \dots, x_l; t]  = \Acal_{\lambda}(x_1,\dots, x_l),
\end{equation}
where $\lambda = (m-\mu_1, m+1-\mu_2, \dots, m+l-1-\mu_l)$, as
follows.  Apply Corollary~\ref{cor:G-nu-sigma-flag-bounds} with flag
bounds $b = (1,2,\dots, l)$.  Consider $T\in\FSST(\nubold,w_0,\mathcal
A_+^b)$ and let $T_1$ (resp. $T_2$) be its restriction to columns $\le
m$ (resp. \(> m\)).  Then $T_2$ is the tableau whose entire $i$-th row
is filled with $i$'s.  This forces $T_1$ to have entries $\le i$ in
row $i$, but every semistandard tableau of shape $(m^l)/\mu$ filled
with $1,\dots, l$ already has this property.  Since
$s_{(m^{l})/\mu}(x_1,\dots, x_l)= s_{(m-\mu_l, \dots,
m-\mu_1)}(x_1,\dots, x_l)$, we can use \eqref{e:flag-bound-G-nu-sigma}
and Lemma \ref{l dom regular atom} to find
\begin{equation}\label{e:total-ragged-2}
\sum_{T\in\FSST(\nubold,w_0,\Acal_+^b)} t^{\inv(T)} x^T = x_2 x_3^2
\cdots x_l^{l-1} s_{(m^{l})/\mu}(x_1,\dots, x_l) =
\Acal_{\lambda}(x_1, \dots, x_l)\,.
\end{equation}
\end{example}

Theorem~\ref{thm:combinatorial-G-nu-sigma} also implies that symmetric
LLT polynomials can be expressed in terms of flagged LLT polynomials.

\begin{cor}\label{cor:Weyl-on-G-nu-sigma}
(a) If $\nubold $ is a tuple of ordinary skew shapes, then
\begin{equation}\label{e:flagged-to-symmetric}
\Gcal _{\nubold ,\sigma }(X,0,\ldots,0;\, t) = \Gcal _{\nubold }(X;\,
t),
\end{equation}
where $\Gcal _{\nubold }(X;\,
t)$ is a symmetric LLT polynomial independent of $\sigma $.

(b) If $\nubold $ is a tuple of ordinary skew shapes, then for any
flag bounds $b$ such that $b_{i}\geq i$ for all $i$, and any $n\geq
b_{l}$, Weyl symmetrization in variables $x_{1},\ldots,x_{n}$ applied
to the specialization $\Gcal _{\nubold ,\sigma
}[X^{b}_{1},\ldots,X^{b}_{l}; t]$ gives the symmetric LLT polynomial
\begin{equation}\label{e:Weyl-on-G-nu-sigma}
\Weyl \, \Gcal _{\nubold ,\sigma }[X^{b}_{1},\ldots,X^{b}_{l}; t] = \Gcal
_{\nubold }[x_{1}+\cdots +x_{n}; t].
\end{equation}

(c) If any component of $\nubold $ is strictly ragged-right, then for
$b_{i}$ and $n$ as in part (b), we have $\Weyl \, \Gcal _{\nubold ,\sigma
}[X^{b}_{1},\ldots,X^{b}_{l}; t] =0$.
\end{cor}

\begin{proof}
If $\nubold $ is a tuple of ordinary skew diagrams, then every box in
$\nubold $ is $1$-admissible.  A flagged tableaux $T$ with all entries
in $\Acal _{1}$ is just a usual super tableau, and $\inv (T)$ reduces
to $\oinv (T)$, as defined in \S \ref{ss:classical-LLT} for symmetric
LLT polynomials, since pairs $(x,y)$ in which $y$ is a flag box do not
contribute.  This given, Theorem~\ref{thm:combinatorial-G-nu-sigma}
proves (a).  Part (b) then follows from
Corollary~\ref{cor:Weyl-on-flagged}.

If any component of $\nubold $ is strictly ragged-right,
then at least one box is not $1$-admissible, and
Theorem~\ref{thm:combinatorial-G-nu-sigma} implies that $\Gcal
_{\nubold ,\sigma }(X,0,\ldots,0;\, t) = 0$.
Hence (c)  follows using Corollary~\ref{cor:Weyl-on-flagged} again.
\end{proof}

The following combinatorial device is sometimes useful for reducing a
flag bound specialization of a flagged LLT polynomial to the case of
flag bounds $b_{i}=i$.

\begin{lemma}\label{lem:empty-rows-trick}
For any flagged LLT indexing data $\nubold ,\sigma $, flag bounds
$0<b_{1}<\cdots <b_{l}$ and $n\geq b_{l}$, one can construct new
indexing data $\nuboldhat ,\sigmahat $ from $\nubold ,\sigma $ by
adjoining trivial components, each consisting of a single empty row,
in such a way that
\begin{equation}\label{e:empty-rows-trick}
\Gcal _{\nubold ,\sigma }[X^{b}_{1},\ldots,X^{b}_{l}; t] = \Gcal
_{\nuboldhat ,\sigmahat }[x_{1},\ldots,x_{n}; t].
\end{equation}
\end{lemma}

\begin{proof}
A trivial component is determined by the end content of its one empty
row.  Let $b_{0} = 0$ by convention.  To construct $\nuboldhat $,
adjoin $n-l$ trivial components after the last component of $\nubold
$, as follows: for $i=1,\ldots,l$, adjoin $b_{i}- b_{i-1} -1$
components with end content equal to the end content of the row of
$\nubold $ assigned flag number $i$ by $\sigma $; take the remaining
$n-b_{l}$ trivial components to have end content less than the end
content of every row of $\nubold $.  Then there is a compatible
$\sigmahat $ such that for $i=1,\ldots,l$, the flag box of $\nubold $
assigned flag number $i$ by $\sigma $ is assigned flag number $b_{i}$
by $\sigmahat $, while the $b_{i}-b_{i-1}-1$ added flag boxes with the
same content are assigned flag numbers $b_{i-1}+1, \ldots, b_{i}-1$,
and the $n- b_{l}$ added flag boxes with content less than that of
every flag box of $\nubold $ are assigned flag numbers
$b_{l}+1,\ldots, n$.

Since $\nuboldhat $ and $\nubold $ have the same non-flag boxes,
fillings $\That :\nuboldhat \rightarrow [n]$ are naturally identified
with fillings $T: \nubold\rightarrow [n]$.  Let $\bhat _{i} = i$ for
$i=1,\ldots,n$.  The flag box $x$ of $\nubold $ assigned flag number
$i$ by $\sigma $, and flag number $p = b_{i}$ by $\sigmahat $,
receives the same entry $T'(x) = b_{i} = \bhat _{p} = \That '(x)$ in
the extended filling associated to either $T$ or $\That $.  Hence
$T\in \FSST (\nubold ,\sigma ,\Acal ^{b}_{+})$ if and only if $\That \in
\FSST (\nuboldhat ,\sigmahat ,\Acal ^{\bhat }_{+})$.  By
Remark~\ref{rem:no-funny-inv}(ii), $\That $ has no attacking
inversions involving flag boxes in the added trivial components.
Hence $\inv (T) = \inv (\That )$.  Now \eqref{e:empty-rows-trick}
follows from Corollary~\ref{cor:G-nu-sigma-flag-bounds}.
\end{proof}

\subsection{Proof of the combinatorial formula}
\label{ss:proof}

We prove Theorem~\ref{thm:combinatorial-G-nu-sigma} at the end of this
section after some preliminary results, starting with lemmas that
relate attacking inversions (Definition~\ref{def:inversions}) to
triples (\S\ref{ss:triples}).  The first two lemmas apply to fillings
more general than tableaux.

\begin{defn}\label{def:incr-triple}
For a signed alphabet $\Acal $ and a tuple of general diagrams
$\nubold $, let $T\colon \nubold \rightarrow \Acal $ be a
row-increasing super filling.  A triple $(x,y,z)$ of $\nubold$ is {\em
increasing} if $T(x) \sgnle T(y) \sgnle T(z)$, in the notation of
\eqref{e:inv-signed}, where if $x$ and/or $z$ is one of the extra
boxes in an extended row, we set $T(x) = -\infty $ if $x\not \in
\nubold $ and $T(z)=\infty $ if $z\not \in \nubold $.  The number of
increasing triples in $T$ is denoted $h(T)$.
\end{defn}

\begin{example}\label{ex:triplesinvetc}
For flagged LLT indexing data $\nubold ,\sigma $ with $\rr=(2,1,2)$,
$\beta=(3,2,1,3,4)$, $\alpha=(1,0,1,1,1)$, and $\sigma = (2,3,5,4,1)$,
the picture on the left below shows the tuple of diagrams $\nubold $,
with flag numbers written in the flag boxes. The picture on the right
shows the extended tableau \(T'\) of a flagged tableau
$T\in\FSST(\nubold,\sigma,\mathcal A)$, where $\Acal = \{1 <
\overline{1} < 2 < \overline{2} < \cdots\}$, with $\Acal ^{+} =
\{1,2,\ldots \}$, $\Acal ^{-} = \{\overline{1}, \overline{2},\ldots,
\}$, and distinguished subsets $\Acal_1 = \{1,\overline{1},2\},\,
\Acal_2 = \{\overline{2}, 3, \overline{3}, 4\},\, \Acal_3 =
\{\overline{4}, 5\},\, \Acal_4 = \{6, 7\},\, \Acal_5 =
\{8,\overline{8},9\}$. Flag bound \(b_i\) is written in the flag box
with flag number \(i\), where we have chosen \(b = ( \overline{2}, 4,
\overline{5},7,9)\).
\begin{gather*}
\begin{tikzpicture}[scale=.5,baseline=2cm]
    \node at (-1,3.5) {$\nubold$};
    \foreach \beta / \alpha / \y / \fb in
      {3/1/0/1,2/0/1/5,1/1/3/4,3/1/5/2,4/1/6/3} {
        \draw (\alpha,\y)  grid (\beta, \y+1);
        \node at (\beta+0.5,\y+0.5) {$\flagwave{\fb}$};
      }
\end{tikzpicture}
\qquad\qquad  \qquad\qquad
\begin{tikzpicture}[scale=.5,baseline=2cm]
    \node at (-1,3.5) {$T$};
    \foreach \beta / \alpha / \y / \fb in
      {3/1/0/{$\mybar{2}$},2/0/1/9,1/1/3/7,3/1/5/4,4/1/6/{$\mybar{5}$}} {
        \draw (\alpha,\y)  grid (\beta, \y+1);
        \node at (\beta+0.5,\y+0.5) {\fb};
      }
    \foreach \r / \c / \entr in {1/2/1, 2/1/8, 2/2/8, 6/3/2, 7/3/4} {
       \node at (\c-0.5, \r-0.5) {\(\entr\)};
    }
    \foreach \r / \c / \entr in {1/3/1, 6/2/1, 7/2/1, 7/4/4} {
       \node at (\c-0.5, \r-0.5) {\(\mybar{\entr}\)};
    }
\end{tikzpicture}
\\[3ex]
\begin{aligned}
h(\nubold) & = 6  & \qquad h(T) & = 2 & \qquad  \oinv (T) & = 4
 & \qquad m(T) & = 4 \\
e(\nubold,\sigma) & = 2 & e(T) & = 1 & \inv(T) & = 5
\end{aligned}
\end{gather*}
For this choice of $\nubold $, $\sigma $, and $T$, the values of the
following statistics are shown: $\inv(T)$ counts attacking inversions
of $T$, as in Definition~\ref{def:inversions}, $\oinv (T)$ counts
attacking inversions in $T$ considered as a filling of $\nubold $
without the flag boxes, as in Definition~\ref{def:inv_old}, $m(T)$
counts negative letters in $T$, $h(\nubold)$ counts triples of
$\nubold$, as in Definition~\ref{def:triple}, $h(T)$ counts increasing
triples in $T$, $e(\nubold ,\sigma )$ counts rising end triples, as in
Definition~\ref{def:end-triples}, and $e(T)$ counts rising end triples
$(x,y,z)$ for $(\nubold ,\sigma )$ which are increasing for the
extended tableau $T'$---that is, in which either $x\not \in \nubold $,
or $T(x) \sgnle T'(y)$.
\end{example}

We will continue to use the notations $\inv (T)$, $\oinv (T)$, $m(T)$,
$h(\nubold )$, $h(T)$, $e(\nubold ,\sigma )$, and $e(T)$ for flagged
tableaux $T$ throughout this subsection.

\begin{lemma}\label{lem:incr-triples}
Let $T\colon \nubold \rightarrow \Acal $ be a row-increasing super
filling of a tuple of general diagrams, as in
Definition~\ref{def:incr-triple}.  For $a\in \Acal$, let $\nubold
|_{a}$ be the unique subdiagram of $\nubold $ such that $T(u) = a$ for
$u \in \nubold |_{a}$, $T(u)<a$ for $u\in \nubold $ to the left of a
row of $\nubold _{a}$, and $T(u)>a$ for $u\in \nubold $ to the right
of a row of $\nubold _{a}$ (note that this specifies the position of
an empty row in $\nubold |_{a}$ when $T$ has no entries equal to $a$
in a row of $\nubold $).  Let $T|_{a}$ be the restriction of $T$ to
$\nubold _{a}$, i.e., the constant filling $T|_{a}\colon \nubold
_{a}\rightarrow \{a \}$.  Then the number of increasing triples $h(T)$
is given by
\begin{equation}\label{e:incr-triples1}
h(T) = \sum _{a\in \Acal } h(T|_{a}),
\end{equation}
where $h(T|_{a})$ is the number of increasing triples of $T|_{a}$ as a
filling of $\nubold |_{a}$.
\end{lemma}

\begin{proof}
We claim that the increasing triples $(x,y,z)$ of $T$ with $T(y) = a$
are exactly the increasing triples of $T|_{a}$.  Then
\eqref{e:incr-triples1} follows.

For any $a\in \Acal$, every triple $(x,y,z)$ of $\nubold |_{a}$ is
clearly also a triple of $\nubold $.  Consider a triple $(x,y,z)$ of
$\nubold $ with $T(y) = a$, and use the convention in
Definition~\ref{def:incr-triple} that $T(x) = -\infty $ (resp.\ $T(z)
= \infty $) if $x$ (resp.\ $z$) is to the left (resp.\ right) of a row
of $\nubold $.

With this convention, if $a\in \Acal ^{+}$, the criterion for
$(x,y,z)$ to be increasing in $T$ is $T(x)\leq a\leq T(z)$.
Equivalently, $x$ and $z$ satisfy (i) $T(x) = a$, or $x$ is the
rightmost box in its row with $T(x)<a$, and (ii) $T(z) = z$, or $z$ is
the leftmost box in its row with $T(z)>a$.  These are exactly the
criteria for $x$ and $z$ to be in an extended row of $\nubold |_{a}$,
and thus for $(x,y,z)$ to be a triple of $\nubold |_{a}$.  For $a\in
\Acal ^{+}$, every triple of $\nubold |_{a}$ is increasing in
$T|_{a}$, so this proves the claim for positive letters $a$.

If $a\in \Acal ^{-}$, with the same convention as before, the
criterion for $(x,y,z)$ to be increasing in $T$ is $T(x)< a< T(z)$,
i.e., $x$ is the rightmost box in its row with $T(x)<a$, and $z$ is
the leftmost box with $T(z)>a$.  This is equivalent to $\{x,z \}$
being the extension of an empty row of $\nubold |_{a}$.  In
particular, it implies that $(x,y,z)$ is a triple of $\nubold |_{a}$.
For $a\in \Acal ^{-}$, a triple $(x,y,z)$ of $\nubold |_{a}$ is
increasing if and only if $\{x,z \}$ is the extension of an empty row
of $\nubold |_{a}$.  This proves the claim for negative letters $a$.
\end{proof}

\begin{cor}\label{cor:incr-triples}
Let $T\colon \nubold \rightarrow \Acal $ be a row-increasing super
filling of a tuple of general diagrams, as in
Definition~\ref{def:incr-triple}.  Given $\Acal = \Acal ' \cup \Acal
''$ with $\Acal '< \Acal ''$, we have $h(T) = h(T')+h(T'')$, where
$T'$, $T''$ are the restrictions of $T$ to the subdiagrams
$T^{-1}(\Acal ')$, $T^{-1}(\Acal '')$.
\end{cor}

The next lemma is a variant of \cite[eq.~(132)]{BHMPS25a}.

\begin{lemma}\label{lem:inv-vs-triples}
For $\Acal $, $\nubold $, and $T$ as in
Definition~\ref{def:incr-triple}, we have the identity
\begin{equation}\label{e:inv-vs-triples}
\oinv (T) = h(\nubold ) - h(T).
\end{equation}
\end{lemma}

\begin{proof}
Let $(x,y,z)$ be a triple of $\nubold $.  If $x\in \nubold $, then
$(x,y)$ is an attacking pair, and if $z\in \nubold $, then $(y,z)$ is
an attacking pair.  Since we set $T(x) = -\infty $ if $x\not \in
\nubold $, we have $T(x) \sgngt T(y)$ if and only if $x\in \nubold $
and $(x,y)$ is an attacking inversion in $T$.  Similarly, $T(y) \sgngt
T(z)$ if and only if $z\in \nubold $ and $(y,z)$ is an attacking
inversion.  We cannot have $T(x) \sgngt T(y) \sgngt T(z)$ because
$T(x) \sgnle T(z)$ by assumption.  Hence if $(x,y,z)$ is increasing,
neither $(x,y)$ nor $(y,z)$ is an attacking inversion, and if
$(x,y,z)$ is not increasing then exactly one of them is an attacking
inversion.

Since every attacking pair is either $(x,y)$ or $(y,z)$ in a unique
triple $(x,y,z)$, it follows that the number of attacking inversions
is equal to the number of triples that are not increasing, giving
\eqref{e:inv-vs-triples}.
\end{proof}

\begin{lemma}\label{lem:rr-incr-triples}
Given flagged LLT indexing data $\nubold ,\sigma $, and subsets $\Acal
_{1}<\cdots <\Acal _{l}$ of a signed alphabet $\Acal $, as in
Theorem~\ref{thm:combinatorial-G-nu-sigma}, let $\nubold '$ be the
tuple of extended diagrams consisting of $\nubold $ together with the
flag box in each row.  For every flagged tableau $T\in \FSST (\nubold
, \sigma , \Acal )$, we have
\begin{equation}\label{e:rr-incr-triples}
\inv (T) = h(\nubold ) + e(\nubold ,\sigma ) - h(T) - e(T),
\end{equation}
with notation as in Example~\ref{ex:triplesinvetc}.
\end{lemma}

\begin{proof}
Let $\inv _{1}(T)$ be the number of attacking inversions $(x,y)$ of
$T$ in which $y$ is a flag box.  By definition, $\inv (T) = \oinv (T)
+ \inv _{1}(T)$.  Using Lemma~\ref{lem:inv-vs-triples}, this implies
that \eqref{e:rr-incr-triples} is equivalent to $\inv _{1}(T) =
e(\nubold ,\sigma ) - e(T)$.

Let $(x,y)$ be an attacking inversion with $x\in \nu ^{(i)}$, $y\in
\nu '\, ^{(j)}$, and $y$ a flag box.  By
Remark~\ref{rem:no-funny-inv}(ii), we cannot have $c(y) = c(x)$, hence
we must have $c(y)= c(x)+1$ and $i>j$.  Let $T'$ be the extended
tableau in \eqref{e:T-extended} and let $z$ be the flag box in the row
of $x$.  Then $T'(y) \sgnlt T'(x) \sgnle T'(z)$, since $(x,y)$ is an
inversion.  If there are any boxes between $x$ and $z$, then
$c(z)>c(x)+1 = c(y)$, contradicting $T'(y) \sgnlt T'(z)$.  Hence $z$
is adjacent to $x$, and $(x,y,z)$ is a rising end triple, but not an
increasing triple in $T'$.  Conversely, if $(x,y,z)$ is a rising end
triple which is not increasing in $T'$, then $x\in \nubold $ and $T(x)
\sgngt T'(y)$ is an attacking inversion with $y$ a flag box.  This
shows that $\inv _{1}(T) = e(\nubold ,\sigma ) - e(T)$, as desired.
\end{proof}

The last ingredient needed for the proof of
Theorem~\ref{thm:combinatorial-G-nu-sigma} is the analogous result for
symmetric LLT polynomials, which equates the combinatorial formula in
Lemma~\ref{lem:signed-classical-LLT} with an algebraic formula given
by \cite[Theorem 5.5.4]{BHMPS25a}.  We briefly recall the latter
formula---see \cite[\S 5]{BHMPS25a} for more details.

Let $\rr $ be a strict composition of $l=|\rr |$.  It follows from the
definition of nonsymmetric Hall-Littlewood polynomials (\S
\ref{ss:ns-HL-pols}) that
\begin{equation}\label{e:E-anti}
E^{-}_{\rr ,\lambda }(z;t^{-1}) \defeq \sum _{w\in \Sfrak _{\rr }}
(-1)^{\ell (w)} E_{w(\lambda )}(z;t^{-1})
\end{equation}
is antisymmetric for the Hecke algebra $\Hcal (\Sfrak _{\rr })$ acting
on $\kk [z_{1}^{\pm 1},\ldots,z_{l}^{\pm 1}]$, i.e., $T_{w}\,
E^{-}_{\rr ,\lambda }(z;\, t^{-1}) = (-1)^{\ell (w)}\, E^{-}_{\rr
,\lambda }(z;\, t^{-1})$ for $w \in \Sfrak _{\rr }$.  Up to sign,
$E^{-}_{\rr ,\lambda }(z;\, t^{-1})$ clearly depends only on the
$\Sfrak _{\rr }$ orbit of $\lambda $, and vanishes if $\lambda $ is
not $\GL _{\rr }$ regular.  The set $\{ E^{-}_{\rr ,\lambda }(z;\,
t^{-1}) \mid \lambda \in X_{++}(\GL _{\rr }) \}$ is a basis of the
space of $\Hcal (\Sfrak _{\rr })$ antisymmetric elements in $\kk [\zz
^{\pm 1}]$.  These elements are $E^{\operatorname{id},-}_{\rr ,\lambda
}(x;t^{-1})$ in the notation of \cite[eq.~(116)]{BHMPS25a}.  Using
\cite[Remark 5.4.2(i)]{BHMPS25a}, we can formulate \cite[Theorem
5.5.4]{BHMPS25a} as follows.

\begin{prop}\label{prop:symm-LLT-alg}
Let $\rr $ be a strict composition of $l = |\rr |$.  If $\gamma ,\eta
\in X_{++}(\GL _{\rr})$ satisfy $\gamma _{i}\leq \eta _{i}$ for all
$i$, then
\begin{equation}\label{e:symm-LLT-alg}
t^{h(\nubold )} \Gcal _{\nubold }(X; t^{-1}) = \langle E^{-}_{\rr,
\eta }(z; t^{-1}) \rangle\, \Omega [X (z_{1}+\cdots +z_{l})]\,
E^{-}_{\rr, \gamma }(z; t^{-1}),
\end{equation}
where $\nubold $ is the diagram of $(\rr ,\gamma ,\eta )$, as in
Definition~\ref{def:diagram-and-weights} (since both $\gamma $ and
$\eta $ are $\GL _{\rr}$ regular, $\nubold $ is a tuple of ordinary
skew diagrams).

If $\gamma ,\eta \in X_{++}(\GL _{\rr })$ do not satisfy $\gamma
_{i}\leq \eta _{i}$ for all $i$, then the right hand side of
\eqref{e:symm-LLT-alg} is equal to zero.
\end{prop}

Regarding formula \eqref{e:symm-LLT-alg}, note that since $\Omega [X
(z_{1}+\cdots +z_{l})] $ is symmetric in $z$, multiplication by it
commutes with the Hecke algebra action and therefore preserves the
space of $\Hcal (\Sfrak _{\rr })$ antisymmetric elements.

\begin{cor}\label{cor:symm-LLT-alg}
For $\gamma \in X_{++}(\GL _{\rr })$, $\eta \in X_{+}(\GL _{\rr })$,
and $v\in \Sfrak _{\rr }$, we have
\begin{equation}\label{e:symm-LLT-alg-F}
\sum _{w\in \Sfrak _{\rr }} (-1)^{\ell (w)} \langle F_{w(\gamma
)}(z;\, t^{-1}) \rangle\, \Omega [X\, \overline{z_{1}+\cdots
+z_{l}}]\, F_{v(\eta )}(z;\, t^{-1}) = (-1)^{\ell (v)} \,
t^{-h(\nubold )}\, \Gcal _{\nubold }(X;\, t)
\end{equation}
if $\eta$ is $\GL _{\rr }$ regular and $\gamma _{i}\leq \eta _{i}$ for
all $i$, where $\nubold $ is diagram of $(\rr ,\gamma , \eta )$.
Otherwise, the left hand side of \eqref{e:symm-LLT-alg-F} is equal to
zero.
\end{cor}

\begin{proof}
Inverting the variables $z_{i}$ and $t$ turns the left hand side of
\eqref{e:symm-LLT-alg-F} into
\begin{equation}\label{e:symm-LLT-alg-2}
\sum _{w\in \Sfrak _{\rr }} (-1)^{\ell (w)} \langle
\overline{F_{w(\gamma )}(z;\, t^{-1})} \rangle\, \Omega [X\,
(z_{1}+\cdots +z_{l})]\, \overline{F_{v(\eta )}(z;\, t^{-1})}.
\end{equation}
By Lemma~\ref{lem:t-dual-bases}(a) and the definition of $E^{-}_{\rr
,\lambda }(z; t^{-1})$, this is equal to
\begin{equation}\label{e:symm-LLT-alg-3}
\begin{aligned}
\langle E^{-}_{\rr ,\gamma }(z;t^{-1}),\, \Omega [X\, (z_{1}& +\cdots
+z_{l})]\, \overline{F_{v(\eta )}(z;\, t^{-1})} \rangle _{t}\\
& = \langle \Omega [X\, (z_{1}+ \cdots +z_{l})]\, E^{-}_{\rr ,\gamma
}(z;t^{-1}),\,
\overline{F_{v(\eta )}(z;\, t^{-1})} \rangle _{t}\\
& = \langle E_{v(\eta )}(z;t^{-1}) \rangle \, \Omega [X\, (z_{1}+
\cdots +z_{l})]\, E^{-}_{\rr ,\gamma }(z;t^{-1}).
\end{aligned}
\end{equation}
Now, for any $\Hcal (\Sfrak _{\rr })$ antisymmetric element
$f(z)\in \kk [\zz ^{\pm 1}]$, the coefficients of $f(z)$ in terms of
$E_{\mu}(z;t^{-1}) $ can be computed by first expanding in the
$E^{-}_{\rr ,\lambda }(z; t^{-1})$ basis and then expanding each
$E^{-}_{\rr ,\lambda }$ into $E_\mu$'s using \eqref{e:E-anti}.  This
gives
\begin{equation}\label{e:anti-symm-expansion}
\langle E_{v(\eta )}(z;t^{-1}) \rangle \,f(z) =
\begin{cases}
(-1)^{\ell (v)}\, \langle E^{-}_{\rr ,\eta }(z;t^{-1}) \rangle\, f(z)
&	\text{if $\eta $ is $\GL _{\rr }$ regular}\\
0&	\text{otherwise},
\end{cases}
\end{equation}
where $\eta \in X_{+}(\GL _{\rr })$.  Applying this with $f(z) =
\Omega [X\, (z_{1}+ \cdots +z_{l})]\, E^{-}_{\rr ,\gamma }(z;t^{-1})$
from the last line of \eqref{e:symm-LLT-alg-3}, we see that the
corollary is equivalent to Proposition~\ref{prop:symm-LLT-alg}.
\end{proof}

We now have the tools to prove the combinatorial formula in
Theorem~\ref{thm:combinatorial-G-nu-sigma}.  The basic strategy is to
show that both the algebraic and the combinatorial side satisfy a
recurrence governed by the decomposition of any flagged tableau into
the part with entries in $\Acal _{1}$ and the part with entries in
$\Acal _{2}\cup \cdots \cup \Acal _{l}$, as sketched in
Figure~\ref{fig:flagged-T-recursion}.

\begin{figure}
\begin{tikzpicture}[scale=.45]
\draw (2,0) -- (3,0) -- (3,1) -- (4,1) -- (4,2) -- (5,2) -- (5,3) --
(-1,3) -- (-1,2) -- (0,2) -- (0,1) -- (2,1) -- (2,0);
\draw [fill=gray!15] (4,1) -- (4,2) -- (5,2) -- (5,3) -- (2,3) --
(2,2) -- (3,2) -- (3,1) -- (4,1);
\node at (1,2) {$\lambold $};
\node at (3.5,2.3) {$\muboldhat $};
\draw (2,4) -- (5,4) -- (5,5) -- (4,5) -- (4,6) -- (5,6) -- (5,7) --
(0,7) -- (0,5) -- (2,5) -- (2,4);
\draw [fill=gray!15] (4,5) -- (4,6) -- (5,6) -- (5,7) --
(0,7) -- (0,6) -- (3,6) -- (3,5) -- (4,5);
\node at (5.5,4.5) {$\flagwave{1}$};
\node at (2.5,5) {$\lambold $};
\node at (3.5,6.3) {$\muboldhat $};
\draw (1,8) -- (5,8) -- (5,9) -- (3,9) -- (3,10) -- (-2,10) -- (-2,9)
--(1,9) -- (1,8);
\draw [fill=gray!15] (5,8) -- (5,9) -- (3,9) -- (3,10) -- (2,10) --
(2,8) -- (5,8);
\node at (0,9.5) {$\lambold $};
\node at (2.5,8.8) {$\muboldhat $};
\end{tikzpicture}
\caption{Recursive description of a flagged tableau $T\in \FSST
(\nubold ,\sigma ,\Acal )$.  Entries of $T$ in $\Acal _{1}$ form a
super tableau $T_{1}$ on a tuple of ordinary skew diagrams $\lambold
$, left-justified in $\nubold $ and containing all of row $m = \sigma
(l)$ with flag number $1$.  After deleting row $m$, entries of $T$ in
$\Acal _{2}\cup \cdots \cup \Acal _{l-1}$ form a flagged tableau
$T_{2}\in \FSST (\muboldhat ,\sigmahat ,\Acalhat ) $ on $\muboldhat
=\nubold /\lambold $, where $\Acalhat$ has distinguished subsets
$\Acal _{2}<\cdots <\Acal _{l}$, and $\sigmahat $ changes flag numbers
$2,\ldots,l$ to $1,\ldots,l-1$.
\label{fig:flagged-T-recursion}}
\end{figure}

\begin{proof}[Proof of Theorem~\ref{thm:combinatorial-G-nu-sigma}]
Given a strict composition $\rr $ of $l$, weights $\gamma \in
X_{++}(\GL _{\rr})$ and $\eta \in X_{+}(\GL _{\rr })$, and a
permutation $\sigma \in \Sfrak _{l}$ satisfying (iii), (iv) in
Definition~\ref{def:indexing-data}, define
\begin{equation}\label{e:G-triples-1}
G_{\rr ,\gamma ,\eta ,\sigma }(x;t)
 = \sum _{w\in \Sfrak _{\rr }} (-1)^{\ell (w)} \langle F_{w(\gamma
)}(z;\, t^{-1}) \rangle\, \overline{T_{\sigma }}\, \Omega [\sum
_{i\leq j} X^\Acal_{i} /z_{l+1-j}] \, z^{\eta _{-}},
\end{equation}
where $X^\Acal_{i} = \sum _{a\in \Acal _{i}^{+}}\, x_{a} - \sum _{b\in \Acal
_{i}^{-}}\, x_{b}$.  If we also have $\gamma _{i}\leq \eta _{i}$ for
all $i$, then $\rr ,\gamma,\eta ,\sigma $ are flagged LLT indexing
data, and the left hand side of \eqref{e:combinatorial-G-nu-sigma} is
given by $\Gcal _{\nubold ,\sigma }[X^\Acal_{1},\ldots, X^\Acal_{l};\, t] =
t^{h(\nubold )+e(\nubold ,\sigma )}\, G_{\rr ,\gamma ,\eta ,\sigma
}(x;t)$, where $\nubold $ is the diagram of $(\rr ,\gamma ,\eta )$.
By Lemma~\ref{lem:rr-incr-triples}, the right hand side of
\eqref{e:combinatorial-G-nu-sigma} is equal to $t^{h(\nubold
)+e(\nubold ,\sigma )}\, C_{\nubold ,\sigma }(x; t)$, where
\begin{equation}\label{e:C-triples}
C_{\nubold ,\sigma }(x; t) = \sum _{T\in \FSST (\nubold ,\sigma
,\Acal )} (-1)^{m(T)}\, t^{-h(T)-e(T)}\, x^{T}.
\end{equation}
Thus, the theorem follows from the stronger assertion
\begin{equation}\label{e:G=C}
G_{\rr ,\gamma ,\eta ,\sigma }(x;t) = \begin{cases} C_{\nubold ,\sigma
}(x; t), & \text{if $\gamma _{i}\leq \eta  _{i}$ for
all $i$, where $\nubold $ is the diagram of $(\rr ,\gamma ,\eta )$,}\\
0&  \text{otherwise},
\end{cases}
\end{equation}
which we will prove by induction on $l$.  In the basis case $l=0$, the
condition $\gamma _{i}\leq \eta _{i}$ holds vacuously, $\nubold $ is
the empty tuple, $\sigma $ is the trivial permutation of $\emptyset $,
the unique element of $\FSST (\nubold ,\sigma ,\Acal )$ is the empty
tableau, and \eqref{e:G=C} holds with $G_{\rr ,\gamma ,\eta ,\sigma
}(x;t) = 1 = C_{\nubold ,\sigma }(x; t)$.

We now derive a recurrence expressing $G_{\rr ,\gamma ,\eta ,\sigma
}(x; t)$ for $|\rr | = l > 0$ in terms of instances for $|\rr | =
l-1$.  We will then prove \eqref{e:G=C} by verifying that the right
hand side satisfies the same recurrence.

Factoring out $\Omega [X^\Acal_{1}\, \overline{z_{1}+\cdots +z_{l}}]$ from
$\Omega [\sum _{i\leq j} X^\Acal_{i} /z_{l+1-j}]$ and using the fact that
functions symmetric in $z$ commute with the Hecke algebra action
yields
\begin{equation}\label{e:G-triples-factored}
G_{\rr ,\gamma ,\eta ,\sigma }(x;t) =
\sum _{w\in \Sfrak _{\rr }} (-1)^{\ell (w)} \langle F_{w(\gamma
)}(z;\, t^{-1}) \rangle\, \Omega [X^\Acal_{1}\, \overline{z_{1}+\cdots
+z_{l}}]\, \overline{T_{\sigma }} \, \Omega [\sum _{2\leq i\leq j}
X^\Acal_{i} /z_{l+1-j}] \, z^{\eta _{-}}.
\end{equation}
By Corollary~\ref{cor:symm-LLT-alg}, this is equal to
\begin{equation}\label{e:G-triples-recur-1}
\sum _{\zeta } t^{-h(\lambold )}\, \Gcal _{\lambold }(X^\Acal_{1};\, t) \sum
_{w\in \Sfrak _{\rr }} (-1)^{\ell (w)} \langle F_{w(\zeta )}(z;
t^{-1}) \rangle\, \overline{T_{\sigma }}\, \Omega [\sum _{2\leq i\leq
j} X^\Acal_{i} /z_{l+1-j}] \, z^{\eta _{-}},
\end{equation}
where the outer sum is over $\zeta \in X_{++}(\GL _{\rr})$ satisfying
$\gamma _{i}\leq \zeta _{i}$ for all $i$, and $\lambold $ is the
diagram of $(\rr ,\gamma ,\zeta )$.

Let $m = \sigma (l)$, and define $\sigmahat $ by
\begin{equation}\label{e:sigmahat}
\sigma = s_{m}\cdots s_{l-1}\, \sigmahat.
\end{equation}
More explicitly, $\sigmahat $ fixes $l$, and for $i<l$ is given by
$\sigmahat (i) = \sigma (i)$ if $\sigma (i)<m$, or $\sigmahat (i) =
\sigma (i)-1$ if $\sigma (i)>m$.  The factorization in
\eqref{e:sigmahat} is reduced, i.e., $\ell (\sigma ) = l-m+\ell
(\sigmahat )$, since $\sigma$ has $l-m$ inversions involving $l$, plus
inversions coming from $\sigmahat $.  Therefore we also have
$\overline{T_{\sigma }} = \overline{T_{m}} \cdots \overline{T_{l-1}}
\,\overline{T_{\sigmahat }},$ allowing the inner sum in
\eqref{e:G-triples-recur-1} to be rewritten as
\begin{equation}\label{e:G-triples-recur-2}
\sum _{w\in \Sfrak _{\rr }} (-1)^{\ell (w)} \langle F_{w(\zeta )}(z;
t^{-1}) \rangle\, \overline{T_{m}} \cdots \overline{T_{l-1}}
\,\overline{T_{\sigmahat }}\, \Omega [\sum _{2\leq i\leq j} X^\Acal_{i}
/z_{l+1-j}] \, z^{\eta _{-}}.
\end{equation}

Since $\eta = \sigma (\eta _{-})$, we have $\eta _{m} = (\eta
_{-})_{l} = \max _{i} (\eta _{i})$ and $z^{\eta _{-}}\in (z_{1}\cdots
z_{l})^{\eta _{m}}\, \kk [z_{1}^{-1},\ldots,z_{l-1}^{-1}]$.  Since
$\Omega [\sum _{2\leq i\leq j} X^\Acal_{i} /z_{l+1-j}]$ is a power
series in the $x$ variables with coefficients in $\kk
[z_{1}^{-1},\ldots,z_{l-1}^{-1}]$, and $\overline{T_{\sigmahat }} \in
\Hcal (\Sfrak _{l-1})$ preserves $\kk
[z_{1}^{-1},\ldots,z_{l-1}^{-1}]$ and commutes with multiplication by
$z_{1}\cdots z_{l}$, it follows that $\overline{T_{\sigmahat }}\,
\Omega [\sum _{2\leq i\leq j} X^\Acal_{i} /z_{l+1-j}]\, z^{\eta _{-}}$
has coefficients in $(z_{1}\cdots z_{l})^{\eta _{m}}\, \kk
[z_{1}^{-1},\ldots,z_{l-1}^{-1}]$.  Using Lemma~\ref{lem:F-basics}(a,
b) and the fact that the $F_\lambda (z;t^{-1})$ indexed by $\lambda
\in -\NN^{l-1}$ form a basis of $\kk
[z_{1}^{-1},\ldots,z_{l-1}^{-1}]$, this implies that
$\overline{T_{\sigmahat }}\, \Omega [\sum _{2\leq i\leq j} X^\Acal_{i}
/z_{l+1-j}]\, z^{\eta _{-}}$ can be expanded in terms of $F_{\mu
}(z;t^{-1})$ for $\mu $ of the form $(\mu _{1},\ldots,\mu _{l-1},\eta
_{m})$ with $\mu _{i}\leq \eta _{m}$ for all $i$, and that the
coefficients in this expansion satisfy
\begin{multline}\label{e:F-coef}
\langle F_{\mu }(z;t^{-1}) \rangle\, \overline{T_{\sigmahat }}\,
\Omega [\sum _{2\leq i\leq j} X^\Acal_{i} /z_{l+1-j}]\, z^{\eta _{-}}\\
 = \langle F_{(\mu _{1},\ldots,\mu _{l-1})}(z;t^{-1}) \rangle\,
\overline{T_{\sigmahat }}\, \Omega [\sum _{2\leq i\leq j} X^\Acal_{i}
/z_{l+1-j}]\, z^{\etahat _{-}},
\end{multline}
where $\etahat =(\eta _{1},\ldots,\eta _{m-1},\eta _{m+1},\ldots, \eta
_{l})$, so $\eta _{-} = (\etahat _{-}; \eta _{m})$.  Note that the
right hand side of \eqref{e:F-coef} does not involve $z_{l}$, and that
if we regard $\sigmahat $ as an element of $\Sfrak _{l-1}$ by
identifying $\Sfrak _{l-1}$ with the subgroup $\Sfrak _{l-1}\cong
\Stab _{\Sfrak _{l}}(l) \subset \Sfrak _{l}$, the operator
$\overline{T_{\sigmahat }}\in \Hcal (\Sfrak _{l-1})$ acting on $\kk
[z_{1}^{\pm 1},\ldots,z_{l-1}^{\pm 1}]$ is just the restriction of the
original $\overline{T_{\sigmahat }}\in \Hcal (\Sfrak _{l})$.

Compatibility condition (iv) in Definition~\ref{def:indexing-data}
implies that $m$ is the first element in its $\rr $-block.
Hence, for $\zeta \in X_{++}(\GL _{\rr})$ and $w\in \Sfrak _{\rr }$,
we have $w(\zeta )_{m} = \eta _{m} = \max _{i} (w(\zeta )_{i})$ if and
only if $w(m) = m$, $\zeta _{m} = \eta _{m}$, and $\zeta _{i}\leq \eta
_{m}$ for all $i$.  These constraints are equivalent to $w(\zeta ) =
s_{m}\cdots s_{l-1}\, \mu $ for $\mu $ of the form $(\mu
_{1},\ldots,\mu _{l-1},\eta _{m})$ with $\mu _{i}\leq \eta _{m}$ for
all $i$, as in \eqref{e:F-coef}.  For this $\mu $, we have $(\mu
_{1},\ldots,\mu _{l-1}) = v(\zetahat )$, where $\zetahat =(\zeta
_{1},\ldots,\zeta _{m-1},\zeta _{m+1},\ldots,\zeta _{l})$ and $v =
s_{l-1}\cdots s_{m}\, w\, s_{m}\cdots s_{l-1}$, and we regard $v$ as
an element of $\Sfrak _{l-1}$ using $v(l) = l$.  It follows from
Lemma~\ref{lem:F-basics}(c) that for such $w$ and $\zeta $, we have
\begin{multline}\label{e:G-triples-recur-3}
\langle F_{w(\zeta )}(z; t^{-1}) \rangle\, \overline{T_{m}} \cdots
\overline{T_{l-1}} \,\overline{T_{\sigmahat }}\, \Omega [\sum _{2\leq
i\leq j} X^\Acal_{i} /z_{l+1-j}] \, z^{\eta _{-}} =\\
t^{-e_{m}(\zeta )}\, \langle F_{v(\zetahat )}(z; t^{-1}) \rangle
\,\overline{T_{\sigmahat }}\, \Omega [\sum _{2\leq i\leq j} X^\Acal_{i}
/z_{l+1-j}] \, z^{\etahat _{-}},
\end{multline}
where $e_{m}(\zeta )$ is the number of indices $i>m$ such that $\zeta
_{i} = \zeta _{m}$.  For this last, note that the exponent $e$ in the
instance of \eqref{e:Tis-on-F} used here is the number of entries
equal to $\mu _{l} = \zeta _{m}$ in $(\mu _{m},\ldots,\mu _{l-1}) =
(w(\zeta ) _{m+1},\ldots,w(\zeta ) _{l})$, and that $w\in \Sfrak _{\rr
}$ and $w(m) = m$ imply that $(w(\zeta ) _{m+1},\ldots,w(\zeta )
_{l})$ is a permutation of $(\zeta _{m+1},\ldots, \zeta _{l})$.  If
$w$ and $\zeta $ do not satisfy the above constraints, then the left
hand side of \eqref{e:G-triples-recur-3} is equal to zero.

Let $k$ be the index of the $\rr $-block with first element $m$, and
let $\rrhat $ be the strict composition of $l-1$ obtained from $\rr $
by changing $r_{k}$ to $r_{k}-1$, or deleting it if $r_{k}=1$.  Then
we have $\zetahat \in X_{++}(\GL _{\rrhat })$, $\etahat \in X_{+}(\GL
_{\rrhat })$, and one can check that $\rrhat ,\, \zetahat ,\, \etahat
\, ,\sigmahat $ satisfy conditions (iii), (iv) in
Definition~\ref{def:indexing-data}.  One can also check that if $w\in
\Sfrak _{\rr }$ and $w(m) = m$, then $v\in \Sfrak _{\rrhat }$, and
$\ell (v) = \ell (w)$.  We take the $l-1$ alphabets $\Acal_{i}$ in the
definition of $G_{\rrhat ,\zetahat ,\etahat ,\sigmahat }(x;t)$ to be
$\Acal_{2},\ldots,\Acal_{l}$, so that
\begin{equation}\label{e:G-triples-recur-4}
G_{\rrhat ,\zetahat ,\etahat ,\sigmahat }(x;t) = \sum _{v\in \Sfrak
_{\rrhat }} (-1)^{\ell (v)} \langle F_{v(\zetahat )}(z; t^{-1})
\rangle \,\overline{T_{\sigmahat }}\, \Omega [\sum _{2\leq i\leq j}
X^\Acal_{i} /z_{l+1-j}] \, z^{\etahat _{-}}.
\end{equation}
Rewriting the sum over $w$ in \eqref{e:G-triples-recur-1} as
\eqref{e:G-triples-recur-2}, then substituting
\eqref{e:G-triples-recur-3} into this, and comparing the resulting
expression for $G_{\rr ,\gamma ,\eta ,\sigma }(x;t)$ with
\eqref{e:G-triples-recur-4}, we obtain
\begin{equation}\label{e:G-triples-recur-5}
G_{\rr ,\gamma ,\eta ,\sigma }(x;t) = \sum _{\zeta } t^{-h(\lambold )
- e_{m}(\zeta )}\, \Gcal _{\lambold }(X^\Acal_{1};\, t) \, G_{\rrhat
,\zetahat ,\etahat ,\sigmahat }(x;t),
\end{equation}
where the sum is over $\zeta \in
X_{++}(\GL _{\rr })$ such that $\gamma _{i}\leq \zeta _{i}\leq \eta
_{m}$ for all $i$ and $\zeta _{m} = \eta _{m}$, and $\lambold $ is the
diagram of $(\rr ,\gamma ,\zeta )$, as before.  Using the induction
hypothesis \eqref{e:G=C} for $G_{\rrhat ,\zetahat ,\etahat ,\sigmahat
}(x;t)$,
this becomes
\begin{equation}\label{e:G-triples-recur-6}
G_{\rr ,\gamma ,\eta ,\sigma }(x;t) = \sum _{\zeta } t^{-h(\lambold )
- e_{m}(\zeta )}\, \Gcal _{\lambold }(X^\Acal_{1};\, t) \, C_{\muboldhat
,\sigmahat }(x;t),
\end{equation}
where the sum is now over $\zeta \in X_{++}(\GL _{\rr })$ satisfying
the stronger condition that $\gamma _{i}\leq \zeta _{i}\leq \eta _{i}$
for all $i$ and $\zeta _{m} = \eta _{m}$, $\muboldhat $ is the diagram
of $(\rrhat ,\zetahat ,\etahat )$, and we define $C_{\muboldhat
,\sigmahat }(x;t)$, using the alphabet $\Acalhat =\Acal $, with
distinguished subsets $\Acal _{2}<\cdots <\Acal _{l}$ in the
definition of $\FSST (\muboldhat ,\sigmahat ,\Acalhat )$.

In particular, if we do not have $\gamma _{i}\leq \eta _{i}$ for all
$i$, then the sum in \eqref{e:G-triples-recur-6} is empty, and $G_{\rr
,\gamma ,\eta ,\sigma }(x;t) = 0$, as claimed.
If $\gamma _{i}\leq \eta _{i}$, we are to show that
\begin{equation}\label{e:C-recurrence}
C_{\nubold ,\sigma }(x;t) = \sum _{\zeta } t^{-h(\lambold ) -
e_{m}(\zeta )}\, \Gcal _{\lambold }(X^\Acal_{1};\, t) \, C_{\muboldhat
,\sigmahat }(x;t)\, ,
\end{equation}
where $\nubold $ is the diagram of $(\rr ,\gamma ,\eta )$, and the
right hand side is the same as in \eqref{e:G-triples-recur-6}.  In
particular, the conditions on $\zeta $ are that the diagram $\lambold
$ of $(\rr ,\gamma ,\zeta )$ is a tuple of ordinary skew diagrams
contained in $\nubold $, left-justified (row $i$ has left endpoint
$\gamma _{i}$ in both $\lambold $ and $\nubold $), and containing all
of row $m$.  The diagram $\mubold = \nubold \setminus \lambold $ of
$(\rr ,\zeta ,\eta )$ is a tuple of ragged-right skew diagrams in
which row $m$ is empty and is the first row in its component $\mu
^{(k)}$.  The diagram $\muboldhat $ of $(\rrhat ,\zetahat ,\etahat )$
is obtained by deleting the empty $m$-th row and shifting the
surviving component $\muhat ^{(k)}$, if any, one unit to the left, to
compensate for the fact that deleting the first row shifts it one unit
down, so that corresponding boxes have the same content in $\muboldhat
$ and in $\mubold $.  If row $m$ is the only row of $\mu ^{(k)}$, then
we just delete this component.

Consider a flagged tableau $T\in \FSST (\nubold ,\sigma ,\Acal )$, and
set $\lambold =T^{-1}(\Acal _{1})$, as in
Figure~\ref{fig:flagged-T-recursion}.  By
Lemma~\ref{lem:i-admissible}, the $1$-admissible region $\nubold _{1}
$ is the largest tuple of ordinary skew diagrams contained in $\nubold
$.  Since entries in $\Acal _{1}$ are smaller than other entries of
$T$, $\lambold $ is a tuple of ordinary skew diagrams left-justified
in $\nubold _{1}$, hence also in $\nubold $.  Since the flag number in
row $m$ is $1$, boxes in this row are not $i$-admissible for $i>1$,
hence $\lambold $ contains all of row $m$.  Letting $\mubold =\nubold
\setminus \lambold $, we have $T|_{\lambold } \in \SSYT _{\pm
}(\lambold ,\Acal _{1})$ and $T|_{\mubold} \in \FSST (\mubold , \sigma
,\Acal )$, with $T|_{\mubold }$ having no entries in $\Acal _{1}$.

Constructing $\muboldhat $ by deleting the empty $m$-th row of
$\mubold $, as above, the flag numbers for $\sigmahat $ in the rows of
$\muboldhat $ are one less than the flag numbers for $\sigma $ in the
corresponding rows of $\mubold $.  Let $T|_{\muboldhat }$ be the
filling of $\muboldhat $ defined by $(T|_{\muboldhat })(\uhat ) =
T(u)$ for $\uhat \in \muboldhat $ corresponding to $u\in \mubold$.
Then $T|_{\muboldhat } \in \FSST (\muboldhat , \sigmahat ,\Acalhat )$,
where the $l-1$ distinguished subsets of $\Acalhat $ are $\Acal
_{2}<\cdots <\Acal _{l}$, as in the definition of $C_{\muboldhat
,\sigmahat }(x;t)$.

Conversely, if $\lambold $, $\mubold $, and $\muboldhat $ are the
diagrams of $(\rr ,\gamma ,\zeta )$, $(\rr ,\zeta ,\eta )$, and
$(\rrhat ,\zetahat, \etahat )$, respectively, as above, every
$T_{1}\in \SSYT _{\pm }(\lambold ,\Acal _{1})$ and $T_{2} \in \FSST
(\muboldhat , \sigmahat ,\Acalhat )$ fit together to form a unique
flagged tableau $T\in \FSST (\nubold ,\sigma ,\Acal )$ with
$T|_{\lambold } = T_{1}$ and $T|_{\muboldhat } = T_{2}$.  By
Lemma~\ref{lem:inv-vs-triples} and the definitions of $\Gcal
_{\lambold }(X;t)$ and $C_{\muboldhat ,\sigmahat }(x;t)$, we have
\begin{gather}\label{e:LLT-triples-T1}
t^{-h(\lambold )}\, \Gcal _{\lambold }(X^\Acal_{1}; t) = \sum _{T_{1}\in
\SSYT _{\pm }(\lambold ,\Acal _{1})} (-1)^{m(T_{1})}\, t^{-h(T_{1})}\,
x^{T_{1}}\\
\label{e:C-T2} C_{\muboldhat ,\sigmahat }(x;t) = \sum _{T_{2} \in
\FSST (\muboldhat , \sigmahat ,\Acalhat )} (-1)^{m(T_{2})}\,
t^{-h(T_{2})-e(T_{2})}\, x^{T_{2}}.
\end{gather}
It is clear that $m(T) = m(T_{1})+m(T_{2})$ and $x^{T} = x^{T_{1}}\,
x^{T_{2}}$, and Corollary~\ref{cor:incr-triples} gives
$h(T)=h(T_{1})+h(T_{2})$.  Triples $(x,y,z)$ of $\muboldhat $ in which
$y$ and $z$ are flag boxes correspond bijectively to triples
$(x',y',z')$ of $\nubold $ in which $y'$ and $z'$ are flag boxes not
in row $m$; specifically, the rows of $\nubold $ with flag boxes $y'$
and $z'$ correspond to the rows of $\muboldhat $ with flag boxes $y$
and $z$.  Among these, the triples $(x,y,z)$ counted by $e(T_{2})$
correspond to triples $(x',y',z')$ counted by $e(T)$ which do not have
boxes in row $m$.  The remaining triples $(x',y',z')$ counted by
$e(T)$ are those in which $y'$ is the flag box with flag number $1$ in
row $m$, $z'$ is a flag box in a higher component with content $c(z')
= c(y')$, and either $x'\not \in \nubold $, or $T(x')\in \Acal _{1}$,
i.e., row $i$ of $\mubold $ is an empty row, where $z'$ is the flag
box in row $i$.  These conditions are equivalent to $i>m$ and $\zeta
_{i} = \zeta _{m}$.  Hence, $e(T) = e(T_{2})+e_{m}(\zeta )$.

It follows that the term in the product $t^{-h(\lambold )-e_{m}(\zeta
)}\, \Gcal _{\lambold }(X^\Acal_{1}; t)\, C_{\muboldhat ,\sigmahat
}(x;t)$ coming from the summands for $T_{1}$ in
\eqref{e:LLT-triples-T1} and $T_{2}$ in \eqref{e:C-T2} is $(-1)^{m(T)}
t^{-h(T)-e(T)} x^{T}$.  Summing over all choices of $\zeta $, $T_{1}$,
and $T_{2}$ gives \eqref{e:C-recurrence}.
\end{proof}

\subsection{Atom positivity conjectures}
\label{ss:atom-positivity}

In Remark~\ref{rem:monomial-positive}, we saw that flagged LLT
polynomials are monomial positive; by
Corollary~\ref{cor:Weyl-on-G-nu-sigma}(a), this property generalizes
the monomial positivity of symmetric LLT polynomials.  We now
conjecture stronger {\em atom positivity} properties for flagged LLT
polynomials, generalizing Schur positivity of symmetric LLT
polynomials.

Given a Laurent polynomial $f(x) = f(x_{1},\ldots,x_{n})$ with
coefficients in $\ZZ [t]$, we denote the coefficients in the expansion
of $f(x)$ in terms of Demazure atoms by
\begin{equation}\label{e:atom-expansion}
\langle \Acal _{\lambda }(x) \rangle\, f(x) = a_{\lambda }(t),\quad
\text{where}\quad f(x)= \sum_{\lambda \in \ZZ^{n}} a_{\lambda}(t)
\Acal_\lambda(x).
\end{equation}
We say that $f(x)$ is {\em atom positive} if $\langle \Acal _{\lambda
}(x) \rangle\, f(x)\in \NN [t]$ for all $\lambda $.

We conjecture that $\Gcal _{\nubold ,\sigma }[x_{1},\ldots,x_{l};t]$
is atom positive in certain cases, and also that the coefficients of
certain atoms $\Acal _{\lambda }$ are positive in all cases.

\begin{conj}\label{conj:atom-pos-standard}
If $\nubold $ is a tuple of ordinary (non-ragged-right) skew diagrams
and $\sigma $ is the standard compatible permutation, then
$\Gcal_{\nubold, \sigma}[x_1, \dots, x_l;t]$ is atom positive.
\end{conj}

\begin{conj}\label{conj:atom-pos-unimodal}
For any flagged LLT indexing data $\nubold ,\sigma $, we have
\begin{equation}\label{e:atom-pos-unimodal}
\langle \Acal _{\lambda }(x) \rangle\, \Gcal_{\nubold, \sigma}[x_1,
\dots, x_l;t] \in \NN [t]
\end{equation}
if the index $\lambda $ is either unimodal (weakly increasing, then
weakly decreasing) or anti-unimodal (weakly decreasing, then weakly
increasing).
\end{conj}

Proposition~\ref{prop:flagged-atom} implies that
Conjecture~\ref{conj:atom-pos-standard} extends to all flag bound
specializations, and Lemma~\ref{lem:empty-rows-trick} implies that
Conjecture~\ref{conj:atom-pos-unimodal} extends to specializations
with strictly increasing flag bounds.

\begin{cor}\label{cor:atom-pos}
(a) If Conjecture~\ref{conj:atom-pos-standard} holds, then for $\nubold $
non-ragged-right and $\sigma $ standard, the flag bound
specialization $\Gcal _{\nubold ,\sigma }[X^{b}_{1},\ldots,X^{b}_{l};
t]$ is atom positive, for all choices of the flag bounds $b_{1}\leq
\cdots \leq b_{l}$.

(b) If Conjecture~\ref{conj:atom-pos-unimodal} holds, then for any
$\nubold ,\sigma $ and strictly increasing flag bounds $0<b_{1}<\cdots <b_{l}$,
we have
\begin{equation}\label{e:atom-pos-unimodal-cor}
\langle \Acal _{\lambda }(x) \rangle\, \Gcal _{\nubold ,\sigma
}[X^{b}_{1},\ldots,X^{b}_{l}; t] \in \NN [t]
\end{equation}
if the index $\lambda $ is unimodal or anti-unimodal.
\end{cor}

\begin{example}\label{ex:atom-pos-standard}
(i) For the tuple of skew diagrams $\nubold = ((1),(2,2),(2)/(1))$ and
standard compatible permutation $\sigma = (1,3,2,4)$, we have the
positive atom expansion
\begin{multline}\label{ex:atom-positive-G}
\Gcal_{\nubold,\sigma}[x_1,x_2,x_3,x_4;t] = t\, \mathcal{A}_{4200} +
t\, \mathcal{A}_{4020} + t\, \mathcal{A}_{3300} + 2\, t^2
\mathcal{A}_{3210}\\
 + t^2 \mathcal{A}_{3201} + 2\, t^2 \mathcal{A}_{3120} + t\,
\mathcal{A}_{3030} + t^2 \mathcal{A}_{3021} + t^2
\mathcal{A}_{2310} + t^3 \mathcal{A}_{2220} + t^3 \mathcal{A}_{2211}\\
 + t^2 \mathcal{A}_{2130} + t^3 \mathcal{A}_{2121} + t^2
\mathcal{A}_{1320} + t^2 \mathcal{A}_{1230} + t^3 \mathcal{A}_{1221}.
\end{multline}

(ii) Atom positivity can fail if $\sigma $ is non-standard, even if
$\nubold $ is a tuple of ordinary skew diagrams.  For instance, for
$\nubold = ((1)/(1),(1)/(1),(1),(1))$ and $\sigma = (3,4,1,2)$, we
have
\begin{multline}\label{e:atoms-nonstandard}
\Gcal_{\nubold,\sigma}[x_1,x_2,x_3,x_4;t] = \Acal_{2000}+t \,
\Acal_{1100}+(t^2 + t^3 -
t^4)\Acal_{1010} \\
+t^3\Acal_{1001}+t^2\Acal_{0200}+t^3\Acal_{0110}+t^4\Acal_{0101}+
t^4\Acal_{0020}+t^5 \Acal_{0011}\, .
\end{multline}
The only terms $\Acal _{\lambda }$ in this example for which $\lambda
$ is neither unimodal nor anti-unimodal are $\Acal _{1010}$ and $\Acal
_{0101}$.  All the others have positive coefficient, illustrating
Conjecture~\ref{conj:atom-pos-unimodal}.
Substituting $x\mapsto
(x_{1},0,x_{2},x_{3})$ in \eqref{e:atoms-nonstandard} gives
\begin{equation}\label{e:atoms-flag-bounds}
\Gcal_{\nubold,\sigma}[X^{b}_1,\ldots, X^{b}_{4};t] = \Acal_{200}+(t^2 + t^3 -
t^4)\Acal_{110} 
+t^3\Acal_{101}+
t^4\Acal_{020}+t^5 \Acal_{011}
\end{equation}
for flag bounds $b = (1,1,2,3)$.  This shows that the extension of
Conjecture~\ref{conj:atom-pos-unimodal} in
Corollary~\ref{cor:atom-pos}(b) does not hold for arbitrary choices of
the flag bounds.

(iii) If $\nubold $ is ragged-right, atom positivity can fail even if
$\sigma $ is the compatible permutation `closest' to standard,
corresponding to flag numbers decreasing in the ordering of flag boxes
by (content, component number), but increasing by rows within each
component, as required.  For instance, for $\nubold =((1,1)/(0,0),\,
(0,1,2)/(0,0,0))$, $\sigma =(2,5,4,3,1)$, with end contents
$(1,0,0,0,0)$ and flag numbers $(1,5,2,3,4)$, we have
\begin{equation}
\Gcal_{\nubold,\sigma}[x_1,x_2,x_3,x_4,x_5;t] = t^4 \Acal _{10121}+t^3
\Acal _{10130}-t^3 \Acal _{20120}.
\end{equation}

(iv) For $\nubold =((1),(1),(1))$, with standard
compatible permutation $\sigma =(1,2,3)$, we have
\begin{equation}
\Gcal _{\nubold, \sigma }[x_{1},x_{2},x_{3}; t] = \Dcal _{300} + t^3\,
\Dcal _{111} + t^2\, \Dcal _{120} + t^2\, \Dcal _{201} +
\left(t-t^2\right) \Dcal _{210},
\end{equation}
showing that the conjectures cannot be strengthened to positive
expansion in terms of full Demazure characters.

Conjecture~\ref{conj:atom-pos-standard} also cannot be strengthened to
{\em multi-headed Demazure positivity}, which means that the atom
coefficients are nonnegative and monotone decreasing in Bruhat order
on each $\Sfrak _{l}$ orbit of weights.  A counterexample to this is
$\nubold =((1),(1),(1),(1),(1),(1))$ with standard compatible
permutation $\sigma =(1,2,3,4,5,6)$, for which we have
\begin{equation}
\langle \Acal _{212100}(x) \rangle \, \Gcal _{\nu ,\sigma }[x; t] -
\langle \, \Acal _{212010}(x) \rangle \, \Gcal _{\nu ,\sigma }[x; t] =
t^{8} + t^{9} - t^{10}\not \in \NN [t].
\end{equation}
\end{example}

By Corollary \ref{cor:Weyl-on-G-nu-sigma}(a), the symmetric LLT
polynomial $\Gcal _{\nubold }(X;t)$ evaluated in $n$ variables is
equal to $\Gcal _{\nubold ,\sigma }[X^{b}_{1},\ldots,X^{b}_{l}; t]$
for flag bounds $b_{1}=\cdots =b_{l} = n$.  Since a symmetric
polynomial is Schur positive if and only if it is atom positive
(Remark \ref{rem:for atom pos}(i)),
Conjecture~\ref{conj:atom-pos-standard} and
Corollary~\ref{cor:atom-pos}(a) generalize the known Schur
positivity of symmetric LLT polynomials.
Conjecture~\ref{conj:atom-pos-unimodal} generalizes Schur positivity
of symmetric LLT polynomials in another way, as seen in
Proposition~\ref{prop:atom-pos-dominant}, below.

\subsubsection*{A reformulation of the conjectures}

Given flagged LLT indexing data $\nubold ,\sigma $, let $(\rr ,\gamma
,\eta )$ be the weight data with diagram $\nubold $, as in
Definition~\ref{def:G-nu-sigma}.  Specializing $X_{i} = x_{i}$ in
\eqref{e:G-nu-sigma} and using \eqref{e:atom-demazure-Cauchy}
yields an identity
\begin{multline}\label{e:atom-coefs-via-D}
t^{-h(\nubold )-e(\nubold ,\sigma )}\, \langle \Acal _{\lambda }(x)
\rangle\, \Gcal _{\nubold ,\sigma } [x_{1},\ldots,x_{l};t] \\
= \sum _{w\in \Sfrak _{\rr }} (-1)^{\ell (w)} \langle F_{w(\gamma
)}(x;\, t^{-1}) \rangle\, \overline{T_{\sigma }}\, \, x^{\eta _{-}}\,
w_{0} \Dcal _{-\lambda }(x)\, .
\end{multline}

In particular, in the case that $\nubold $ is a tuple of single-row
shapes with equal end contents, and $\sigma $ is the standard
compatible permutation, we have $\rr =(1^{l})$, $\eta =(d^{l})$, and
$\sigma =\operatorname{id}$.  Hence, by \eqref{e:atom-coefs-via-D},
Conjecture~\ref{conj:atom-pos-standard} reduces in this case to
$\langle F_{\kappa }(x;\, t^{-1}) \rangle\, w_{0} \Dcal _{\mu }(x)\in
\NN [t]$ for all $\mu $, that is, to the conjecture that opposite
Demazure characters are $F_{\kappa }$ positive.  In general,
\eqref{e:atom-coefs-via-D} expresses the atom coefficients in
Conjectures~\ref{conj:atom-pos-standard} and
\ref{conj:atom-pos-unimodal} as signed sums of the coefficients
\begin{equation}\label{e:D-to-F-with-twist}
\langle F_{\kappa }(x;\, t^{-1}) \rangle\, \overline{T_{\sigma }}\, \,
x^{\eta _{-}}\, w_{0} \Dcal _{\mu }(x)
\end{equation}
in the expansion of an opposite Demazure character $w_{0}\Dcal _{\mu
}$, acted upon with an antidominant weight monomial $x^{\eta _{-}}$
and a Hecke algebra operator $\overline{T_{\sigma }}$, in terms of
nonsymmetric Hall-Littlewood polynomials $F_{\kappa }$,
and yields the following reformulation of the conjectures.

\begin{prop}\label{prop:atom-pos-via-D}
(a) Conjecture~\ref{conj:atom-pos-standard} is equivalent to the
positivity condition
\begin{equation}\label{e:atom-pos-via-D}
\sum _{w\in \Sfrak _{\rr }} (-1)^{\ell (w)} \langle F_{w(\gamma
)}(x;\, t^{-1}) \rangle\, \overline{T_{\sigma }}\, \, x^{\eta _{-}}\,
w_{0} \Dcal _{\mu }(x)\in \NN [t]
\end{equation}
holding for all $\mu $, and all $\gamma, \eta_-, \sigma$ satisfying (i) $\gamma \in X_{++}(\GL _{\rr})$ is
$\GL _{\rr}$-dominant and regular, (ii) $\eta _{-}\in -X_{+}(\GL _{l})$ is
antidominant, (iii) $\sigma $ is maximal in $\Sfrak _{\rr }\, \sigma $,
and (iv) $\sigma $ is minimal in $\sigma \Stab (\eta _{-})$.

(b) Conjecture~\ref{conj:atom-pos-unimodal} is equivalent to
\eqref{e:atom-pos-via-D} holding for all $\mu $ unimodal or
anti-unimodal, and all $\gamma, \eta_-, \sigma$ satisfying (i)--(iii).
\end{prop}

\begin{proof}
By Definition~\ref{def:indexing-data}, if we set $\eta = \sigma (\eta
_{-})$, then $(\rr ,\gamma ,\eta ,\sigma )$ are flagged LLT indexing
data if and only $\rr$, $\eta _{-}$, $\sigma $ satisfy assumptions (i)--(iii)
and the extra condition $\gamma _{i}\leq \eta _{i}$ for all $i$.
However, the second clause of \eqref{e:G=C}, established in the proof
of Theorem~\ref{thm:combinatorial-G-nu-sigma}, shows that the
expression in \eqref{e:atom-pos-via-D} vanishes if we do not have
$\gamma _{i}\leq \eta _{i}$ for all $i$.  Hence
\eqref{e:atom-pos-via-D} holds for $\rr $, $\eta _{-}$, $\sigma $
satisfying (i)--(iii) if and only if it holds when $(\rr ,\gamma ,\eta
,\sigma )$ are flagged LLT indexing data.  The additional assumption
(iv) is the definition of $\sigma $ being the standard compatible
permutation.

This given, the equivalence of the conjectures with the specified
instances of \eqref{e:atom-pos-via-D} is immediate from
\eqref{e:atom-coefs-via-D}.
\end{proof}

Conjecture~\ref{conj:atom-pos-standard} can be reduced further, using
the following property of Demazure characters, which follows from a
theorem of Polo \cite{Polo} (see also \cite{Mathieu},
\cite[\S2]{MathieuPositivity}, \cite[Ch. 6]{vanderKallenBModule}).

\begin{lemma}\label{lem:dom-times-Dem}
If $\eta _{+}$ is a dominant weight, then the coefficients in the
expansion $x^{\eta _{+}}\, \Dcal_{\mu }(x) = \sum _{\lambda
}c_{\lambda } \Dcal _{\lambda }(x)$ are nonnegative.
\end{lemma}

Let $\Lambda _{i} = \varepsilon _{1}+\cdots +\varepsilon _{i}$
($i=1,\ldots,l-1$) be fundamental weights for $\GL _{l}$.  Given
$\sigma \in \Sfrak _{l}$, set $\eta _{0}(\sigma ) = - \sum _{i\colon
\sigma _{i}>\sigma _{i+1}} \Lambda _{i}\in -X_{+}(\GL _{l})$.

\begin{prop}\label{prop:atom-pos-via-D-reduced}
In Proposition~\ref{prop:atom-pos-via-D}(a), it suffices to take $\eta
_{-} = \eta _{0}(\sigma )$.
\end{prop}

\begin{proof}
For any $\eta _{-}\in -X_{+}(\GL _{l})$, we have $\sigma $ minimal in
$\sigma \Stab (\eta _{-})$ if and only if $(\eta _{-})_{i}<(\eta
_{-})_{i+1}$ for all $i$ such that $\sigma _{i}>\sigma _{i+1}$.  Hence
the set of such weights $\eta _{-}$ is equal to $\eta _{0}(\sigma ) -
X_{+}(\GL _{l})$.  For any such $\eta _{-}$, writing $\eta _{-} = \eta
_{0}(\sigma ) - \zeta $ with $\zeta \in X_{+}(\GL _{l})$, we have
$x^{\eta _{-}}\, w_{0}\Dcal _{\mu }(x) = x^{\eta _{0}(\sigma )}
w_{0}(x^{-w_{0}(\zeta )} \Dcal _{\mu })$.  By
Lemma~\ref{lem:dom-times-Dem}, $x^{-w_{0}(\zeta )} \Dcal _{\mu }$ is a
nonnegative combination of Demazure characters $\Dcal _{\lambda }$.
Hence the positivity in \eqref{e:atom-pos-via-D} holds for a given
$\eta _{-}$ and $\Dcal _{\mu }$ if it holds with $\eta _{0}(\sigma )$
and $\Dcal _{\lambda }$ in place of $\eta _{-}$ and $\Dcal _{\mu }$,
for each $\Dcal _{\lambda }$ that occurs in the expansion of
$x^{-w_{0}(\zeta )} \Dcal _{\mu }$.
\end{proof}

Note that we cannot use this argument to reduce part (b) of
Proposition~\ref{prop:atom-pos-via-D} to $\eta _{-}$ of a special
form, because the space spanned by Demazure characters $\Dcal _{\mu }$
with unimodal or anti-unimodal index is not closed under
multiplication by dominant weight monomials.
Conjecture~\ref{conj:atom-pos-unimodal} does admit the following
simplification.

\begin{prop}\label{prop:unimodal-vs-antiunimodal}
The validity of Conjecture~\ref{conj:atom-pos-unimodal} in the
unimodal case is equivalent to its validity in the anti-unimodal case.
\end{prop}

\begin{proof}
The change of variables $x_{i}\mapsto x_{w_{0}(i)}^{-1}$ sends $\Dcal
_{\mu }(x)$ to $\Dcal _{-w_{0}(\mu )}(x)$ and $F_{\kappa }(x;t^{-1})$
to $F_{-w_{0}(\kappa )}(x;t^{-1})$, commutes with the action of
$w_{0}$ on the variables, preserves antidominant weight monomials, and
changes $\overline{T_{\sigma }}$ to $\overline{T_{w_{0}\sigma
w_{0}}}$.  Changing $\Sfrak _{\rr }$ to $w_{0}\, \Sfrak _{\rr}\, w_{0}
= \Sfrak _{\rr ^{\op }}$, we see that the change of variables
$x_{i}\mapsto x_{w_{0}(i)}^{-1}$ preserves the assumptions on
\eqref{e:atom-pos-via-D} in
Proposition~\ref{prop:atom-pos-via-D-reduced}(b), while exchanging
unimodal weights $\mu $ with anti-unimodal weights $-w_{0}(\mu )$.
\end{proof}

\begin{remark}\label{rem:other-special-mu}
Computations suggest that unimodal and anti-unimodal weights $\mu $
are not the only weights for which the positivity in
Proposition~\ref{prop:atom-pos-via-D-reduced}(b) holds---or,
equivalently, for which the atom $\Acal _{\lambda } = \Acal _{-\mu }$
has $\langle \Acal_{\lambda }(x) \rangle\, \Gcal _{\nubold ,\sigma
}[x_{1},\ldots,x_{l}; t]\in \NN [t]$ for every flagged LLT polynomial
$\Gcal _{\nubold ,\sigma }$.  However, we do not have a specific
conjecture describing a larger class of weights with this property.
\end{remark}

\subsubsection*{Special cases}

To help motivate Conjecture~\ref{conj:atom-pos-unimodal}, we prove it
in the two extreme cases when the weight $\lambda $ is dominant or
antidominant (thus, both unimodal and anti-unimodal).

\begin{prop}\label{prop:atom-pos-dominant}
If $\lambda \in X_{+}(\GL _{l})$ is dominant, then $\langle
\Acal_{\lambda }(x) \rangle\, \Gcal _{\nubold ,\sigma
}[x_{1},\ldots,x_{l}; t]\in \NN [t]$ for all flagged LLT data $\nubold
,\sigma $ of length $l$.
\end{prop}

\begin{proof}
As we saw in Remark \ref{rem:for atom pos}(ii), for any Laurent
polynomial $f(x)$, the coefficient $\langle \Acal _{\lambda }(x)
\rangle\, f(x)$ for $\lambda $ dominant is equal to the coefficient
$\langle \chi _{\lambda }(x) \rangle\, \Weyl f(x)$ of the irreducible
$\GL _{l}$ character $\chi _{\lambda }$ in the Weyl symmetrization of
$f(x)$.  By Corollary~\ref{cor:Weyl-on-G-nu-sigma}(b, c), the dominant
weight case of Conjecture~\ref{conj:atom-pos-unimodal} is therefore
equivalent to the known Schur positivity of symmetric LLT polynomials.
\end{proof}

\begin{prop}\label{prop:atom-pos-antidominant}
If $\lambda \in -X_{+}(\GL _{l})$ is antidominant, then $\langle
\Acal_{\lambda }(x) \rangle\, \Gcal _{\nubold ,\sigma
}[x_{1},\ldots,x_{l}; t]\in \NN [t]$ for all flagged LLT data $\nubold
,\sigma $ of length $l$.
\end{prop}

\begin{proof}
The antidominant case of Conjecture~\ref{conj:atom-pos-unimodal} is
equivalent to the positivity in
Proposition~\ref{prop:atom-pos-via-D}(b) holding for $\mu = -\lambda $
dominant.  But then $w_{0}\Dcal _{\mu } = x^{-w_{0}(\lambda )}$ is an
antidominant weight monomial, so \eqref{e:atom-pos-via-D} in this case
is the same as in the case with $\eta _{-}$ replaced by $\eta
_{-}-w_{0}(\lambda )$, and $\mu =(0^{l})$.

This reduces the result to the case $\lambda =(0^{l})$, which holds by
Proposition~\ref{prop:atom-pos-dominant}.
\end{proof}

\subsubsection*{Computational evidence for the conjectures}

By Proposition~\ref{prop:atom-pos-via-D-reduced},
Conjecture~\ref{conj:atom-pos-standard} follows from positivity in
\eqref{e:atom-pos-via-D} for all $\mu $ when $\eta _{-} = \eta
_{0}(\sigma )$, $\sigma $ is maximal in $\Sfrak _{\rr }\, \sigma $,
and $\gamma \in X_{++}(\GL _{\rr})$.  This can be checked for a given
$\mu $ by means of a finite computation: namely, for each $\sigma \in
\Sfrak _{l}$, expand $\overline{T_{\sigma }}\, \, x^{\eta _{0}(\sigma
)}\, w_{0} \Dcal _{\mu }(x)$ in terms of the $F_{\kappa }(x;t^{-1})$,
and then evaluate the signed coefficient sums for all Young subgroups
$\Sfrak _{\rr} \subseteq \Sfrak _{l}$ such that $\sigma $ is maximal
in $\Sfrak _{\rr }\, \sigma $.

Using this method, we verified the reduction of
Conjecture~\ref{conj:atom-pos-standard} in
Proposition~\ref{prop:atom-pos-via-D-reduced} by computer in the
following cases: (a) for all $\mu \in \NN ^{l}$ satisfying $|\mu |\leq
8$ for $l=4,5$; and (b) for randomly chosen $\mu \in [0,m]^{l}$ in
several hundred examples with $m=8$ and $l = 5,6$, several dozen with
$m=6$ and $l=7,8$, and several with $m=8$ and $l=8$.

Conjecture~\ref{conj:atom-pos-unimodal} takes more work to check,
since we must allow both $\mu $ and $\eta _{-}$ to vary.  By
Proposition~\ref{prop:unimodal-vs-antiunimodal}, the anti-unimodal
case suffices.  We verified the positivity in \eqref{e:atom-pos-via-D}
for anti-unimodal $\mu $ and varying $\eta _{-}$ in hundreds of cases
for $l=6,7$, and a few with $l=8$.

\subsection{Signed flagged LLT polynomials}
\label{ss:signed}

For signed alphabets of a specific form, the specializations of
flagged LLT polynomials in Theorem~\ref{thm:combinatorial-G-nu-sigma}
interact well with the nonsymmetric plethysm $\Pi _{t,x}$ and the
Hecke algebra action, and will be needed for the study of nonsymmetric
Macdonald polynomials in \S \ref{s:ns-Mac-pols}.  We now define these
specializations and develop their properties.

\begin{defn}\label{def:signed-G}
Given flagged LLT indexing data $\nubold , \sigma $ of length $l$, the
{\em signed flagged LLT polynomial with flag bounds
$b=(b_1\leq \cdots
\leq b_{l})$} is the specialization
\begin{equation}\label{e:signed-G}
\Gcal _{\nubold ,\sigma }[X_{1}^b - t\, \Xpb{1},\, \ldots,
X_{l}^b - t\, \Xpb{l};\, t^{-1}],
\end{equation}
where $X^b_i = x_{b_{i-1}+1} + \cdots + x_{b_i}$, as in
Definition~\ref{def:flag-bounds}, and we set $X' =
(0,x_{1},x_{2},\ldots)$, so ${X'_{i\,}}^b = x_{b_{i-1}}+ \cdots +
x_{b_i-1}$.  In the special case $b_{i} = i$, we use the notation
\begin{equation}\label{e:G-minus}
\Gcal _{\nubold ,\sigma }^{-}(x_{1},\ldots,x_{l};\, t^{-1}) \defeq
\Gcal _{\nubold ,\sigma }[x_{1},\, x_{2} -t\, x_{1}, \ldots, x_{l} -
t\, x_{l-1};\, t^{-1}].
\end{equation}
\end{defn}

Note that we define signed flagged LLT polynomials to be
specializations of flagged LLT polynomials with parameter $t^{-1}$ in
place of $t$.  This turns out to be the correct definition to use in
\S \ref{s:ns-Mac-pols}, and is also needed for results below.

Our next result is a signed analog of
Corollary~\ref{cor:G-nu-sigma-flag-bounds}.

Given flag bounds $b_{1}\leq \cdots \leq b_{l}\leq n$, we write $\Acal
_{\pm}^{b}$ for the signed alphabet $\Acal = \{1<\overline{1}< \cdots
< n < \overline{n}\} $
with $\Acal^+ = \{1,\ldots, n\}$, $\Acal^- = \{\overline{1},\ldots,
\overline{n}\}$, and the distinguished subsets
\begin{equation}\label{eq:A plus minus alphabets}
\Acal_1 = \{1, \overline 1, 2,\dots, b_1\}, \,\Acal_2 =
\{\overline{b_1}, b_1+1, \dots, b_2\},\, \ldots,\, \Acal_l =
\{\overline{b_{l-1}},b_{l-1}+1, \dots, b_l\}\, .
\end{equation}
The set of flagged tableaux on these alphabets is then denoted $\FSST
(\nubold ,\sigma,\Acal_{\pm}^b)$.
Explicitly, by Lemma~\ref{lem:T-prime}, a tableau $T \in \FSST
(\nubold ,\sigma,\Acal_{\pm }^b)$ is a filling $T\colon \nubold
\rightarrow \{1, \overline{1}, \ldots, n, \overline{n}\}$ such that
the extended filling $T'$ with fixed entry $b_{i}$ in the flag box
with flag number $i$ is a ragged-right super tableau.

\begin{prop}\label{prop:signed-G-tableaux}
Given $b_{1}\leq \cdots \leq b_{l}\leq n$, the signed flagged LLT
polynomial with flag bounds $b_{i}$ is given by
\begin{equation}\label{e:signed-G-tableaux}
\Gcal _{\nubold ,\sigma }[X_{1}^b - t\,  \Xpb{1},\, \ldots, X_{l}^b - t\,
 \Xpb{l\,};\, t^{-1}] = \sum _{T\in \FSST (\nubold , \sigma , \Acal^b_{\pm} )}
(-t)^{m(T)} t^{-\inv (T)} x^{|T|},
\end{equation}
where $x^{|T|} = \prod _{u\in \nubold } x_{|T(u)|}$, with
$|\overline{a}| = |a| \defeq a$, and $m(T)$ is the number of negative
entries.
\end{prop}

\begin{proof}
For $\Acal = \Acal ^{b}_{\pm }$, the plethystic alphabet $X^{\Acal
}_{i}$ in Theorem~\ref{thm:combinatorial-G-nu-sigma} becomes
$X^{b}_{i} - t\, \Xpb{i}$ upon substituting $t\, x_{a}$ for
$x_{\overline{a}}$.  Hence, \eqref{e:signed-G-tableaux} follows by
substituting $t^{-1}$ for $t$ in \eqref{e:combinatorial-G-nu-sigma}
and then setting $x_{\overline{a}} = t\, x_{a}$.
\end{proof}

The next result is immediate from
Propositions~\ref{prop:PiA-inv-flag-bounds} and
\ref{prop:G-nu-flagged}, and indeed holds with any flagged symmetric
function in place of $\Gcal _{\nubold ,\sigma }(x;t^{-1})$.

\begin{prop}\label{prop:Pit-signed-to-unsigned}
For strictly increasing flag bounds $b_{1}<\cdots <b_{l}$, the
signed and unsigned specializations are related by nonsymmetric plethysm:
\begin{equation}\label{e:Pit-signed-to-unsigned}
\Pi _{t,x}\, \Gcal _{\nubold ,\sigma }[X_{1}^b - t\, \Xpb{1},\, \ldots,
X_{l}^b - t\, \Xpb{l\,};\, t^{-1}] = \Gcal _{\nubold ,\sigma
}[X_{1}^b,\ldots,X_{l}^b;\, t^{-1}].
\end{equation}
In particular, when $b_{i} = i$ for all $i$, we have
\begin{equation}\label{e:Pit-signed-to-unsigned-X=x}
\Pi _{t,x}\, \Gcal _{\nubold ,\sigma }^{-}(x_{1},\ldots,x_{l};\,
t^{-1}) = \Gcal _{\nubold ,\sigma }[x_{1},\ldots,x_{l};\, t^{-1}].
\end{equation}
\end{prop}

\begin{example}\label{ex:ns-HHL=signed-LLT}
Propositions~\ref{prop:signed-G-tableaux} and
\ref{prop:Pit-signed-to-unsigned} give a combinatorial formula for the
nonsymmetric Hall-Littlewood polynomials $E_\mu$ in terms of signed
flagged tableaux on a tuple of rows. Specifically, combining
\eqref{e:rows-G} and \eqref{e:Pit-signed-to-unsigned-X=x}, we have
\begin{equation}\label{e:ns-HHL=signed-LLT}
E_{\mu}(x;t^{-1}) = t^{h(\nubold)} \Gcal_{ \nubold,
\operatorname{id} }^-(x_1,\dots, x_l ; t^{-1})
= t^{h(\nubold)} \sum
_{T\in \FSST (\nubold , \operatorname{id} , \Acal^b_{\pm} )}
(-t)^{m(T)} t^{-\inv (T)} x^{|T|},
\end{equation}
where $\nubold$ is the tuple of rows $((0)/(-\mu_l),\dots,
(0)/(-\mu_1) )$ and the flag bounds $b$ for $\Acal^b_{\pm}$ are
$(1,2,\dots, l)$.  For example, for $\mu = (2,1)$, illustrating $\FSST
(\nubold , \operatorname{id} , \Acal^b_{\pm} )$ with the same
conventions as in Example \ref{ex:flag-LLT-ex-2}, we have
\begin{equation}
\label{e:signed tabs sigma12}
\begin{array}{ccccccc}
&
\begin{tikzpicture}[scale=.5,baseline=0]
  \foreach \beta / \alpha / \y in {2/1/0,2/0/2}
    \draw (\alpha,\y)  grid (\beta, \y+1);
  \node at (1.5,0.5) {1};
  \node at (2.5,0.5) {2};
  \node at (0.5,2.5) {1};
  \node at (1.5,2.5) {1};
  \node at (2.5,2.5) {1};
\end{tikzpicture}
&
\begin{tikzpicture}[scale=.5,baseline=0]
  \tikzstyle{vertex}=[outer sep=1.5pt]
  \foreach \beta / \alpha / \y in {2/1/0,2/0/2}
     \draw (\alpha,\y)  grid (\beta, \y+1);
   \node[vertex] (a) at (1.5,0.5) {2};
   \node at (2.5,0.5) {2};
   \node at (0.5,2.5) {1};
   \node[vertex] (b) at (1.5,2.5) {1};
   \node at (2.5,2.5) {1};
   \draw[-Latex] (a) -- (b);
\end{tikzpicture}
&
\begin{tikzpicture}[scale=.5,baseline=0]
  \tikzstyle{vertex}=[outer sep=1.5pt]
  \foreach \beta / \alpha / \y in {2/1/0,2/0/2}
     \draw (\alpha,\y)  grid (\beta, \y+1);
   \node[vertex] (a) at (1.5,0.5) {\mybar{1}};
   \node at (2.5,0.5) {2};
   \node at (0.5,2.5) {1};
   \node[vertex] (b) at (1.5,2.5) {1};
   \node at (2.5,2.5) {1};
   \draw[-Latex] (a) -- (b);
\end{tikzpicture} \\[.5ex]
E_{21}(x;t^{-1})\; = & \hspace{-3ex}
t\;(\,x_1^3  & + \; t^{-1} x_1^2x_2 & +\; (-t) t^{-1} x_1^3\;) & = x_1^2x_2
\end{array}
\end{equation}
\end{example}

\begin{example}\label{ex:signed-llts}
Continuing Example \ref{ex:flag-LLT-ex-2}, we confirm
\eqref{e:Pit-signed-to-unsigned-X=x} for $\nubold = ((2)/(1), (2))$
and $\sigma = (2,1)$.  Specializing $x_{\overline{1}} = t \, x_1$ in
\eqref{e ex signG ii} gives
\begin{equation}
\Gcal_{\nubold,\sigma}^-(x_1,x_2;t^{-1}) =x_1^3 + t^{-1} x_1^2 x_2 +
t^{-2} x_1 x_2^2 - t^{-1} x_1^2 (t\, x_1) - t^{-2} x_1 (t \, x_1) x_2
= t^{-2} x_1 x_2^2.
\end{equation}
Hence, by \eqref{e nspleth n2 t}, \(\Pi_{t,x} \,
\Gcal_{\nubold,\sigma}^-(x_1,x_2;t^{-1}) = x_1^3 + t^{-1} x_1^2 x_2 +
t^{-2}x_1 x_2^2\), which agrees with our computation of
$\Gcal_{\nubold,\sigma}[x_1, x_2; t^{-1}]$ in
\eqref{eq:flag-LLT-ex-1}.
\end{example}

Another property of the polynomials $\Gcal _{\nubold ,\sigma
}^{-}(x_{1},\ldots,x_{l};\, t^{-1})$ is that Hecke algebra operators
have the effect of exchanging flag bounds whenever it makes sense to
do so.

\begin{prop}\label{prop:Ti-action}
(a) Given flagged LLT indexing data $(\nubold ,\sigma )$, assume that
the flag boxes with flag numbers $i$ and $i+1$ have the same content
and are in different components $\nu ^{(j)}$ and $\nu ^{(k)}$ of
$\nubold $, so that $\sigma ' = \sigma s_{l-i}$ is also compatible
with $\nubold $, and $(\nubold , \sigma ')$ differs from $(\nubold
,\sigma )$ by exchanging flag numbers $i$ and $i+1$.  Then
\begin{equation}\label{e:Ti-action-a}
\Gcal _{\nubold ,\sigma ' }^{-}(x_{1},\ldots,x_{l};\, t^{-1}) =
\begin{cases}
t\, \overline{T_{i}}\, \Gcal _{\nubold ,\sigma}^{-}(x_{1}, \ldots,
x_{l};\, t^{-1})& \text{if $j<k$},\\[1ex]
t^{-1} \, T_{i} \, \Gcal _{\nubold ,\sigma}^{-}(x_{1},\ldots,x_{l};\,
t^{-1})& \text{if $j>k$}.
\end{cases}
\end{equation}

(b) If the flag boxes with flag numbers $i$ and $i+1$ have the same
content and are in the same (strictly ragged-right) component of
$\nubold $, then
\begin{equation}\label{e:Ti-action-b}
T_{i}\, \Gcal _{\nubold ,\sigma }^{-}(x_{1},\ldots,x_{l};\, t^{-1}) =
- \Gcal _{\nubold ,\sigma }^{-}(x_{1},\ldots,x_{l};\, t^{-1}).
\end{equation}
\end{prop}

\begin{proof}
By Corollary~\ref{cor:G-nu-sigma-ter} and
Proposition~\ref{prop:Pit-signed-to-unsigned}, we have
\begin{equation}\label{e:G-signed-ter}
t^{h(\nubold )+e(\nubold ,\sigma )} \Gcal _{\nubold ,\sigma
}^{-}(x_{1},\ldots,x_{l};\, t^{-1}) = x^{\eta _{+}}\, T_{w_{0}\,
\sigma ^{-1} \, w_{0}}\, w_{0} \sum _{w\in \Sfrak _{\rr }} (-1)^{\ell
(w)} \, \overline{E_{w(\gamma )}(x;\, t)},
\end{equation}
where $\nubold $ is the diagram of $(\rr ,\gamma, \eta )$.  Flag boxes
numbered $i$ and $i+1$ having the same content means that $(\eta
_{+})_{i} = (\eta _{+})_{i+1}$, so $T_{i}$ commutes with $x^{\eta
_{+}}$.

For part (a), either case in \eqref{e:Ti-action-a} implies the other,
since $t^{-1}\, T_{i} = (t\, \overline{T_{i}})^{-1}$, so we can assume
that $j<k$.  Then $\sigma '<\sigma $, so $w_{0}\, \sigma ^{-1}\,
w_{0}$ has a reduced factorization $s_{i}\cdot (w_{0}\, (\sigma
')^{-1}\, w_{0})$, which implies $T_{w_{0}\, (\sigma' )^{-1}\, w_{0}}
= \overline{T_{i}}\, T_{w_{0}\, \sigma ^{-1}\, w_{0}}$.  We also have
$e(\nubold ,\sigma ) = e(\nubold ,\sigma ') + 1$, with flag numbers
$i$, $i+1$ contributing one more rising end triple for $(\nubold
,\sigma )$ than for $(\nubold ,\sigma ')$.  The $j<k$ case of
\eqref{e:Ti-action-a} now follows from \eqref{e:G-signed-ter}.

For part (b), flag numbers $i$ and $i+1$ must be in consecutive rows,
$j$ and $j+1$, say.  Then $w_{0}\, \sigma\, w_{0}(i)=l - j+1$ and
$w_{0}\, \sigma\, w_{0}(i+1)=l-j$, which implies that $v\cdot s_{l-j}$
and $s_{i}\cdot v$ are reduced factorizations of $w_{0}\, \sigma
^{-1}\, w_{0}$, where $v = (w_{0}\, \sigma ^{-1}\, w_{0})\, s_{l-j} =
s_{i}\, (w_{0}\, \sigma ^{-1}\, w_{0})$.  Hence $T_{w_{0}\, \sigma
^{-1}\, w_{0}} = T_{v} \, T_{l-j} = T_{i}\, T_{v}$, and therefore
$T_{i}\, T_{w_{0}\, \sigma ^{-1}\, w_{0}} = T_{w_{0}\, \sigma ^{-1}\,
w_{0}}\, T_{l-j}$.  As in \eqref{e:E-anti}, the sum
\begin{equation}\label{e:sign-sum}
w_{0} \sum _{w\in \Sfrak _{\rr }} (-1)^{\ell (w)} \,
\overline{E_{w(\gamma )}(x;\, t)} = \sum _{w\in w_{0}\, \Sfrak _{\rr }\,
w_{0}} (-1)^{\ell (w)} \, E_{-w \, w_{0}(\gamma )}(x;\, t^{-1})
\end{equation}
is antisymmetric for the Hecke algebra of $w_{0}\, \Sfrak _{\rr }\,
w_{0} = \Sfrak _{\rr ^{\op }}$, where $\rr ^{\op}$ is the reverse of
$\rr $.  Since $j$ and $j+1$ are in the same $\rr $ block, $l-j$ and
$l-j+1$ are in the same $\rr ^{\op }$ block.  Hence the sum in
\eqref{e:sign-sum} is antisymmetric for $T_{l-j}$.  This given,
\eqref{e:G-signed-ter} implies \eqref{e:Ti-action-b}.
\end{proof}

\begin{example}
Consider \(\nubold = ((2)/(1),(2)/(0))\) and set \(\sigma' = (2,1)\),
$\sigma = \operatorname{id}$, so that \(\sigma' = \sigma s_1\).  We
have $\Gcal^{-}_{\nubold,\sigma'}(x_1,x_2;t^{-1}) = t^{-2}x_1x_2^2$
from Example~\ref{ex:signed-llts}, and a similar computation gives
$\Gcal^{-}_{\nubold,\sigma'} =t^{-1} x_1^2 x_2$.  One checks that $
t^{-1} x_1^2 x_2 = t \, \overline{T_1}(t^{-2} x_1 x_2^2)$, confirming
the assertion of Proposition~\ref{prop:Ti-action} that
\(\Gcal^{-}_{\nubold,\sigma'}(x_1,x_2;t^{-1}) = t \,
\overline{T_1}\,\Gcal^{-}_{\nubold,\sigma}(x_1,x_2;t^{-1})\) in this
case.
\end{example}

A feature of the combinatorial formula in
Proposition~\ref{prop:signed-G-tableaux} is the existence of a
sign-reversing involution that permits the sum in
\eqref{e:signed-G-tableaux} to be restricted to tableaux $T$
satisfying a non-attacking condition.

\begin{defn}\label{def:non-attacking}
With notation as in Proposition~\ref{prop:signed-G-tableaux}, a
flagged tableau $T\in \FSST (\nubold ,\sigma ,\Acal ^{b}_{\pm })$ is
{\em non-attacking} if there is no attacking pair $(x,y)$ for $\nubold
$ such that $|T'(x)| = |T'(y)|$, where $T'$ is the extended tableau
with fixed entry $b_{i}$ in the flag box with flag number $i$, and $
|\overline{a}| = |a| \defeq a$.
\end{defn}

\begin{prop}\label{prop:non-attacking}
With notation as in Proposition~\ref{prop:signed-G-tableaux}, if the
flag bounds $b_{1}<\cdots <b_{l}$ are strictly increasing, then formula
\eqref{e:signed-G-tableaux} also holds with the sum restricted to
non-attacking tableaux:
\begin{equation}\label{e:non-attacking}
\Gcal _{\nubold ,\sigma }[X^{b}_{1} - t\, \Xpb{1},\, \ldots, X^{b}_{l}
- t\, \Xpb{l\,};\, t^{-1}] = \sum _{\substack{T\in \FSST (\nubold ,
\sigma ,
\Acal ^{b}_{\pm })\\
\text{{\rm non-attacking}}}} (-t)^{m(T)} t^{-\inv (T)} x^{|T|}.
\end{equation}
In particular, this holds for $\Gcal _{\nubold ,\sigma
}^{-}(x_{1},\ldots,x_{l};\, t^{-1})$.
\end{prop}

\begin{proof}
Adapting the proof of \cite[Lemma~5.1]{HagHaiLo05} to flagged
tableaux, we will construct a sign-reversing, weight-preserving
involution $\Psi $ on tableaux $T\in \FSST (\nubold ,\sigma ,\Acal
^{b}_{\pm })$ that violate the non-attacking condition; then the terms
in \eqref{e:signed-G-tableaux} for $T$ and $\Psi (T)$ cancel, leaving
only the terms in \eqref{e:non-attacking}.

The definition of attacking pair $(x,y)$ is equivalent to: $x$ and $y$
belong to different components of $\nubold '$, $x$ is not a flag box,
and $x<y<z$ in the content/row reading order, where $z$ is the box
immediately to the right of $x$.

Given $T\in \FSST (\nubold ,\sigma ,\Acal ^{b}_{\pm })$ which fails to
be non-attacking, let $x$ be the last box in reading order that occurs
in an attacking pair $(x,y)$ with $|T'(x)| = |T'(y)|$, and let $a =
|T'(x)|$.  We claim that for this $x$, the other member $y$ of the
pair is unique.

Suppose to the contrary that $(x,y)$ and $(x,y')$ are attacking pairs
with $|T'(x)| = |T'(y)| = |T'(y')| = a$ and (without loss of
generality) that $y<y'$ in reading order.  If $y$ is a flag box with
flag number $j$, then the flag number $i$ in the row of $y'$ is less
than $j$, because the flag box in the row of $y'$ follows $y$ in
reading order.  But then $|T'(y')|\leq b_{i}<b_{j}= a$.  Hence $y$ is
not a flag box.  For use below, note that this argument shows that if
$|T'(u)| = a$ for any $u$ following $y$ in reading order, then $y$ is
not a flag box.

If $y$ and $y'$ are in the same component of $\nubold '$, then they
have the same content, either $c(y) = c(y') = c(x)$ or $c(y) = c(y') =
c(x)+1$, depending on whether the component containing $y$ and $y'$ is
above or below the component containing $x$.  But then we cannot have
$|T'(y)| = |T'(y')| = a$, because $\nubold $ contains a box $v$ above
$y$ and to the left of $y'$, the condition on rows in a flagged
tableau implies that $T(v)\leq a$, and the condition on columns
implies that $T(v)\geq \overline{a}$.  Hence $y$ and $y'$ are in
different components.

Since $x<y<y'<z$ in reading order, where $z$ is the box
immediately to the right of $x$, and $y$ is not a flag box, and $y$
and $y'$ are in different components, it follows that $(y,y')$ is an
attacking pair.  But $y$ follows $x$ in reading order, so this
contradicts the choice of $x$.  This shows that $y$ is unique, as
claimed.

We construct $S = \Psi (T)$ by changing the sign of $T(x)$, that is,
$S(x) = a$ if $T(x)=\overline{a}$, or $S(x) = \overline{a}$ if $T(x) =
a$, and $S(u) = T(u)$ for $u\not =x$.  We claim that this defines a
valid flagged tableau $S$.  If not, then $S(u)$ and $S(x)$ would have
to violate the row or column conditions for a box $u$ adjacent to $x$
in its row or column.  We now check and rule this out for each of the
four possible boxes $u$.

If $u$ is the box to the left of $x$ or above $x$, let $S(u) = T(u) =
b$.  If $u$ is the box to the left of $x$, then $S(u) \sgnle S(x)$
because $b\sgnle a \Leftrightarrow b \sgnle \overline{a}$ for all
$b\in \Acal $.  Similarly, if $u$ is the box above $x$, then $S(u)
\sgngt S(x)$ because $b\sgngt a \Leftrightarrow b \sgngt \overline{a}$
for all $b\in \Acal $.

If $u$ is the box to the right of $x$ ($u$ might be a flag box), the
only way to have $S(x) \not \sgnle S'(u)$ is if $S'(u) = T'(u) = a$
and $T(x) = a$, so $S(x) = \overline{a}$.  Similarly, if $u$ is the
box below $x$, the only way to have $S(x) \not \sgngt S'(u)$ is if
$S'(u) = T'(u) = \overline{a}$ and $T(x) = \overline{a}$, so $S(x) =
a$.  In either case, $c(u) = c(x)+1$, so $x<y<u$ in reading order.
Hence, by the argument used previously as part of the proof that $y$
is unique, it follows that $y$ is not a flag box.  But then $(y,u)$ is
an attacking pair with $|T'(y)| = |T'(u)|$, contrary to the choice of
$x$.

Since $|S| = |T|$, the tableau $S$ again violates the non-attacking
condition, and the distinguished attacking pair in the construction of
$\Psi (S)$ is the same pair $(x,y)$ as for $T$.  Hence $\Psi (S) = T$,
i.e., $\Psi $ is an involution.

Clearly, $x^{|S|} = x^{|T|}$.  Exchanging $S$ and $T$ if needed, we
can assume that $T(x) = a$ and $S(x) = \overline{a}$.  Then $(x,y)$ is
an inversion in $S$ but not in $T$.  We have already shown that if
$(x,y')$ is another attacking pair with $y\not =y'$, then $|T'(y')|
\not = a$, hence $(x,y')$ is an inversion in $S$ if and only if it is
an inversion in $T$.  Since $b \sgngt a \Leftrightarrow b \sgngt
\overline{a}$ for all $b\in \Acal $, we also have that any attacking
pair $(w,x)$ is an inversion in $S$ if and only if it is an inversion
in $T$.  The same obviously holds for any attacking pair not
containing $x$.  Thus, $S$ and $T$ have the same inversions except for
$(x,y)$, giving $\inv (S) = \inv (T)+1$.  Finally, since we changed
one entry from positive to negative, $m(S) = m(T)+1$.  It follows that
$(-t)^{m(S)} t^{-\inv (S)} x^{|S|} = - (-t)^{m(T)} t^{-\inv (T)}
x^{|T|}$, so the terms for $S$ and $T$ in \eqref{e:signed-G-tableaux}
cancel.
\end{proof}

\begin{example}
In Example \ref{ex:ns-HHL=signed-LLT}, the only non-attacking tableau
in \eqref{e:signed tabs sigma12} is the one in the middle.  The terms
for the other two tableaux cancel, illustrating
Proposition~\ref{prop:non-attacking}.
\end{example}

Lastly, we have the following $t$-deformation of
Corollary~\ref{cor:Weyl-on-flagged} for the action of a normalized
Hecke symmetrizer on the signed flag bound evaluation of any flagged
symmetric function.

\begin{lemma}\label{lem:symm-on-signed-flagged}
Let $X_{i}^{b}$ and $\Xpb{i}$ be as in Definition~\ref{def:signed-G},
for flag bounds $b_{1}\leq \cdots \leq b_{l}\leq N$ such that $b_{j-1}
< b_{j} =r$ and $r+n <b_{j+1}$, for some $j<l$, $r$, and $n$.  Let
$\hsym _{[r,r+n]}$ be the normalized Hecke symmetrizer in
\eqref{e:normalized-Hecke-symmetrizer} for the Young subgroup $\Sfrak
_{[r,r+n]}\subseteq \Sfrak _{N}$ of permutations of $[r,r+n]$.  Then
for any flagged symmetric function $f(X_{1},\ldots,X_{l})$ with
coefficients in a ring containing $\QQ (t)$, we have
\begin{multline}\label{e:signed-flag-r-n}
\hsym _{[r,r+n]}\, f[X_{1}^b - t\, \Xpb{1},\, \ldots, 
X_{l}^b - t\, \Xpb{l}] \\
\equiv f[X_{1}^b - t\, \Xpb{1},\, \ldots, \,
-t x_{b_{j-1}} + (1-t)(x_{b_{j-1}+1}+\cdots +x_{r+n}),\\
(1-t)(x_{r+n+1}+\cdots +x_{b_{j+1} -1}) + x_{b_{j+1}},\, \ldots,
X_{l}^b - t\, \Xpb{l}]\, ,
\end{multline}
where the two plethystic alphabets written out in full on the right
hand side replace $X_{i}^{b} - t \Xpb{i}$ for $i = j,\, j+1$, modulo
polynomials whose coefficients vanish to order $>n$ at $t=0$.
\end{lemma}

\begin{proof}
We can assume by linearity that $f = \fh _{\aA }$.  Then
\begin{equation}\label{e:signed-spec-lhs}
f[X_{1}^b - t\, \Xpb{1},\, \ldots, X_{l}^b - t\, \Xpb{l}] = \prod
_{i=1}^{l} h_{a_{i}}[x_{b_{i}}+ (1-t) (x_{1}+\cdots + x_{b_{i}-1})]\, ,
\end{equation}
and the evaluation of $f$ on the right hand side of
\eqref{e:signed-flag-r-n} is
\begin{equation}\label{e:signed-spec-rhs}
h_{a_{j}}[(1-t)(x_{1}+\cdots +x_{r+n})]\, 
\prod
_{i\not =j} h_{a_{i}}[x_{b_{i}}+ (1-t) (x_{1}+\cdots + x_{b_{i}-1})]\, .
\end{equation}
The factors for $i\not =j$ in \eqref{e:signed-spec-lhs} and
\eqref{e:signed-spec-rhs} are $\Sfrak _{[r,r+n]}$ invariant, and the
factor for $i=j$ in \eqref{e:signed-spec-lhs} is
$h_{a_{j}}[(1-t)(x_{1}+\cdots +x_{r-1}) +x_{r}]$.  Hence the result
follows if we show that
\begin{equation}\label{e:Srn-identity}
\hsym _{[r,r+n]} \, h_{a}[Y+x_{r}] \equiv h_{a}[Y+(1-t)(x_{r}+\cdots
+x_{r+n})]\, ,
\end{equation}
where $Y$ does not involve the variables $x_{r},\ldots,x_{r+n}$.
Using $h_{a}[Y+Z] = \sum _{p} h_{a-p}[Y]h_{p}[Z]$,
\eqref{e:Srn-identity} follows from $\hsym _{[r,r+n]}\, x_{r}^{0} = 1$
and the identity
\begin{equation}\label{e:Sl-x1-a}
\hsym _{[r,r+n]}\, x_{r}^{p} = \frac{h_{p}[(1-t)(x_{r}+\cdots
+x_{r+n})]}{1 - t^{n+1}},
\end{equation}
for $p>0$, which is equivalent to \eqref{e:Sl-to-symmetric-HL} for
$\lambda =(p)$, using the formula \cite[III,~(2.10)]{Macdonald95}
\begin{equation}\label{e:Pa}
P_{(p)}(x_{1},\ldots,x_{n+1}; t) = h_{p}[(1-t)(x_{1}+\cdots
+x_{n+1})]/(1-t). \qedhere
\end{equation}
\end{proof}

\section{Nonsymmetric Macdonald polynomials}
\label{s:ns-Mac-pols}
In this section, we show that Conjecture~\ref{conj:atom-pos-standard}
implies an atom positivity conjecture for nonsymmetric plethystic
transformations applied to Knop's integral form nonsymmetric Macdonald
polynomials $\Ecal _{\mu }(x;q,t)$, generalizing positivity of the
Kostka-Macdonald coefficients $K_{\lambda ,\mu }(q,t)$ for symmetric
Macdonald polynomials.  Our positivity conjecture requires that we
first stabilize $\Ecal _{\mu }(x;q,t)$ to be symmetric in extra
variables, as do other nonsymmetric generalizations of $K_{\lambda
,\mu }(q,t)$ positivity conjectured by Knop~\cite{Knop07} and
Lapointe~\cite{Lapointe22}.  In the present context, we can explain
the need for stabilization as emerging naturally from the way
nonsymmetric Macdonald polynomials connect with flagged LLT
polynomials and nonsymmetric plethysm.  In subsequent papers, we will
apply the methods of this paper to obtain new results on the
conjectures in \cite{Knop07} and \cite{Lapointe22}.

Throughout this section we set $\kk= \QQ(q,t)$.

\subsection{Integral forms}
\label{ss:integral-forms}

The {\em row diagram} $\nubold = \dg (\mu )$ of $\mu \in \NN ^{m}$ is
the tuple of single rows $\nu ^{(i)} = (0/-\mu _{m+1-i})$ with end
content $0$ and length $\mu _{m+1-i}$.  Since our convention is to
number rows from bottom to top, $\dg (\mu )$ has row lengths $\mu
_{1},\ldots,\mu _{m}$ from top to bottom.  The {\em arm } of a box
$u\in \dg (\mu )$ of content $-c$ in a row of length $\mu _{i}$ is the
number $a(u) = \mu _{i}-c$ of boxes to the left of $u$ in the same
row.  The {\em leg} $l(u)$ is the number of indices $j<i$ such that $c
\leq \mu _{j}\leq \mu _{i}$ plus the number of indices $j>i$ such that
$c \leq \mu _{j}+1\leq \mu _{i}$.

A more visual way to describe the leg of a box is as follows.  Define
the {\em hand} of a row of length $\mu _{i}$ to be the box of content
$-(\mu _{i}+1)$ immediately to its left.  If $v$ is the hand in the
row of $u$, then $l(u)$ is the number of hands strictly between $v$
and $u$ in the content/row reading order.

\begin{example}\label{ex:arms-and-legs}
The row diagram of $\mu =(2,1,4,4,0,1,4)$ is shown below, with hands
indicated by circles.  For the box labeled $u$, the arm $a(u)=2$
counts the boxes marked $*$, and the leg $l(u)=3$ counts the shaded
hands.
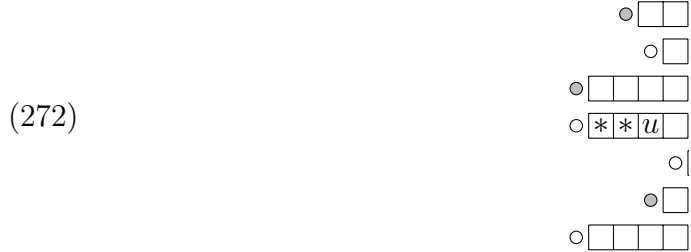
\begin{equation}\label{e:arms-and-legs}
\begin{tikzpicture}[scale=.33,baseline=5*.33cm]
\foreach  \mu / \y in {2/9, 1/7.5, 4/6, 4/4.5, 0/3, 1/1.5, 4/0}
  \draw[xshift=-.5cm, yshift=\y cm] (-\mu,0) grid (0,1);
\node at (-2,5) {$u$};
\node at (-3,5) {$*$};
\node at (-4,5) {$*$};
\filldraw[color=black, fill=gray!50] (-3,9.5) circle (.25) (-5,6.5)
circle (.25) (-2,2) circle (.25);
\draw (-2,8) circle (.25) (-5,5) circle (.25) (-1,3.5) circle (.25)
(-5,.5) circle (.25);
\end{tikzpicture}
\end{equation}
\end{example}

Our definitions of diagram, arms, and legs are equivalent to those in
Knop \cite{Knop97}, taking into account that Knop reverses the index
$\mu $.

\begin{defn}[\cite{Knop97}]\label{def:integral-form}
The {\em integral form nonsymmetric Macdonald polynomials} are
\begin{equation}\label{e:integral-form}
\Ecal _{\mu }(x;q,t) = \bigl(\prod \nolimits_{\, u\in \dg (\mu )} (1-
q^{a(u)+1}\, t^{l(u)+1}) \bigr)\, E_{\mu }(x;q,t).
\end{equation}
\end{defn}

By \cite[Corollary~5.2]{Knop97}, $\Ecal _{\mu }(x;q,t)$ has
coefficients in $\ZZ [q,t]$, whence the term {\em integral form}.  A
combinatorial formula for $\Ecal _{\mu }(x;q,t)$ was given in
\cite{HagHaiLo08}.  We will state it here in terms of fillings of the
row diagram $\dg (\mu )$, adjusting for the fact that
\cite{HagHaiLo08} uses column diagrams, and swaps the terms `arm' and
`leg' versus the usage here and in \cite{Knop97}.

Fix $\mu \in \NN ^{m}$, and regard $\dg (\mu )$ as a tuple of
ragged-right diagrams, with a flag box of content $0$ in each row, as
in Definition~\ref{def:flag-numbers}.  Let $\dg (\mu )'$ be the
extended tuple including the flag boxes.  Given a filling $U\colon \dg
(\mu )\rightarrow [m]$ with entries in $[m]=\{1,2,\ldots,m \}$, we
define the {\em extended filling} $U':\dg (\mu )'\rightarrow [m]$ as
in \eqref{e:T-extended}, taking the fixed entries in the flag boxes to
be $m,\ldots,1$, so that the row with $i$ in its flag box has length
$\mu _{i}$.

We define attacking pairs for $\dg (\mu )$ as in
Definition~\ref{def:inversions}---thus, the first box in an attacking
pair must be in $\dg (\mu )$, but the second may be a flag box.  An
{\em attacking inversion} in $U$ is an attacking pair $(x,y)$ such
that $U(x)>U'(y)$; the number of attacking inversions is denoted $\inv
(U)$.  A filling $U\colon \dg (\mu )\rightarrow [m]$ is {\em
non-attacking} if the extended filling $U'$ satisfies $U(x)\not =
U'(y)$ for all attacking pairs $(x,y)$.  These definitions are
essentially the same as in Definition~\ref{def:inversions} and
Definition~\ref{def:non-attacking}, except that the filling $U$ is
unsigned and is not required to be a flagged tableau.

For any box $u\in \dg (\mu )$, let $e(u)\in \dg (\mu )'$ denote the
box immediately to its right in the same row; thus, if $u$ has content
$-c$, then $e(u)$ has content $-c+1$, and is a flag box if $c=1$.
Given a filling $U\colon \dg (\mu )\rightarrow [m]$, define
\begin{align}\label{e:Des}
\Des (U) & = \{ u\in \dg (\mu ) \mid U(u)>U'(e(u))\} \\
\label{e:maj}
\maj (U) & = \sum \nolimits_{u\in \Des (U)} (a(u)+1)\\
\label{e:coinv}
\coinv (U) & = - \inv (U) + n(\mu _{+}) + \sum \nolimits_{u\in \Des (U)} l(u),
\end{align}
where $\mu _{+}$ is the rearrangement of $\mu $ to a partition, and
$n(\mu _{+}) = \sum _{j}(j-1)(\mu _{+})_{j}$, as in
\eqref{e:n(lambda)}.

We note that $\inv $ and $\coinv $ are defined a little differently in
\cite{HagHaiLo08}, in such a way that direct translation into the
notation here would give
\begin{equation}\label{e:coinv-original}
\coinv (U)  = - \inv (U) + N + \sum \nolimits_{u\in \Des (U)} l(u),
\end{equation}
where $N = \bigl(\sum _{u\in \dg (\mu )} l(\mu ) \bigr) - |\{(i<j)
\mid \mu _{i} >\mu _{j} \}|$.  In fact, $N = n(\mu _{+})$, since
switching adjacent parts $\mu _{i}$, $\mu _{i+1}$ does not change $N$,
and it is immediate that $N = n(\mu _{+})$ if $\mu $ is weakly
increasing.

\begin{prop}[{\cite[Corollary~3.5.2]{HagHaiLo08}}]\label{prop:original-HHL}
With notation as above, for any $\mu \in \NN ^{m}$, we have
\begin{equation}\label{e:original-HHL}
\Ecal _{\mu }(x;q,t)\, =
\sum _{\substack{U \colon \dg (\mu )\rightarrow [m]\\
\text{non-attacking}}} \hspace{-1ex}  q^{\maj (U)}\, t^{\coinv (U)}
\hspace{-2ex} \prod _{\substack{u\in \dg (\mu )\\ U(u) = U'(e(u))}}
\hspace{-3ex} (1 - q^{a(u)+1} t^{l(u)+1}) \hspace{-2ex}
\prod _{\substack{u\in \dg (\mu )\\
U(u) \not = U'(e(u))}} \hspace{-3ex} (1-t) \;\; x^{U}\, .
\end{equation}
\end{prop}

\begin{example}\label{ex:original-HHL}
For $\mu = (2,0,1)$, we display the arm and leg of each box in the row
diagram $\dg (\mu )$, the attacking pairs in the extended diagram $\dg
(\mu )'$, and all non-attacking fillings $U\colon \dg (\mu
)\rightarrow [3]$, with their extensions $U'$.
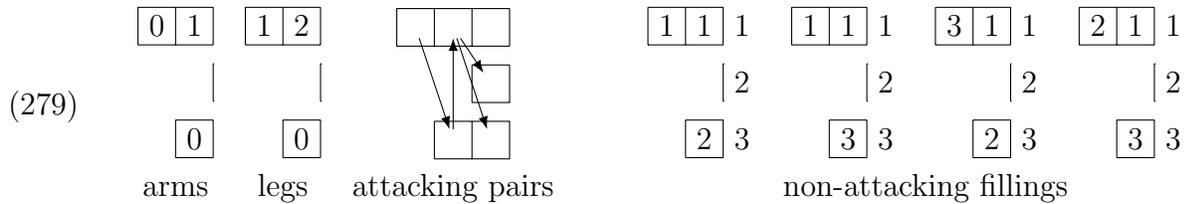
\begin{equation}\label{e:original-HHL-fillings}
\begin{array}{c@{\quad }c@{\quad }c@{\qquad \quad }c@{\quad }c@{\quad }
c@{\quad }c}
\begin{tikzpicture}[scale=.5]
  \foreach \beta / \alpha / \y / \z in {0/-1/1/0, 0/0/1/1.5, 0/-2/1/3}
     \draw[yshift=\z cm] (\alpha,0) grid (\beta, \y);
   \node at (-.5,.5) {0};
   \node at (-1.5,3.5) {0};
   \node at (-.5,3.5) {1};
\end{tikzpicture}
&
\begin{tikzpicture}[scale=.5]
  \foreach \beta / \alpha / \y / \z in {0/-1/1/0, 0/0/1/1.5, 0/-2/1/3}
     \draw[yshift=\z cm] (\alpha,0)  grid (\beta, \y);
   \node at (-.5,.5) {0};
   \node at (-1.5,3.5) {1};
   \node at (-.5,3.5) {2};
\end{tikzpicture}
&
\begin{tikzpicture}[scale=.5]
  \foreach \beta / \alpha / \y / \z in {1/-1/1/0, 1/0/1/1.5, 1/-2/1/3}
     \draw[yshift=\z cm] (\alpha,0)  grid (\beta, \y);
  \node (a) at (-1.5,3.5) {};
  \node (b) at (-.5,.5) {};
  \node (c) at (-.5,3.5) {};
  \node (d) at (.5,.5) {};
  \node (e) at (.5,2) {};
  \draw[-Latex] (a) -- (b); \draw[-Latex] (b) -- (c); \draw[-Latex] (c) -- (d);
  \draw[-Latex] (c) -- (e);
\end{tikzpicture}
&
\begin{tikzpicture}[scale=.5]
  \foreach \beta / \alpha / \y / \z in {0/-1/1/0, 0/0/1/1.5, 0/-2/1/3}
     \draw[yshift=\z cm] (\alpha,0)  grid (\beta, \y);
   \node at (.5,.5) {3}; \node at (.5,2) {2}; \node at (.5,3.5) {1};
   \node at (-1.5,3.5) {1}; \node at (-.5,3.5) {1}; \node at (-.5,.5) {2};
\end{tikzpicture}
&
\begin{tikzpicture}[scale=.5]
  \foreach \beta / \alpha / \y / \z in {0/-1/1/0, 0/0/1/1.5, 0/-2/1/3}
     \draw[yshift=\z cm] (\alpha,0)  grid (\beta, \y);
   \node at (.5,.5) {3}; \node at (.5,2) {2}; \node at (.5,3.5) {1};
   \node at (-1.5,3.5) {1}; \node at (-.5,3.5) {1}; \node at (-.5,.5) {3};
\end{tikzpicture}
&
\begin{tikzpicture}[scale=.5]
  \foreach \beta / \alpha / \y / \z in {0/-1/1/0, 0/0/1/1.5, 0/-2/1/3}
     \draw[yshift=\z cm] (\alpha,0)  grid (\beta, \y);
   \node at (.5,.5) {3}; \node at (.5,2) {2}; \node at (.5,3.5) {1};
   \node at (-1.5,3.5) {3}; \node at (-.5,3.5) {1}; \node at (-.5,.5) {2};
\end{tikzpicture}
&
\begin{tikzpicture}[scale=.5]
  \foreach \beta / \alpha / \y / \z in {0/-1/1/0, 0/0/1/1.5, 0/-2/1/3}
     \draw[yshift=\z cm] (\alpha,0)  grid (\beta, \y);
   \node at (.5,.5) {3}; \node at (.5,2) {2}; \node at (.5,3.5) {1};
   \node at (-1.5,3.5) {2}; \node at (-.5,3.5) {1}; \node at (-.5,.5) {3};
\end{tikzpicture}\\
\text{arms}& \text{legs}& \text{attacking pairs}
 & \multicolumn{4}{c}{\text{non-attacking fillings}}
\end{array}
\end{equation}
In this example, $\mu _{+} = (2,1)$ and $n(\mu _{+}) = 1$.  For the
four non-attacking fillings $U$ in the order displayed, we have $\maj
(U) = 0,0,1,1$; $\inv (U)=1,1,2,1$; $\sum _{u\in \Des (U)}l(u) =
0,0,1,1$; $\coinv (U)=0,0,0,1$; hence, $q^{\maj (U)} t^{\coinv (U)} =
1,1,q,qt$.  Formula \eqref{e:original-HHL} now reads
\begin{multline}\label{e:original-HHL-term-by-term}
\Ecal _{(2,0,1)}(x;q,t) =
(1-t)(1-q\, t^2)(1-q^2 t^3)x_1^2  x_2
+  (1- q\,  t) (1-q\,  t^2) (1-q^2 t^3) x_1^2 x_3\\
+ q\, (1-t)^{2}(1-q^2 t^3) x_1 x_2 x_3 + q\, t\, (1-t) (1-q\, t)
(1-q^2 t^3) x_1 x_2 x_3.
\end{multline}
The last two terms could be combined and simplified, but we have left
them as they are so as to keep the correspondence with non-attacking
fillings visible.
\end{example}

Expanding the products in \eqref{e:original-HHL} leads to an
equivalent formula expressed as a sum over signed fillings $S\colon
\dg (\mu )\rightarrow [m]_{\pm }$, where $[m]_{\pm } =
\{1<\overline{1}<\cdots <m<\overline{m} \}$ with $([m]_{\pm })^{+} =
\{1, \ldots, m\}$ and $([m]_{\pm })^{-} = \{\overline{1},\ldots,
\overline{m}\}$.  We define a signed filling $S$ to be non-attacking
if the unsigned filling $U = |S|$ is non-attacking.  Then, taking
$\inv (S)$ to be the number of attacking pairs $(x,y)$ for $\dg (\mu
)$ such that $S(x)\sgngt S'(y)$, as in
Definition~\ref{def:inversions}, we have $\inv (S) = \inv (U)$, since
the non-attacking condition $|S(x)|\not = | S'(y)|$ implies
$S(x)\sgngt S'(y) \Leftrightarrow U(x)>U'(y)$.

\begin{cor}\label{cor:signed-HHL}
With the above definitions for signed fillings, for any $\mu \in \NN
^{m}$, we have
\begin{equation}\label{e:signed-HHL}
\Ecal _{\mu }(x;q,t) = t^{n(\mu _{+})} \sum _{\substack{S \colon \dg
(\mu ) \rightarrow [m]_{\pm }\\
\text{non-attacking}}}  (-t)^{m(S)}\,
t^{-\inv (S)}\, x^{|S|}\,
\!\!\prod _{\substack{u\in \dg (\mu )\\
S(u) \sgngt S'(e(u))}} \!\!\!\! q^{a(u)+1}\, t^{l(u)},
\end{equation}
where $m(S)$ is the number of negative entries in $S$.
\end{cor}

\begin{proof}
Letting $U = |S|$, we have $\inv (S) = \inv (U)$, and $S(u) \sgngt
S'(e(u))$ for all $u\in \Des (U)$.  We also have $S(u) \sgnle
S'(e(u))$ if $U(u) < U'(e(u))$---that is, if $u\not \in \Des (U)$ and
$U(u)\not =U'(e(u))$.  Hence, the term for $S$ in
\eqref{e:signed-HHL} factors as
\begin{equation}\label{e:signed-HHL-factored}
\biggl(
t^{n(\mu _{+})}\, t^{-\inv (U)}\,  x^{U}\, \! \prod _{u\in \Des (U)}
q^{a(u)+1}\, t^{l(u)} \biggr)
\times \biggl((-t)^{m(S)} \!\!\! \prod _{\substack{u\in \dg (\mu )\\
U(u) = U'(e(u))\\
S(u) \sgngt S'(e(u))}} \!\!\!\! \,q^{a(u)+1}\,
t^{l(u)} \biggr)\,.
\end{equation}
The first factor is equal to $q^{\maj (U)}\, t^{\coinv (U)}\, x^{U}$.
When $U(u) = U'(e(u))$, we have $S(u)\sgngt S'(e(u))$ if and only if
$S(u)$ is negative.  Hence, if $S(u)$ is positive, the contribution
from $u$ in the second factor is $1$, and if is $S(u)$ is negative, it
is either $-t$, if $U(u) \not = U'(e(u))$, or $-q^{a(u)+1}t^{a(u)+1}$,
if $U(u) = U'(e(u))$.  Summing this factor over all choices of signs
for a fixed $|S| = U$ therefore gives the product
\begin{equation}\label{e:U-factors}
 \prod _{\substack{u\in \dg (\mu )\\ U(u) = U'(e(u))}}
\hspace{-3ex} (1 - q^{a(u)+1} t^{l(u)+1}) \hspace{-3ex}
\prod _{\substack{u\in \dg (\mu )\\
U(u) \not = U'(e(u))}} \hspace{-3ex} (1-t)
\end{equation}
in the term for $U$ on the right hand side of \eqref{e:original-HHL}.
\end{proof}

Later we will need the following properties of the integral forms.

\begin{prop}[\cite{Knop07}]\label{prop:Knop-antisymmetry}
For any $\mu \in \NN ^{m}$ and $1\leq i<m$, the difference $\Ecal
_{s_{i} \mu }(x;q,t) - \Ecal _{\mu }(x;q,t)$ is $T_{i}$ antisymmetric,
i.e.,
\begin{equation}\label{e:Knop-antisymmetry}
(T_{i}+1) (\Ecal _{s_{i} \mu }(x;q,t) - \Ecal _{\mu }(x;q,t)) = 0.
\end{equation}
\end{prop}

\begin{proof}
By \cite[Lemma~11.5]{Knop07}, $\Ecal _{s_{i}\mu } - \Ecal _{\mu }$
is a scalar multiple of $(T_{i}-t)\Ecal _{\mu }$.
\end{proof}

\begin{prop}\label{prop:rotation-int}
The integral forms satisfy the  rotation
identity
\begin{equation}\label{e:rotation-int}
\Ecal _{(\mu _{m}+1,\mu _{1},\ldots,\mu _{m-1})}(x;q,t) = (1-q^{\mu
_{m}+1} t^{l+1})\, q^{\mu _{m}}\, x_{1}\, \Ecal _{\mu
}(x_{2},\ldots,x_{m},q^{-1} x_{1}; q,t)\, ,
\end{equation}
where $l = |\{i<m \mid \mu _{i}\leq \mu _{m} \}|$.
\end{prop}

\begin{proof}
Let $\muhat = (\mu _{m}+1,\mu _{1},\ldots,\mu _{m-1})$.  Then $\dg
(\muhat )$ is constructed by rotating the bottom row of $\dg (\mu )$
to the top and shifting it left one box, inserting a new box $v$ of
content $-1$ to extend it to length $\mu _{m}+1$.  If $u$ is a box of
$\dg (\mu )$, the box $\uhat \in \dg (\muhat )$ corresponding to $u$
via this construction has $a(\uhat ) = a(u)$ and $l(\uhat ) = l(u)$.
The new box $v$ has $a(v) = \mu _{m}$ and $l(v) = l$.  The product
$\prod _{u}(1-q^{a(u)+1} t^{l(u)+1})$ in \eqref{e:integral-form} for
$\muhat $ is therefore $(1 - q^{\mu _{m}+1}t^{l+1})$ times the product
for $\mu $.  This given, \eqref{e:rotation-int} is equivalent to
\eqref{e:rotation}.
\end{proof}

\subsection{Flagged LLT formula for nonsymmetric Macdonald
polynomials}
\label{ss:ns-HHL}

We now show that the formulas in Proposition~\ref{prop:original-HHL}
and Corollary~\ref{cor:signed-HHL} can be reinterpreted as expressing
the nonsymmetric integral forms $\Ecal _{\mu }(x;q,t)$ in terms of
signed flagged LLT polynomials for tuples of ribbon-shaped diagrams,
parallel to results of \cite{HagHaiLo05} in the symmetric case.

A {\em ribbon} is a connected skew diagram containing no $2\times 2$
square---equivalently, containing no two boxes of the same content.
An {\em augmented ribbon} is a ragged-right diagram $\nu $ such that
the flag box in the first row has content $0$, and this box and $\nu$
form a ribbon.  We call the flag box in the first row the {\em
augmentation box}, and let $\nu _{\aug }$ denote the ribbon with the
augmentation box included.  If $|\nu |=s$, the boxes in $\nu _{\aug }$
have contents $0,-1,\ldots,-s$ from southeast to northwest, starting
with the augmentation box.  Each pair of boxes with consecutive
contents forms a horizontal or vertical domino.

The only possible overhanging box in an augmented ribbon $\nu $ is the
box of content $-1$, which may form a vertical domino with the
augmentation box.  If the box of content $-1$ and the augmentation box
form a horizontal domino, or if $\nu $ is empty, then $\nu $ is an
ordinary skew diagram.  Two examples of size $|\nu | = 6$ are shown
here, with augmentation box indicated by $\bullet $.  Each one has
four rows.  The first is an ordinary skew diagram; the second is
strictly ragged-right, with first row $0/0$.
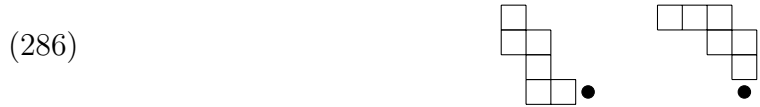
\begin{equation}\label{e:augmented-ribbons}
\begin{tikzpicture}[scale=.33,baseline=.66cm]
  \foreach \beta / \alpha / \z in {-3/-4/3,-2/-4/2, -2/-3/1, -1/-3/0}
     \draw[yshift=\z cm] (\alpha,0)  grid (\beta, 1);
  \filldraw (-.5,.5) circle (.25);
\end{tikzpicture}
\qquad
\begin{tikzpicture}[scale=.33,baseline=.66cm]
  \foreach \beta / \alpha / \z in {-1/-4/3, 0/-2/2, 0/-1/1}
     \draw[yshift=\z cm] (\alpha,0)  grid (\beta, 1);
  \filldraw (-.5,.5) circle (.25);
\end{tikzpicture}
\end{equation}

If $\nubold = (\nu ^{(1)},\ldots,\nu ^{(m)})$ is a tuple of augmented
ribbons, we have $\nubold \subset \nubold _{\baug } \subseteq \nubold
'$, where $\nubold '$ is the usual extended diagram containing all
flag boxes, and $\nubold _{\baug }$ is the tuple of ribbons $\nu
^{(i)}_{\, \aug }$ containing the augmentation boxes.  Let $ \mu
^{\op} = (\mu _{m},\ldots,\mu _{1}) = (|\nu ^{(1)}|,\ldots,|\nu
^{(m)}|)$.  If all the ribbons $\nu ^{(i)} _{\, \aug }$ are
horizontal, then $\nubold = \dg (\mu )$ is the row diagram of $\mu $.
In general, to every box $u\in \nubold $, there corresponds a unique
box $\uhat $ of the same content as $u$ in the corresponding row of
$\dg (\mu )$.  We generalize the definition of arms and legs from row
diagrams to augmented ribbon tuples by setting $a(u) = a(\uhat )$,
$l(u) = l(\uhat )$.  Thus, $a(u)$ is the number of boxes northwest of
$u$ in its ribbon, and if we adjoin a `hand' to each ribbon, of
content $-(\mu _{i}+1)$ for a ribbon of length $\mu _{i}$, then $l(u)$
is the number of hands strictly between $v$ and $u$ in the content/row
reading order, where $v$ is the hand attached to the ribbon containing
$u$.

By Lemma~\ref{lem:flag-numbers}(a), if $\nubold = (\nu ^{(1)},\ldots,
\nu ^{(m)})$ is a tuple of ragged-right skew shapes such that the
first row of each $\nu ^{(i)}$ has end content zero, then there exists
a compatible permutation $\sigma $ that assigns any permutation of the
flag numbers $1, \ldots, m$ to the flag boxes in these rows.  For a
tuple of augmented ribbons, these flag boxes are the augmentation
boxes.

\begin{thm}\label{thm:ns-HHL}
The integral form nonsymmetric Macdonald polynomials are given in
terms of flagged LLT polynomials by
\begin{equation}\label{e:ns-HHL}
\Ecal _{(\mu _{1},\ldots,\mu _{m})}(x;q,t) = t^{n(\mu _{+})} \sum
_{(\nubold ,\sigma )\in R(\mu )} \bigl( \prod \nolimits_{\domino }
q^{a(u)+1}\, t^{l(u)} \bigr)\, \Gcal _{\nubold ,\sigma
}^{-}(x_{1},\ldots,x_{m},0,\ldots,0;\, t^{-1})\, ,
\end{equation}
where $R(\mu )$ contains one pair $(\nubold ,\sigma )$ for each tuple
$\nubold $ of augmented ribbons of lengths $|\nu ^{(i)}| = (\mu
^{\op})_{i} = \mu _{m+1-i}$, and $\sigma $ assigns flag numbers
$m,\ldots,2,1$ to the augmentation boxes.  The product of arm and leg
factors $q^{a(u)+1}\, t^{l(u)}$ in each term is over the north boxes
$u$ in all vertical dominoes in $\nubold _{\baug }$.
\end{thm}

\begin{proof}
Let $\nubold $ be a tuple of augmented ribbons of sizes $\mu ^{\op }$.
By Proposition~\ref{prop:non-attacking}, $\Gcal _{\nubold ,\sigma
}^{-}(x_{1},\ldots,x_{l};t^{-1})$ is the sum of $(-t)^{m(T)}\,
t^{-\inv (T)}\, x^{|T|}$ over non-attacking flagged tableaux $T \in
\FSST (\nubold ,\sigma ,\Acal ^{b}_{\pm })$, where $\Acal ^{b}_{\pm }$
is defined by \eqref{eq:A plus minus alphabets} with flag bounds
$b_{i} = i$.  Setting $x_{i} = 0$ for $i>m$ gives the sum of the terms
for tableaux $T$ with entries in $[m]_{\pm } =
\{1,\overline{1},\ldots,m,\overline{m} \}$.

If $z\in \nubold ' \setminus \nubold _{\baug }$ is a flag box which
is not an augmentation box, then no box overhangs $z$, and the flag
bound in $z$ is greater than $m$.  These `irrelevant' flag boxes
impose no constraint on tableaux $T$ with entries in $[m]_{\pm }$.
They are also irrelevant in that no attacking pair $(u,z)$ with $z$ as
the second box can be an attacking inversion, or violate the
non-attacking condition.  In particular, $\Gcal _{\nubold ,\sigma
}^{-} (x_{1}, \ldots, x_{m}, 0,\ldots, 0;t^{-1})$ is independent of
the specific choice of $\sigma $ assigning flag numbers $m,\ldots,1$
to the augmentation boxes.

Given $T\colon \nubold \rightarrow [m]_{\pm }$, let $\Phi (T)$ be the
signed filling $S\colon \dg (\mu )\rightarrow [m]_{\pm }$ such that
$S(\uhat ) = T(u)$ for each box $u\in \nubold $, where $\uhat $ is the
box of the same content as $u$ in the row of $\dg (\mu )$
corresponding to the component of $\nubold $ containing $u$.  We also
have $S'(\uhat ) = T'(u)$ for the fixed flag bound entries
$m,\ldots,1$ in the augmentation boxes $u\in \nubold _{\baug }$ and
corresponding flag boxes $\uhat \in \dg (\mu )'$.  Setting aside the
non-attacking condition for the moment, the only constraints for a
filling $T\colon \nubold \rightarrow [m]_{\pm }$ to be a flagged
tableau are that for every domino $\{u, v\}$ in $\nubold _{\baug }$,
with $v$ the box of larger content, we have $T'(u) \sgnle T'(v)$ if
$\{u, v\}$ is a horizontal domino, and $T'(u) \sgngt T'(v)$ if it is
vertical.  It follows that every signed filling $S\colon \dg (\mu
)\rightarrow [m]_{\pm }$ is $\Phi (T)$ for a unique flagged tableau
$T\colon \nubold \rightarrow [m]_{\pm }$ on a unique augmented ribbon
tuple $\nubold $: namely, the ribbon tuple such that the domino $\{u,v
\}$ corresponding to each pair $\{\uhat , e(\uhat )\}$ in $\dg (\mu
)'$ is horizontal if $S(\uhat )\sgnle S'(e(\uhat ))$, and vertical if
$S(\uhat )\sgngt S'(e(\uhat ))$.

From the definitions, attacking pairs $(u,v)$ for $\nubold $ such that
$v \in \nubold _{\baug }$ correspond to attacking pairs $(\uhat ,\vhat
)$ for $\dg (\mu )$.  It follows that $T$ is non-attacking if and only
if $S = \Phi (T)$ is non-attacking, and that $\inv (T) = \inv (S)$.
Clearly, $x^{|T|} = x^{|S|}$.  By definition, $a(u) = a(\uhat )$ and
$l(u) = l(\uhat )$ for all $u$ in $\nubold $.  Since $u\in \nubold $
is the north box in a vertical domino if and only if $S(\uhat ) \sgngt
S'(e(\uhat ))$, the product $\prod _{u\in \dg (\mu ): S(u) \sgngt
S'(e(u))} q^{a(u)+1}\, t^{l(u)}$ in \eqref{e:signed-HHL} is equal to
the product $\prod _{\smalldomino } q^{a(u)+1}\,
t^{l(u)}$ in \eqref{e:ns-HHL} for $\nubold $.  Hence, after expanding
each $\Gcal _{\nubold ,\sigma }^{-} (x_{1}, \ldots, x_{m}, 0,\ldots,
0;t^{-1})$ in \eqref{e:ns-HHL} as a sum over non-attacking flagged
tableaux, the right hand sides of \eqref{e:signed-HHL} and
\eqref{e:ns-HHL} agree term by term.
\end{proof}

\begin{example}\label{ex:ns-HHL}
For $\mu =(2,0,1)$, there are eight augmented ribbon tuples, shown
here with flag bounds $3,2,1$ written in the augmentation boxes, and
the monomial $t^{n(\mu _{+})} \prod _{\smalldomino }
q^{a(u)+1}\, t^{l(u)}$ to the left of each tuple.
\begin{equation}\label{e:ns-HHL-ribbon-tuples}
\setlength{\arraycolsep}{0cm}
\begin{array}{rcrcrcrcrcrcrcrc}
t \! &
\begin{tikzpicture}[scale=.35,baseline=.85cm]
  \draw[yshift=4.25cm] (-2,0) grid (0,1); \node at (.5,4.75) {1};
  \draw[yshift=2.5cm] (0,0) grid (0,1); \node at (.5,3) {2};
  \draw[yshift=.5cm] (-1,0) grid (0,1); \node at (.5,1) {3};
\end{tikzpicture}
& \quad \;  q\, t \!\!\!\! &
\begin{tikzpicture}[scale=.35,baseline=.85cm]
  \draw[yshift=4.25cm] (-2,0) grid (0,1); \node at (.5,4.75) {1};
  \draw[yshift=2.5cm] (0,0) grid (0,1); \node at (.5,3) {2};
  \draw[yshift=.75cm] (0,0) grid (1,1); \node at (.5,.25) {3};
\end{tikzpicture}
& \quad \; q\, t^{2} &
\begin{tikzpicture}[scale=.35,baseline=.85cm]
  \draw[yshift=4.25cm] (-1,0) grid (0,2); \node at (.5,4.75) {1};
  \draw[yshift=2.5cm] (0,0) grid (0,1); \node at (.5,3) {2};
  \draw[yshift=.5cm] (-1,0) grid (0,1); \node at (.5,1) {3};
\end{tikzpicture}
& \quad \;  q^{2} t^{2} \! &
\begin{tikzpicture}[scale=.35,baseline=.85cm]
  \draw[yshift=4.25cm] (-1,0) grid (0,2); \node at (.5,4.75) {1};
  \draw[yshift=2.5cm] (0,0) grid (0,1); \node at (.5,3) {2};
  \draw[yshift=.75cm] (0,0) grid (1,1); \node at (.5,.25) {3};
\end{tikzpicture}
& \quad \; q^{2} t^{3} &
\begin{tikzpicture}[scale=.35,baseline=.85cm]
  \draw[yshift=4.25cm] (-1,1) grid (1,2); \node at (.5,4.75) {1};
  \draw[yshift=2.5cm] (0,0) grid (0,1); \node at (.5,3) {2};
  \draw[yshift=.5cm] (-1,0) grid (0,1); \node at (.5,1) {3};
\end{tikzpicture}
& \quad \; q^{3} t^{3} \! &
\begin{tikzpicture}[scale=.35,baseline=.85cm]
  \draw[yshift=4.25cm] (-1,1) grid (1,2); \node at (.5,4.75) {1};
  \draw[yshift=2.5cm] (0,0) grid (0,1); \node at (.5,3) {2};
  \draw[yshift=.75cm] (0,0) grid (1,1); \node at (.5,.25) {3};
\end{tikzpicture}
& \quad \;  q^{3} t^{4} \!&
\begin{tikzpicture}[scale=.35,baseline=.85cm]
  \draw[yshift=4.25cm] (0,1) grid (1,3); \node at (.5,4.75) {1};
  \draw[yshift=2.5cm] (0,0) grid (0,1); \node at (.5,3) {2};
  \draw[yshift=.5cm] (-1,0) grid (0,1); \node at (.5,1) {3};
\end{tikzpicture}
& \quad \; q^{4} t^{4} \;\; &
\begin{tikzpicture}[scale=.35,baseline=.85cm]
  \draw[yshift=4.25cm] (0,1) grid (1,3); \node at (.5,4.75) {1};
  \draw[yshift=2.5cm] (0,0) grid (0,1); \node at (.5,3) {2};
  \draw[yshift=.75cm] (0,0) grid (1,1); \node at (.5,.25) {3};
\end{tikzpicture}
\\
&\nubold _{1} && \nubold _{2} && \nubold _{3} && \nubold _{4} &&
\nubold _{5} && \nubold _{6} && \nubold _{7} && \nubold _{8}
\end{array}
\end{equation}
Formula \eqref{e:ns-HHL} now reads
\begin{multline}\label{e:ns-HHL-written-out}
\Ecal _{(2,0,1)}(x;q,t) = t\, \Gcal _{\nubold _{1}, \sigma
_{1}}^{-}(x_{1},x_{2},x_{3}; t^{-1}) +
q\, t\, \Gcal _{\nubold _{2}, \sigma _{2}}^{-}(x_{1},x_{2},x_{3},0; t^{-1})\\
+ \cdots + q^{4}t^{4}\, \Gcal _{\nubold _{8}, \sigma
_{8}}^{-}(x_{1},x_{2},x_{3},0,0,0; t^{-1})\, .
\end{multline}
Picking out the term for $\nubold _{5}$, the non-attacking flagged
tableaux $T\colon \nubold _{5}\rightarrow [3]_{\pm }$ give
\begin{equation}\label{e:GR5}
\begin{array}{rcccccc}
&&
\begin{tikzpicture}[scale=.45,baseline=0cm]
  \foreach \beta / \alpha / \y / \z in {0/-1/1/0, 0/0/1/1.5, 1/-1/1/4}
     \draw[yshift=\z cm] (\alpha,0)  grid (\beta, \y);
   \node at (.5,.5) {3}; \node at (.5,2) {2}; \node at (.5,3.5) {1};
   \node at (.5,4.5) {\mybar{1}}; \node at (-.5,4.5) {1};
   \node at (-.5,.5) {2};
\end{tikzpicture}
&&
\begin{tikzpicture}[scale=.45,baseline=0cm]
  \foreach \beta / \alpha / \y / \z in {0/-1/1/0, 0/0/1/1.5, 1/-1/1/4}
     \draw[yshift=\z cm] (\alpha,0)  grid (\beta, \y);
   \node at (.5,.5) {3}; \node at (.5,2) {2}; \node at (.5,3.5) {1};
   \node at (.5,4.5) {\mybar{1}}; \node at (-.5,4.5) {1};
   \node at (-.5,.5) {\mybar{2}};
\end{tikzpicture}
&&
\begin{tikzpicture}[scale=.45,baseline=0cm]
  \foreach \beta / \alpha / \y / \z in {0/-1/1/0, 0/0/1/1.5, 1/-1/1/4}
     \draw[yshift=\z cm] (\alpha,0)  grid (\beta, \y);
   \node at (.5,.5) {3}; \node at (.5,2) {2}; \node at (.5,3.5) {1};
   \node at (.5,4.5) {\mybar{1}}; \node at (-.5,4.5) {1};
   \node at (-.5,.5) {3};
\end{tikzpicture}
\\[1ex]
\Gcal _{\nubold_{5}, \sigma _{5}}^{-}(x_{1},x_{2},x_{3},0;t^{-1}) & =
& (-t)\, t^{-1} x_{1}^{2}x_{2} & + & (-t)^{2}\, t^{-1}\,
x_{1}^{2}x_{2} & + & (-t)\, t^{-1} x_{1}^{2}x_{3}\, .
\end{array}
\end{equation}
The map $\Phi $ in the proof of Theorem~\ref{thm:ns-HHL} sends these
flagged tableaux to the non-attacking signed fillings $S$ of $\dg (\mu
)$ shown here.
\begin{equation}\label{e:Phi-map}
\begin{array}{c@{\qquad }c@{\qquad }c}
\begin{tikzpicture}[scale=.45,baseline=0cm]
  \foreach \beta / \alpha / \y / \z in {0/-1/1/0, 0/0/1/1.5, 0/-2/1/3}
     \draw[yshift=\z cm] (\alpha,0)  grid (\beta, \y);
   \node at (.5,.5) {3}; \node at (.5,2) {2}; \node at (.5,3.5) {1};
   \node at (-.5,3.5) {\mybar{1}}; \node at (-1.5,3.5) {1};
   \node at (-.5,.5) {2};
\end{tikzpicture}
&
\begin{tikzpicture}[scale=.45,baseline=0cm]
  \foreach \beta / \alpha / \y / \z in {0/-1/1/0, 0/0/1/1.5, 0/-2/1/3}
     \draw[yshift=\z cm] (\alpha,0)  grid (\beta, \y);
   \node at (.5,.5) {3}; \node at (.5,2) {2}; \node at (.5,3.5) {1};
   \node at (-.5,3.5) {\mybar{1}}; \node at (-1.5,3.5) {1};
   \node at (-.5,.5) {\mybar{2}};
\end{tikzpicture}
&
\begin{tikzpicture}[scale=.45,baseline=0cm]
  \foreach \beta / \alpha / \y / \z in {0/-1/1/0, 0/0/1/1.5, 0/-2/1/3}
     \draw[yshift=\z cm] (\alpha,0)  grid (\beta, \y);
   \node at (.5,.5) {3}; \node at (.5,2) {2}; \node at (.5,3.5) {1};
   \node at (-.5,3.5) {\mybar{1}}; \node at (-1.5,3.5) {1};
   \node at (-.5,.5) {3};
\end{tikzpicture}
\end{array}
\end{equation}
In Example~\ref{ex:original-HHL}, the four terms in
\eqref{e:original-HHL-term-by-term} correspond to unsigned fillings
$U$ of $\dg (\mu )$.  Multiplying out the products in
\eqref{e:original-HHL-term-by-term} gives eight terms for each $U$,
corresponding as in Corollary~\ref{cor:signed-HHL} to signed fillings
$S$ such that $|S| = U$.  The first two signed fillings in
\eqref{e:Phi-map} correspond in this way to $(-q^{2}t^{3})(1- t)\,
x_{1}^{2}x_{2}$ from the first term in
\eqref{e:original-HHL-term-by-term}, and the third to $-q^{2}t^{3}\,
x_{1}^{2}x_{3}$ from the second term in
\eqref{e:original-HHL-term-by-term}.  As the reader can verify, we get
the same terms by multiplying the expression for $\Gcal _{\nubold
_{5},\sigma _{5}}^{-}(x_{1},x_{2},x_{3},0;t^{-1})$ in \eqref{e:GR5} by
the corresponding factor $q^{2}t^{3}$ from
\eqref{e:ns-HHL-ribbon-tuples}.
\end{example}

\begin{example}\label{ex:E-F}
Knop \cite[Corollary~5.7]{Knop97} and Ion \cite[Theorem~4.8]{Ion08},
expressed in terms of integral form nonsymmetric Macdonald polynomials
$\Ecal _{\mu }(x;q,t)$ and the nonsymmetric Hall-Littlewood
polynomials in \S \ref{ss:ns-HL-pols}, give the specializations
\begin{gather}\label{e:ns-Mac-to-ns-E}
\Ecal _{\mu }(x;\, 0,t) = E_{\mu }(x;t^{-1})\\
\label{e:ns-Mac-to-ns-F}
\bigl(q^{n((\mu _{+})^{*})+|\mu |}\, \Ecal _{\mu }(x;q^{-1},t)
\bigr)\bigr|_{q=0} = (-t)^{|\mu |} \, t^{n(\mu _{+}) + \inv (\mu )}
F_{\mu }(x;t^{-1})\, ,
\end{gather}
where in \eqref{e:ns-Mac-to-ns-F} we used $\sum _{u\in \dg (\mu )}
a(u) = n((\mu _{+})^{*})$ and $\sum _{u\in \dg (u)} l(u) = n(\mu _{+})
+ \inv (\mu )$.

Combining Theorem~\ref{thm:ns-HHL} with \eqref{e:ns-Mac-to-ns-E}
yields
\begin{equation}\label{e:ns-E=signed-LLT}
E_{\mu}(x;t^{-1}) = t^{n(\mu _{+})} \Gcal_{ \nubold,
\operatorname{id} }^-(x_1,\dots, x_m; t^{-1}),
\end{equation}
for $\mu \in \NN ^{m}$, where $\nubold = \dg (\mu )$ is a tuple of
single-row augmented ribbons of sizes $\mu ^{\op }$.  One can check
that $h(\nubold ) = n(\mu _{+})$, so this is equivalent to the formula
in Example~\ref{ex:ns-HHL=signed-LLT}.

Combining Theorem~\ref{thm:ns-HHL} with \eqref{e:ns-Mac-to-ns-F}
yields
\begin{equation}\label{e:ns-F=signed-LLT}
 F_{\mu}(x;t^{-1}) =(-t)^{-|\mu |} \, t^{n(\mu _{+})} \Gcal_{ \nubold,
\operatorname{id} }^-(x_1,\dots, x_m,0,\ldots,0 ; t^{-1}),
\end{equation}
where $\nubold $ is a tuple of single-column augmented ribbons of
sizes $\mu ^{\op }$.

\end{example}

\subsection{Hecke twists of integral forms}
\label{ss:more-about-Ecal}

Theorem~\ref{thm:ns-HHL} can be combined with results on flagged LLT
polynomials to deduce additional properties of the integral forms
$\Ecal _{\mu }(x;q,t)$.  In particular, we can generalize formulas
\eqref{e:original-HHL} and \eqref{e:ns-HHL} to $T_{w} \, \Ecal _{\mu
}(x;q,t)$.

\begin{cor}\label{cor:ns-HHL-twisted}
Let $\mu \in \NN ^{m}$ and $w\in \Sfrak _{m}$ be given.

(a) Modify notation in Proposition~\ref{prop:original-HHL} by taking
the extended filling $U'$ for $U\colon \dg (\mu )\rightarrow [m]$ to
have $w(m),\ldots,w(1)$ instead of $m,\ldots,1$ in the flag boxes, and
defining $\maj (U)$ and $\coinv (U)$ as in
\eqref{e:maj}--\eqref{e:coinv} using $\inv (U)$ and $\Des (U)$ for the
modified $U'$.  Then
\begin{equation}\label{e:original-HHL-twisted}
T_{w}\, \Ecal _{\mu }(x;q,t) = t^{\ell (w)}\hspace{-2ex}
\sum _{\substack{U \colon \dg (\mu )\rightarrow [m]\\
\text{non-attacking}}} \hspace{-1.5ex}  q^{\maj (U)}\, t^{\coinv (U)}
\hspace{-3ex} \prod _{\substack{u\in \dg (\mu )\\ U'(u) = U'(e(u))}}
\hspace{-3ex} (1 - q^{a(u)+1} t^{l(u)+1}) \hspace{-3ex}
\prod _{\substack{u\in \dg (\mu )\\
U'(u) \not = U'(e(u))}} \hspace{-3ex} (1-t)  \; x^{U}\, .
\end{equation}

(b) Define $R_{w}(\mu )$ similarly to $R(\mu )$ in Theorem~\ref{thm:ns-HHL}, but
choose each pair $(\nubold, \sigma )$ so that the flag numbers in
the augmentation boxes of $\nubold $ are $w(m),\ldots,w(1)$ instead of
$m,\ldots,1$.  Then
\begin{equation}\label{e:ns-HHL-twisted}
T_{w}\, \Ecal _{\mu }(x;q,t)\\
 = t^{\ell (w)+n(\mu _{+})} \hspace{-1ex} \sum _{(\nubold, \sigma )\in
R_{w}(\mu )} \bigl( \prod \nolimits_{\domino } q^{a(u)+1}\, t^{l(u)}
\bigr)\, \Gcal _{\nubold, \sigma
}^{-}(x_{1},\ldots,x_{m},0,\ldots,0;\, t^{-1})\, .
\end{equation}
\end{cor}

\begin{proof}
Since $T_{w}$ commutes with setting $x_{i} = 0$ for $i>m$, part (b)
follows from Theorem~\ref{thm:ns-HHL} and
Proposition~\ref{prop:Ti-action}.  For part (a), we need to see that
the right hand sides of \eqref{e:original-HHL-twisted} and
\eqref{e:ns-HHL-twisted} become equal term by term after multiplying
out the products in \eqref{e:original-HHL-twisted}, and expressing
each $\Gcal _{\nubold ,\sigma }^{-}(x_{1},\ldots,x_{m},0,\ldots,0;\,
t^{-1})$ in \eqref{e:ns-HHL-twisted} as a sum over non-attacking
flagged tableaux.  When $w$ is the identity, this was shown in the
proofs of Corollary~\ref{cor:signed-HHL} and Theorem~\ref{thm:ns-HHL}.
The argument goes through essentially unaltered with flag bounds
$w(m),\ldots,w(1)$ instead of $m,\ldots,1$ in the flag boxes of $\dg
(\mu )$ and the augmentation boxes of $\nubold $.
\end{proof}

\begin{remark}\label{rem:coinv}
In formulas \eqref{e:original-HHL}, \eqref{e:signed-HHL},
\eqref{e:ns-HHL}, and
\eqref{e:original-HHL-twisted}--\eqref{e:ns-HHL-twisted}, each term in
the sum on the right hand side clearly has coefficients in $\ZZ
[q,t^{\pm 1}]$, but it is not obvious {\it a priori} that the
individual terms have coefficients $\ZZ [q,t]$, even though the sum as
a whole does.

The next lemma implies that every term on the right hand side of
\eqref{e:ns-HHL} and \eqref{e:ns-HHL-twisted} has coefficients in $\ZZ
[q,t]$, and more generally that in these formulas, the full flagged
LLT term
\begin{equation}\label{e:ns-HHL-twisted-full-term}
 t^{\ell (w)+n(\mu _{+})}  \bigl( \prod \nolimits_{\domino }
q^{a(u)+1}\, t^{l(u)} \bigr)\, \Gcal _{\nubold, \sigma
}(X_{1},\ldots,X_{l};\, t^{-1})
\end{equation}
has coefficients in $\ZZ
[q,t]$ even before specializing to $\Gcal ^{-}_{\nubold, \sigma
}(x_{1}, \ldots, x_{m}, 0, \ldots, 0;\, t^{-1})$

Since the proofs of Theorem~\ref{thm:ns-HHL} and
Corollaries~\ref{cor:signed-HHL} and \ref{cor:ns-HHL-twisted} equate
individual terms, it follows that every term on the right hand sides
of \eqref{e:original-HHL}, \eqref{e:signed-HHL}, and
\eqref{e:original-HHL-twisted} also has coefficients in $\ZZ [q,t]$.
In particular, this implies that the modified definition of $\coinv
(U)$ involving $w$ in \eqref{e:original-HHL-twisted} satisfies $\coinv
(U) + l(w)\geq 0$.  When $w$ is the identity, this was also shown in
\cite[Proposition~3.6.2]{HagHaiLo08}.
\end{remark}

\begin{lemma}\label{lem:inv-bound}
Let $\nubold = (\nu ^{(1)}, \ldots, \nu ^{(m)}) $ be a tuple of
augmented ribbons of lengths $|\nu ^{(i)}| = (\mu ^{\op })_{i}$, and
let $\sigma $ be a compatible permutation such that the flag numbers
in the augmentation boxes are $w(m),\ldots,w(1)$, where $w\in \Sfrak
_{m}$.  Then, for any signed alphabet $\Acal $, every tableau $T\in
\FSST (\nubold ,\sigma ,\Acal )$ satisfies $\inv (T)\leq n(\mu _{+}) +
\ell (w) + \sum _{\smalldomino } l(u)$, where the sum
is over the north boxes $u$ in all vertical dominoes in $\nubold
_{\baug }$.
\end{lemma}

\begin{proof}
Let $I$ be the set of all pairs $(x,y)$ in $\nubold _{\baug }$ with
$c(y) = c(x)$ and $i<j$, where $x\in \nu _{\, \aug } ^{(i)}$, $y\in
\nu _{\, \aug } ^{(j)}$.  The number of such pairs with content $c(y)
= c(x) <0$, i.e., with $x,y\in \nubold $, is $n(\mu _{+})$.  In the
pairs with content $c(y) = c(x) =0$, $x$ and $y$ are augmentation
boxes.  Let $T'$ be the extended tableau in \eqref{e:T-extended},
where (as is always possible) we take the letters $b_{j}$ to be
strictly increasing.  Then the number of pairs $(x,y)\in I$ with
content $c(x) = c(y) = 0$ and $T'(x)\sgnle T'(y)$ is $l(w)$.

By Remark~\ref{rem:no-funny-inv}(ii), if $(x,y)$ is an attacking
inversion with $c(y) = c(x)$, then $y$ is not a flag box.  Hence the
attacking inversions with $c(y) = c(x)$ for $T$ are precisely the
pairs $(x,y)\in I$ such that $x,y\in \nubold $ and $T(x)\sgngt T(y)$.
Let $\inv _{1}(T)$ be the number of these pairs, and let $g = |H|$,
where $H = \{(x,y) \in I\mid T'(x) \sgnle T'(y) \}$.  By the preceding
considerations, $\inv _{1}(T) + g = n(\mu _{+}) + l(w)$.  It therefore
remains to show that $\inv _{2}(T) \leq g + \sum _{\smalldomino }
l(u)$, where $\inv _{2}(T)$ is the number of attacking inversions
$(x,y)$ for $T$ with $c(y) = c(x)+1$ and $i>j$, where $x\in \nu
^{(i)}$ and $y\in \nu '\, ^{(j)}$.  Let $J$ be the set of attacking
inversions $(x,y)$ of this latter type.

Another way to describe the leg of a box $u \in \nubold $, as in
\cite[\S 5]{Knop97} or \cite[\S 4]{KnopSahi97}, is that $l(u)$ is the
number of boxes $v \in \nubold _{\baug }$ such that we either have
$u\in \nu ^{(j)}$, $v\in \nu _{\, \aug }^{(i)}$, $|\nu ^{(j)}| \geq
|\nu ^{(i)}|$, and $c(v) = c(u)$, or we have $u\in \nu ^{(i)}$, $v\in
\nu _{\, \aug }^{(j)}$, $|\nu ^{(j)}| < |\nu ^{(i)}|$, and $c(v) =
c(u)+1$, where $i>j$ in either case.  Let $K$ be the set of such pairs
$(u,v)$, where $u$ is the north box in a vertical domino in $\nubold
_{\baug }$, so that $|K| = \sum _{\smalldomino } l(u)$.  We will show
that $|J|\leq |H|+|K|$, as desired, by constructing an injective map
$(x,y)\mapsto (x',y')\colon J\hookrightarrow H\coprod K $.

Given $(x,y)\in J$ with $x\in \nu ^{(i)}$ and $y\in \nu '\, ^{(j)}$,
where $i>j$, we consider four cases.  Suppose first that $|\nu ^{(j)}|
\geq |\nu ^{(i)}|$.  Then there is a unique box $u\in \nu ^{(j)}$ with
content $c(u) = c(x)$, which forms a domino with $y$.  Case (1a) is
that $\{u ,y \}$ is a horizontal domino.  Then $(u, x)\in H$, since
$T(u) \sgnle T'(y) \sgnlt T(x)$.  Case (1b) is that $\{u, y \}$ is a
vertical domino.  Then $(u, x)\in K$, since $|\nu ^{(j)}| \geq |\nu
^{(i)}|$.  In either case, we set $(x',y') = (u, x)$, considered as an
element of $H$ or $K$, respectively.

Otherwise, $|\nu ^{(j)}| < |\nu ^{(i)}|$.  Since $c(x)<0$, there is a
unique box $z\in \nu _{\, \aug }^{(i)}$ with content $c(z) = c(y)$,
which forms a domino with $x$.  Case (2a) is that $\{x ,z \}$ is a
horizontal domino.  Then $(y,z)\in H$, since $T'(y) \sgnlt T(x) \sgnle
T(z)$.  In this case we set $(x',y') = (y,z)$.  Case (2b) is that
$\{x,z \}$ is a vertical domino.  Then $(x, y)\in K$, since $|\nu
^{(j)}| < |\nu ^{(i)}|$, and we set $(x',y') =(x,y)$, considered as an
element of $K$.

Since we always map $(x,y)$ to a pair $(x',y')$ in the same two
components of $\nubold _{\baug }$ as $(x,y)$ itself, $(x',y')$
determines $i$ and $j$.  Then $(x,y)$ belongs to Case (1a) or (1b) if
$|\nu ^{(j)}| \geq |\nu ^{(i)}|$, and Case (2a) or (2b) if $|\nu
^{(j)}| < |\nu ^{(i)}|$.  By definition, we record whether $(x',y')$
is in $H$ or $K$, so we can further distinguish Case (1a) from (1b)
and Case (2a) from (2b).  Once we know which case $(x,y)$ is in, it is
clear that we can recover $(x,y)$ from $(x',y')$.  Hence the
map we have constructed is injective.
\end{proof}

\subsection{Stabilized nonsymmetric Macdonald polynomials}
\label{ss:stabilization}

There are two natural ways to construct a polynomial symmetric in some
block of variables from a nonsymmetric Macdonald polynomial.  The
first---{\em specialization}---uses the fact that if $\mu _{i} = \mu
_{i+1}$, then $\Ecal _{\mu }(x;q,t)$ is symmetric in $x_{i}, x_{i+1}$
(\S \ref{ss:ns-Macs}). Hence, given $\eta \in \NN ^{r}$, $\lambda \in
\NN ^{k}$, and $\kappa \in \NN ^{s}$, the polynomial
\begin{equation}\label{e:specialized-ns-mac}
\Ecal _{(\eta ;\, 0^{n};\, \lambda ;\, \kappa )}(x_{1}, \ldots, x_{r},
y_{1}, \ldots, y_{n}, 0^{k}, z_{1}, \ldots, z_{s};\, q, t)
\end{equation}
is symmetric in new variables $y_{1},\ldots,y_{n}$ which replace the
variables $x_{i}$ indexed by the positions of $\lambda
_{1},\ldots,\lambda _{k}$.  By specializing the integral form $\Ecal
_{\mu }(x;q,t)$ rather than $E_{\mu }(x;q,t)$, we get a polynomial
with coefficients in $\ZZ [q,t]$.  We insert $0^{n}$ before $\lambda $
in $(\eta ;\, 0^{n};\, \lambda ;\, \kappa )$, and not after, because
$\Ecal _{(\eta ;\, \lambda ;\, 0^{n} ;\, \kappa
)}(x_{1},\ldots,x_{r},0^{k},y_{1},\ldots,y_{n},z_{1},\ldots,z_{s};\,
q,t) = 0$ if $\eta _{1},\ldots,\eta _{r},\lambda _{1}>0$.

The second construction---{\em symmetrization}---uses the symmetric
linear combination of nonsymmetric Macdonald polynomials in
\eqref{e:symmetric-combo}--\eqref{e:monic-symmetric-combo} for the
Young subgroup $\Sfrak _{\rr } = \Sfrak _{[i,j]}$ which permutes a
block of indices.  In particular, given $\eta $, $\lambda $, $\kappa $
as above, for any $n\geq k$ there exists a polynomial of the form
\begin{equation}\label{e:symmetrized-ns-mac}
\sum _{v} c_{v}\, E_{(\eta ;\, v(0^{n-k}; \, \lambda ); \, \kappa
)}(x_{1}, \ldots, x_{r}, y_{1}, \ldots, y_{n}, z_{1}, \ldots, z_{s};
q, t)\, ,
\end{equation}
symmetric in $y_{1},\ldots,y_{n}$, where the sum is over permutations
$v(0^{n-k};\, \lambda )$ of $(0^{n-k};\, \lambda )$.  Such a
polynomial is unique up to a scalar factor; its monic normalization,
in which the leading term $x^{\eta }\, m_{\lambda _{+} }(y) z^{\kappa
}$ has coefficient $1$, is given more explicitly by
\begin{equation}\label{e:monic-symmetrized-ns-mac}
\bigl(\sum \nolimits _{w\in \Sfrak _{[r+1,r+n]}/ \Sfrak _{\sS }}
T_{w}\, E_{(\eta ;\, \lambda _{+}; \, 0^{n-k}; \, \kappa )}(x;q,t)
\bigr) \big|_{x\mapsto(x_{1}, \ldots, x_{r}, y_{1}, \ldots, y_{n},
z_{1}, \ldots, z_{s})}\, ,
\end{equation}
where $\Sfrak _{[r+1,r+n]}/ \Sfrak _{\sS }$ is the set of minimal
coset representatives for the stabilizer $\Sfrak _{\sS } = \Stab
_{\Sfrak _{[r+1,r+n]}}((\eta ;\, \lambda _{+}; \, 0^{n-k}; \, \kappa
))$.  We will show that the specialization
\eqref{e:specialized-ns-mac} and the symmetrization
\eqref{e:monic-symmetrized-ns-mac} each converge in a suitable sense
to a symmetric function in $Y = y_{1},y_{2},\cdots $ as $n\rightarrow
\infty $, and that, up to a scalar factor, the limit is the same for
either construction.

Fixing nonsymmetric variables $x_{1},\ldots,x_{r}$ and taking the
symmetric limit in the rest---i.e., taking $\kappa $ empty in
\eqref{e:specialized-ns-mac}--\eqref{e:monic-symmetrized-ns-mac}---yields
a basis of {\em right stabilized} Macdonald polynomials, indexed by
pairs $\eta ,\lambda $, where $\eta \in \NN ^{r}$ and $\lambda $ is a
partition, for the algebra $\kk [x_{1},\ldots,x_{r}]\otimes \Lambda
(Y)$ of {\em $r$-nonsymmetric polynomials}.

Alternatively, taking $\eta $ empty yields a basis for $\kk
[z_{1},\ldots,z_{s}]\otimes \Lambda (Y)$ consisting of {\em left
stabilized} Macdonald polynomials.  Although left stabilization is
simpler in some ways, right stabilization is more general, because the
rotation identity \eqref{e:rotation} can be used to transform the left
stabilized basis into a subset of the right stabilized basis.  Our
results and conjectures in \S \S \ref{ss:right stable
unmod}--\ref{ss:Weyl-and-Hecke-symmetrization} cover the more general
right stabilized case.

Both flavors of stabilized Macdonald polynomials have appeared before
in the literature, and parts of the unified treatment we will present
here are already known.  Remark~\ref{rem:previous-stable-Macs}, below,
summarizes connections with previous results.

We now develop properties of the above constructions in detail.  We
begin by using a variant of the formula in Theorem~\ref{thm:ns-HHL} to
define the functions that will turn out to be the limit of the
specialized integral forms in \eqref{e:specialized-ns-mac} as
$n\rightarrow \infty$. To this end, generalizing the notation in
\eqref{e:G-minus}, given flagged LLT data $\nubold , \sigma $ of
length $l$, define
\begin{multline}\label{e:G-minus-Y}
\Gcal _{\nubold ,\sigma }^{-}(x_{1}, \ldots, x_{r}, Y, z_{1}, \ldots,
z_{l-r};\, t^{-1})\\
 = \Gcal _{\nubold ,\sigma }[x_{1},\, x_{2} - t\, x_{1}, \ldots, x_{r}
- t\, x_{r-1},\\
z_{1} + (1-t)Y -t\, x_{r},\, z_{2} -t\, z_{1}, \ldots, z_{l-r} - t\,
z_{l-r-1};\, t^{-1}],
\end{multline}
a polynomial in $x_{1},\ldots,x_{r},z_{1},\ldots,z_{l-r}$ and
symmetric function in $Y$.  Evaluated in a finite plethystic alphabet
$Y = y_{1}+\cdots +y_{n}$, this becomes the signed flagged LLT
polynomial in \eqref{e:signed-G}, in variables
$x_{1},\ldots,x_{r},y_{1},\ldots,y_{n},z_{1},\ldots,z_{l-r}$, for flag
bounds $b_{i} = i$ for $i\leq r$, $b_{i} = n+i$ for $i>r$.

\begin{defn}\label{def:symm-int-form}
Given $\eta \in \NN ^{r}$, $\lambda \in \NN ^{k}$ and $\kappa \in \NN
^{s}$, the {\em partially symmetric integral form Macdonald
polynomial} $\Jns _{\eta |\lambda |\kappa
}(x_{1},\ldots,x_{r},Y,z_{1},\ldots,z_{s};\, q,t)$ is given by
\begin{multline}\label{e:J-eta-lambda-kappa}
\Jns _{\eta |\lambda |\kappa }(x,Y,z;\, q,t) = \\
t^{n(\mu _{+})}\hspace{-1ex} \sum _{(\nubold ,\sigma )\in R'(\mu )}
\bigl( \prod \nolimits_{\domino } q^{a(u)+1}\, t^{l(u)} \bigr) \Gcal
_{\nubold ,\sigma }^{-}(x_{1}, \ldots, x_{r}, Y, 0^{k}, z_{1}, \ldots,
z_{s},0,\ldots,0;\, t^{-1}),
\end{multline}
where $\mu =(\eta ;\lambda ;\kappa )\in \NN ^{m}$, with $m=r+k+s$, and
$R'(\mu )$ is the same as $R(\mu )$ in \eqref{e:ns-HHL}, except that
$R'(\mu )$ only contains pairs $(\nubold ,\sigma )$ in which the last
$k+r$ augmented ribbons $\nu ^{(m-r-k+1)},\ldots,\nu ^{(m)}$, of sizes
$(\lambda ^{\op}; \eta ^{\op})$, are non-ragged-right, i.e., no
augmentation box in these ribbons is in a vertical domino.
\end{defn}

By Remark~\ref{rem:coinv}, $\Jns _{\eta |\lambda |\kappa }(x,Y,z;\,
q,t)$ has coefficients in $\ZZ [q,t]$.

\begin{prop}\label{prop:stable-insertion}
The specialized integral form Macdonald polynomials in
\eqref{e:specialized-ns-mac}, symmetric in $y_{1},\ldots,y_{n}$,
converge $t$-adically (as defined in \S \ref{ss:limits}, (ii)) to
\begin{equation}\label{e:stable-insertion}
\lim _{n\rightarrow \infty } \Ecal _{(\eta ;\, 0^{n};\, \lambda; \,
\kappa )}(x_{1},\ldots,x_{r},y_{1},\ldots,y_{n}, 0^{k}, z_{1}, \ldots,
z_{s}; \, q,t) = \Jns _{\eta |\lambda |\kappa }(x,Y,z;\, q,t).
\end{equation}
For $r=0$ and $\eta =\emptyset $, the limit converges strongly: i.e., for
all $n$, we have
\begin{equation}\label{e:strong-stable-insertion}
\Ecal _{(0^{n};\, \lambda ;\, \kappa )}(y_{1},\ldots,y_{n}, 0^{k},
z_{1}, \ldots, z_{s}; \, q,t) = \Jns _{\emptyset |\lambda |\kappa
}[y_{1}+\cdots +y_{n},z;\, q,t].
\end{equation}
\end{prop}

\begin{proof}
By Theorem~\ref{thm:ns-HHL}, the quantity inside the limit in
\eqref{e:stable-insertion} is a sum of terms 
\begin{equation}\label{e:lhs-term}
t^{n((\eta ;\lambda ;\kappa )_{+})} \, \bigl(\prod \nolimits _{\domino
} q^{a(u)+1} t^{l(u)}\bigr) \, \Gcal ^{-}_{\nubold ,\sigma
}(x_{1},\ldots,x_{r}, y_{1},\ldots,y_{n}, 0^{k}, z_{1},\ldots,z_{s},
0,\ldots,0;\, t^{-1})
\end{equation}
for tuples $\nubold $ of augmented ribbons of sizes $(\kappa ^{\op
};\, \lambda ^{\op };\, 0^{n};\, \eta ^{\op })$ with flag numbers
$m,\ldots,1$ in the augmentation boxes, where $m=r+n+k+s$.  For any
such tuple, let $\nuboldhat = (\nu ^{(1)}, \ldots, \nu ^{(k+s)}, \nu
^{(m-r+1)}, \ldots, \nu ^{(m)})$ be the tuple of length $\mhat =r+k+s$
obtained by deleting the ribbons corresponding to the $n$ inserted
zeroes in $(\eta ;0^{n};\lambda ;\kappa )$, and let $\sigmahat $ be a
compatible permutation that assigns flag numbers $\mhat ,\ldots,1$ to
the augmentation boxes of $\nuboldhat $.

To prove \eqref{e:stable-insertion}, we will show that the term in
\eqref{e:lhs-term} vanishes $t$-adically as $n\rightarrow \infty $ if
any the last $k+r$ ribbons in $\nuboldhat $ have their augmentation
box in a vertical domino, and that if these ribbons are
non-ragged-right, then \eqref{e:lhs-term} is equal to the
term
\begin{equation}\label{e:rhs-term}
t^{n((\eta ;\lambda ;\kappa )_{+})} \, \bigl(\prod \nolimits _{\domino
} q^{a(u)+1} t^{l(u)}\bigr) \, \Gcal ^{-}_{\nuboldhat ,\sigmahat
}(x_{1},\ldots,x_{r}, Y, 0^{k}, z_{1},\ldots,z_{s}, 0,\ldots,0;\,
t^{-1})
\end{equation}
for $\nuboldhat $ in \eqref{e:J-eta-lambda-kappa}, evaluated with $Y =
y_{1}+\cdots +y_{n}$.

Setting $x_{r+i} = y_{i}$ for $1 \leq i \leq n$, $x_{r+n+i} = 0$ for
$1 \leq i \leq k$, and $x_{r+n+k+i} = z_{i}$ for $1 \leq i \leq s$,
the flagged LLT polynomial in \eqref {e:lhs-term} is given by
\begin{equation}\label{e:lhs-llt}
\Gcal ^{-}_{\nubold ,\sigma }(x, y, 0^{k}, z,
0,\ldots,0;\, t^{-1}) = \sum _{T} (-t)^{m(T)} t^{-\inv (T)} x^{|T|},
\end{equation}
where the sum is over non-attacking flagged tableaux $T\in \FSST
(\nubold , \sigma ,\Acal _{\pm }^{b})$ for flag bounds $b_{i} = i$,
with entries $a\in [m]_{\pm }$ such that $|a|\not \in [r+n+1, r+n+k]$.
The extension $T'$ of $T$ to $\nubold _{\baug }$ is then a
non-attacking super tableau with flag bounds $m,\ldots,1$ in the
augmentation boxes.  Here we sketch $\nubold _{\baug }$ with its flag
bounds, to aid in following the argument below.
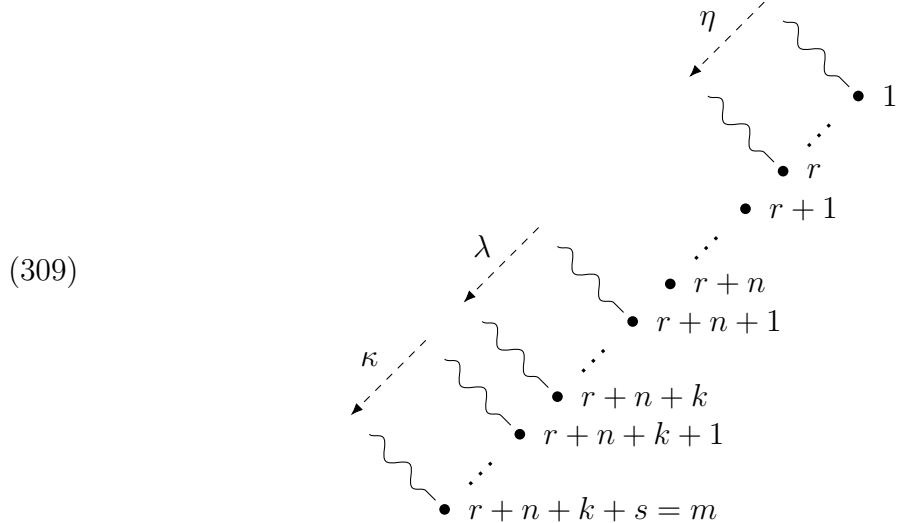
\begin{equation}\label{e:nu-sketch-old}
\begin{array}{c}
\begin{tikzpicture}[scale=.5]
\tikzset{decoration={snake,segment length=5mm, amplitude=1mm}}
\node[circle, fill=black, inner sep=.5mm, outer sep=1mm,
      label=right:{$r+n+k+1$}] (a) at (-1,-1) {};
\node[circle, fill=black, inner sep=.5mm, outer sep=1mm,
      label=right:{$r+n+k+s=m$}] (b) at (-3,-3) {};
\draw[decorate] (-3,1) -- (a);
\draw[decorate] (-5,-1) -- (b);
\draw[loosely dotted, very thick, shorten <=3mm, shorten >=3mm] (b) -- (a);
\draw[-Latex, dashed] (-3.5,1.5) -- (-5.5,-.5);
\node  at (-5,1) {$\kappa $};
\node[circle, fill=black, inner sep=.5mm, outer sep=1mm,
      label=right:{$r+n+1$}] (a) at (2,2) {};
\node[circle, fill=black, inner sep=.5mm, outer sep=1mm,
      label=right:{$r+n+k$}] (b) at (0,0) {};
\draw[decorate] (0,4) -- (a);
\draw[decorate] (-2,2) -- (b);
\draw[loosely dotted, very thick, shorten <=3mm, shorten >=3mm] (b) -- (a);
\draw[-Latex, dashed] (-.5,4.5) -- (-2.5,2.5);
\node  at (-2,4) {$\lambda $};
\node[circle, fill=black, inner sep=.5mm, outer sep=1mm,
      label=right:{$r+1$}] (c) at (5,5) {};
\node[circle, fill=black, inner sep=.5mm, outer sep=1mm,
      label=right:{$r+n$}] (d) at (3,3) {};
\draw[loosely dotted, very thick, shorten <=3mm, shorten >=3mm] (d) -- (c);
\node[circle, fill=black, inner sep=.5mm, outer sep=1mm,
      label=right:{$1$}] (e) at (8,8) {};
\node[circle, fill=black, inner sep=.5mm, outer sep=1mm,
      label=right:{$r$}] (f) at (6,6) {};
\draw[decorate] (6,10) -- (e);
\draw[decorate] (4,8) -- (f);
\draw[loosely dotted, very thick, shorten <=3mm, shorten >=3mm] (f) -- (e);
\draw[-Latex, dashed]  (5.5,10.5) -- (3.5,8.5);
\node  at (4,10) {$\eta $};
\end{tikzpicture}
\end{array}
\end{equation}
If $u$ is a box of content $-1$ in a ribbon $\nubold ^{(m+1-i)}$ for
$i\in [r+n+1, r+n+k]$, corresponding to a part of $\lambda $, then $u$
attacks the augmentation boxes with flag bounds $r+n+k+1,\ldots,m$,
hence $|T(u)|\leq r+n+k$.  Since $T$ has no entries with $|a|\in
[r+n+1, r+n+k]$, this implies $|T(u)|\leq r+n$.  Hence $u$ and the
adjacent augmentation box, which has flag bound $i>r+n$, must form a
horizontal domino.  This shows that \eqref{e:lhs-term} vanishes if any
ribbon corresponding to a part of $\lambda $ has its augmentation box
in a vertical domino.

Let $\uhat $ denote the box of $\nuboldhat _{\baug }$ corresponding to
a box $u\in \nubold _{\baug }$, and define $\That '$ on $\nuboldhat
_{\baug }$ by $\That '(\uhat ) = T'(u)$.  Then $\That '$ is the
extension of a flagged tableau $\That $ on $\nuboldhat $ with flag
bounds $r+n+k+s,\ldots, r+n+1, r,\ldots,1$ in the augmentation boxes,
matching those on the components of $\nubold $ not deleted in
$\nuboldhat $.  Equivalently, $\That \in \FSST (\nuboldhat ,\sigmahat
,\Acal _{\pm }^{\bhat })$, where $\bhat _{i} = i$ for $i\leq r$
and $\bhat _{i} = n+i$ for $i>r$.  For boxes
$u$ of content $-1$ in the last $r$ ribbons of $\nubold $,
corresponding to parts of $\eta $, the non-attacking condition implies
$|T(u)|\leq r$.  Hence $T$ has no attacking inversions involving
augmentation boxes deleted in $\nubold$, so $\inv (T) = \inv (\That
)$.

We have $a(u) = a(\uhat )$ for all $u\in \nubold $, and $l(u) =
l(\uhat )$ unless $u$ is a box of content $-1$ in one of the last $r$
ribbons.  In that case, we have $l(u) = l(\uhat )+n$, due to the
contribution from ribbons of length zero in $\nubold $ that are
deleted in $\nuboldhat $.  Lemma~\ref{lem:inv-bound} implies that for
$\That $ we have $n((\eta ; \lambda; \kappa ) _{+}) + \sum
_{\smalldomino } l(\uhat ) - \inv (\That )\geq 0$.  Hence, if the
augmentation box in any of the last $r$ ribbons in $\nubold $ is in a
vertical domino, then for all $T$ we have $n((\eta ; \lambda; \kappa )
_{+}) + \sum _{\smalldomino } l(u) - \inv (T)\geq n$, and
\eqref{e:lhs-term} vanishes $t$-adically as $n\rightarrow \infty $.

Assume now that all of the ribbons in $\nubold $ corresponding to
parts of $\eta $ and $\lambda $, which are the last $k+r$ ribbons in
$\nuboldhat $, are non-ragged-right.  Then we have $|T(u)|\leq r$ for
boxes $u$ of content $-1$ in the last $r$ ribbons, which implies $\inv
(T) = \inv (\That )$ even without the non-attacking condition, hence
$T\mapsto \That $ is a weight-preserving bijection from tableaux $T\in
\FSST (\nubold ,\sigma ,\Acal _{\pm }^{b})$ to tableaux $\That \in
\FSST (\nuboldhat ,\sigmahat ,\Acal _{\pm }^{\bhat })$, both with
entries $a\in [m]_{\pm }$ such that $|a|\not \in [r+n+1, r+n+k]$.
This shows that
\begin{equation}\label{e:horizontal-term}
\Gcal _{\nubold ,\sigma }^{-}(x, y, 0^{k}, z, 0,\ldots,0;\, t^{-1}) =
\Gcal _{\nuboldhat ,\sigmahat }^{-}[x, y_{1}+\cdots +y_{n}, 0^{k}, z,
0,\ldots,0; t^{-1}]\, ,
\end{equation}
as claimed, and proves \eqref{e:stable-insertion}.  For $r=0$, the
same calculation gives \eqref{e:strong-stable-insertion}, since
passage to the $t$-adic limit as $n\rightarrow \infty $ was only used
for the vanishing of \eqref{e:lhs-term} when one of the last $r$
ribbons in $\nubold $ has its augmentation box in a vertical domino.
\end{proof}

\begin{lemma}\label{lem:lambda-independence}
The specialized integral form Macdonald polynomial $\Ecal _{(\eta ;\,
0^{n};\, \lambda ;\, \kappa )}(x,y,0^{k},z;\, q,t)$ in
\eqref{e:specialized-ns-mac} is independent of permuting the parts of
$\lambda $.
\end{lemma}

\begin{proof}
Let $\mu = (\eta ;\, 0^{n};\, \lambda ;\, \kappa )$ and let $i\in
[r+n+1,r+n+k-1]$, so $s_{i}$ acts on $\mu $ by switching two parts of
$\lambda $.  By Proposition~\ref{prop:Knop-antisymmetry}, $\Ecal
_{s_{i} \mu }(x;q,t) - \Ecal _{\mu }(x;q,t)$ is Hecke antisymmetric,
hence divisible by $t\, x_{i} - x_{i+1}$ (\S \ref{ss:hecke}).
Therefore $\Ecal _{\mu }(x;q,t)$ and $\Ecal _{s_{i} \mu }(x;q,t)$
become equal upon setting $x_{i}= x_{i+1} = 0$.
\end{proof}

\begin{cor}\label{cor:lambda-independence}
The function $\Jns _{\eta |\lambda |\kappa }(x;Y;z;\, q,t)$ is
independent of permuting the parts of $\lambda $, and also of deleting
any zeroes in $\lambda $.
\end{cor}

\begin{proof}
Lemma~\ref{lem:lambda-independence} and
Proposition~\ref{prop:stable-insertion} give independence of permuting
the parts of $\lambda $.  If there are $p$ zeroes in $\lambda $, we
can assume that $\lambda = (0^{p};\lambda _{+})$.  Then $\Ecal _{(\eta
;\, 0^{n};\, \lambda ;\, \kappa )}(x,y,0^{k},z;\, q,t) = \Ecal _{(\eta
;\, 0^{n+p};\, \lambda _{+};\, \kappa )}(x,y,0^{k-p},z;\,
q,t)|_{y_{n+1},\ldots,y_{n+p} = 0}$, so
Proposition~\ref{prop:stable-insertion} implies $\Jns _{\eta |\lambda
|\kappa }(x,Y,z;\, q,t) = \Jns _{\eta |\lambda _{+} |\kappa
}(x,Y,z;\, q,t)$.
\end{proof}

In view of this corollary, we usually assume from now on that the
index $\lambda $ for $\Jns _{\eta |\lambda |\kappa }(x;Y;z;\, q,t)$ is
a partition, without trailing zeroes.

We now compare the above construction using specialization with the
one using symmetrization.  One reason for doing so is that
symmetrization is more commonly used in other literature.  A second is
that the comparison will allow us to see that $\Jns _{\eta |\lambda
|\kappa }(x,Y,z; q,t)$ has leading term $x^{\eta }\, m_{\lambda }(Y)\,
z^{\kappa }$, and to determine the leading coefficient.  In
particular, this will show that the functions $\Jns _{\eta |\lambda
|\kappa }(x,Y,z; q,t)$ form a basis of $\kk
[x_{1},\ldots,x_{r}]\otimes \Lambda _{\kk } (Y) \otimes \kk
[z_{1},\ldots, z_{s}]$ (recall that $\kk = \QQ(q,t)$ in this section).

\begin{prop}\label{prop:limits-of-symmetrizers}
Let $\hsym _{[i,j]}$ denote the normalized Hecke symmetrizer in
\eqref{e:normalized-Hecke-symmetrizer} for the group $\Sfrak _{[i,j]}$
of permutations of $[i,j]$.  Define plethystic alphabets $Y_{r,r+n} =
x_{r+1}+\cdots +x_{r+n}$ and set $Y_{r} = x_{r+1}, x_{r+2}, \ldots $.
Then the following hold.

(a) For $\eta \in \NN ^{r-1}$, we have
\begin{equation}\label{e:one-S}
\hsym _{[r,r+n]}\, \Jns _{(\eta ,b)|\lambda |\kappa }[x_{1}, \ldots,
x_{r}, Y_{r,r+n} ,z;\, q,t] \equiv \Jns _{\eta\, |(b,\lambda) |\,
\kappa }[x_{1},\ldots,x_{r-1},Y_{r-1,r+n},z;\, q,t] \,
\end{equation}
modulo polynomials whose coefficients in $\QQ (q,t)$ vanish to order
$>n$ at $t=0$.

(b) For $\eta \in \NN ^{r}$, $\lambda \in \NN ^{k}$, and $\kappa \in
\NN ^{s}$, we have the $t$-adically convergent limits (with $n\geq k$)
\begin{equation}\label{e:many-S-limit}
\begin{aligned}
\Jns _{\eta |\lambda |\kappa }(x_{1},\ldots,x_{r},\, & Y_{r},z;\, q,t)\\
& =\lim_{n\rightarrow \infty } \hsym _{[r+1,r+n]} \, \Jns _{(\eta
;\lambda ) |\emptyset |\kappa
}[x_{1},\ldots,x_{r+k},Y_{r+k,r+n},z;\,
q,t]\\
& = \lim_{n\rightarrow \infty } \hsym _{[r+1,r+n]}\, \Ecal _{(\eta
;\, \lambda ;\, 0^{n-k};\, \kappa
)}(x_{1},\ldots,x_{r+n},z_{1},\ldots,z_{s};\, q,t)\, .
\end{aligned}
\end{equation}
\end{prop}

\begin{proof}
For part (a), Lemma~\ref{lem:symm-on-signed-flagged} implies that the
functions $\Gcal ^{-}_{\nubold ,\sigma }(x,Y,z; t^{-1})$ in
\eqref{e:G-minus-Y} satisfy
\begin{equation}\label{e:symm-G-minus-Y}
\begin{aligned}
\hsym _{[r,r+n]}\, &\Gcal ^{-}_{\nubold ,\sigma
}[x_{1},\ldots,x_{r},Y_{r,r+n},z_{1},z_{2},\ldots ; t^{-1}]\\
& \equiv \Gcal _{\nubold ,\sigma }[x_{1}, \ldots, x_{r-1}-t x_{r-2},\,
(1-t)Y_{r-1,r+n} - t x_{r-1},\, z_{1}, z_{2} - t z_{1},\ldots ;
t^{-1}]\\
& = \Gcal ^{-}_{\nubold ,\sigma
}[x_{1},\ldots,x_{r-1},Y_{r-1,r+n},0,z_{1}, z_{2},\ldots ; t^{-1}]\, .
\end{aligned}
\end{equation}
Substituting $(z_{1},z_{2},\ldots )\mapsto
(0^{k},z_{1},\ldots,z_{s},0,\ldots )$ and then applying this to each
term in formula \eqref{e:J-eta-lambda-kappa} for $\Jns _{(\eta
,b)|\lambda |\kappa }[x_{1},\ldots,x_{r}, Y_{r,r+n}, z; q,t]$ gives
\eqref{e:one-S}.

For part (b), using $\hsym _{[r+1,r+n]} = \hsym _{[r+1,r+n]}\,
\cdots \, \hsym _{[r+k,r+n]}$, we get 
\begin{equation}\label{e:repeat-one-S}
\hsym _{[r+1,r+n]}\, \Jns _{(\eta ;\lambda )|\emptyset |\kappa
}[x_{1},\ldots,x_{r+k}, Y_{r+k,r+n},z; q,t] \equiv \Jns _{\eta
|\lambda|\kappa }[x_{1},\ldots, x_{r}, Y_{r,r+n},z; q,t]
\end{equation}
by repeated applications of \eqref{e:one-S}.  This gives the first
limit in \eqref{e:many-S-limit}.

Proposition~\ref{prop:stable-insertion} with $k\mapsto 0$, $\lambda
\mapsto \emptyset $, $\eta \mapsto (\eta ;\lambda )$, $n \mapsto n-k$,
and $y_{i} = x_{r+k+i}$ shows that $\Ecal _{(\eta ;\, \lambda ;\,
0^{n-k};\, \kappa )}(x_{1},\ldots, x_{r+n},z_{1},\ldots,z_{s}; q,t) $
and $\Jns _{(\eta ;\lambda ) |\emptyset
|\kappa}[x_{1},\ldots,x_{r+k},Y_{r+k,r+n},z; q,t]$ are congruent
modulo polynomials whose coefficients vanish to arbitrarily high order
at $t=0$ as $n\rightarrow \infty $.  The same then holds for their
normalized symmetrizations, giving the second limit in
\eqref{e:many-S-limit}.
\end{proof}

Given $\eta \in \NN ^{r}$, $\kappa \in \NN ^{s}$, and a partition
$\lambda $, the term for each $n\geq \ell (\lambda )$ in the second
limit in \eqref{e:many-S-limit} is a scalar multiple of the monic
symmetrization
\begin{equation}\label{e:monic-eta-lambda-kappa}
\sum \nolimits _{w \in \Sfrak _{[r+1,r+n]}/ \Sfrak _{\sS }} T_{w}\, E
_{(\eta ;\, \lambda ;\, 0^{n - \ell (\lambda )};\, \kappa
)}(x;q,t)\bigr|_{x\mapsto (x_{1},\ldots,x_{r+n},
z_{1},\ldots,z_{s})}\, ,
\end{equation}
where $\Sfrak _{\sS } = \Stab _{\Sfrak _{[r+1,r+n]}} ((\eta ;\,
\lambda ;\, 0^{n - \ell (\lambda )};\, \kappa ))$.  We now calculate
the $t$-adic limit of the scalar factor, as $n\rightarrow \infty $.
By \eqref{e:many-S-limit}, this gives the coefficient $\langle x^{\eta
}\, m_{\lambda }(Y)\, z^{\kappa } \rangle \, \Jns _{\eta |\lambda
|\kappa }(x,Y,z;q,t)$.

The scalar factor in question is the product of two pieces.  One
piece is the factor $W_{\sS }(t)/[n]_{t}!$ arising from the ratio
between the normalized symmetrization using $\hsym _{[r+1,r+n]}$ and
the reduced symmetrization in \eqref{e:monic-eta-lambda-kappa}.  Let
\begin{equation}\label{e:W-Stab(lam)}
W_{\Stab (\lambda )}(t)\defeq \sum \nolimits
_{w\in \Stab _{\Sfrak _{\ell (\lambda )}}(\lambda )} t^{l(w)}\, .
\end{equation}
Then $W_{\sS }(t)/[n]_{t}! = W_{\Stab (\lambda )}(t)\, [n-\ell
(\lambda )]_{t}!/[n]_{t}!$.  Since $\lim_{n\rightarrow \infty }\, [n -
\ell (\lambda )]_{t}!/[n]_{t}! = (1-t)^{\ell (\lambda )}$, this
contributes a factor $(1-t)^{\ell (\lambda )} W_{\Stab (\lambda )}(t)$
in the limit.  The other piece is the normalizing factor $\Pi _{u\in
\dg ((\eta ;\, \lambda ;\, 0^{n-\ell (\lambda )};\, \kappa ))}
(1-q^{a(u)+1} t^{l(u)+1})$ in the definition of $\Ecal _{(\eta ;\,
\lambda ;\, 0^{n-\ell (\lambda )};\, \kappa )}(x; q,t)$.  With the
obvious identifications between boxes of $\dg ((\eta ;\, \lambda ;\,
0^{n-\ell (\lambda )};\, \kappa ))$ for different values of $n$, the
factor $(1-q^{a(u)+1} t^{l(u)+1})$ is independent of $n$, except for
boxes $u$ of content $c(u)=-1$ in the top $r+\ell (\lambda )$ rows,
corresponding to the parts of $(\eta ;\lambda )$.  For these boxes,
$a(u)$ is independent of $n$, and $l(u) = l_{0}(u)+n-\ell (\lambda )$,
where $l_{0}(u)$ is the leg of $u$ when $n=\ell (\lambda )$.  The
factor in this case is therefore $(1-q^{a(u)+1} t^{l_{0}(u)+n -\ell
(\lambda )+1})$, which has $t$-adic limit $1$.  Combining the two
pieces, we see that $\langle x^{\eta }\, m_{\lambda }(Y)\, z^{\kappa }
\rangle \, \Jns _{\eta |\lambda |\kappa }(x,Y,z;q,t)$ is given by
\begin{equation}\label{e:proto-Goodberry}
\Pi _{\eta ,\lambda ,\kappa }(q,t) \defeq (1-t)^{\ell (\lambda )}\,
W_{\Stab (\lambda )}(t)\, \prod _{u\in S} (1 - q^{a(u)+1}
t^{l(u)+1})\, ,
\end{equation}
where $S\subseteq \dg ((\eta ;\lambda ;\kappa ))$ excludes boxes of
content $c(u)=-1$ in the top $r+\ell (\lambda )$ rows.

Since this coefficient is non-zero, \eqref{e:many-S-limit} further
implies that $x^{\eta }\, m_{\lambda }(Y)\, z^{\kappa }$ is the
leading term of $\Jns _{\eta |\lambda |\kappa }(x,Y,z;q,t)$, giving the
following corollary.

\begin{cor}\label{cor:J-alt}
For $\eta \in \NN ^{r}$, $\kappa \in \NN ^{s}$, and partition $\lambda
$, we have
\begin{equation}\label{e:J-alt}
\Jns _{\eta |\lambda |\kappa }(x,Y,z;q,t) = \Pi _{\eta ,\lambda
,\kappa }(q,t)\, E_{\eta |\lambda |\kappa }(x,Y,z;q,t)\, ,
\end{equation}
where the monically normalized partially symmetric Macdonald
polynomial $E_{\eta |\lambda |\kappa }(x,Y,z;q,t)$ is the scalar
multiple of $\Jns _{\eta |\lambda |\kappa }(x,Y,z;q,t)$ in which the
leading term $x^{\eta }\, m_{\lambda }(Y)\, z^{\kappa }$ has
coefficient $1$, and is given by the $t$-adically convergent limit
\begin{equation}\label{e:E-stable}
E_{\eta |\lambda |\kappa }(x,Y,z;q,t)
= \lim_{n\rightarrow \infty } \bigl( \sum _{w \in \Sfrak _{[r+1,r+n]}/
\Sfrak _{\sS } \hspace{-6ex}} T_{w}\, E _{(\eta ;\, \lambda ;\,
0^{n-\ell(\lambda )};\, \kappa )}(x;q,t)\, \bigr) \bigr|_{x\mapsto
(x_{1},\ldots,x_{r}, y_{1},\ldots,y_{n}, z_{1},\ldots,z_{s})}
\end{equation}
of the monically normalized symmetrizations in
\eqref{e:monic-symmetrized-ns-mac}.
\end{cor}

\begin{proof}
The preceding calculation gives \eqref{e:J-alt}.  The limit in
\eqref{e:E-stable} follows from the second limit in
\eqref{e:many-S-limit} by dividing each term by the scalar factor that
makes it monic, and using the fact that these scalar factors converge
$t$-adically to $\Pi _{\eta ,\lambda ,\kappa }(q,t)$.
\end{proof}

\begin{cor}\label{cor:basis}
The functions $\Jns _{\eta |\lambda |\kappa }(x,Y,z;q,t)$ for $\eta \in \NN
^{r}$, $\kappa \in \NN ^{s}$, and all partitions $\lambda $ are a
basis of $\kk [x_{1},\ldots,x_{r}] \otimes \Lambda _{\kk }(Y)\otimes
\kk [z_{1},\ldots,z_{s}]$.
\end{cor}

\begin{proof}
Immediate from $\Jns _{\eta |\lambda |\kappa }(x,Y,z;q,t)$
having leading term $x^{\eta }\, m_{\lambda }(Y)\, z^{\kappa }$ with
non-vanishing coefficient.
\end{proof}

Note, however, that although the functions $\Jns _{\eta |\lambda
|\kappa }(x,Y,z;q,t)$ have coefficients in $\ZZ [q,t]$, they are not a
basis of $\ZZ [q,t] [x_{1},\ldots,x_{r}] \otimes \Lambda _{\ZZ
[q,t]}(Y)\otimes \ZZ [q,t] [z_{1},\ldots,z_{s}]$.

\begin{remark}\label{rem:Goodberry}
For $r=0$ and $\eta =\emptyset $, the quantity inside the limit in
\eqref{e:E-stable} reduces to
\begin{equation}\label{e:E-strong-stable}
f_{n}(y,z;q,t) \defeq \bigl( \sum _{w \in \Sfrak _{[1,n]}/ \Sfrak
_{\sS } \hspace{-3ex}} T_{w}\, E _{(\lambda ;\, 0^{n-\ell (\lambda
)};\, \kappa )}(x;q,t)\, \bigr) \bigr|_{x\rightarrow
(y_{1},\ldots,y_{n}, z_{1},\ldots,z_{s})}\, .
\end{equation}
Now, $f_{n}(y,z;q,t)$ is the unique linear combination of the
$E_{(v(\lambda ;\, 0^{n-\ell (\lambda )});\, \kappa )}(y,z;q,t)$ for
$v\in \Sfrak _{n}$ that is symmetric in $y$ and has coefficient
$\langle m_{\lambda }(y)z^{\kappa } \rangle\, f_{n}(y,z;q,t) = 1$.
Lemma~\ref{lem:x1=0} implies that $f_{n+1}(0,y,z;q,t)$ is also a
linear combination of the $E_{(v(\lambda ;\, 0^{n-\ell (\lambda )});\,
\kappa )}(y,z;q,t)$.  Since this is again symmetric in $y$ with
leading coefficient $1$, we have $f_{n+1}(0,y,z ;q,t) = f_{n}(y,z
;q,t)$, i.e., the limit in \eqref{e:E-stable} converges strongly in
this case.

Since \eqref{e:strong-stable-insertion} also converges strongly, we
have
\begin{multline}\label{e:Goodberry-identity}
\Ecal _{(0^{n};\, \lambda ;\, \kappa )}(y_{1},\ldots,y_{n},0^{\ell
(\lambda )},z_{1},\ldots,z_{s}; q,t)\\
 = \Pi _{\emptyset |\lambda |\kappa }(q,t) \bigl( \sum _{w \in \Sfrak
_{[1,n]}/ \Sfrak _{\sS } \hspace{-3ex}} T_{w}\, E _{(\lambda ;\,
0^{n-\ell (\lambda )};\, \kappa )}(x;q,t)\, \bigr)
\bigr|_{x\rightarrow (y_{1},\ldots,y_{n}, z_{1},\ldots,z_{s})}\, ,
\end{multline}
for all $n\geq \ell (\lambda )$, with $\Jns _{\emptyset |\lambda
|\kappa }[y_{1}+\cdots +y_{n},z;q,t]$ given by the expression on
either side.  The results in this case were obtained previously by
Goodberry \cite{Goodberry22, Goodberry24}.
\end{remark}

The normalizing factor $\Pi _{\eta ,\lambda ,\kappa }(q,t)$ in
\eqref{e:proto-Goodberry} can be recast in a form more closely
resembling the factor in Definition~\ref{def:integral-form}.  If
$r_{i}$ are the part multiplicities of $\lambda = (1^{r_{1}},
2^{r_{2}},\ldots )$, then $\Stab (\lambda ) = \Sfrak _{\rr }$, and the
first two factors in \eqref{e:proto-Goodberry} are given by
\begin{equation}\label{e:modified-Knop-a=0}
(1-t)^{\ell (\lambda )} \, W_{\Stab (\lambda )}(t) \, = \, \prod _{i}
\prod _{j=1}^{r_{i}} (1-t^{j})\; = \!  \prod _{a_{\lambda }(u) = 0}
(1-t^{l'(u)+1})\, ,
\end{equation}
where the last product is over boxes $u\in \dg (\lambda )$ with arm
$a_{\lambda }(u) = 0$, and $l'(u)$ is the number of indices $i<j$ such
that $\lambda _{i} = \lambda _{j}$, where $u$ is in the row
corresponding to $\lambda _{j}$.

Let $S_{\eta }, S_{\lambda }, S_{\kappa }\subseteq \dg ((\eta ;\lambda
;\kappa ))$ be the sets of boxes in rows corresponding to $\eta $,
$\lambda $, $\kappa $, respectively, and let $S'_{\eta }\subseteq
S_{\eta }$, $S'_{\lambda }\subseteq S_{\lambda }$ consist of boxes $u$
with $c(u)\not =-1$, so we have $S = S'_{\eta }\cup S'_{\lambda }\cup
S_{\kappa }$ in \eqref{e:proto-Goodberry}.  For $u\in S_{\lambda }$,
define the {\em reduced leg} $l'(u)$ to be the leg of $u$ in $\dg
((\eta ;\lambda ;\kappa ))$ with hands of content $c(u)-1$ in rows of
$S_{\eta } $ excluded, along with all hands of content $c(u)$.

For $v\in S'_{\lambda }$, we have $e(v)\in S_{\lambda }$.  Setting
$u=e(v)$, we have $a(u) = a(v)+1$.  Since $\lambda $ is a partition,
there are no hands of content $c(u)-1=c(v)$ in rows of $S_{\lambda }$
above the row containing $u$ and $v$, so the hands excluded in the
definition of $l'(u)$ are exactly those not included in $l(v)$. Hence
$l'(u) = l(v)$.  For $u\in S_{\lambda }$ with arm $a(u) = 0$, i.e.,
$u$ not of the form $e(v)$, $l'(u)$ is the same as in
\eqref{e:modified-Knop-a=0}. Hence we have
\begin{equation}\label{e:Goodberry}
\Pi _{\eta ,\lambda ,\kappa }(q,t) = \prod _{u\in S'_{\eta }\cup
S_{\kappa }} (1-q^{a(u)+1} t^{l(u)+1})\times \prod _{u\in S_{\lambda
}} (1-q^{a(u)} t^{l'(u)+1})\, ,
\end{equation}
where $l'(u)$ denotes the reduced leg, as defined in the preceding
paragraph.

\begin{remark}\label{rem:previous-stable-Macs}
(i) In the case $\eta =\emptyset $, the results on left stabilized
integral form Macdonald polynomials $E_{\emptyset |\lambda |\kappa
}(Y,z; q,t)$ and $\Jns _{\emptyset |\lambda |\kappa }(Y,z; q,t)$
summarized in Remark~\ref{rem:Goodberry} are due to Goodberry
\cite{Goodberry22, Goodberry24}, as already noted.  Lapointe
\cite{Lapointe22} studies the same left stabilized polynomials as
introduced by Goodberry.  Knop \cite{Knop07} had earlier studied the
special case of left stabilized polynomials in which
$\lambda = \emptyset $,
and shown that these converge strongly (the
polynomials in \cite{Knop07} look right stabilized, as the variables
are reversed).

(ii) In the next subsection, we examine the case $\kappa =\emptyset $
of {\em right stabilized} partially symmetric Macdonald polynomials.
This case was previously studied by Bechtloff Weising
\cite{BechtloffWeising23}, who constructed the stable limits for a
variant of $t$-adic convergence, related them to operators arising
in the work of Carlsson and Mellit \cite{CarlMell18}, and
characterized them in terms of
Hecke symmetrization as the
 limit $\lim_{n\rightarrow \infty } \hsym
_{[r+1,r+n]}\, E_{(\eta;\, \lambda ;\, 0^{n-\ell (\lambda)})}(x;\,
q^{-1},t)$.
He gives a
formula in the spirit of Proposition~\ref{prop:original-HHL} for his
stable limits in the case $\lambda =\emptyset $ and poses the problem
of giving a similar formula for general $\eta |\lambda$, a problem
which is solved by Proposition~\ref{prop:right-stable-HHL}.
Bechtloff Weising does not construct integral forms.  In terms of the
monically normalized partially symmetrized Macdonald polynomials
$E_{\eta|\lambda |\kappa }(x,Y,z; q,t)$ in Corollary~\ref{cor:J-alt}, his
stable limits are $(1-t)^{\ell (\lambda )}W_{\Stab (\lambda )}(t)
E_{\eta|\lambda |\emptyset }(x,Y;q^{-1},t)$.
\end{remark}

\subsection{Right stable \texorpdfstring{$r$}{r}-nonsymmetric
Macdonald polynomials} \label{ss:right stable unmod}

\begin{defn}\label{def:right-stable}
The algebra of {\em $r$-nonsymmetric polynomials} with coefficients in
$\kk =\QQ (q,t)$ is $\kk [x_{1},\ldots,x_{r}]\otimes \Lambda _{\kk
}(Y)$.  For $\eta \in \NN ^{r}$ and $\lambda $ a partition, we define
the {\em right stabilized integral form $r$-nonsymmetric Macdonald
polynomial}
\begin{equation}\label{e:right-stable}
\Jns _{\eta |\lambda }(x,Y;q,t) \defeq \Jns _{\eta |\lambda |\emptyset
}(x,Y,z;q,t) \in \kk [x_{1},\ldots,x_{r}]\otimes \Lambda _{\kk }(Y).
\end{equation}
Note that $\Jns _{\eta |\lambda }(x,Y;q,t)$ has coefficients in $\ZZ
[q,t]$, and that there are no variables $z_{i}$ in the formula, since
$\kappa = \emptyset$ and $s=0$.
\end{defn}

Restating the definition in this case gives the following formula.

\begin{prop}\label{prop:right-stable-HHL}
For $\eta \in \NN ^{r}$ and $\lambda $ a partition, we have
\begin{equation}\label{e:right-stable-HHL}
\Jns _{\eta |\lambda }(x,Y;\, q,t) = t^{n((\eta ;\lambda) _{+})}
\hspace{-2ex} \sum _{(\nubold ,\sigma )\in R'(\mu )} \bigl( \prod
\nolimits_{\domino } q^{a(u)+1}\, t^{l(u)} \bigr) \Gcal _{\nubold
,\sigma }^{-}(x_{1}, \ldots, x_{r}, Y, 0,\ldots,0;\, t^{-1}),
\end{equation}
where $R'(\mu )$ is the same as $R(\mu )$ in \eqref{e:ns-HHL}, with
$\mu =(\eta ;\lambda)\in \NN ^{m}$ and $m=r+\ell(\lambda)$, except
that $R'(\mu )$ only contains pairs $(\nubold ,\sigma )$ in which all
augmented ribbons are non-ragged-right, i.e., no augmentation box is
in a vertical domino.
\end{prop}

\begin{remark}\label{rem:rearranging-lambda}
By Corollary~\ref{cor:lambda-independence}, $\Jns _{\eta |\lambda }(x,Y;\,
q,t)$ is also equal to $\Jns _{\eta |\zeta | \emptyset}(x,Y,z;q,t)$ for
any weak composition $\zeta $ such that $\lambda =\zeta _{+}$, i.e.,
$\zeta $ is a permutation of the parts of $\lambda $ and possible
added zeroes.  Hence \eqref{e:right-stable-HHL} also holds with $\mu
=(\eta ;\zeta )$ in place of $(\eta ;\lambda )$.
\end{remark}

By Proposition~\ref{prop:stable-insertion} the definition is
equivalent to the specialization formula
\begin{equation}\label{e:right-stable-spec}
\Jns _{\eta |\lambda }(x,Y;q,t)  =
\lim_{n\rightarrow \infty } \Ecal _{(\eta ;\, 0^{n};\, \lambda
)}(x_{1},\ldots,x_{r}, y_{1},\ldots,y_{n},0^{\ell (\lambda )}; q,t)\, ,
\end{equation}
and by Proposition~\ref{prop:limits-of-symmetrizers} and
Corollary~\ref{cor:J-alt}, also to the symmetrization formulas
\begin{equation}\label{e:right-stable-symm}
\begin{aligned}
\Jns _{\eta |\lambda }(x,Y;\, & q,t) \\
& = \lim_{n\rightarrow \infty } \bigl(\hsym_{[r+1,r+\ell (\lambda
)+n]} \Ecal _{(\eta ;\, \lambda ;\, 0^{n})}(x;q,t)\, \bigr)
|_{x\mapsto
(x_{1},\ldots,x_{r}, y_{1},\ldots,y_{n+\ell (\lambda )})}\\
& = \Pi _{\eta ,\lambda ,\emptyset }(q,t) \lim_{n\rightarrow \infty }
\bigl( \sum _{w \in \Sfrak _{[r+1, r+\ell(\lambda )+n ]}/ \Sfrak _{\sS
} \hspace{-6ex}} T_{w}\, E _{(\eta ;\, \lambda ;\, 0^{n})}(x;q,t)\,
\bigr) |_{x\mapsto (x_{1},\ldots,x_{r}, y_{1},\ldots,y_{n+\ell
(\lambda )})}\, ,
\end{aligned}
\end{equation}
where $\Pi _{\eta ,\lambda ,\emptyset }(q,t)$ is the product in
\eqref{e:proto-Goodberry} and \eqref{e:Goodberry}, $\Sfrak _{\sS }$ is
the stabilizer in $\Sfrak _{[r+1, r+n+\ell (\lambda ) ]}$ of $(\eta
;\, \lambda ;\, 0^{n})$, and all of the above limits converge
$t$-adically.  In particular, $\Jns _{\eta |\lambda }(x,Y;q,t)$ has
leading term $x^{\eta } \, m_{\lambda }(Y)$ with coefficient $\Pi
_{\eta ,\lambda ,\emptyset }(q,t)$.  For $\eta \in \NN ^{r}$ and all
partitions $\lambda $, these functions form a $\kk $-basis of $\kk
[x_{1},\ldots,x_{r}]\otimes \Lambda _{\kk }(Y)$.

\subsection{The modified \texorpdfstring{$r$}{r}-nonsymmetric
Macdonald polynomials \texorpdfstring{$\mE_{\eta|\lambda}$}{}}
\label{ss:r mod nsmac}

We come now to the crucial innovation of this section: an operation
$\Pisf _{r}$ on $r$-nonsymmetric polynomials, mixing symmetric and
nonsymmetric plethysm, which we view as the right way to generalize
the symmetric plethysm $f(X)\mapsto f[X/(1 - t)]$.  We use $\Pisf
_{r}$ to define the plethystically modified stable $r$-nonsymmetric
Macdonald polynomials that are the subject of our atom positivity
conjectures in \S \ref{ss:ns-Mac-atom-positivity}.

In the formulas for $\Jns _{\eta |\lambda }(x,Y;q,t)$, it is natural
to interpret $Y$ as a set of additional $x$ variables $Y = Y_{r} =
x_{r+1}, x_{r+2}$, so that $r$-nonsymmetric polynomials $f(x,Y)$ are
identified with functions of $X = x_{1}, x_{2}, \ldots $ symmetric in
all but the first $r$ variables.
We can express $f(x,Y) = f(x_1,\dots, x_r, Y)$ uniquely as a function
of $x_1, \dots, x_r$ and symmetric functions in $X$, by the formula
\begin{equation}\label{e:change-Y-to-X}
f(x_1, \dots, x_r,Y)=f[x_1, \dots, x_r,X-(x_{1}+\cdots +x_{r})].
\end{equation}
This defines an isomorphism $\kk [x_{1},\ldots,x_{r}]\otimes \Lambda
_{\kk }(Y) \cong \kk [x_{1},\ldots,x_{r}]\otimes \Lambda _{\kk }(X)$,
whose inverse, expressing $g(x_1, \dots, x_r,X)\in \kk
[x_{1},\ldots,x_{r}]\otimes \Lambda _{\kk }(X)$ as a function of $x_1,
\dots, x_r$ and $Y$, is given by
\begin{equation}\label{e:change-X-to-Y}
g(x_1, \dots, x_r,X)=g[x_1, \dots, x_r,Y+x_{1}+\cdots x_{r}].
\end{equation}

\begin{defn}\label{def:Pir}
Let $\Pi _{t,X} f(X) = f[X/(1-t)]$ for symmetric functions $f(X)\in
\Lambda _{\kk }(X)$ in the formal alphabet $X$.  The operator $\Pisf
_{r}$ of {\em $r$-nonsymmetric plethysm} is the composite $\Pisf _{r}
= \Pi _{t,x}\otimes \Pi _{t,X}$ of the (commuting) operations $\Pi
_{t,x}$ and $\Pi _{t,X}$ on $r$-nonsymmetric polynomials expressed as
functions $g(x,X)\in \kk [x_{1},\ldots,x_{r}]\otimes \Lambda _{\kk
}(X)$ of $x_{1},\ldots,x_{r}$ and the formal alphabet $X$.  More
explicitly,
\begin{equation}\label{e:Pir}
\Pisf _{r} \bigl(h(x_{1},\ldots,x_{r})\, f(X)\bigr) = (\Pi
_{t,x_{1},\ldots,x_{r}}\, h(x_{1},\ldots,x_{r}))\, f[X/(1-t)].
\end{equation}
\end{defn}

\begin{example}\label{ex:Pir}
Taking $r=2$, we calculate $\Pisf _{2}\, (x_{2}\, e_{1}(Y))$.  Using
\eqref{e:change-Y-to-X}, we write
\begin{equation}\label{e:Pir-step-1}
x_{2}\, e_{1}(Y) = x_{2}\, e_{1}[X-(x_{1}+x_{2})] = x_{2}\, e_{1}(X) -
x_{2}\, (x_{1}+x_{2})\, .
\end{equation}
Now we apply $\Pi _{t,x}\otimes \Pi _{t,X}$, obtaining
\begin{equation}\label{e:Pir-step-2}
\begin{aligned}
\Pisf _{2}\, (x_{2}\, e_{1}(Y)) & = (\Pi _{t,x}\, x_{2})\,
e_{1}[X/(1-t)] - \Pi _{t,x}\, (x_{1}\, x_{2} +x_{2}^{2})\\
& = (t\, x_{1}+x_{2})\, e_{1}(X) /(1-t) - ((t^{2} + t)\,
x_{1}^{2}+(t+1)\, x_{1}x_{2}+x_{2}^{2})\, .
\end{aligned}
\end{equation}
To express this in terms of $x$ and $Y$ again, we use
\eqref{e:change-X-to-Y}, obtaining
\begin{equation}\label{e:Pir-step-3}
\begin{aligned}
\Pisf _{2}\, (x_{2}\, e_{1}(Y)) & = (t\, x_{1}+x_{2})\, e_{1}[Y +
x_{1}+x_{2}]/(1-t) - ((t^{2} + t)\, x_{1}^{2}+(t+1)\, x_{1}x_{2}+x_{2}^{2})\\
& = \bigl((t\, x_{1}+x_{2})\, e_{1}(Y) + t^{3} x_{1}^{2}+(t^{2} +t)\,
x_{1}x_{2}+ t\, x_{2}^{2} \bigr)/(1-t)\, .
\end{aligned}
\end{equation}
\end{example}

Theorem~\ref{thm:Pit-symm} implies that $\Pisf _{r}$ is a stable limit
of nonsymmetric plethysms $\Pi
_{t,(x_{1},\ldots,x_{r},y_{1},\ldots,y_{n})}$ acting on polynomials
symmetric in the $y$ variables.  The following lemma makes this
precise.

\begin{lemma}\label{lem:Pit-vs-Pir}
Suppose that $\phi _{n}(x,y,z)\in \kk[x_{1}, \ldots, x_{r}, y_{1},
\ldots, y_{n}, z_{1}, \ldots, z_{s}]$ are symmetric in $y_{1},
\ldots, y_{n}$, have coefficients
with no pole at $t=0$,
and converge $t$-adically to
\begin{equation}\label{e:lim-assumption}
\lim_{n\rightarrow \infty } \phi _{n}(x,y,z) = \phi (x,Y,z) \in
\kk[x_{1}, \ldots, x_{r}, z_{1}, \ldots, z_{s}]\otimes \Lambda
_{\kk }(Y)\, .
\end{equation}
Then the nonsymmetric plethysms $\Pi _{t,(x,y,z)}\, \phi _{n}(x,y,z)$,
with coefficients expanded as formal power series in $t$, converge
formally in $x$, $y$, $z$, and $t$ to
\begin{equation}\label{e:lim-conclusion-z}
\lim_{n\rightarrow \infty } \Pi _{t,(x,y,z)}\, \phi _{n}(x,y,z) =
\Pisf _{r} \, \Pi _{t,z}\, \phi (x,Y,z)
\end{equation}
(note that $\Pisf _{r}$, acting on $x,Y$, commutes with $\Pi _{t,z}$).
As a consequence, the following limits also converge formally in $x$,
$y$, and $t$:
\begin{equation}\label{e:lim-conclusions-0}
\lim_{n\rightarrow \infty } \bigl((\Pi _{t,(x,y,z)}\, \phi
_{n}(x,y,z))|_{z = 0}\bigr) = \lim_{n\rightarrow \infty } \Pi
_{t,(x,y)}\, \phi _{n}(x,y,0) = \Pisf _{r}\, \phi (x,Y,0).
\end{equation}
\end{lemma}

\begin{proof}
The $t$-adic convergence in \eqref{e:lim-assumption} means that for
any $e$, we have $\phi _{n}(x,y,z) \equiv \phi [x,y_{1}+\cdots
+y_{n},z]$ (mod $(t^{e})$), and therefore also $\Pi _{t,(x,y,z)}\phi
_{n}(x,y,z) \equiv \Pi _{t,(x,y,z)}\phi [x,y_{1}+\cdots +y_{n},z]$
(mod $(t^{e})$), for $n$ sufficiently large.  Hence, if we show that
\begin{equation}\label{e:lim-with-z}
\lim_{n\rightarrow \infty } \Pi _{t,(x,y,z)}\phi [x,y_{1}+\cdots
+y_{n},z] = \Pisf _{r} \, \Pi _{t,z}\, \phi (x,Y,z)\, ,
\end{equation}
with the limit converging formally in $x$, $y$, $z$ and $t$, then
\eqref{e:lim-conclusion-z} follows.  For \eqref{e:lim-with-z}, we can
assume by linearity that $\phi (x,Y,z)$ has the factored form $\phi
(x,Y,z) = g(x)f[X]h(z)$ in Theorem~\ref{thm:Pit-symm}, where $X=
x_{1}+\cdots +x_{r}+Y$.
Then Corollary~\ref{cor:other-limits2} implies
\begin{equation}\label{e:lim-with-z-2}
\begin{aligned}
\lim_{n\rightarrow \infty } \Pi _{t,(x,y,z)}\phi [x,y_{1}+\cdots
+y_{n},z] & = (\Pi _{t,x} g(x))f[X/(1-t)] (\Pi _{t,z} h(z)) \\
& = \Pisf _{r} \, \Pi _{t,z}\, \phi (x,Y,z)\, .
\end{aligned}
\end{equation}

We always have $(\Pi _{t,z}\, h(z))|_{z=0} = h(0)$, i.e., the operator
$\Pi _{t,z}$ commutes with setting $z=0$, because $\Pi _{t,z}$ is
homogeneous of degree zero and $\Pi _{t,z}\, 1 = 1$. We also have
$(\lim_{n\to\infty} f_n(x,y,z,t))|_{z=0} = \lim_{n\to\infty}
f_n(x,y,0,t)$ for any sequence $f_n$ which converges formally in $x$,
$y$, $z$, and $t$.  Hence setting $z=0$ in \eqref{e:lim-conclusion-z}
gives $\lim_{n\rightarrow \infty } \bigl((\Pi _{t,(x,y,z)}\, \phi
_{n}(x,y,z))|_{z = 0}\bigr) = \Pisf _{r}\, \phi (x,Y,0).$ The second
equality in \eqref{e:lim-conclusions-0} follows from the case of
\eqref{e:lim-conclusion-z} with no $z$ variables, applied to
$\phi_{n}(x,y,0)$ and $\phi (x,Y,0)$.
\end{proof}

We now get a stable version of
Proposition~\ref{prop:Pit-signed-to-unsigned}.

\begin{lemma}\label{lem:Pir-signed-to-unsigned}
For any flagged LLT data $\nubold , \sigma $ of length at least $r$,
we have
\begin{equation}\label{e:Pir-signed-to-unsigned}
\Pisf _{r}\, \Gcal ^{-} _{\nubold ,\sigma
}(x_{1},\ldots,x_{r},Y,0,\ldots,0;\, t^{-1}) = \Gcal _{\nubold ,\sigma
}[x_{1},\ldots,x_{r},Y,0,\ldots,0;\, t^{-1}]\, ,
\end{equation}
where $ \Gcal ^{-} _{\nubold ,\sigma}$ is as in \eqref{e:G-minus-Y}.
\end{lemma}

\begin{proof}
By Proposition~\ref{prop:Pit-signed-to-unsigned}, for any $n$, we have
\begin{multline}\label{e:xyz-signed-to-unsigned}
\Pi _{t,(x,y,z)}\, \Gcal ^{-} _{\nubold ,\sigma
}[x_{1},\ldots,x_{r},y_{1}+\cdots +y_{n},z_{1},\ldots,z_{s};\, t^{-1}]
\\
= \Gcal _{\nubold ,\sigma }[x_{1},\ldots,x_{r},y_{1}+\cdots
+y_{n},z_{1}, \ldots, z_{s};\, t^{-1}]\, .
\end{multline}
By \eqref{e:lim-conclusions-0} in Lemma~\ref{lem:Pit-vs-Pir}, if we
set $z = 0$ in the expression on the left hand side, the result
converges formally as $n\rightarrow \infty $ to the left hand side of
\eqref{e:Pir-signed-to-unsigned}.  After setting $z=0$, the right hand
side plainly converges strongly to the right hand side of
\eqref{e:Pir-signed-to-unsigned}.
\end{proof}

\begin{defn}\label{def:mod-mac}
The {\em modified $r$-nonsymmetric Macdonald polynomials} are the
$r$-nonsymmetric plethystic transformations of the right stabilized
integral forms
\begin{equation}\label{e:H-eta-lambda}
\mE _{\eta |\lambda }(x,Y;\, q,t) \defeq \Pisf _{r}\, \Jns _{\eta
|\lambda }(x,Y;\, q,t)\, .
\end{equation}
\end{defn}

\begin{thm}\label{thm:mod-mac}
For all $\eta \in \NN ^{r}$ and $\lambda $ a partition, the
modified $r$-nonsymmetric Macdonald polynomials
are given by either of the following identities.

(a) $\mE _{\eta |\lambda }(x,Y;\, q,t)$ is the formally convergent
limit
\begin{equation}\label{e:mod-mac-lim}
 \mE _{\eta |\lambda }(x,Y;\, q,t) = \lim_{n\rightarrow \infty } \Pi
_{t,(x;y)}\, \Ecal _{(\eta ;0^{n};\lambda
)}(x_{1},\ldots,x_{r},y_{1},\ldots,y_{n}, 0^{\ell (\lambda )};\, q,t)
\end{equation}
of nonsymmetric plethysms of specialized integral form Macdonald
polynomials.

(b) We have the following `unsigned' version of
\eqref{e:right-stable-HHL}:
\begin{equation}\label{e:mod-mac-HHL}
\mE _{\eta |\lambda }(x,Y;\, q,t) = t^{n((\eta ;\lambda ) _{+})}
\hspace{-3ex} \sum _{(\nubold ,\sigma )\in R'((\eta ;\lambda
))}\hspace{-1ex} \bigl( \prod \nolimits_{\domino } q^{a(u)+1}\,
t^{l(u)} \bigr)\, \Gcal _{\nubold ,\sigma }[x_{1}, \ldots, x_{r}, Y,
0,\ldots,0;\, t^{-1}],
\end{equation}
where $R'((\eta ;\lambda ))$ only contains pairs $(\nubold ,\sigma )$
in which all augmented ribbons $\nu ^{(i)}$ are non-ragged right, as
in Proposition~\ref{prop:right-stable-HHL}.
\end{thm}

\begin{proof}
Part (a) follows from formula \eqref{e:right-stable-spec} and
Lemma~\ref{lem:Pit-vs-Pir}.  Part (b) follows from
Proposition~\ref{prop:right-stable-HHL} and
Lemma~\ref{lem:Pir-signed-to-unsigned}.
\end{proof}

\begin{remark}\label{rem:rearranging-lambda-2}
(i) By Remark~\ref{rem:rearranging-lambda}, formula \eqref{e:mod-mac-HHL}
also holds with any weak composition $\zeta $ such that $\zeta _{+} =
\lambda $ in place of $\lambda $ on the right hand side.

(ii) Since $\nubold $ is a tuple of ordinary skew diagrams for every
$(\nubold ,\sigma )$ indexing a term in \eqref{e:mod-mac-HHL}, we can
further assume that $\sigma $ is the standard compatible permutation,
which does assign the required flag numbers $m,\ldots,1$ to the
augmentation boxes.
\end{remark}

\begin{remark}\label{rem:mod-mac-HHL explicit}
We can make \eqref{e:mod-mac-HHL} more explicit by using Theorem
\ref{thm:combinatorial-G-nu-sigma} to express the flagged LLT
polynomials in terms of tableaux.  We set $y_{i}=x_{r+i}$, so that $Y
= x_{r+1}+x_{r+2}+\cdots $ as a plethystic alphabet.  Fix the positive
alphabet $\Acal = \Acal^+ = \{1 < 2 < \cdots \}$, with $\Acal_i =
\{i\}$ for $i \le r$, $\Acal_{r+1} = \{r+1, r+2, \dots\}$, and
$\Acal_i=\emptyset$ for $i > r+1$.  This corresponds to flag bounds
$b_{i}=i$ for $i=1,\ldots,r$ and effectively infinite flag bounds
$b_{i}>\Acal $ for $i>r$.  For this $\Acal $, the plethystic alphabets
in Theorem \ref{thm:combinatorial-G-nu-sigma} are $X_{i}^{\Acal } =
x_{i}$ for $i\leq r$, $X_{r+1}^{\Acal } = Y$, and $X_{i}^{\Acal } = 0$
for $i>r+1$, giving
\begin{equation}\label{e:}
\Gcal_{\nubold ,\sigma }[x_{1}, \ldots, x_{r}, Y, 0,\ldots,0;\,
t^{-1}] =  \sum _{T\in \FSST (\nubold , \sigma , \Acal
)} t^{-\inv (T)} x^{T}\, .
\end{equation}
Since all ribbons in $\nubold$ are non-ragged-right for $(\nubold,
\sigma) \in R'((\eta ;\lambda ))$, $\FSST(\nubold , \sigma , \Acal )$
consists of ordinary semistandard tableaux $T$ on $\nubold$, subject
only to the flag condition that for $i = 1,\dots, r$, if $\eta
_{i}\not =0$, the southeastern-most box $u$ in the ribbon
$\nu^{(m+1-i)}$ of length $|\nu^{(m+1-i)}| = \eta_i$, with flag bound
$b_{i} = i$ in its augmentation box, has entry $T(u)\leq i$.
Moreover, $\inv (T)$ has no contributions from flag boxes, i.e.,
$\inv(T)= \oinv (T)$.  This follows from
Remark~\ref{rem:no-funny-inv}(i) since we can assume that $\sigma$ is
the standard compatible permutation.
\end{remark}

\begin{remark}\label{rem:why-stabilize}
The main ingredients in the proof of Theorem~\ref{thm:mod-mac} are the
flagged LLT expansion \eqref{e:ns-HHL} for the integral forms $\Ecal
_{\mu }(x;q,t)$, and the property in
Proposition~\ref{prop:Pit-signed-to-unsigned} that nonsymmetric
plethysm converts signed flagged LLT polynomials into unsigned ones.
Neither of these results requires stabilizing with additional
variables $Y$.  The need to stabilize arises from the fact that the
ribbon tuples in \eqref{e:ns-HHL} have more total rows than the number
of variables $m$ (the number of ribbons in the tuple), resulting in
some variables being set equal to zero in the flagged LLT polynomials
in \eqref{e:ns-HHL}.  Specifically,
Lemma~\ref{lem:Pir-signed-to-unsigned} depends on interchanging the
plethysm with setting variables to zero, which only works when the
variables set equal to zero immediately follow a stabilized formal
alphabet $Y$.
\end{remark}

\subsection{Atom positivity conjecture for
\texorpdfstring{$\mE_{\eta|\lambda}$}{H}}
\label{ss:ns-Mac-atom-positivity}

We now formulate our atom positivity
conjecture for the modified $r$-nonsymmetric Macdonald polynomials
$\mE_{\eta|\lambda}(x,Y;q,t)$, and show that it would follow from our
atom positivity conjecture for flagged LLT polynomials,
Conjecture~\ref{conj:atom-pos-standard}.

By Remark~\ref{rem:for atom pos}(iv), if $f(x,Y)$ is an
$r$-nonsymmetric polynomial, then for any $\gamma \in \NN ^{r}$ and
$\lambda \in \NN ^{k}$, the coefficient $\langle \Acal_{(\gamma ;\,
\lambda ;\, 0^{n-k})}(x;y) \rangle\, f[x,y_{1}+\cdots +y_{n}]$ is
independent of permuting or adjoining zeroes to $\lambda $; hence it
is also independent of $n$ for $n\geq k$.  Taking $\lambda $ to be a
partition, the Weyl symmetrizations in the $y$ variables converge
strongly to a {\em stable atom}
\begin{equation}\label{e:stable-atom}
\Acal _{\gamma |\lambda }(x,Y) = \lim_{n\rightarrow \infty } \Weyl
_{y} \, \Acal_{(\gamma ;\, \lambda ;\, 0^{n-\ell (\lambda )})}(x;y),
\end{equation}
such that $\Acal _{\gamma |\lambda }[x,y_{1}+\cdots +y_{n}]$ is the
sum of all atoms $\Acal _{(\gamma ;\mu )}(x,y)$ with $\mu $ a
permutation of $(\lambda ;0^{n-\ell (\lambda )})$.  The stable atoms
form a basis of the algebra of $r$-nonsymmetric polynomials such that
the coefficients of any $f(x,Y)$ satisfy
\begin{equation}\label{e:stable-atom-coefs}
\langle \Acal _{\gamma |\lambda }(x,Y) \rangle\, f(x,Y) =
\langle \Acal_{(\gamma ;\,
\lambda ;\, 0^{n-\ell (\lambda )})}(x;y) \rangle\, f[x,y_{1}+\cdots +y_{n}]
\end{equation}
for all $n\geq \ell (\lambda )$.  For any $k$, $\Acal _{\gamma
|\lambda }(x,Y)$ is also given by the specialization
\begin{equation}\label{e:stable-atom=flagged-atom}
\Acal _{\gamma |\lambda }(x,Y) = \Acal _{(\gamma ;\, \lambda;\, 0^{k}
)}[x_{1},\ldots,x_{r},Y,0^{k+\ell (\lambda )-1}]
\end{equation}
of the flagged symmetric function $\Acal _{(\gamma ;\lambda ;0^{k}
)}(X_{1},\ldots,X_{r+\ell (\lambda )+k}) = \xi ^{-1} \Acal _{(\gamma
;\lambda ;0^{k} )}(x)$ from Proposition~\ref{prop:flagged-atom}.  To
see this, observe first that since $\Acal _{(\gamma ;\lambda ;0^{k}
)}(x)$ is a function of only the first $r+\ell (\lambda )$ variables,
$\xi ^{-1} \Acal _{(\gamma ;\lambda ;0^{k} )}(x)$ is a function of
only the first $r+\ell (\lambda )$ formal alphabets $X_{i}$, so the
right hand side of \eqref{e:stable-atom=flagged-atom} does not depend
on $k$.  Now the identity follows because
Corollary~\ref{cor:Weyl-on-flagged} implies that for any $n\geq \ell
(\lambda )$, we have
\begin{equation}\label{e:stable-atom-flagged-1}
\Weyl _{y}\, \Acal _{(\gamma ;\lambda ;0^{n-\ell (\lambda
)})}(x_{1},\ldots,x_{r}, y_{1},\ldots, y_{n}) = \Acal _{(\gamma
;\lambda ;0^{n-\ell (\lambda )})}[x_{1},\ldots,x_{r}, y_{1}+\cdots
+y_{n},0^{n-1}]\, .
\end{equation}

\begin{cor}\label{cor:stable-mac-atom-pos}
For $\gamma ,\eta \in \NN ^{r}$ and partitions $\lambda $, $\mu $, let
\begin{equation}\label{e:cor stable-atom-coefs}
K_{(\gamma |\lambda ),(\eta |\mu )}(q,t) \defeq \langle \Acal _{\gamma
|\lambda }(x,Y) \rangle\, \mE _{\eta |\mu }(x,Y;\, q,t)
\end{equation}
denote the stable atom coefficients of the modified
$r$-nonsymmetric Macdonald polynomials.  Then, as $n\rightarrow
\infty $, the atom coefficients
\begin{equation}\label{e:modMac-atom}
\langle \Acal _{(\gamma ;\, \lambda ;\, 0^{n-\ell (\lambda )} )}(x,y)
\rangle \, \Pi _{t,(x;y)}\, \Ecal _{(\eta ;0^{n};\mu )}(x_{1},\ldots,
x_{r}, y_{1},\ldots,y_{n}, 0^{\ell (\mu )};\, q,t)
\end{equation}
of the nonsymmetric plethysm of a specialized integral form
nonsymmetric Macdonald polynomial converge formally in $t$ to
$K_{(\gamma |\lambda ),(\eta |\mu )}(q,t)$ (more generally, the
coefficient of $\Acal _{(\gamma ;\, \nu ;\,  0^{n-\ell (\nu )})}(x,y) $
converges to $K_{(\gamma |\lambda ),(\eta |\mu )}(q,t)$ for any $\nu $
whose nonzero parts are a permutation of $\lambda $).

If Conjecture~\ref{conj:atom-pos-standard} is true, then we have {\em
stable atom positivity:} $K_{(\gamma |\lambda ),(\eta |\mu )}(q,t)\in
\NN [q,t]$.
\end{cor}

\begin{proof}
By Theorem~\ref{thm:Pit-symm}, for any $e$, the result of the
nonsymmetric plethysm in \eqref{e:modMac-atom} is symmetric (mod
$(t^{e})$) in all but the last $p$ of the variables $y_{i}$, for $p$
not depending on $n$.  This given, the first part of the corollary is
Theorem~\ref{thm:mod-mac}(a), stated in terms of atom coefficients.

Using Remark~\ref{rem:rearranging-lambda-2}(ii), the positivity
implication follows from Corollary~\ref{cor:atom-pos}(a) and
Theorem~\ref{thm:mod-mac}(b).
\end{proof}

For $\eta =(0^{r})$, if we set $y_{i} = x_{r+i}$, then $\mE
_{0^{r}|\mu }(x,Y;q,t) = \mE _{\emptyset |\mu }(X; q,t)$, independent
of $r$, and \cite[Theorem~5.2.1]{HagHaiLo08} implies that $\mE
_{\emptyset |\mu }(X; q,t)$ is the classical symmetric modified
Macdonald polynomial $H_{\mu }(X; q,t) = J_{\mu }[X /(1-t);q,t]$.
Hence, in this case, we have
\begin{equation}\label{e:K-lam-mu-classical}
K_{(\gamma |\lambda ),(0^{r} |\mu )}(q,t) = K_{(\gamma ;\lambda
)_{+}, \mu }(q,t),
\end{equation}
where $K_{\lambda ,\mu }(q,t)$ are the classical Kostka-Macdonald
coefficients.  Thus, Conjecture~\ref{conj:atom-pos-standard}, via its
consequence in Corollary~\ref{cor:stable-mac-atom-pos}, generalizes
the known positivity $K_{\lambda ,\mu }(q,t)\in \NN [q,t]$ from
\cite{Haiman01}.

While stable atom positivity of $\mE _{\eta |\mu }(x,Y;\, q,t)$
depends on Conjecture~\ref{conj:atom-pos-standard}, the following
weaker result, which follows immediately from
Theorem~\ref{thm:mod-mac}(b) and Remark~\ref{rem:monomial-positive},
already provides some evidence to suggest that our way of defining
modified $r$-nonsymmetric Macdonald polynomials is the right one.

\begin{cor}\label{cor}
The modified $r$-nonsymmetric Macdonald polynomials are monomial-positive,
i.e., the coefficients $\langle x^{\gamma }\, m_{\lambda }(Y) \rangle\,
\mE _{\eta |\mu }(x,Y;\, q,t)$ are in $\NN [q,t]$.
\end{cor}

\begin{example}\label{ex:K-stable}
One can calculate
\begin{equation}\label{e:K-stable-ex}
\mE _{(1)|(1)}(x,Y;q,t) = \Acal _{(1)|(1)}(x,Y) + t\, \Acal
_{(2)|\emptyset }(x,Y)\, ,
\end{equation}
giving $K_{((1)|(1)),\, ((1)|(1))}(q,t) = 1$ and $K_{((2)|(\emptyset
)),\, ((1)|(1))}(q,t) = t$.  Approximating these by the non-stabilized
atom coefficients in \eqref{e:modMac-atom}, one finds
\begin{multline}\label{e:K-approx-n=4}
\Pi _{t,(x;y)}\, \Ecal _{((1) ;0^{4};(1) )}(x,y, 0;\, q,t) \\
= (1-t^{3})(1-q\, t^{5}) \Acal _{1,1,0^{3}}(x,y) + t (1-t^{3})(1-q\,
t^{5})\, \Acal _{2,0^{4}}(x,y)+\cdots \, ,
\end{multline}
\begin{multline}\label{e:K-approx-n=5}
\Pi _{t,(x;y)}\, \Ecal _{((1) ;0^{5};(1) )}(x,y, 0;\, q,t)\\
 = (1-t^{4})(1-q\, t^{6}) \Acal _{1,1,0^{4}}(x,y) + t (1-t^{4})(1-q\,
t^{6})\, \Acal _{2,0^{5}}(x,y)+\cdots \, ,
\end{multline}
for $n=4$ and $n=5$ (we have omitted many other terms).  The pattern
continues similarly for larger $n$.  The behavior seen here is
typical: the coefficients in \eqref{e:modMac-atom} almost always
have non-positive terms involving powers of $t$ that grow with $n$.
\end{example}

\begin{example}
We compute $\mE_{(2)|(1)}(x,Y;q,t)$ using \eqref{e:mod-mac-HHL} and
Remark \ref{rem:mod-mac-HHL explicit}.  The set $R'((2,1))$ of
non-ragged-right augmented ribbon tuples is displayed below with flag
bound ($=$ flag number) $1$ written in the augmentation box of
$\nu^{(2)}$ (in general, the relevant flag boxes are the $r$
augmentation boxes in the ribbons $\nu^{(m)}, \dots, \nu^{(m+1-r)}$ of
lengths $\eta$), and the monomial $t^{n((\eta;\lambda)_{+})} \prod
_{\smalldomino } q^{a(u)+1}\, t^{l(u)}$ shown to the left of each
tuple.
\begin{equation}\label{e:H-HHL-ribbon-tuples}
\setlength{\arraycolsep}{0cm}
\begin{array}{cccc}
t \ \ &
\begin{tikzpicture}[scale=.45,baseline=.5cm]
  \foreach \beta / \alpha /  \z in {0/-1/0, 0/-2/2.2}
     \draw[yshift=\z cm] (\alpha,0)  grid (\beta, 1);
     \node at (.5,.5+2.2) {$1$};
      \end{tikzpicture}
& \qquad q\, t^{2} \ \ &
\begin{tikzpicture}[scale=.45,baseline=.5cm]
  \foreach \beta / \alpha /  \z in {0/-1/0, 0/-1/2.2, 0/-1/3.2}
     \draw[yshift=\z cm] (\alpha,0)  grid (\beta, 1);
     \node at (.5,.5+2.2)  {$1$};
      \end{tikzpicture}
\end{array}
\end{equation}
Infinitely many flagged tableaux with entries in $\Acal =\{1<2<\cdots
\}$ contribute to $\mE_{(2)|(1)}$, so for convenience in
displaying them below, we divide them into five groups determined by
the indicated inequalities, two on the first ribbon tuple above, and
three on the second.
\begin{equation*}
\begin{array}{rcccccccccc}
\raisebox{5mm}{$T$} &&
\hspace{-2ex}
\begin{tikzpicture}[scale=.45,baseline=0cm]
  \foreach \beta / \alpha /  \z in {0/-1/0, 0/-2/2.2}
     \draw[yshift=\z cm] (\alpha,0)  grid (\beta, 1);
     \node at (-1.5,.5+2.2) {1};  \node at (-.5,.5+2.2) {1};  \node at (.5,.5+2.2) {1};
   \node at (-.5,.5) {1};
      \end{tikzpicture}
&&
\begin{tikzpicture}[scale=.45,baseline=0cm]
  \foreach \beta / \alpha /  \z in {0/-1/0, 0/-2/2.2}
     \draw[yshift=\z cm] (\alpha,0)  grid (\beta, 1);
     \node at (-1.5,.5+2.2) {1};  \node at (-.5,.5+2.2) {1};   \node at (.5,.5+2.2) {1};
   \node at (-.5,.5) {$a$};
      \end{tikzpicture}
&&
\begin{tikzpicture}[scale=.45,baseline=0cm]
  \foreach \beta / \alpha /  \z in {0/-1/0, 0/-1/2.2, 0/-1/3.2}
     \draw[yshift=\z cm] (\alpha,0)  grid (\beta, 1);
     \node at (-.5,.5+1+2.2) {$a$};  \node at (-.5,.5+2.2) {1};   \node at (.5,.5+2.2) {1};
   \node at (-.5,.5) {1};
      \end{tikzpicture}
&&
\begin{tikzpicture}[scale=.45,baseline=0cm]
  \foreach \beta / \alpha /  \z in {0/-1/0, 0/-1/2.2, 0/-1/3.2}
     \draw[yshift=\z cm] (\alpha,0)  grid (\beta, 1);
     \node at (-.5,.5+1+2.2) {$a$};  \node at (-.5,.5+2.2) {1};   \node at (.5,.5+2.2) {1};
   \node at (-.5,.5) {$b$};
      \end{tikzpicture}
&&
\begin{tikzpicture}[scale=.45,baseline=0cm]
  \foreach \beta / \alpha /  \z in {0/-1/0, 0/-1/2.2, 0/-1/3.2}
     \draw[yshift=\z cm] (\alpha,0)  grid (\beta, 1);
     \node at (-.5,.5+1+2.2) {$b$};  \node at (-.5,.5+2.2) {1};   \node at (.5,.5+2.2) {1};
   \node at (-.5,.5) {$a$};
      \end{tikzpicture}
  \\[1mm]
& &&& a > 1 && a > 1 && b \ge a > 1 && b > a > 1\\[3mm]
t^{-\inv(T)} && 1  && t^{-1} && t^{-1} && t^{-1} && t^{-2} \\[2mm]
\mE_{(2)|(1)}(x,Y;q,t) &=& t \, x_1^3  & \hspace{-1ex}+& x_1^2
s_1(Y) &+& qt \, x_1^2 s_1(Y) &+& qt \, x_1 h_2(Y) &+&
q \, x_1 e_2(Y)
\end{array}
\end{equation*}
It so happens that each term above is a stable atom (although
organizing the tableaux by inequalities would not usually have this
effect), yielding the stable atom expansion
\begin{equation}\label{e:stable-atom-expandsion}
\mE_{(2)|(1)}(x,Y;\, q,t) = t \, \Acal_{(3)|\emptyset} +
\Acal_{(2)|(1)} + qt \, \Acal_{(2)| (1)} + qt \,\Acal_{(1)| (2)} + q
\, \Acal_{(1)|(1,1)}.
\end{equation}
\end{example}

\subsection{Weyl and Hecke Symmetrization}
\label{ss:Weyl-and-Hecke-symmetrization}

We define a partial Weyl symmetrization operator $\Weyl_{x_r, Y}$ on
$r$-nonsymmetric polynomials by
\begin{align}
\label{e Weylr strong converge} \Weyl_{x_r,Y} f(x,Y) = \lim_{n \to
\infty} \Weyl_{x_r, y_1,\dots,y_n} f[x, y_1+ \dots + y_n],
\end{align}
where $x = x_1,\dots, x_r$ and $\Weyl_{x_r, y_1,\dots, y_n}$ is the
Weyl symmetrization operator in the variables $x_r, y_1,\dots, y_n$.
The sequence on the right hand side converges strongly to a
well-defined limit, as can be verified by reducing to the case that
$r=1$ and $f(x_{1},Y) = x_{1}^{k}\, g[x_{1}+Y]$, where $g$ is a
symmetric function.  Similarly, we define the full Weyl symmetrization
operator $\Weyl$ on $r$-nonsymmetric polynomials by
\begin{align}
\label{e Weylr strong converge 2} \Weyl f(x,Y) = \lim_{n \to \infty}
\Weyl_{x_1,\dots, x_r, y_1,\dots,y_n} f[x, y_1+ \dots + y_n],
\end{align}
where again we have strong convergence to the limit.

\begin{lemma}\label{l Weyl on flag sym}
For any flagged symmetric function $f(X_1, \dots, X_{l})$,
\begin{align}\label{e Weyl on flag sym}
\Weyl_{x_r, Y} f[x_1,\dots, x_r, Y, 0^k]& = f[x_1,\ldots, x_{r-1},
x_r+ Y, 0^{k+1}],\\
\label{e Weyl on flag sym2} \Weyl f[x_1,\dots, x_r, Y, 0^k] &= f[x_1+
\cdots + x_r+ Y, 0^{r+k}],
\end{align}
where  $k = l-r-1$.
\end{lemma}
\begin{proof}
The first identity follows from combining \eqref{e Weylr strong
converge}, Corollary \ref{cor:Weyl-on-flagged}, and the fact that
$\lim_{n \to \infty} f[x_1,\dots, x_r+ y_1+\cdots + y_n, 0^{k+1}]$
converges strongly to $f[x_1,\dots, x_r+ Y, 0^{k+1}]$.  The second
identity is proved similarly.
\end{proof}

For stable atoms, this implies the following analog of
\eqref{e:for atom pos 1} and Remark \ref{rem:for atom pos}(ii).
\begin{lemma}
\label{l stable atoms}
For  $\gamma \in \NN ^{r}$ and partition $\lambda $,
\begin{align}
\label{e:Weyl stable atom}
\Weyl \Acal_{\gamma | \lambda}(x,Y)
=
\begin{cases}
s_{(\gamma;\lambda)}[x_{1}+\cdots +x_{r}+Y] & \text{if $\gamma_1 \ge
\cdots \ge \gamma_r \ge
\lambda_1$,} \\
0 & \text{ otherwise}.
\end{cases}
\end{align}
Hence, for any $r$-nonsymmetric polynomial $f(x,Y)$, the coefficient
of $s _{\mu}[x+Y]$ in the Schur function expansion of $\Weyl f(x,Y)$
is equal to the coefficient of the stable atom $\Acal_{(\mu_1,\dots,
\mu_r)| (\mu_{r+1}, \dots, \mu_k)}(x, Y)$ in the stable atom expansion
of $f(x,Y)$ (if $k<r$, this means $\Acal_{(\mu
;0^{r-k}) | \emptyset }(x, Y)$).
\end{lemma}

\begin{proof}
By \eqref{e:stable-atom=flagged-atom} and \eqref{e Weyl on flag sym2},
$\Weyl\, \Acal_{\gamma | \lambda}(x,Y) = \Acal _{(\gamma ;\lambda
)}[x_{1}+\cdots +x_{r}+Y,0,\ldots,0]$.  For $Y = y_{1}+\cdots +y_{n}$,
this and the right hand side of \eqref{e:Weyl stable atom} are both
equal to $\Weyl _{x;y}\, \Acal _{(\gamma ;\lambda)}(x;y)$.  Given
\eqref{e:Weyl stable atom}, the rest is clear.
\end{proof}

To distinguish between $r$-nonsymmetric and $(r-1)$-nonsymmetric
polynomials in the next proposition, we use the notation $Y_{r}$,
$Y_{r-1}$ as in Proposition~\ref{prop:limits-of-symmetrizers} in place
of $Y$.

\begin{prop}\label{p:Weyl modnsmac}
Partial Weyl symmetrization acts in the following simple way on
modified $r$-nonsymmetric Macdonald polynomials: for any $\eta \in
\NN^r$ and partition $\lambda$,
\begin{equation}\label{e Weyl modnsmac}
\Weyl_{x_{r},Y_{r}}\,  \mE_{\eta | \lambda}(x_{1},\ldots,x_{r},
Y_{r};q,t) = \mE_{\etahat\,  |\,  (\eta_r;\lambda)_+}(x_1,\dots, x_{r-1},
Y_{r-1};q,t)\, ,
\end{equation}
where $\etahat  = (\eta _{1},\ldots,\eta _{r-1})$,
$Y_{r} = x_{r+1},x_{r+2},\dots $, and $Y_{r-1} = x_{r}, x_{r+1}, \ldots$.
Iterating this, we also have
\begin{align}
\label{e Weyl modnsmac2} \Weyl \, \mE_{\eta | \lambda}(x_1,\dots, x_r,
Y_{r}; q,t) = \mE_{\emptyset | (\eta ;\lambda)_+}(X;q,t) =
H_{(\eta;\lambda)_+}(X;q,t).
\end{align}
where $X = x_{1},x_{2},\ldots $.
\end{prop}

\begin{proof}
Applying $\Weyl_{x_{r}, Y_{r}}$ to both sides of \eqref{e:mod-mac-HHL}
and using \eqref{e Weyl on flag sym} in each
term yields
\begin{multline}\label{e Weyl LLT}
\Weyl_{x_r, Y_{r}}\,  \mE_{\eta | \lambda}(x, Y_{r};\, q,t) \\
= t^{n((\eta ;\lambda ) _{+})} \hspace{-2ex}\sum_{(\nubold ,\sigma
)\in R'(\eta;\lambda)} \bigl( \prod \nolimits_{\domino } q^{a(u)+1}\,
t^{l(u)} \bigr)\, \Gcal _{\nubold ,\sigma }[x_{1}, \ldots,x_{r-1},
Y_{r-1}, 0,\ldots,0;\, t^{-1}].
\end{multline}
The right hand side is the generalized version of formula
\eqref{e:mod-mac-HHL} in Remark~\ref{rem:rearranging-lambda-2}(i), for
$\mE_{\etahat\, |\, (\eta_r;\lambda)_+}(x_1, \dots, x_{r-1}, Y_{r-1}
;q, t)$, with $\zeta =(\eta _{r};\lambda )$.
\end{proof}

\begin{remark}
\label{rem:K-lam-mu-classical}
We can now see that the conjectured stable atom positivity of
$\mE_{\eta|\mu }(x,Y;q,t)$ generalizes Macdonald positivity in a sense
stronger than in \eqref{e:K-lam-mu-classical}: namely, equation
\eqref{e Weyl modnsmac2} and Lemma \ref{l stable atoms} imply that
$K_{(\gamma |\lambda ),(\eta |\mu )}(q,t) = K_{(\gamma;\lambda),(\eta;
\mu )_+}(q,t)$ whenever $(\gamma;\lambda)$ is a partition.  Thus, not
only are the classical Kostka-Macdonald coefficients special cases of
$K_{(\gamma |\lambda ),(\eta |\mu )}(q,t)$ when $\eta =(0^{r})$, but for
every pair $(\eta |\mu )$, the Kostka-Macdonald
coefficients $K_{\lambda ,(\eta |\mu )_{+}}(q,t)$ for all $\lambda $
are special cases of
stable atom coefficients of
$\mE_{\eta|\mu }(x,Y;q,t)$.
\end{remark}

\begin{example}\label{ex:}
Applying the Weyl symmetrizer $\Weyl$ to
\eqref{e:stable-atom-expandsion} we obtain
\begin{equation}
 t \, s_{3} + s_{21} + qt \, s_{21} + 0 + q \, s_{111} =
\mE_{\emptyset|(2,1)}(X;q,t) = H_{21}(X;q,t),
\end{equation}
illustrating Proposition \ref{p:Weyl modnsmac}.
\end{example}

Proposition~\ref{prop:limits-of-symmetrizers} implies that
Proposition~\ref{p:Weyl modnsmac} has an analog for the unmodified
$r$-nonsymmetric integral forms $\Jns _{\eta |\lambda }$, in which
Hecke symmetrization replaces Weyl symmetrization.

Let $\hsym _{z_{1},\ldots,z_{l}}$ denote the normalized Hecke
symmetrizer $\hsym _{(l)}$ in \eqref{e:normalized-Hecke-symmetrizer}
for the full permutation group $\Sfrak _{l}$, in variables
$z_{1},\ldots,z_{l}$.  It is shown in \cite[\S
4.3]{BechtloffWeising23} that there is a well-defined partial Hecke
symmetrization operator $\hsym_{x_r, Y}$ on $r$-nonsymmetric
polynomials given by the $t$-adically convergent limit
\begin{equation}\label{e H sym converge}
\hsym_{x_r,Y} f(x,Y) = \lim_{n \to
\infty} \hsym_{x_r, y_1,\dots,y_n} f[x, y_1+ \dots + y_n].
\end{equation}
Similarly, there is a full Hecke symmetrization operator $\hsym $ on
$r$-nonsymmetric polynomials given by the $t$-adically convergent
limit
\begin{equation}\label{e H sym full}
\hsym f(x,Y) = \lim_{n \to \infty} \hsym_{x_1,\dots, x_r,
y_1,\dots,y_n} f[x, y_1+ \dots + y_n]\, .
\end{equation}

The following result is similar to \cite[Lemma 4.28 and Proposition
5.1]{BechtloffWeising23}. In particular, \cite[Lemma
4.28]{BechtloffWeising23} shows that \eqref{e Hsym J2} below holds up
to a scalar factor.  However, notably, the scalar factor goes away
when working with the integral forms.

\begin{prop}\label{p:Hsym J}
Partial Hecke symmetrization acts in the following simple way on
the right stabilized integral form $r$-nonsymmetric Macdonald
polynomials: for any $\eta \in
\NN^r$ and partition $\lambda$,
\begin{equation}\label{e Hsym J}
\hsym_{x_{r},Y_{r}}\,  \Jns_{\eta | \lambda}(x_{1},\ldots,x_{r},
Y_{r};q,t) = \Jns_{\etahat\,  |\,  (\eta_r,\lambda)_+}(x_1,\dots, x_{r-1},
Y_{r-1};q,t)\, ,
\end{equation}
where $\etahat  = (\eta _{1},\ldots,\eta _{r-1})$,
$Y_{r} = x_{r+1},x_{r+2},\dots $, and $Y_{r-1} = x_{r}, x_{r+1}, \ldots$.
Iterating this, we also have
\begin{align}\label{e Hsym J2} \hsym \, \Jns_{\eta | \lambda}(x_1,\dots, x_r,
Y_{r}; q,t) = \Jns_{\emptyset | (\eta ;\lambda)_+}(X;q,t) =
\Jns_{(\eta;\lambda)_+}(X;q,t),
\end{align}
where $X = x_{1},x_{2},\ldots $.
\end{prop}

\begin{proof}
To prove \eqref{e Hsym J} and, consequently, the first equality in
\eqref{e Hsym J2}, let $n\rightarrow \infty $ in
Proposition~\ref{prop:limits-of-symmetrizers}(a).  The second equality
in \eqref{e Hsym J2} follows from the discussion before
\eqref{e:K-lam-mu-classical} and the fact that for $r=0$, $\Pisf_r$ is
the ordinary plethystic transformation $f[X] \mapsto f[X/(1-t)]$.
\end{proof}

\subsection{Left stable \texorpdfstring{$r$}{r}-nonsymmetric Macdonald polynomials}
\label{ss:left-stabilization}

We conclude with a brief discussion of how left stabilized Macdonald
polynomials fit into the context of our work above.

\begin{defn}\label{def:left-stable}
The {\em left stabilized $r$-nonsymmetric integral form Macdonald
polynomial} indexed by a partition $\lambda $ and $\kappa \in \NN
^{r}$ is the case
\begin{equation}\label{e:left-stable}
\Jns _{\emptyset | \lambda | \kappa }(Y,z;q,t)\in \Lambda _{\kk
}(Y)\otimes \kk [z_{1},\ldots,z_{r}]
\end{equation}
of the partially symmetric integral form Macdonald polynomial in
Definition~\ref{def:symm-int-form}.  Note that since $r=0$ and $\eta
=\emptyset $, there are no $x$ variables in this case.
\end{defn}

Restating the definition gives the following formula, which also holds
with $\mu =(\zeta ;\kappa )$ in place of $(\lambda ;\kappa )$ for any
weak composition $\zeta $ such that $\zeta _{+} = \lambda $, as in
Remark~\ref{rem:rearranging-lambda}.

\begin{prop}\label{prop:left-stable-HHL}
For $\lambda $ a partition and $\kappa \in \NN ^{r}$, we have
\begin{multline}\label{e:left-stable-HHL}
\Jns _{\emptyset | \lambda | \kappa }(Y,z;q,t)\\
 = t^{n((\lambda ;\kappa ) _{+})} \hspace{-2ex} \sum _{(\nubold
,\sigma )\in R'(\mu )} \bigl( \prod \nolimits_{\domino } q^{a(u)+1}\,
t^{l(u)} \bigr) \Gcal _{\nubold ,\sigma }^{-}(Y, 0^{\ell (\lambda )},
z_{1}, \ldots, z_{r}, 0, \ldots,0;\, t^{-1}),
\end{multline}
where $R'(\mu )$ is the same as $R(\mu )$ in \eqref{e:ns-HHL}, with
$\mu =(\lambda ;\kappa )\in \NN ^{m}$ and $m=r+\ell(\lambda)$, except
that $R'(\mu )$ only contains pairs $(\nubold ,\sigma )$ in which the
last $\ell (\lambda )$ augmented ribbons, corresponding to the parts
of $\lambda $, are non-ragged-right, i.e., no augmentation box in
these ribbons is in a vertical domino.
\end{prop}

As noted in Remark~\ref{rem:Goodberry}, we have specialization
and symmetrization formulas
\begin{align}\label{e:left-stable-spec}
\Jns _{\emptyset | \lambda | \kappa }(Y,z; q,t) & = \lim_{n\rightarrow
\infty } \Ecal _{(0^{n};\, \lambda ;\, \kappa
)}(y_{1},\ldots,y_{n},0^{\ell
(\lambda )},z_{1},\ldots,z_{r}; q,t)\\
\label{e:left-stable-symm}
& = \Pi _{\emptyset |\lambda |\kappa }(q,t)\, \lim_{n\rightarrow
\infty } \bigl( \sum _{w \in \Sfrak _{[1,n]}/ \Sfrak _{\sS }
\hspace{-3ex}} T_{w}\, E _{(\lambda ;\, 0^{n-\ell (\lambda )};\,
\kappa )}(x;q,t)\, \bigr) \bigr|_{x\rightarrow (y_{1},\ldots,y_{n},
z_{1},\ldots,z_{r})}\, ,
\end{align}
with both limits converging strongly.  The normalized
symmetrization formula
\begin{equation}\label{e:left-stable-normal}
\Jns _{\emptyset | \lambda | \kappa }(Y,z; q,t) = \lim_{n\rightarrow
\infty } \bigl( \hsym _{[1,n]}\, \Ecal _{(\lambda ;\, 0^{n-\ell
(\lambda )};\, \kappa )}(x;q,t) \bigr) \bigr|_{x\rightarrow
(y_{1},\ldots,y_{n}, z_{1},\ldots,z_{r})}\,
\end{equation}
also holds, but only converges $t$-adically.

The next proposition shows that, up to minor adjustments, the left
stabilized forms reduce to the cases $\eta \in (\ZZ _{>0})^{r}$ of the
right stabilized forms $\Jns _{\eta |\lambda }(x,Y;q,t)$.

\begin{prop}\label{prop:stable-rotation}
The left stabilized integral form $r$-nonsymmetric Macdonald
polynomials are given in terms of the right stabilized ones by the
identity
\begin{equation}\label{e:stable-rotation}
q^{|\kappa |+r} \, z_{1}\cdots z_{r}\, \Jns _{\emptyset |\lambda
|\kappa }(Y, z_{1}, \ldots, z_{r};q,t) = \Jns _{\kappa +(1^{r})
|\lambda }(q\, z_{1},\ldots,q\, z_{r},Y;q,t)
\end{equation}
\end{prop}

\begin{proof}
Immediate from Proposition~\ref{prop:rotation-int},
\eqref{e:right-stable-spec}, and \eqref{e:left-stable-spec}.
\end{proof}

It follows that the left stabilized forms can be transformed into the
modified forms using a $q$-shifted variant of $\Pisf _{r}$ defined by
\begin{equation}\label{e:q-Pir}
\Pisf '_{r} \bigl(f[Y+q(z_{1}+\cdots +z_{r})]\,  g(z) \bigr) = 
f\left[\frac{Y+q(z_{1}+\cdots +z_{r})}{1-t}\right]\, \Pi _{t,z} g(z).
\end{equation}

\begin{cor}\label{cor:stable-rot}
With $\Pisf '_{r}$ as above, We have the identity
\begin{equation}\label{e:stable-rot}
\Pisf '_{r} \bigl( z_{1}\cdots z_{r}\, \Jns _{\emptyset |\lambda
|\kappa }(Y,z;q,t) \bigr) = q^{-|\kappa |-r} \mE _{\kappa
+(1^{r})|\lambda }(q \, z_{1},\ldots,q\, z_{r}, Y; q,t).
\end{equation}
\end{cor}

Thus, if Conjecture~\ref{conj:atom-pos-standard} holds, the left
stabilized polynomials $\Jns _{\emptyset |\lambda |\kappa }(Y,z ;q,t)$,
transformed as in \eqref{e:stable-rot}, become positive in terms of
stable atoms $\Acal _{\lambda }(q\, z,Y)$.

In the left stabilized case, the more straightforward limit of
nonsymmetric plethysms $\lim_{n\rightarrow \infty } \, \Pi _{t,(y;z)}\, 
\Jns _{\emptyset |\lambda |\kappa }[y_{1}+\cdots +y_{n},
z_{1},\ldots,z_{r}; q,t]$ is not positive, as the following simple
counterexample shows.

\begin{example}\label{ex:left-is-bad}
One calculates $\Ecal _{(0^{n},1)}(y_{1},\ldots,y_{n},z_{1};q,t) =
(1-t)(y_{1}+\cdots +y_{n})+(1-q\, t) z_{1}$, giving $\Jns _{\emptyset
|\emptyset |(1)}(Y,z_{1};q,t) = (1-t)e_{1}(Y)+(1-q\, t)z_{1}$.
Corollary~\ref{cor:other-limits2}
then gives the formally convergent limit
\begin{equation}\label{e:left-is-bad}
\lim_{n\rightarrow \infty }\, \Pi _{t,(y_{1},\ldots,y_{n},z_{1})} \,
\Jns _{\emptyset |\emptyset |(1)}[y_{1}+\cdots +y_{n}, z_{1}; q,t] =
e_{1}(Y) +(1-q\, t) z_{1}\, .
\end{equation}
Since this is not monomial positive, there can be no formula like
\eqref{e:mod-mac-HHL} for $\lim_{n\rightarrow \infty } \, \Pi
_{t,(y;z)}\, \Jns _{\emptyset |\lambda |\kappa }[y_{1}+\cdots +y_{n},
z_{1},\ldots,z_{r}; q,t]$ as a positive combination of unsigned
flagged LLT polynomials.
\end{example}

\begin{remark}\label{rem:obstacle}
As noted in Remark~\ref{rem:why-stabilize}, an obstacle to using
nonsymmetric plethysm to convert the flagged LLT polynomials in
\eqref{e:left-stable-HHL} from signed to unsigned is that some
variables set equal to zero do not immediately follow the formal
alphabet $Y$.  Example~\ref{ex:left-is-bad} shows that this obstacle
cannot be removed by some alternative argument.
\end{remark}

\bibliographystyle{amsplain}
\bibliography{references}

\end{document}